\documentclass[12pt,oneside,reqno]{amsart}

\usepackage[T1]{fontenc}
\usepackage[utf8]{inputenc}
\usepackage{amsmath}
\usepackage{lipsum}
\usepackage{amssymb}
\usepackage{amsfonts}
\usepackage{enumitem}
\usepackage{stackengine}
\usepackage{graphicx}
\usepackage{calc}

\usepackage{hyperref}
\usepackage{ragged2e}
\usepackage{bbm}
\usepackage{xcolor}
\hypersetup{colorlinks=true}
\usepackage[a4paper, left=2.5cm, right=2.5cm, top=2.5cm, bottom=2.5cm]{geometry}
\usepackage{amsthm}
\usepackage{indentfirst}
\makeatletter
\DeclareMathAlphabet{\mathcal}{OMS}{zplm}{m}{n}

\newcommand{\Ima}{\mathrm{Im}}
\newcommand{\tmod}{\mathrm{mod}}
\newcommand{\x}{\mathbf{x}}
\newcommand{\bd}{\mathbf{d}}

\newcommand{\p}{\mathbf{p}}
\newcommand{\y}{\mathbf{y}}
\newcommand{\bb}{\mathbf{b}}
\newcommand{\z}{\mathbf{z}}
\newcommand{\ba}{\mathbf{a}}
\newcommand{\bfb}{\mathbf{b}}
\newcommand{\R}{\mathbb{R}}
\newcommand{\T}{\mathbb{T}}
\newcommand{\bC}{\mathbb{C}}

\newcommand{\Q}{\mathbb{Q}}
\newcommand{\N}{\mathbb{N}}

\newcommand{\D}{\mathbb{D}}

\newcommand{\Z}{\mathbb{Z}}

\newcommand{\cB}{\mathcal{B}}
\newcommand{\cM}{\mathcal{M}}
\newcommand{\cH}{\mathcal{H}}
\newcommand{\cL}{\mathcal{L}}

\newcommand{\cE}{\mathcal{E}}

\newcommand{\cF}{\mathcal{F}}
\newcommand{\cK}{\mathcal{K}}
\newcommand{\cG}{\mathcal{G}}
\newcommand{\cC}{\mathcal{C}}
\newcommand{\cO}{\mathcal{O}}

\newcommand{\cT}{\mathcal{T}}
\newcommand{\cJ}{\mathcal{J}}
\newcommand{\cI}{\mathcal{I}}
\newcommand{\cY}{\mathcal{Y}}

\newcommand{\cP}{\mathcal{P}}
\newcommand{\cS}{\mathcal{S}}
\newcommand{\cA}{\mathcal{A}}
\newcommand{\cX}{\mathcal{X}}

\newcommand{\td}{\mathrm{d}}

\newcommand{\ubar}[1]{\text{\underbar{$#1$}}}
\newcommand\barbelow[1]{\stackunder[1.6pt]{$#1$}{\rule{1.2ex}{.1ex}}}

\newcommand{\dimH}{\dim_{\mathrm{H}}}
\newcommand{\dist}{\mathrm{dist}}
\newcommand{\norm}[1]{{\left\Vert #1
		\right\Vert}}
\newcommand{\linte}[1]{{\left\lfloor #1
		\right\rfloor}}

\newcommand{\one}{\mathbbm{1}}
\newcommand{\spt}{\mathrm{spt}}
\newcommand{\rmi}{\mathrm{i}}

\theoremstyle{plain}

\newtheorem{theorem}{Theorem}
\newtheorem{lemma}{Lemma}
\newtheorem{proposition}{Proposition}

\newtheorem{corollary}{Corollary}

\newtheorem*{claim*}{Claim}

\theoremstyle{definition}
\newtheorem{definition}{Definition}
\newtheorem{example}{Example}

\newtheorem*{examples*}{Examples}
\newtheorem*{example*}{Example}
\newtheorem{question}{Question}

\newtheorem*{notations*}{Notations}
\newtheorem*{notation*}{Notation}

\theoremstyle{remark}
\newtheorem{remark}{{Remark}}

\numberwithin{equation}{section}
\numberwithin{question}{section}
\numberwithin{theorem}{section}
\numberwithin{thm}{section}
\numberwithin{lemma}{section}
\numberwithin{proposition}{section}
\numberwithin{cor}{section}
\numberwithin{corollary}{section}
\numberwithin{claim}{section}
\numberwithin{definition}{section}
\numberwithin{example}{section}
\numberwithin{remark}{section}
\numberwithin{notations}{section}
\numberwithin{notation}{section}
\numberwithin{claim}{section}
\numberwithin{problem}{section}
\numberwithin{figure}{section}

\allowdisplaybreaks
\newcommand{\delete}[1]{{\red{\sout{#1}}}}
\renewcommand{\emph}[1]{{\it#1}}

\newcommand{\red}[1]{\textcolor[rgb]{1.00,0.00,0.00}{#1}}
\newcommand{\sv}[1]{{\color{blue}   {#1}}}

\newcommand{\slv}[1]{{\color{teal}  {#1}}}
\newcommand{\junjie}[1]{{\color{teal}  {#1}}}

\newtheorem{thmA}{{\rm \textbf{Theorem~A}}  \!\!\!\!\!\!}

\newtheorem{claimf}{{\rm \textbf{Claim~F}}  \!\!\!\!\!\!}

\title{Exponential mixing  for  Gibbs measures on self-conformal sets and  applications}

\author[J. Huang]{Junjie Huang}
\address[J. Huang]{School of Mathematics, South China University of Technology, Wushan Road 381, Tianhe District, Guangzhou, China}
\email{h1135778478@163.com}

\author[B. Li]{Bing Li}
\address[B. Li]{School of Mathematics, South China University of Technology,  Wushan Road 381, Tianhe District, Guangzhou, China}
\email{scbingli@scut.edu.cn}

\author[S. Velani]{Sanju Velani}
\address[S. Velani]{Department of Mathematics,
University of York, Heslington, York, YO10
5DD, England.}
\email{sanju.velani@york.ac.uk}

\begin{document}
\begin{abstract}
     In this paper, we show that  Gibbs measures on self-conformal sets generated by  a $C^{1+\alpha}$ conformal IFS on $\R^d$ satisfying the OSC are exponentially mixing. We exploit this to obtain essentially sharp asymptotic counting statements for the  recurrent  and  the  shrinking target subsets  associated with any such set. In particular, we provide explicit examples of dynamical systems for which the  recurrent sets  exhibit (unexpected) behaviour that is not present in the shrinking target setup.
In the process of establishing our exponential mixing result we extend Mattila’s rigidity theorem for self-similar sets to self-conformal sets without any separation condition and for arbitrary Gibbs measures.

 \end{abstract}
	\maketitle
	\section{Introduction}
\subsection{Background and motivation }   \label{bandi}
 Let $(X,\cB,\mu)$ be a probability space and let $T:X\to X$ be a measurable map. We say that $(X,\cB,\mu,T)$ is a \emph{measure-preserving system} if $\mu$ is $T$-invariant in the sense that $\mu(T^{-1}E)=\mu(E)$ for all $E\in\cB$. Within the area of ergodic theory and dynamical systems,  it is widely recognized  that ``mixing'' in its various forms  is a highly sought-after trait for a system to exhibit.  In particular, the ``rate'' of mixing is closely linked to fundamental properties of the system; such as ergodicity. To some extent, an exponential rate  represents the most desirable form of mixing and this exponential form of mixing together with its applications will be the main focus of this paper.  We start by defining the weaker notion  of $\Sigma$-mixing introduced in \cite{li2023} and consider some of its consequences.
 \begin{definition}\label{defmix}
     Let $(X,\cB,\mu, T)$ be a measure-preserving system. Let $\cC$ be a collection of measurable subsets of $X$. For any integer $n\geq1$, define
     \begin{eqnarray*}\label{phin}
         \phi(n):=\sup\left\{\Big|\frac{\mu(E\cap T^{-n}F)}{\mu(F)}-\mu(E)\Big|:E\in\cC,~F\in\cB,~\mu(F)>0\right\}.
     \end{eqnarray*}
     We say that $\mu$ is \emph{$\Sigma$-mixing (short for summable-mixing) with respect to $(T,\cC)$} if the sum $\sum_{n=1}^{\infty}\phi(n)$ converges. In particular, we say that $\mu$ is \emph{exponentially mixing with respect to $(T,\cC)$} if there exists $\gamma\in (0,1)$ such that $\phi(n)=O(\gamma^n)$.
 \end{definition}
One  powerful consequence of summable-mixing  is that if $\{E_n \}_{n \in \mathbb{N}} $ is a sequence of subsets in $\mathcal{C}$ and $\mu$ is $\Sigma$-mixing with respect to $(T,\cC)$, then in the language of probability theory,  the sets $$A_n:=T^{-n}E_n   \qquad (n \in \mathbb{N}) $$  are essentially  pairwise independent on average.  More precisely,  for any pair of positive integers $a < b$, it can be verified that
\begin{eqnarray}\label{qionav}
\sum_{a\le m, n\le b}\mu\left(A_m\cap A_n\right)
& \le   &  \left(\sum_{a\le n\le b} \!\! \mu(A_n) \right)^2    \ +  \ (2\kappa +1) \sum_{a\le   n \le   b }\mu(A_n) \  \,
\end{eqnarray}
where  $ \kappa := \sum_{n=1}^\infty\phi(n)$.  Note that  without the second term on the right hand side of \eqref{qionav}, we would have full independence as in the classical sense.  Furthermore,  observe that  $\mu(A_n) = \mu(E_n) $ since $T$ is measure preserving and if  the sum $\sum_{n \in \N }\mu(A_n)$ diverges (the interesting case),  then the first term dominates  and the second `error' term  is, up to constants,  the square root of the main term.   Independence is of course a well known fundamental notion in probability theory, as in statistics and the theory of stochastic processes.

With essentially full  independence as \eqref{qionav} at our disposal, it is relatively straightforward to exploit the quantitative form of the Borel-Cantelli Lemma (see Section~\ref{CRR}: Lemma~\ref{countlem}\footnote{We will also present a more versatile form of the familiar Lemma~\ref{countlem}.  The more versatile Lemma~\ref{genharmanlem} is required  for establishing one of our ``application'' results in Section~\ref{apptoprove} and is  potentially  of independent interest.}) to obtain the following  elegant and essentially sharp  asymptotic statement concerning the counting function
    \smallskip
\begin{equation} \label{countdef}
R(x,N) := \# \big\{ 1\le n \le   N :    x \in A_n \} \,  \qquad \ \ \  ( x \in X \, , N \in \N).
\end{equation}
Obviously,  by definition,  \eqref{countdef} is equivalent to $R(x,N) = \# \big\{ 1\le n \le   N :    T^n(x) \in E_n \}$.


\smallskip

\begin{thmA}\label{gencountthm}
 \emph{Let $(X,\mathcal{B},\mu,T)$ be a measure-preserving dynamical system  and $\mathcal{C}$ be a collection of measurable subsets of $X$. Suppose that $\mu$ is $\Sigma$-mixing with respect to $(T,\mathcal{C})$ and let  $\{E_n \}_{n \in \mathbb{N}} $ be a sequence of subsets  in $\mathcal{C}$.   Then, for any given $\varepsilon>0$, we have that
\begin{equation} \label{def_Psi}R(x,N)=\Psi(N)+O\left(\Psi^{1/2}(N) \ (\log\Psi(N))^{3/2+\varepsilon}\right)
\end{equation}
\noindent for $\mu$-almost all $x\in X$, where
$\Psi(N):=\sum_{n=1}^N \mu(A_n) \,  .
$
}
\end{thmA}

 The details of deriving  \eqref{qionav} and in turn Theorem~A from the notion of summable mixing,  can be found in \cite[Section~2]{li2023}.   The upshot of Theorem~A is that if the rate of mixing is summable  and   the measure sum $\Psi(N)$ diverges (as $N \to \infty$),  then for  $\mu$--almost all $x \in X$  the points $x$ lie in the sets $A_n$, or equivalently  the orbits $\{T^nx\}_{n\in\N}$ hit  the target sets $E_n$,  the `expected' number of times. In particular,
 $\lim_{N \to \infty} R(x,N) = \infty  $  for $\mu$--almost all $x \in X$  and so a simple consequence of the theorem is  the following zero-full measure criterion for the associated $\limsup$ set $A_\infty :=\limsup_{n \to \infty} A_n $:
 \begin{eqnarray}\label{appdiv}
		\mu\big(A_\infty  \big)=
		\begin{cases}
			0 &\text{if}\ \  \sum_{n=1}^\infty \mu\big(A_n\big)<\infty\\[2ex]
			1 &\text{if}\ \ \sum_{n=1}^\infty \mu\big(A_n\big)=\infty.
		\end{cases}
	\end{eqnarray}
For some readers, this ``corollary'' is maybe    more familiar than its stronger quantitative form. It resembles,  for example, the standard
\begin{itemize}
  \item  Borel-Cantelli Lemma in probability theory   \cite{bvbc,bill},\vspace*{1ex}
  \item   measure characterization  of shrinking target and recurrent sets in ergodic theory and dynamical systems (see Section~\ref{appintro} below),  and\vspace*{1ex}
  \item   Khintchine-type theorems in metric number theory \cite{BRV2016,harman1998,khintchine1924}.
\end{itemize}
 It is worth highlighting that the latter includes the  Duffin-Schaeffer conjecture recently  proved in the ground breaking work of Koukoulopoulos $\&$ Maynard \cite{koukoulopoulos2020}.
 In pursing a zero-full measure criterion such as \eqref{appdiv}, we can usually get away with a lot less than summable mixing.  Indeed, if we already know  by some other means (such as Kolmogorov's theorem \cite[Theorems 4.5 \& 22.3]{bill} or ergodicity \cite[Section 24]{bill}) that the $\limsup $ set  $A_{\infty}$ satisfies a \emph{zero-one law} (that is to say that  $\mu(A_{\infty}) =     0 \ \text{or} \ 1 $),
then to show full measure it suffices to show that $\mu(A_{\infty}) > 0$ when the measure sum diverges. To prove this,  the main ingredient required is a significantly weaker form of mixing equivalent  to the  pairwise quasi-independent on average statement
\begin{equation}  \label{pqoa}
\limsup_{N\to\infty}\frac{\left(\sum_{n=1}^N\mu(A_n)\right)^2}{\sum_{n,m=1}^N\mu(A_n\cap A_m)} > 0 \, .
\end{equation}
For details of this and for the interested reader, its  related converse  see \cite[Section~1.1]{bvbc} and references within.  Clearly,  \eqref{qionav} implies pairwise quasi-independent on average.  For the sake of completeness and to give a concrete example, we mention that the $\limsup$ set of well approximable numbers  associated with the  Duffin-Schaeffer conjecture satisfies a zero-one law and  in short,  Koukoulopoulos $\&$ Maynard established  \eqref{pqoa} to prove the conjecture.    This independence has subsequently been strengthened  \cite{ABH,KMY} to establish  the  quantitative form of the Duffin-Schaeffer conjecture in the spirit of Theorem~A.   This brings to an end  our brief discussion demonstrating  the ``power of mixing''.   As mentioned at the start,   mixing  is a highly sought-after trait for a system to exhibit. It has  various deep consequences.  We have deliberately chosen to highlight Theorem~A,   since it is in line with the applications we have in mind of our main mixing result (see Theorem~\ref{Main2}) for self-conformal dynamical systems -- see  Section~\ref{SEC5}.


Given the fact that exponential mixing is regarded as a  desirable property for a dynamical system to possess, it is natural to ask: when do we have it? From this point onward, unless specified otherwise,   we assume that  $X$ is a  metric space and we will  use the  reasonably standard protocol to say that  `$\mu$ is exponentially mixing' to mean that $\mu$ is exponentially mixing with respect to $(T,\cC)$ with  $\cC$ equal to the collection of balls in $X$. In other words, we simply say that  \emph{$\mu$ is exponentially mixing} if there exist constants $C> 0 $ and  $\gamma\in (0,1)$   such that
\begin{equation}\label{expmixballs}
  \Big|  \mu(B\cap T^{-n}F)   -   \mu(F)\mu(B) \Big|  \; \leq \;  C  \, \gamma^n   \mu(F)      \qquad  \forall  \ n \in \N \, ,
\end{equation}
for all balls $B$ in  $X$ and $\mu$-measurable sets $F$ in $X$.   Note that in certain situations, such as when $X$ is a subset of $\R^d$ and $\mu$ is a Radon measure,  the above restriction to balls  suffices  to imply \emph{strong-mixing}; that is
\[
\lim_{n\to \infty} \mu(E\cap T^{-n}F)=\mu(E)\mu(F)   \qquad \forall  \  E,F \in\cB \, .
\]
This in turn implies  ergodicity.
\noindent   The following  are examples   of  naturally occurring   measure preserving systems  in the literature that are known to possess the exponentially mixing  property.
\begin{itemize}

    \item Let $X=[0,1]$ and let $T : X \to X$ where $Tx=mx~\tmod ~1$ with $m\in\N_{\geq2}$, or  $Tx=\beta x~\tmod ~1$ with $\beta\in\R_{>1}$  or $Tx=\frac{1}{x}~\tmod ~1$ is the continued fraction map.   Accordingly, let $\mu$ be  Lebesgue measure, or Parry measure or Gauss measure.  In 1967,  Philipp \cite{philipp1967} showed that for each  of these situations the measure  $\mu$ is  exponentially mixing. \smallskip
    \item In 1981, Pommerenke \cite{Pommerenke1981} studied the ergodic properties of inner functions acting on the  unit disc $\mathbb{D}$ of the complex plane.   Indeed, let $f: \D  \to \D$ be an inner function such that  $f(0)=0$ and $f$ is not a rotation. Let $f^*: \partial\D  \to \partial\D $ denote the  radial boundary extension of $f$ and let  $\mu$  be normalised Lebesgue measure on $\partial\D$.  Then a consequence of \cite[Lemma~3]{Pommerenke1981} is that if $|f^{\prime}(0) |  < 1 $, then $ \mu$ is exponentially mixing. In fact, Pommerenke's showed that this is true with $f$ replaced by the composition $f_n \circ \ldots \circ f_1$ of $n$ inner functions.
     \smallskip
    \item In 1996, Liverani, Saussol and Vaienti \cite{liverani1998}  considered a class of systems in  abstract totally ordered sets that cover many one-dimensional systems. For convenience, we restrict our attention to the unit interval. Let $X=[0,1]$  and let $T:X\to X$ be a piecewise monotonic transformation. Let $\varphi$ be a contracting potential. Then, under various assumptions on $T$, their results \cite[Theorem 3.1 and 3.2]{liverani1998} implies that there exists a unique equilibrium state $\mu_{\varphi}$ with respect to $\varphi$ and that it is exponentially mixing. This generalizes the  work of Philipp described above.\smallskip
    \item In 2017, Wang, Wu and Xu \cite[Lemma 3.2]{wang2017}  proved that the Cantor measure on middle-third Cantor set is exponentially mixing with respect to the~$\times 3$ map: $Tx=3x~\tmod~1$. This was subsequently  extended to a range of  self-similar sets on $\R$ --  see for instance  \cite[Section 7]{kleinbock2023} \smallskip
    \item In 2023,  by exploiting a deep result of Saussol \cite{saussol2000}, Li, Liao, Velani and Zorin \cite[Proposition~1]{li2023} proved that  for a certain class of  expanding matrix transformations $T$ of the $d$-dimensional torus $\T^d$ there exists an absolutely continuous invariant measure $\mu$ that is exponentially mixing.   In particular, suppose   $T$ is a real, non-singular $d \times d $ matrix with all eigenvalues of modulus strictly greater than one.  Then the ``certain class'' includes $T$ if in addition it satisfies one of the following conditions: (i) all eigenvalues are of modulus strictly larger than $1 + \sqrt{d}$, (ii) $T$ is diagonal, (iii) $T$ is an integer matrix.
\end{itemize}
We highlight the fact that in  all the aforementioned works, the  focus is either on one-dimensional systems or in higher dimensions on systems equipped with absolutely continuous   invariant measures (with respect to Lebesgue measure).   In short, we are  not aware of any non-trivial situation in two-dimensions (let alone arbitrary dimensions)  for which  the invariant measure is fractal and exponentially mixing.  By  a fractal measure we mean a measure that is   not absolutely continuous with respect to Lebesgue measure.  By non-trivial, we mean that   the support  of the measure is not contained  in the union of finitely many straight lines. This simply avoids the higher dimensional situation reducing to a known  one-dimensional setup for which we have exponentially mixing.  

The upshot of the above discussion is the following question.

\begin{question}\label{ques}
    Let $d \geq 2$ and $ X$ be a subset of $\R^d$.  Is there a measure-preserving dynamical system $(X,\mathcal{B},\mu,T)$  for which $\mu$ is  a non-trivial fractal measure on $\R^d$ and $\mu$ is exponentially mixing?
\end{question}

\noindent This question is pretty much the motivating factor behind this paper. With this in mind,  we  show that  Gibbs measures on a large class of  self-conformal sets generated by iterated function schemes (IFS) on $\R^d$  are exponentially mixing.  The precise statement is given by Theorem~\ref{Main2}.

\medskip

\begin{remark} \label{proddy}
Although not explicitly stated in the above discussion, it is worth emphasizing that unless otherwise specified the default metric on $\mathbb{R}^d$ is the $L^2$-norm (i.e., the Euclidean norm). This point is particularly relevant in the context of Question 1.1. Indeed, had we instead employed the $L^\infty$-norm (i.e., the maximum norm), then in referring to a non-trivial fractal measure, we would also need to exclude the case in which $\mu$ is a product measure constructed from one-dimensional measures that are each exponentially mixing. The reason is that it is relatively straightforward to demonstrate that such “manufactured” higher-dimensional measures exhibit exponential mixing with respect to $L^\infty$-balls in $\mathbb{R}^d$. For completeness, we provide the details of this construction in Appendix~\ref{AppA}. We further elaborate on exponential mixing with respect to $L^\infty$-balls (i.e., hypercubes) in the subsequent appendix.   Briefly, in Appendix~\ref{AppH}, this discussion also serves to place our results within the context of absolutely friendly measures supported on fractal subsets of $\mathbb{R}^d$.
\end{remark}

\medskip

Before describing the setup and formally stating our results,  we  say a few words concerning the  closely related  theory of  decay of correlations.
 Let $(X,\cB,\mu,T)$ be a measure-preserving system. The correlation function $   C_{f_1,f_2} : \N \to \R$  associated with $f_1,f_2\in L^2(\mu)$ is defined as
\begin{eqnarray}\label{corredef}
    C_{f_1,f_2}(n):=\int_Xf_1\cdot f_2\circ T^n~\td\mu-\int_Xf_1~\td\mu\int_Xf_2~\td\mu  \, .
\end{eqnarray}
It is well known  \cite[Theorem 1.23]{walters2000} that  $(X,\cB,\mu,T)$ is strong-mixing if and only if
 $\lim_{n \to \infty} C_{f_1,f_2}(n)=0$ for every  $f_1,f_2\in L^2(\mu)$.  In general, nothing can be said regarding the decay rate of convergence for arbitrary $L^2(\mu)$ functions, so it is usually the case that functions are restricted to smaller Banach spaces $\cE$  of functions. Indeed,   $\cE$ is typically taken to be the space of (i) $\alpha$-Hölder continuous functions (e.g. \cite{baladi2000,bowen2008,parry1990}),  (ii) functions of bounded variation for one dimensional systems (e.g.  \cite{liverani1998}),  and (iii) quasi-Hölder functions for higher dimensional systems (e.g. \cite{saussol2000}).  In these cases, if the dynamical system satisfies suitable assumptions, it can be shown that the rate of decay of $C_{f_1,f_2}(n)$  is exponential for all $f_1\in\cE$ and $f_2\in L^1(\mu)$.    To the best of our  knowledge, the existing results on the rate of decay do not address Question \ref{ques}.  For this, we would need to show  the existence of a measure-preserving system $(X,\cB,\mu,T)$ where  $\mu$ is  a non-trivial fractal measure on $\R^d$ $(d \ge 2)$  and for which  there  exist constants $C> 0 $ and  $\gamma\in (0,1)$   such that
\[
|C_{f_1,f_2}(n)|\leq  C  \, \gamma^n  \cdot\|f_2\|_{L^1(\mu)}    \qquad  \forall  \ n \in \N \, ,
\]
for all functions  $f_1=\one_B$, $f_2=\one_F$  where $B$ is any ball in $X$ and $F \in \cB$ (cf. Theorem~\ref{meaannimpexp}   in Section~\ref{Sec2}).  The main issue in showing this lies in the fact that the characteristic function $\one_B$ of a ball is not continuous  and the sought after  measure $\mu$  is not absolutely continuous with respect to Lebesgue measure.  With this in mind,   Theorem~1.2 provides a sufficient condition (namely \eqref{collection}) under which the exponential decay of correlations for Hölder continuous functions can be transferred to the characteristic functions of balls.

\begin{remark}
For the sake of completeness, it is worth mentioning that there is also an abundance of work on the decay of correlations for continuous-time dynamical systems  (also called flows).  In this setting the correlation function is defined  similarly to that in \eqref{corredef}, but with $n\in\N$ replaced by $t\in[0,+\infty)$ and $T^n$ replaced by the `flow' $\phi^t$.  Anosov and Sinai \cite{anosov1967}  in sixties proved that any $C^2$ Anosov flow that preserves Lebesgue measure is strong mixing unless the stable and unstable foliations are jointly integrable.   Regarding the rate of decay of correlations,
the Bowen-Ruelle conjecture essentially  states  that a mixing Anosov flow should have exponential decay for smooth functions.  The original conjecture  \cite{bowen1975} dates back to the seventies and was made for the wider class of Axiom A flows but this was shown to be false soon after  \cite{ruelle1983}.   The  modified conjecture for Anosov flows represents a fundamental problem in the area of  continuous-time dynamical systems.  It has been the catalyst for much groundbreaking research, see for example   \cite{araujo2016,avila2006,butterley2020,chernov1998,dolgopyat1998,dolgopyat2024, field2007,khalil2023,li2023exponential,nonnenmacher2015,tsujii2023}
and references within. To date,  the Bowen-Ruelle conjecture remains unsolved.  The introductions to  \cite{butterley2020,tsujii2023}  are particularly informative of its current state.
\end{remark}



\subsection{Main results}  \label{MRext}

Let $(\Phi,K,\mu,T)$  be  a self-conformal system on $\R^d$.   For  the definition and basic properties see Section~\ref{Sec2}. In short, the set up considered is one in which:
\begin{itemize}
    \item  $\Phi=\{\varphi_j\}_{1\leq j\leq m}$ ($m\geq2$) is a $C^{1+\alpha}$ conformal IFS on $\R^d$ satisfying the open set condition (see Definition \ref{IFSdef} in Section~\ref{cifs}).
    \vspace*{1ex}
    \item $K\subseteq\R^d$ is the self-conformal set generated by $\Phi$ (see (\ref{defofselfconset}) in Section~\ref{cifs}).
      \vspace*{1ex}
    \item $\mu$ is a  Gibbs measure on $K$ (see Definition \ref{Gibbsmeasonk} in Section~\ref{ruelle conformal}).
      \vspace*{1ex}
    \item $T:\R^d\to\R^d$  is a natural map induced by $\Phi$ such that $T|_K:K\to K$ is conjugate to the shift map on the symbolic space $\{1,2,...,m\}^{\N}$  (see \eqref{defoft} in Section~\ref{ruelle conformal}).
\end{itemize}

\medskip

\noindent In order to state our main results  addressing Question \ref{ques}, we  introduce some standard notation that will be used throughout the paper. Given   $\x\in\R^d$ and a nonempty set $E\subseteq\R^d$,
we denote the distance between $\x$ and $E$ as
\begin{eqnarray}\label{defofdxe}
    \mathbf{d}(\x,E):=\inf\big\{|\x-\y|:\y\in E\big\}\,.
\end{eqnarray}
      The topological boundary of $E$ under the usual topology on $\R^d$ is denoted by $\partial E$. Given  $\varrho>0$, the symbol
	\[
	(E)_{\varrho}:=\big\{\x\in\R^d: \bd(\x,E)<\varrho\big\}
	\]
   represents  the $\varrho$-neighborhood of $E$.

\medskip

The following theorem provides a sufficient condition for a  self-conformal system  to exhibit exponentially mixing with respect to  a given collection $\cC$ of measurable subsets of $K$.

 \begin{theorem}\label{Main}
    Let $d\geq1$ be an integer. Let $(\Phi, K, \mu, T)$ be a self-conformal system on $\R^d$. Let $\cC$ be a collection of $\mu$-measurable subsets in $\R^d$ and suppose there exist constants  $C>0$ and $\delta>0$ so that
\begin{eqnarray}\label{collection}
\mu\big((\partial E)_{\varrho}\big)  \ \leq  \ C  \, \varrho^{\delta},   \qquad \forall \ \varrho>0, \quad \forall \  E \in\cC.
    \end{eqnarray}
    Then $\mu$ is exponentially mixing with respect to $(T,\cC)$.
 \end{theorem}

 \noindent  
 In the main body of the paper, we derive Theorem~\ref{Main} as a consequence of a more general statement for measure-preserving systems. This formulation makes explicit the connection between exponential decay of correlations and exponential mixing—see the discussion following Remark~\ref{proddy}. For completeness, we also provide a direct proof of Theorem~\ref{Main} in Appendix~\ref{appendix:directproof}, ensuring the paper remains self-contained within the setting of self-conformal systems. In this context, as shown in the appendix,  the result is a relatively straightforward consequence of exponentially mixing restricted to  cylinder sets.  In order to state the general result we introduce various pieces of useful and relatively standard notation.
 Throughout, given a  bounded  subset $X\subseteq\R^d$ and $\beta>0$, we let $\cC^{\beta}(X)$ denote the collection of all $\beta$-H\"{o}lder continuous functions on $X$  and we let $\|\cdot\|_{\beta}$ denote the $\beta$-H\"{o}lder norm on $\cC^{\beta}(X)$.     Also recall that given a measure-preserving system $(X,\cB,\mu,T)$ on $\R^d$ and $\beta>0$, we say that \emph{$\mu$ has exponential decay of correlations with respect to $\cC^{\beta}(X)$} if there exist $C>0$ and $0<\gamma<1$ such that
    \begin{equation}\label{expforholfun}
        \left|C_{f_1,f_2}(n)\right|\ \leq\  C\gamma^n\|f_1\|_{\beta}\int_{X}|f_2|\,\td\mu
    \end{equation}
    for any $n\in\N$, $f_1\in\cC^{\beta}(X)$ and $f_2\in L^1(\mu)$  where 
 $   C_{f_1,f_2} : \N \to \R$  is the  correlation function  defined  by \eqref{corredef}. 

\begin{theorem}\label{meaannimpexp}
    Let $(X,\cB,\mu,T)$ be a measure-preserving system on $\R^d$ and suppose that  $\mu$ has exponential decay of correlations with respect to $\cC^{\beta}(X)$ for some $0<\beta\leq1$. Let $\cC$ be a collection of $\mu$-measurable subsets in $\R^d$ and suppose there exist constants  $C>0$ and $\delta>0$ so that \eqref{collection} holds.
   Then $\mu$ is exponentially mixing with respect to $(T,\cC)$.
\end{theorem}

In the case of balls, we show that \eqref{collection} is satisfied  and  thus the following   provides an affirmative answer to Question~\ref{ques}.

 \begin{theorem}\label{Main2}
     Let $d\geq1$ be an integer. Let $(\Phi, K, \mu, T)$ be a self-conformal system on $\R^d$ and let $\cC$ be a collection of  balls in $\R^d$. Then $\mu$ is exponentially mixing with respect to $(T,\cC)$.
 \end{theorem}

 Note that for balls, the  left hand side of \eqref{collection}  corresponds to the $\mu$-measure of annuli and the desired upper bound is easily verified if $\mu$ is  absolutely continuous with respect to Lebesgue measure or $d=1$.    However,  non-trivial difficulties appear when $\mu$ is a general measure in higher dimensional space. In short, it turns out that  we need to  fully utilize the geometric structure of self-conformal sets to estimate the  measures of annuli. The following classification or equivalently rigidity  of  self-conformal sets is  an important  ingredient in obtaining  the desired estimates and is potentially  of independent interest. To state the result, we introduce the definition of analytic curve in the plane: the set $\Gamma\subseteq\R^2$ is said to be an \emph{analytic curve} if there exist an open set $\cO\subseteq\R^2$ containing $[0,1]\times\{0\}$ and a conformal map $f:\cO\to\R^2$ such that $\Gamma=f([0,1]\times\{0\})$.

 \begin{theorem}\label{thmrigidity}
     Let $(\Phi,K,\mu,T)$ be a self-conformal system on $
    \R^d$ with $d\geq 2$. Given $\ell\in\{1,...,d-1\}$, then one of the following statements hold:
     \begin{enumerate}[label=(\roman*)]
         \item $\mu(K\cap M)=0$ for any  $\ell$-dimensional $C^1$ submanifold $M\subseteq\R^d$;  \vspace*{1ex}
 \item $K$ is contained in a $\ell$-dimensional affine subspace or $\ell$-dimensional geometric sphere;  \vspace*{1ex}
 \item $K$ is contained in a finite disjoint union of analytic curves and this may happen  only when $d=2$ and $\ell=1$.
     \end{enumerate}
 \end{theorem}

\medskip

 \begin{remark} \label{thy}
      Mattila \cite{mattila1982} obtained the analogous  statement   for self-similar sets satisfying the open set condition and  with $\mu$ being the Hausdorff measure. It is often referred to as Mattila’s rigidity theorem for self-similar sets.    Subsequently, Käenmäki \cite{kaenmaki2003} extended his work to self-conformal systems. Compared with their results,  our statement  is
      \begin{itemize}
    \item  valid for  a broader class of measures (namely Gibbs measures);
    \vspace*{1ex}
    \item  valid without the open set condition  (indeed, see Proposition~\ref{rigidity} for the more general result from which Theorem~\ref{thmrigidity} follows).
\end{itemize}
Furthermore, our result corrects an oversight  in the statement of   \cite[Theorem~2.1]{kaenmaki2003}.  Basically,  Käenmäki claimed  that $K$ is contained in a single analytic curve  when (i) or (ii) are not the case.   However, this is not true even with the open set condition assumption and $\mu$ restricted to Hausdorff measures.  For details of a counterexample  see Example~\ref{counterexam} in Section~\ref{Secrig}.
 \end{remark}

\begin{remark} \label{whyadb}  As already alluded to in Remark~\ref{proddy}, Appendix~\ref{AppH} contains a discussion of how our results relate to the theory of absolutely friendly measures supported on fractal subsets of $\mathbb{R}^d$. This connection provides additional insight into the applicability of our framework within the broader context of fractal geometry and Diophantine approximation.
\end{remark}

\subsection{Application to shrinking target and recurrent sets}  \label{appintro}

In this section we discuss  applications of
Theorem~\ref{Main2} to the  shrinking target and related  recurrent problem for  self-conformal dynamical systems.
Further details  including  proofs of statements presented  will be the subject of Section~\ref{SEC5}.


Throughout, $(X,d)$ is a compact metric space and $(X,\mathcal{B},\mu,T)$ is an ergodic probability measure-preserving system.   We start with describing the application of
Theorem~\ref{Main2} to the shrinking target problem.   For obvious reasons, from the onset,  we restrict our attention the setup in which the ``target sets''  are balls rather than arbitrary measurable sets $E_n$ as in the  general setup of Theorem~A.   With this in mind, given   a sequence of points  $\cY=\{y_n\}_{n\in\N}  \subset X$,    and  a real, positive function $\psi:\N\to[0,+\infty)$ let
\begin{equation*}
            W(\cY,\psi):=  \left\{x\in X:T^nx\in B(y_n,\psi(n))~\text{for infinitely many}~n\in\N\right\}
       \end{equation*}
denote the associated \emph{shrinking target set}.   If $\psi=c$ (a constant) and $\cY$ is contained in the support of $\mu$,  the Ergodic Theorem implies that
$
 \mu( W(\cY,c))  \, = 1      $.
In view of this, it is natural to ask:
what is the  $\mu$-measure of  the set  if $\psi(n) \to 0 $ as $n \to \infty$?
In turn, whenever the $\mu$-measure is zero,  it is natural to ask about the Hausdorff dimension of the sets under consideration.   The shrinking target problem  was introduced  in \cite{hill1995}, with a focus on the dynamics of expanding rational maps.  Subsequently, there has been activity  encompassing a wide range of dynamical systems.   We refer to \cite{aspenberg2019,barany2018,bugeaud2014,fang2020,ghosh2024,li2023,li2014,persson2019} and the references within for dimensional results  and to  \cite{allen2021,fernandez2012,galatolo2015,kim2007,tseng2008} and the references within for measure-theoretical statements.   Concentrating our attention solely on the $\mu$-measure question, note  that
$$
W(\cY,\psi) = \limsup_{n\to\infty}A_n(\cY,\psi)
$$
where for $n \in \N$,
\begin{eqnarray*}
A_n(\cY,\psi)& :=  & \left\{x\in X:T^nx\in B(y_n,\psi(n))\right\}  \\[1ex]  & = &  T^{-n} \big(B(y_n,\psi(n))\big) .
   \end{eqnarray*}
Then, on combining Theorem~\ref{Main2} and   Theorem~A (with $E_n :=  B(y_n,\psi(n))$) we  immediately  see that we  have the following  quantitative shrinking target statement for  self-conformal dynamical systems.

\begin{theorem}  \label{shrinktarg}
    Let $(\Phi,K,\mu,T)$ be a self-conformal system on $\R^d$, let $\psi:\R\to\R_{\geq0}$ be a real positive function  and let $\cY=\{\y_n\}_{n\in\N}$ be a sequence of points in $\R^d$. Then, for any $\epsilon>0$, we have
   \begin{equation} \label{shrinktargcount}
        \sum_{n=1}^N\one_{B(\y_n,\psi(n))}(T^n\x)=\Psi(N)+O\left(\Psi(N)^{1/2}\log^{\frac{3}{2}+\epsilon}(\Psi(N))\right)
    \end{equation}
    for $\mu$-almost all $\x\in K$, where 
    \begin{equation} \label{shrinktargsum}  \Psi(N):=\sum_{n=1}^N\mu\big(B(\y_n,\psi(n))\big)   \, . \end{equation}
\end{theorem}

 Several comments are in order. Given $ N \in \N$ and $\x \in K$,   note that the left hand side of  \eqref{shrinktargcount} is simply  the  counting function
\begin{eqnarray}
W(\x,N;\cY,\psi)& :=  & \# \big\{ 1\le n \le   N :    \x \in A_n(\cY,\psi) \big\} \, \\[1ex]  & = & \# \big\{ 1\le n \le   N :    T^n\x \in B(\y_n, \psi(n))  \big\}   \nonumber
   \end{eqnarray}
and that the measure sum \eqref{shrinktargsum} is equivalent to
\begin{equation} \label{shrinktargsum2}  \Psi(N):=\sum_{n=1}^N\mu\big(A_n(\cY,\psi)\big)   \, . \end{equation}
Thus,  the theorem shows that for $\mu$-almost all $\x \in K $,  the asymptotic behaviour  of the counting function $W(\x,N;\cY,\psi)$  is determined by the behaviour of   the measure sum  $ \Psi(N)$ involving the sets $A_n(\cY,\psi)$ associated with the $\limsup$ set  $W(\cY,\psi)$. This together with the fact that  $\Psi(N)$ is independent of $\x \in K $,
is well worth keeping in mind for future  comparison with the analogous recurrent  problem.     Next note that  by definition, $ \x \in W(\cY,\psi) $ if and only if $\lim_{N \to \infty}  W(\x,N;\cY,\psi) = \infty$ and  so  an immediate consequence of Theorem~\ref{shrinktarg} is the following zero-full  measure criterion (which naturally  is in line with \eqref{appdiv} in the general setup of measurable sets).

\begin{corollary}  \label{shrinktargcor}
Let $(\Phi,K,\mu,T)$ be a self-conformal system on $\R^d$, let $\psi:\R\to\R_{\geq0}$ be a real positive function  and let $\cY=\{\y_n\}_{n\in\N}$ be a sequence of points in $\R^d$. Then,
\begin{eqnarray*}
		 \mu\left(W(\cY,\psi)\right)=
		\begin{cases}
			0 &\text{if}\ \  \sum_{n=1}^{\infty}\mu\big(B(\y_n,\psi(n))\big)<\infty,\\[2ex]
			1 &\text{if}\ \ \sum_{n=1}^{\infty}\mu\big(B(\y_n,\psi(n))\big)=\infty.
		\end{cases}
	\end{eqnarray*}
\end{corollary}

\medskip

We now consider the analogue of the shrinking target problem for the recurrence framework.  As above,  the general scene is one in which   $(X,d)$ is  a compact metric space and $(X,\mathcal{B},\mu,T)$ is a probability measure-preserving system.  We do not need the system to be ergodic to pose the recurrent problem.   Given   a real, positive function $\psi:\N\to[0,+\infty)$ let
\begin{equation*}
            R(\psi):=  \left\{x\in X:T^nx\in B(x,\psi(n))~\text{for infinitely many}~n\in\N\right\}
       \end{equation*}

\noindent denote the associated \emph{recurrent set}.  If $\psi=c$ (a constant),   the Poincar\'{e} Recurrence  Theorem implies  that
$
 \mu( R(c) )  \, = 1      $
 and  it is natural to determine the
$\mu$-measure of  the set $R(\psi)$  if $\psi(n) \to 0 $ as $n \to \infty$ and if it is
zero, its size in terms of  Hausdorff dimension. The first results date back to the pioneering work of  Boshernitzan \cite{boshernitzan1993} who studied the case $\psi(n)=n^{-1/\beta}$ ($\beta>0$). For subsequent activity  we refer to \cite{baker2021,baker2024,chang2019,hussain2022,kleinbock2023}  and references within  for measure-theoretical statements and to  \cite{he2022,hu2024,seuret2015,tan2011,wu2022} and references within for dimensional results.  As with the shrinking target problem,  we will concentrate our attention on the $\mu$-measure aspect of the recurrent problem. With this in mind,  we first note that  $R(\psi)$ is clearly  also a  $\limsup$  set; namely
$$
R(\psi) = \limsup_{n\to\infty}R_n(\psi)
$$
where for $n \in \N$,
\begin{equation} \label{anrpsidef}
R_n(\psi) :=   \left\{x\in X:T^nx\in B(x,\psi(n))\right\}  \, .
\end{equation}

\noindent Furthermore, given $ N \in \N$ and $x \in X$,  if we consider  the  associated counting function
\begin{eqnarray} \label{countdefst}
R(x,N;\psi)& :=  & \# \big\{ 1\le n \le   N :    x \in R_n(\psi) \big\} \, \nonumber  \\[1ex]  & = & \# \big\{ 1\le n \le   N :    T^nx \in B(x , \psi(n))  \big\} \nonumber \\[1ex] & = &  \textstyle{\sum_{n=1}^{N}}\mathbbm{1}_{B(x,\psi(n))}(T^nx)  \, ,
   \end{eqnarray}
 then in line with the shrinking target framework (and more  generally that of quantitative Borel-Cantelli), it would not be particularly  outrageous to  suspect (under
suitable but natural assumptions) that for $\mu$-almost all $x \in X$,    the asymptotic behaviour  of the counting function  is determined by the behaviour of the   $\mu$-measure sum  of the sets $R_n(\psi)$.  Let us make this precise in the setting of  self-conformal dynamical systems.

 \begin{claimf}  \label{claimf}
  \emph{ Let $(\Phi,K,\mu,T)$ be a self-conformal system on $\R^d$, let $\psi:\R\to\R_{\geq0}$ be a real positive function and assume that $\sum_{n=1}^{\infty}\mu(R_n(\psi)) $ diverges.  Then, for $\mu$-almost  all $\x \in  K$
 \begin{equation}\label{sbc}
 \lim_{N\to\infty}\frac{\sum_{n=1}^{N}
 \mathbbm{1}_{B(\x,\psi(n))}(T^n\x)}{\sum_{n=1}^{N}\mu(R_n(\psi))} \, = \, 1 \, .
 \end{equation}}
 \end{claimf}

 \noindent Such a claim was also  alluded to in \cite[Section~1]{levesley2024} and it was  shown to be true for a large class of piecewise linear maps in $\R^d$. However, as we shall demonstrate, it turns out that in general the claim is  false (hence the label ``F'') in a rather strong sense.  Indeed, in Section~\ref{examplesec} we provide  explicit examples of self-conformal systems for which the $\mu$-measure of the $\limsup$ set $R(\psi)$ is one but the limit
appearing in \eqref{sbc}  is not even a constant let alone one (cf. Example~ABB below). In
other words, even after excluding a set of $\mu$-measure zero, the limit in \eqref{sbc}  depends on $x$
and thus for these self-conformal systems the associated recurrent sets exhibit (unexpected
and extreme) behaviour that is not present for shrinking target sets.   To the best of our knowledge this phenomena seems not to have been observed previously or at least not explicitly documented. The following summarises the counterexamples to the claim  given in Section~\ref{examplesec}.

\begin{itemize}
  \item  In Example~\ref{egcantor},  we start with
 $\Phi=\{\varphi_1,\varphi_2\}$ where
   $\varphi_1:[0,1]\to[0,1/3]$ and $\varphi_2:[0,1]\to[2/3,1]$ are  given by
    \[
    \varphi_1(x)=\frac{x}{3},  \qquad \varphi_2(x)=\frac{x+2}{3}  \qquad \forall \  x\in[0,1]  .
    \]
   This gives rise to  the ``natural'' associated  self-conformal system $(\Phi,K,\mu,T)$ in which $K$ is  the standard middle-third Cantor set and $\mu$ is the Cantor measure.    Then, for   the constant function $\psi:\R\to\R_{\geq0}$ given by $\psi(x) \, := \,  \textstyle{ \Large{\frac{1}{3}+\frac{2}{3^2} }}  \, , $
   we show that:  for $\mu$--almost all  $x\in K$
   \begin{eqnarray*}
      \lim_{N\to\infty}\frac{\sum_{n=1}^{N}
 \mathbbm{1}_{B(x,\psi(n))}(T^nx)}{\sum_{n=1}^{N}\mu(R_n(\psi))}
      =\left\{
\begin{aligned}
    &\textstyle{\frac{4}{5}  \quad \text{if}  \quad x\in\left(\big[0,\frac{1}{9}\big]\cup\big[\frac{8}{9},1\big]\right)\cap K} ,\\[2ex]
    &\textstyle{\frac{6}{5}   \quad \text{if} \quad x\in\left(\big[\frac{2}{9},\frac{1}{3}\big]\cup\big[\frac{2}{3},\frac{7}{9}\big]\right)\cap K}  \, .
\end{aligned}
      \right.
   \end{eqnarray*}

\medskip

  \item In Example~\ref{egfull} we start with $\Phi=\{\varphi_1,\varphi_2,\varphi_3,\varphi_4\}$  where $\varphi_i$ $(i=1,2,3,4)$ are defined on $[0,1]$ and  given  by
    \[
   \varphi_1(x)=\frac{1}{4}x,\ \varphi_2(x)=\frac{1}{2(1+x)},\ \varphi_3(x)=\frac{1+x}{2+x},\ \varphi_4(x)=\frac{2}{2+x}\,.
    \]
    We show that this gives rise to a  self-conformal system $(\Phi,K,\mu,T)$ in which $K$ is  the unit interval and $\mu$ is the natural Gibbs measure supports on $K$ that is absolutely continuous with respect to Lebsegue measure.  In turn, for any  real positive function $\psi:\R\to\R_{\geq0}$ such that $\psi(x) \to 0 $ as $x \to \infty$ and $\sum_{n=1}^{\infty}\mu(R_n(\psi)) $ diverges, we show that: for $\mu$--almost all  $x\in[0,1]$
\begin{equation*}
     \lim_{N\to\infty}\frac{\sum_{n=1}^{N}
 \mathbbm{1}_{B(x,\psi(n))}(T^nx)}{\sum_{n=1}^{N}\mu(R_n(\psi))} \, = \,
     \frac{2\log 2}{1+x}.
\end{equation*}
\end{itemize}

\medskip

\noindent While the first more familiar ``Cantor''  example requires less sophisticated tools to setup and execute,  it does rely on $ \psi$ being a constant function.

\medskip

\begin{remark}
    The counterexamples show that even though we have exponentially mixing (Theorem~\ref{Main2})  we can not in general guarantee  that the sets $R_n(\psi)$ are pairwise independent on average (in the sense of  \eqref{qionav}) as in the shrinking target framework. The point is that if it did then the quantitative form of the Borel-Cantelli Lemma (see Section~\ref{CRR}: Lemma~\ref{countlem}) would  establish Claim~F  (very much in the same way we deduce Theorem~A from \eqref{qionav}).
\end{remark}

\medskip

Note that in both Example~\ref{egcantor} and \ref{egfull},  we still have that
 $\lim_{N \to \infty} R(x,N;\psi) = \infty  $  for $\mu$--almost all $x \in X$  and so  $\mu(R(\psi))=1$.  Moreover, we  highlight  the fact that in both  the measure $\mu$ is Ahlfors regular  and that for such measures this phenomena (under the assumption that $\sum_{n=1}^{\infty}\mu(R_n(\psi)) $ diverges) is known to hold for any self-conformal system  (see \cite{baker2021})
  and indeed for more general systems (see \cite{hussain2022}).  Recall, a measure $\mu$ on a metric space $(X,d)$ is \emph{$\tau$-Ahlfors regular} if there exists a constant  $C\ge 1$  such that for any ball $B(x, r)  \subset X$ with $x \in X$
\begin{equation}  \label{tar}
C^{-1}r^{\tau}\leq \mu(B(x,r))\leq Cr^{\tau}
\,\qquad  \forall \ 0<r\leq |X|
\, ,  \medskip
\end{equation}
where $|X|$ denotes the diameter of $X$.
The upshot of the above is that given a self-conformal system $(\Phi,K,\mu,T)$ on $\R^d$ for which the Gibbs measure $\mu$ is Ahlfors regular, and a real positive function  $\psi:\R\to\R_{\geq0}$,  then
\begin{eqnarray}  \label{divcondsv1}
		 \mu\left(R(\psi)\right)=
		\begin{cases}
			0 &\text{if}\ \  \sum_{n=1}^{\infty}\mu\big(R_n(\psi)\big)<\infty,\\[2ex]
			1 &\text{if}\ \ \sum_{n=1}^{\infty}\mu\big(R_n(\psi) \big)=\infty.
		\end{cases}
\end{eqnarray}\\
 The convergent part is  a straightforward consequence of the standard convergent Borel-Cantelli Lemma in probability theory.    In view of this, it is tempting to
    suspect that  at the coarser level of a zero-full measure criterion the analogue of  Claim~F is true; that is to say that \eqref{divcondsv1} is true for any self-conformal system on $\R^d$.  Clearly, such a  statement  would correspond to the analogue of Corollary~\ref{shrinktargcor} for recurrent sets.
     However, this  turns out not to be the case.  In a recent  beautiful paper,  Allen, Baker $\&$ B\'{a}r\'{a}ny \cite{allen2025}  consider the recurrent problem within the symbolic dynamics setting for topologically mixing sub-shifts of finite type. More precisely, in this setting they provide sufficient conditions for $\mu(R(\psi))$ to be zero or one when $\mu$ is assumed to be a non-uniform Gibbs measure and thus is not Ahlfors regular.  In terms of Bernoulli measures defined on the full shift,  the condition on the measure means that the  components of the defining probability vector  are not all equal.    As a consequence of their main result, in the introduction  \cite[Section~1]{allen2025} they provide  a class of  examples  within the symbolic dynamics setting for  which the sum of  $\mu(R_n(\psi))$ diverges  but $\mu(R(\psi))$ is equal to zero. In particular, these examples show that \eqref{divcondsv1} is not true  for non-uniform Gibbs measures associated with topologically mixing shifts of finite type. In the final section of \cite{allen2025}, the authors   outline how their theorems  can be transferred, via a relatively standard  argument involving the coding map, to the setting of dynamics on homogenous self-similar sets satisfying the strong separation condition  and for which the corresponding  Gibbs measures are assumed to be non-uniform.  Although not explicitly mentioned, in the same spirit the examples from \cite[Section~1]{allen2025}  can  also be naturally transferred across and when specialised to the middle-third  Cantor we obtain the following concrete example that shows that \eqref{divcondsv1} is not true for any self-conformal system.

\medskip

  \noindent \textbf{Example~ABB.} \ Let $\Phi=\{\varphi_1 \, , \varphi_2 \}$, $T$ and  $K$ be as in Example~\ref{egcantor}.  Recall, $K$ is the  standard  middle-third Cantor. Now let $\mu$ be the weighted Cantor measure  associated with the probability vector  $(p_1,p_2)$ with $p_1\neq p_2$.  Let $\alpha>0$ and $\psi_{\alpha}(n)=3^{-\linte{\alpha\log n}}$. If
    \[
    \frac{1}{-(p_1\log p_1+p_2\log p_2)}<\alpha<\frac{1}{-\log(p_1^2+p_2^2)},
    \]
    then
    \[
    \sum_{n=1}^{\infty}\mu(R_n(\psi_{\alpha}))=\infty\qquad\text{but}\qquad \mu(R(\psi_{\alpha}))=0.
    \]

\begin{remark}
To be precise,  in the above example,   $\mu:=\ubar{\mu}\circ\pi^{-1}$ where $\ubar{\mu}$ is the Bernoulli measure on $\Sigma^{\N}:=\{1,2\}^{\N}$ associated with the probability vector  $(p_1,p_2)$ with $p_1\neq p_2$ and $\pi:\Sigma^{\N}\to K$ is the coding map  associated  to $\Phi$ (see \eqref{cmap} for the definition).
\end{remark}

\medskip

A straightforward consequence of  Example~ABB is that for $\mu$--almost all $x \in X$
    \[
   \sum_{n=1}^{\infty}\one_{B(x,\psi_{\alpha}(n))}
(T^nx)  \, \ll \, 1   \quad {\rm and \ so \ }  \quad
    \lim_{N\to\infty}\frac{\sum_{n=1}^N\one_{B(x,\psi_{\alpha}(n))}
(T^nx)}{\sum_{n=1}^N\mu(R_n(\psi_{\alpha}))}=0  \, .
    \]
 In other words, even though the limit is a constant  for $\mu$--almost all $x \in X$,  it is not one (cf.  Claim~F). Note that  Examples 7.1 and 7.2 show that Claim~F is false  even when  $\mu(R(\psi))=1$ and that for $\mu$-almost all $x \in K$, the limit under consideration is dependent on $x$ and thus  not a constant; that is to say that Claim~F is false on a large scale!


 \medskip

Given that Claim~F is false, it is natural to  attempt to establish   an appropriate ``modified'' statement  that is true for the  full range of dynamical systems under consideration (namely, self-conformal systems).  Such a statement  would obviously follow on establishing the analogue  of Theorem~\ref{shrinktarg} for recurrent sets.  Indeed,  this is the ultimate goal as it  would provide an asymptotic result with an error term.   With this in mind, in order to state our first main result (for recurrent sets) we need to introduce a particular function that will determine the appropriate setup and thus the  asymptotic behaviour.   As usual,  let
$(\Phi,K,\mu,T)$   be a self-conformal system and a $\psi:\R\to\R_{\geq0}$ be a real,  positive function. Then for each  $n\in \N$,  we define the function
%
 $$
 t_n(\cdot)=t_n(\cdot,\psi):K\to\R_{\geq0}
 $$
  by
  \begin{eqnarray}\label{defofrn}
    t_n(\x)=t_n(\x,\psi):=\inf\left\{r\geq0:\mu(B(\x,r))\geq\psi(n)\right\}
\end{eqnarray}
if $ \psi(n)  \leq  1 $  and  we put $ t_n(\x)$ equal to the diameter of the bounded set  $K$ otherwise.
  With the definition of $t_n$ in mind, Theorem~\ref{Main2}  enables us to establish the following analogue  of Theorem~\ref{shrinktarg} for recurrent sets.

\begin{theorem}\label{quantrec}
    Let $(\Phi,K,\mu,T)$ be a self-conformal system on $\R^d$ and let $\psi:\R\to\R_{\geq0}$ be a real positive function such that $\psi(x) \to 0 $ as $x \to \infty$. Furthermore, for $n \in \N$
    let $t_n :K\to\R_{\geq0}$ be given by  (\ref{defofrn}). Then for any $\epsilon>0$, we have
    \begin{equation} \label{recurrentcount}
        \sum_{n=1}^{N}\one_{B(\x,t_n(\x))}(T^n\x)=\Psi(N)+O\left(\Psi(N)^{1/2}\log^{\frac{3}{2}+\epsilon}(\Psi(N))\right)
    \end{equation}
    for $\mu$-almost all $\x\in K$, where
     \begin{equation} \label{recurrentsum}
    \displaystyle\Psi(N):=\sum_{n=1}^N\psi(n)   \, .
    \end{equation}
\end{theorem}

\noindent
\begin{remark} \label{remarkthm1.5}  Several comments are in order.
\begin{enumerate}
  \item[(i)] It turns out (see Lemma~\ref{lem5.1} in Section~\ref{CRR}) that for all  $ \x \in K $ and all sufficiently large  $n \in \N$
\begin{equation}  \label{omg}   \mu(B(\x,t_n(\x))) = \psi(n)   \, .   \end{equation}
  Thus, up to an additive constant, the sum \eqref{recurrentsum} is simply the sum of the $\mu$-measure of the ``target balls'' $B(\x,t_n(\x)) $  associated with the modified  counting function appearing  on the left hand side of~\eqref{recurrentcount}.  In short, if the measure $\mu$ is non-uniform then the measure of a ball $B(\x, r)$  depends on its location $\x$  and not just its radius $r$.  In order to take this into account,  for  $n$ large, the radii of the target balls within the framework of  Theorem~\ref{quantrec}   are adjusted so that they all have the same measure  (namely $\psi(n)$) regardless of location.
  \medskip
  \item[(ii)] Let $ \hat{R}_n(\x,N;\psi) $   denote the  modified  counting function appearing  on the left hand side of~\eqref{recurrentcount}.  Then by definition,
  \begin{equation*} \label{countdefst-hat}
\hat{R}_n(\x,N;\psi)  =  \# \big\{ 1\le n \le   N :    \x \in \hat{R}_n(\psi) \big\} \, ,
   \end{equation*}
   where
   $$
 \hat{R}_n(\psi)   :=    \big\{\x\in K :    T^n\x \in B(\x , t_n(\x))  \big\} \, .
   $$
 It turns out  (see Lemma~\ref{measan} in Section~\ref{CRR}) that there exists a constant $0< \gamma < 1 $ such that
$$ \mu(\hat{R}_n(\psi))=\psi(n)+O(\gamma^n)    \, . $$
The upshot of this and the equality \eqref{omg} appearing in (i) above is   that
the sum   \eqref{recurrentsum}  appearing in the theorem and the measure sums $\sum_{n=1}^N\mu(B(\x,t_n(\x))) $ and $ \sum_{n=1}^N \mu(\hat{R}_n(\psi)) $  are all equal up to an additive constant.
   \medskip

  \item[(iii)] The theorem is valid for any self-conformal system on $\R^d$. The price we  seemingly have to pay for this generality is that the radii of the target balls  $B(\x,t_n(\x))$ associated with the modified counting function $ \hat{R}_n(\x,N;\psi) $  are dependant on their centres $\x \in K$.  This is   clearly unlike the situation for the ``pure'' counting function  $R_n(\x,N;\psi)$   for which we know that Claim~F is false for all self-conformal systems.\medskip
  \item[(iv)]  A simple consequence of Theorem~\ref{quantrec} is the following asymptotic statement that ``fixes'' Claim~F:
 \emph{ Let $(\Phi,K,\mu,T)$ be a self-conformal system on $\R^d$ and let $\psi:\R\to\R_{\geq0}$ be a real positive function such that $\sum_{n=1}^{\infty} \psi(n)$ diverges. Then for $\mu$-almost all $ \x \in  K$}
 \begin{equation}\label{sbccc}
 \lim_{N\to\infty}\frac{\sum_{n=1}^{N}\one_{B(\x,t_n(\x))}
 (T^n\x)}{\sum_{n=1}^{N}\psi(n)} \, = \, 1 \, .
 \end{equation} \\
 Note that in view of the discussion in (ii) above this ``corrected'' statement simply corresponds to Claim~F in which the counting function $R_n(\x,N;\psi) $ is replaced by the modified counting function $ \hat{R}_n(\x,N;\psi) $    and with $R_n(\psi) $ replaced by $\hat{R}_n(\psi)$.
 \medskip
 \item[(v)] Recently,  under various  growth conditions on the function  $\psi$,  Persson \cite{persson2023} has proved a result in a similar  vein to \eqref{sbccc} for a large class of dynamical systems  with exponential decay of correlations on the unit interval. Subsequently, his work (with the various growth conditions) was extended by Sponheimer \cite{Sponheimer} to  more general dynamical systems including Axiom A diffeomorphisms. We stress that Theorem~\ref{quantrec}, which implies \eqref{sbccc},  is free of growth conditions on $\psi$ and provides  an essentially optimal error term.  At the point of  completing this paper, the preprint \cite{persson2025} of Persson $\&$ Sponheimer   appeared.  In this, under a `short return time assumption' and `$3$-fold  exponential decay' they  essentially remove the growth conditions on $\psi$ imposed in their previous works.
\end{enumerate}
\end{remark}
\medskip



Even though Theorem~\ref{quantrec}  is in some sense a ``complete'' result, it  fails to  directly deal with the main purpose  of  Claim~F.   Indeed, it remains highly desirable  to  obtain asymptotic information regarding the  behaviour of the ``pure''  counting function \eqref{countdefst} in which the radii of the target balls are independent of their centres.  We reiterate that this is not the case within the framework of Theorem~\ref{quantrec}.   In short, our second main result (for recurrent sets) shows that we are in reasonably  good shape for systems with Gibbs measures equivalent to restricted Hausdorff measures $\cH^{\tau}|_K$.  Here and throughout, we say that Borel measures $\mu$ and $\nu$ on a metric space $(X,d)$ are equivalent if there exists a constant $C\geq1$  such that
$
C^{-1}\nu(E)\leq\mu(E)\leq C\nu(E)  $   for any Borel subset $ E\subseteq X$.


\begin{theorem} \label{toprove}
     Let $(\Phi,K,\mu,T)$ be a self-conformal system on $\R^d$ with $\mu$ being a Gibbs measure  equivalent to $\cH^{\tau}|_K$ where $\tau:=\dimH K$. Let $\psi:\R\to\R_{\geq0}$ be a real positive function such that $\psi(x) \to 0 $ as $x \to \infty$. Then for any $\eta>0$ and $\epsilon>0$, we have
     \begin{equation} \label{recurrentcountHM}
        \sum_{n=1}^{N}\one_{B(\x,\psi(n))}(T^n\x)=\sum_{n=1}^N\mu\big(B(\x,\psi(n))\big)+O\left(\Psi_{\eta}(N)^{1/2}(\log\Psi_{\eta}(N))^{\frac{3}{2}+\epsilon}\right)
    \end{equation}
    for $\mu$-almost all $\x\in K$, where
     \begin{equation} \label{recurrentsumHM}
    \Psi_{\eta}(N):=\sum_{n=1}^N\psi(n)^{(1-\eta)\tau}   \, .
    \end{equation}
\end{theorem}

\medskip

\begin{remark}~  \label{newrem} Several comments are in order.

\begin{enumerate}
  \item[(i)]
It is easily versified that within the setup of self-conformal systems, the notion of  $\mu$ being equivalent to $\cH^{\tau}|_K$  and $\mu$ being $\tau$-Ahlfors regular (see \eqref{tar}) coincide  -- for the details, see the proof of Theorem~2.7 in \cite{fan1999}.  In general, we only have that the latter implies  the former.   \\

\item[(ii)] Theorem~\ref{toprove},  could, in principle, be stated with $\eta = \epsilon$ for a cleaner formulation. However, the parameters $\eta$ and $\epsilon$ serve distinct roles. The presence of $\epsilon > 0$ arises necessarily  from the application of the quantitative  Borel–Cantelli Lemma (see Section~\ref{apptoprove}, Lemma~\ref{genharmanlem}) and  cannot be removed.   Notably, this same $\epsilon$ appears in the statements of Theorem~\ref{shrinktarg} and Theorem~\ref{quantrec}.  In contrast, we strongly believe that theorem remains valid with $\eta = 0$, and that the introduction of $\eta>0$ is merely a technical artifact.  Specifically, it arises in Lemma~\ref{bddalholder} to facilitate the proof.  We shall soon see that  this  belief is justified  if we are content with asymptotic statements without error term.
  \\

\item[(iii)] In view of (i) it follows that
\begin{equation}  \label{allcomp101}
\sum_{n=1}^N \mu\big(B(\x,\psi(n))\big)   \;  \asymp \;    \sum_{n=1}^N \psi(n)^{\tau} \,  :=  \, \Psi(N)   \, ;
\end{equation}
that is the main term in \eqref{recurrentcountHM} is comparable to  \eqref{recurrentsumHM} with $\eta =0$.  In turn, it follows that due to the presence of $\eta > 0$ in the error term  in \eqref{recurrentcountHM} we can not always conclude that the main term dominates the error term without imposing a condition on the decay rate of $\psi$.  We give a simple example that we hope clearly illustrates the point being made.  Suppose $d=1$ and $\mu$ is one-dimensional Lebesgue measure. For $\alpha > 0$,  consider the function $\psi_\alpha:\R\to\R_{\geq0} $ given by
$$
\psi_{\alpha} (x) \, :=   \,  x^{-\alpha}  \, .
$$
Then, $   \mu\big(B(x,\psi_{\alpha}(n))\big) = 2 n^{-\alpha} $ for any $x \in K$  and so $\sum_{n=1}^\infty  \mu\big(B(x,\psi_{\alpha}(n))\big)$ diverges for any $\alpha \in (0,1]$. Moreover,
\begin{eqnarray*}\label{}
	\lim_{N\to\infty} \; \frac{{\rm Error \ Term \ in \ \eqref{recurrentcountHM} }}{\sum_{n=1}^N\mu\big(B(x,\psi_{\alpha}(n))\big)}  \ = \
\begin{cases}
			0 &\text{if}\ \  \alpha \in (0,1)\\[2ex]
			\infty &\text{if}\ \   \alpha =1 .
		\end{cases}
	\end{eqnarray*}
Thus, Theorem \ref{toprove} does not yield the desired asymptotic statement at the critical exponent $\alpha =1$.    Nevertheless, apart from this flaw, for reasons outlined earlier,  Theorem~\ref{toprove} is on the whole  a more desirable analogue of Theorem~\ref{shrinktarg} than Theorem~\ref{quantrec} for self-conformal systems with Gibbs measures $\mu$ equivalent to  $\cH^{\tau}|_K$.   Under certain additional conditions on the measure (such as $\mu$ being absolutely continuous with respect to  Lebesgue measure) we are able to show that  the flaw  is not present if we are content with asymptotic statements without error terms.
\end{enumerate}
\end{remark}

Note that the theorem shows that  for $\mu$-almost all $x \in K $,  the asymptotic behaviour  of the counting function $R(\x,N;\psi)$  is determined by the behaviour of the measure sum
\begin{equation} \label{recurrentsumHMx}
\sum_{n=1}^N\mu\big(B(\x,\psi(n))\big)  \, ,
\end{equation}
  which, a priori,  is dependant on $x$.    The point is that  if the measure  $\mu$ is non-uniform,   the measure of the ``target balls''  $B(\x,\psi(n)) $  associated with   $R(\x,N;\psi)$  depends on $x$.
 This is unlike the situation in the shrinking  target framework in which the  measure of the ``target balls''  $B(\y_n,\psi(n)) $  associated with the counting function  $W(\x,N;\cY,\psi)$ are independent of $\x$.  On a slightly different but related note, we point out that the Gibbs measures associated with the  explicit counterexamples (Examples~\ref{egcantor} \& \ref{egfull})  to Claim~F satisfy the conditions of Theorem~\ref{toprove}.    Thus,    the $\mu$-measure sum \eqref{recurrentsumHMx}   can not in general coincide with the  $\mu$-measure sum involving the sets $R_n(\psi)$ associated with the  recurrent $\limsup$ set  $R(\psi)$.
However, it is the case  (see Lemma~\ref{topromuofrnpsi})  that the sums \eqref{recurrentsumHM} with $\eta =0$,  \eqref{recurrentsumHMx} and $ \sum_{n=1}^N\mu\big(R_n(\psi)\big)$ are all comparable\footnote{For the sake of comparison, recall that in the setting of Theorem~\ref{quantrec} the analogous three sums are asymptotically equivalent  (see comment (ii) in Remark~\ref{remarkthm1.5}).}; that is
\begin{equation}  \label{allcomp}
\sum_{n=1}^N \mu\big(B(\x,\psi(n))\big)   \;  \asymp \;   \sum_{n=1}^N \mu\big(R_n(\psi)\big) \;  \asymp \;  \Psi(N):= \sum_{n=1}^N \psi(n)^{\tau}  \, .
\end{equation}
Now with Remark~\ref{remarkthm1.5}\,(ii) in mind,  it follows that if $\mu$ is  $\tau$-Ahlfors regular  then for all $n \in \N$
\begin{equation}  \label{allcomp2}
 \hat{R}_n\big(C^{-1} \psi(n)^{\tau}\big)
 \; \subseteq   \;  R_n(\psi) \;   \subseteq  \;  \hat{R}_n\big(C \psi(n)^{\tau}\big)    \, ,
  \end{equation}
  where $C \ge 1 $ is the ``Ahlfors regular'' constant appearing in \eqref{tar}.  Then on making use of \eqref{allcomp} and
  \eqref{allcomp2},  it is easily verified that  Theorem~\ref{quantrec} implies  the   following  zero-full measure criterion  which validates \eqref{divcondsv1} whenever $\mu$ is equivalent to $\cH^{\tau}|_K$.  Indeed, it  coincides with the main result of  Baker $\&$ Farmer \cite{baker2021} discussed within the context of~\eqref{divcondsv1}.

   \begin{corollary} \label{toprovecor}
     Let $(\Phi,K,\mu,T)$ be a self-conformal system on $\R^d$ with $\mu$ being a Gibbs measure equivalent to $\cH^{\tau}|_K$ where $\tau:=\dimH K$. Let $\psi:\R\to\R_{\geq0}$ be a real positive function.  Then \begin{eqnarray}\label{0-1law}
		 \mu\left(R(\psi)\right)=
		\begin{cases}
			0 &\text{if}\ \  \sum_{n=1}^{\infty}\psi(n)^{\tau} <\infty,\\[2ex]
			1 &\text{if}\ \ \sum_{n=1}^{\infty} \psi(n)^{\tau}  =\infty.
		\end{cases}
	\end{eqnarray}
\end{corollary}

\noindent  As we have seen, the corollary follows in a fairly straightforward manner from Theorem~\ref{quantrec}. However, it would follow in an entirely trivial way if we were able to eliminate the dependence on $\eta > 0$ in the statement of Theorem~\ref{toprove}. At present, this is not possible, and so we must invoke Theorem~\ref{quantrec} as stated.

We now point out that in the case $\mu$ is equivalent to $\cH^{\tau}|_K$, beyond  implying  the above zero-full measure criterion,   Theorem~\ref{quantrec} can  also  be utilized to explicitly obtain information regarding the behaviour of the counting function \eqref{countdefst}.  In order to state precisely what exactly can be obtained,  we need to introduce the following notion of upper and lower densities.  Let $\psi:\R\to\R_{\geq0}$ be a real positive function.   Then, for each $\tau>0$, each probability measure $\mu$ on $\R^d$ and each $\x\in\R^d$,  we define the $\tau$-lower and $\tau$-upper densities of $\mu$ at $\x$ associated with $\psi$ by
\[
\Theta^{\tau}_*(\mu, \psi,\x):=\liminf_{n\to \infty}\frac{\mu(B(\x,\psi(n)))}{\psi(n)^{\tau}} \, ,  \quad \Theta^{*\tau}(\mu, \psi,\x):=\limsup_{n\to\infty}\frac{\mu(B(\x,\psi(n)))}{\psi(n)^{\tau}}.
\]
With this in mind, the following can be deduced directly from Theorem~\ref{quantrec}.    We provide the details of its deduction from Theorem~\ref{quantrec} in Section~\ref{weakquantcount}.

\begin{theorem} \label{quantcount}
   Let $(\Phi,K,\mu,T)$ be a self-conformal system on $\R^d$ with $\mu$ being a Gibbs measure equivalent to $\cH^{\tau}|_K$ where $\tau:=\dimH K$. Let $\psi:\R\to\R_{\geq0}$ be a real positive function such that $\psi(x) \to 0 $ as $x \to \infty$ and assume that $\sum_{n=1}^{\infty}\psi(n)^{\tau}
    $ diverges.  Then, for $\mu$-almost  all $ \x \in  K$
    \begin{eqnarray*}\label{hauscount}
    \begin{split}
\frac{\Theta^{\tau}_*(\mu,\psi,\x)}{\Theta^{*\tau}(\mu,\psi,\x)}
&\leq\liminf_{N\to\infty}
\frac{\sum_{n=1}^N\one_{B(\x,\psi(n))}(T^n\x)}{\sum_{n=1}^N\mu(B(\x,\psi(n)))}\\
&\leq\limsup_{N\to\infty}
\frac{\sum_{n=1}^N\one_{B(\x,\psi(n))}(T^n\x)}{\sum_{n=1}^N\mu(B(\x,\psi(n)))}
\leq\frac{\Theta^{*\tau}(\mu,\psi,\x)}{\Theta^{\tau}_*(\mu,\psi,\x)}  \, .
    \end{split}
    \end{eqnarray*}
\end{theorem}

 \medskip

 \noindent Clearly, this result is  weaker  than Theorem~\ref{toprove} whenever the main term in \eqref{recurrentcountHM} dominates. In such cases, Theorem~\ref{toprove} yields an asymptotic statement with a quantified error term, whereas Theorem~\ref{quantcount} provides, at best, an unquantified asymptotic statement. However, when the error term in \eqref{recurrentcountHM} dominates, Theorem~\ref{toprove} becomes ineffective, while Theorem~\ref{quantcount} may still yield meaningful information—and is potentially stronger. In particular, when the lower and upper densities of $\mu$ associated with $\psi$ coincide, we can derive the following asymptotic statement directly from Theorem~\ref{quantcount}.


\begin{corollary} \label{quantcountcor1}
   Let $(\Phi,K,\mu,T)$ be a self-conformal system on $\R^d$ with $\mu$ being a Gibbs measure equivalent to $\cH^{\tau}|_K$ where $\tau:=\dimH K$. Let $\psi:\R\to\R_{\geq0}$ be a real positive function such that $\psi(x) \to 0 $ as $x \to \infty$ and assume that $\sum_{n=1}^{\infty}\psi(n)^{\tau}
    $ diverges and that
     $ \Theta^{\tau}_*(\mu, \psi,\x) =  \Theta^{*\tau}(\mu, \psi,\x) $
    for $\mu$-almost  all $ \x \in  K$.    Then, for $\mu$-almost  all $ \x \in  K$
    \begin{eqnarray*}
    \begin{split}
\lim_{N\to\infty}
\frac{\sum_{n=1}^N\one_{B(\x,\psi(n))}(T^n\x)}{\sum_{n=1}^N\mu(B(\x,\psi(n)))} =1
    \end{split}
    \end{eqnarray*}
\end{corollary}

\medskip

We now consider the special case in which the Gibbs measure is absolutely continuous with respect to $d$-dimensional Lebesgue measure $\cL^d$.  For convenience,  let   $c_d:=\cL^d\big(B(0,1)\big)$ and suppose that $\mu$ is a Gibbs measure equivalent to $\cL^{d}|_K$ with density function $h$.
Then,   the  Lebesgue density theorem
implies  that for $\mu$-almost all $\x\in K$
\smallskip
\begin{equation}  \label{ldt}
    \mu\big(B(\x,\psi(n))\big)= \big(h(\x)+\epsilon_n(\x)\big)\cdot c_d  \, \psi(n)^d  \, ,
\end{equation}

\noindent where $\epsilon_n(\x)\to 0$ as $n\to\infty$.   The upshot of this is the following statement for  absolutely continuous  measures.  The first part is  a rewording  of Theorem~\ref{toprove} while the second part is a  rewording of  Corollary~\ref{quantcountcor1}


\begin{corollary}
\label{quantcountabcont}
     Let $(\Phi,K,\mu,T)$ be a self-conformal system on $\R^d$ and suppose that $\dimH K=d$. Let  $\mu$  be a  Gibbs measure equivalent to $\cL^{d}|_K$ with density function $h$.  Let $\psi:\R\to\R_{\geq0}$ be a real positive function such that $\psi(x) \to 0 $ as $x \to \infty$. Then the following are true.
\begin{itemize}
  \item[(i)] For any $\eta>0$ and $\epsilon>0$, we have
  \begin{eqnarray} \label{lebcount}
        \sum_{n=1}^{N}\one_{B(\x,\psi(n))}(T^n\x) \;
        &=&  \;   c_d h(\x)\Psi(N)   +   c_d\sum_{n=1}^N\epsilon_n(\x)\psi(n)^d  \nonumber  \\[1ex]  &~&  \hspace*{10ex} +\ O\left(\Psi_{\eta}(N)^{1/2}(\log\Psi_{\eta}(N))^{\frac{3}{2}+\epsilon}\right).
    \end{eqnarray}
    for $\mu$-almost all $\x\in K$, where   $\Psi_{\eta}(N) :=  \sum_{n=1}^N\psi(n)^{(1-\eta)d}  \, $,   $\Psi(N) := \Psi_{0}(N)  \, $  and  $\epsilon_n(\x)\to 0$ as $n\to\infty$ satisfies~\eqref{ldt}.\\

  \item[(ii)]  If  \ $\sum_{n=1}^\infty \psi(n)^{d}$ diverges,  then
   \begin{equation}  \label{ghjk}
\lim_{n \to \infty} \frac{\sum_{n=1}^{N}\one_{B(\x,\psi(n))}(T^n\x) }{\Psi(N) }  \, =  \, h(\x)  \quad \text{for $\mu$-almost all $\x\in K$.  }    \end{equation}
\end{itemize}

\end{corollary}


Note that  in general we do not have any information regarding the rate at which  $\epsilon_n(\x)\to0$, so it is not possible to compare the size of the second and third terms appearing on the right hand side of
\eqref{lebcount}.  However, if $\mu= \cL^{d}$ then $\epsilon_n(\x)= 0$ for all $n \in \N $ and $\x \in K$ and so the second term is zero.
With this in mind, it follows that Corollary~\ref{quantcountabcont} is in line with the main result established in \cite{levesley2024} for piecewise linear maps of $[0,1]^d$.   Furthermore, with Theorem~\ref{Main2} at our disposal,  the   asymptotic statement  \eqref{ghjk}
can be  directly derived from the recent impressive  work of He \cite{he2024}. In short, He obtains \eqref{ghjk}  for a class of  measure-preserving systems  for which $\mu$ is exponentially mixing and absolutely continuous with respect to  Lebesgue measure.

\bigskip

We bring this section to an end with a brief discussion concerning  the  recurrent problem  beyond self-conformal systems, or rather beyond the structure inherited by such systems.    In view of Theorem A, we know that  exponential mixing underpins the asymptotic behaviour of the counting function within the setup of the shrinking target problem. Currently, we see no obvious counterexample that  shows that this is not enough within the recurrent framework.  Adding a safety net, by  restricting to Hausdorff  measures, it remains plausible that the following  ``strengthening''  of Theorem~\ref{toprove} is true.   In short it would suggest  that the  key aspect of the system under consideration is that it is exponentially mixing and nothing else.

\noindent\textbf{Claim T.  \ }
 \emph{Let $(X,\mathcal{B},\mu,T)$ be a measure-preserving dynamical system in $\R^d$  with $\mu$ being a  $\tau$-Ahlfors regular measure where $\tau:=\dimH X$.  Let $\mathcal{C}$ be a collection of balls in $\R^d$ and suppose that $\mu$ is exponentially-mixing with respect to $(T,\mathcal{C})$.  Let $\psi:\R\to\R_{\geq0}$ be a real positive function such that $\psi(x) \to 0 $ as $x \to \infty$.  Then, for any given $\varepsilon>0$, we have that
 \begin{equation*}
        \sum_{n=1}^N\one_{B(\x,\psi(n))}(T^n\x)=\Psi(N,\x)+O\left(\Psi(N,\x)^{1/2}\log^{\frac{3}{2}+\epsilon}(\Psi(N,x))\right)
    \end{equation*}
    for $\mu$-almost all $\x\in K$, where
$\Psi(N,\x):=\displaystyle{\sum_{n=1}^N}\mu\big(B(\x,\psi(n))\big)  $ \, .  }

\bigskip

\noindent Several comments are in order.
\begin{enumerate}
\item[(i)]  Recall Remark~\ref{newrem}\,(i), namely that within the setup of self-conformal systems, the notion of  $\mu$ being equivalent to $\cH^{\tau}|_X$ and $\mu$ being $\tau$-Ahlfors regular coincide.
\medskip

\item[(ii)] Clearly, under the assumption that  $\mu$ is a  $\tau$-Ahlfors regular measure as in Claim~T,  we can replace the quantity  $\Psi(N,x)$ by $\sum_{n=1}^N\psi(n)^{\tau}  $  in the error term  and thus making it independent of $ \x \in K$.    The reason that we have not done this is that there is a possibility that the conclusion of the claim is true without the Ahlfors regular assumption and in such generality the error may depend on $\x \in K$; that is to say that  $\Psi(N,x)$ may not be comparable to a sum that is independent of $\x$.
\medskip

\item[(iii)] With the previous comment in mind, it is worth pointing out that \eqref{allcomp}  is in fact true under the hypothesis of Claim~T  (see \cite[Lemma~2.5]{hussain2022}).
Indeed, it is easily checked that all that is essentially required to  establish \eqref{allcomp}  is that $\mu$ is $\tau$-Ahlfors regular and that $\mu$ is exponentially-mixing.
\end{enumerate}


\noindent Even if Claim~T turns out to be false, it does not rule out the following strengthening of Corollary~\ref{toprovecor} which is of independent interest.

\noindent\textbf{Claim 0-\!1. }  \emph{Let $(X,\mathcal{B},\mu,T)$ be a measure-preserving dynamical system in $\R^d$  with $\mu$ being a  $\tau$-Ahlfors regular measure where $\tau:=\dimH X$.  Let $\mathcal{C}$ be a collection of balls in $\R^d$ and suppose that $\mu$ is exponentially-mixing with respect to $(T,\mathcal{C})$.  Let $\psi:\R\to\R_{\geq0}$ be a real positive function.  Then \begin{eqnarray}\label{claim0-1law}
		 \mu\left(R(\psi)\right)=
		\begin{cases}
			0 &\text{if}\ \  \sum_{n=1}^{\infty}\psi(n)^{\tau} <\infty,\\[2ex]
			1 &\text{if}\ \ \sum_{n=1}^{\infty} \psi(n)^{\tau}  =\infty.
		\end{cases}
	\end{eqnarray}}

\medskip

\noindent As already mentioned,  currently we see no obvious counterexample that  shows that Claim~T is false, let alone a counterexample to Claim 0-\!1.

\begin{remark}  \label{nonono}
As mentioned in the discussion leading up to Claim~T, the actual statement of the claim is erring on the side of caution. Indeed, we see no obvious counter example to either Claim~T or Claim~0-1 even if we remove the assumption that the measure $\mu$ is Ahlfors regular.  Obviously, without the latter assumption,  in Claim~0-1 we would  replace the sum appearing in  \eqref{claim0-1law} by  $\sum_{n=1}^{\infty}\mu\big(B(\x,\psi(n))\big)  $.  It is worth pointing out that   a relatively painless  calculation shows that within the context of Example~ABB, we have that
$$
\sum_{n=1}^{\infty}\mu\big(B(\x,\psi_{\alpha}(n))\big) \, <  \, \infty
$$
for $\mu$-almost all $\x \in K$ (see Proposition~\ref{propnocounter}  in Appendix~\ref{appendixABB}  for the details).
Thus,  Example~ABB is not a counterexample to the bolder statement in which the Ahlfors regular assumption is dropped.  Finally,  at the very basic level,  as far as we are aware, it is not known whether or not $\mu(R(\psi))$ satisfies a zero-one law; i.e.  $\mu(R(\psi))= 0 $ or $1$.
\end{remark}

\subsection{Organizations of the paper}
This paper is organized as follows. In Section \ref{Sec2}, we introduce the background knowledge, including concepts and basic results regarding  conformal maps, conformal  iterated
function schemes and Ruelle operators on symbolic spaces and  self-conformal sets. In Section \ref{Secrig}, we prove Theorem~\ref{thmrigidity} in a more general framework that does not require the open set condition (by definition, this is implicit in the framework of a self-conformal system). In addition, with reference to Remark~\ref{thy},  we provide a counterexample to Käenmäki's result in dimension two.
 In Section~\ref{SEC3}, we prove Theorem~\ref{Main} modulo  Theorem~\ref{meaannimpexp} and in Section~\ref{SEC3A} we prove the latter.\footnote{In Appendix~\ref{appendix:directproof}, we present a direct proof of Theorem~\ref{Main} using the exponentially mixing property for cylinder sets, thereby keeping the paper self-contained within the framework of self-conformal systems. }
 Sections \ref{proofM2} $\&$ \ref{annulusec}   are devoted to the proof of Theorem~\ref{Main2} -  our main result.
In  Section~\ref{SEC5}, we establish the statements presented  in  Section~\ref{appintro}  regarding the  applications of
Theorem~\ref{Main2} to the recurrent problem for  self-conformal dynamical systems. This involves establishing a more versatile form of the standard quantitative  Borel-Cantelli Lemma.

	\section{Self conformal systems:  the preliminaries}\label{Sec2}
 For convenience, various pieces of notation  that are  frequently used throughout the paper are listed below:
 \begin{itemize}
     \item $d\geq1$ is an integer. \vspace*{1ex}
     \item For any $\x=\big(x_1,...,x_d\big)\in\R^d$, denote by  $|\x|:=(x_1^2+\cdots+x_d^2)^{1/2}$ the Euclidean norm of $\x$.\vspace*{1ex}
     \item If $A$ is a $d\times d$ real matrix, the maximal norm of $A$ is denoted  by $$|A|:=\sup\big\{|A\x|:|\x|=1\big\}.$$ 
     \item For a function $f\in C^1(\Omega)$ on an open set $\Omega\subseteq\R^d$, the symbol $f'(\x)$  represents the Jacobian matrix of $f$ at $\x\in\Omega$.  It is also common to  use the notation  $D_{\x}f $.    \vspace*{1ex}
     \item The diameter of a set $E\subseteq\R^d$ under the Euclidean norm is denoted by  $|E|$, and we write $\overline{E}$ for the closure of $E$ in the topology induced by this norm. \vspace*{1ex}
     \item For any $\x\in\R^n$ and $r>0$, we use $B(\x,r)$ to denote the open ball centered at $\x$ with radius $r$ under the Euclidean norm.
 \end{itemize}

    \subsection{Conformal maps}

\begin{definition}
    Let $\Omega\subseteq\R^d$ be an open set. We say that  $f\in C^1(\Omega)$ is \emph{conformal} on $\Omega$ if $f$ is injective and
    \[
        f'(\x)\neq0  \quad \text{and}  \quad    |(f'(\x))(\y)|=|f'(\x)|\cdot|\y|,  \qquad \forall  \,  \x\in\Omega,~\forall \, \y\in\R^d.
    \]
\end{definition}

Let $\Omega\subseteq\R^d$ be a connected open set. We shall recall the rigidity of conformal maps on $\Omega$. When $d=1$, a map $f:\Omega\to\R$ is a conformal map  if and only if $f\in C^1(\Omega)$ and $f'(x)\neq0$ for all $x\in\Omega$. When $d=2$, if we  view $\R^2$ as the complex plane $\bC$, then an injective map $f:\Omega\to \bC$ is conformal if and only if $f$ is holomorphic (or anti-holomorphic) on $\Omega$. 
When $d\geq3$, by Liouville's theorem (see \cite[Section 3.8]{gehring2017}), a map $f:\Omega\to\R^d$ is conformal if and only if it is a restriction to $\Omega$ of a  Möbius transformation on $\overline{\R}^d:=\R^d\cup\{\infty\}$, that is
\[
    f(\x)=\mathbf{b}+\frac{c}{|\x-\mathbf{a}|^{\epsilon}}\cdot A(\x-\mathbf{a}),
\]
where $\mathbf{a},\mathbf{b}\in\R^d$, $c\in\R$, $\epsilon\in\{0,2\}$ and $A$ is a $d\times d$ orthogonal matrix. We end this section with the following useful result.
	\begin{lemma}\label{concon}
		Let $d\geq2$ be an integer and let $\Omega\subseteq\R^d$ be a bounded connected open set.   Suppose $\{f_n\}_{n\in\N}$ is a sequence of conformal maps on $\Omega$ such  that:
  \begin{enumerate}[label=(\roman*)]
      \item\label{concon1} there exists $C>1$ so that for any $n\geq1$ we have
      \[
          C^{-1}|\x-\y|\,\leq\,|f_n(\x)-f_n(\y)|\,\leq\, C\,|\x-\y| \, , \qquad \forall \; \x, \,  \y \in \Omega \,
      \]
      \item\label{concon2} $\{f_n\}_{n\in\N}$ is uniformly bounded on $\Omega$; that is to say that there exists $M>0$ so that for any $n\geq1$ we have
      $$|f_n(\x)|\leq M    \, , \qquad \forall \; \x \, \in \Omega   \, . $$
  \end{enumerate}
  Then,  there exist a subsequence $\{f_{n_k}\}\subseteq\{f_n\}$ and a conformal map $g$ on $\Omega$ such that $f_{n_k}\to g$ uniformly on $\Omega$.
	\end{lemma}
\begin{proof}
    Let $\{f_n\}_{n\in\N}$ satisfy the above conditions \ref{concon1} and \ref{concon2}. Then by the Arzelà–Ascoli Theorem, there exist a subsequence $\{f_{n_k}\}\subseteq\{f_n\}$ and a continuous map $g:\Omega\to\R^d$ such that $f_{n_k}\to g$ uniformly on $\Omega$. We now show  that $g$ is conformal.  When $d=2$, we view $\R^2$ as the complex plane $\bC$ and it follows from  \cite[Theorem \uppercase\expandafter{\romannumeral3} 1.3]{busam2009} that $g$ is holomorphic on $\Omega$. By condition \ref{concon1}, we know that $g'(z)\neq0$ for all $z\in \Omega$ and  hence $g$ is conformal. If $d\geq3$, then $\{f_n\}_{n\in\N}$ is a sequence of Möbius maps on $\overline{\R}^d$. Moreover, by combining condition \ref{concon1} and \cite[Corollary 3.6.6]{gehring2017},   the sequence $\{f_n\}_{n\in\N}$ is a normal family over $\overline{\R}^d$. So we can assume that $f_{n_k}\to g$ uniformly on $\overline{\R}^d$ under the chordal metric (see \cite[Page 7]{gehring2017} for the definition of chordal metric). Finally, we conclude that $g$ is a Möbius map by means of \cite[Theorem 3.6.7]{gehring2017} and thus conformal.
\end{proof}

	\subsection{$C^{1+\alpha}$ conformal IFS}  \label{cifs}
	Let $d\geq1$ be an integer. In this section, we introduce the definition of a $C^{1+\alpha}$ conformal IFS on $\R^d$ and bring  together  some simple but useful properties that are frequently used throughout the paper.

 \begin{definition}
     Let $\Omega\subseteq\R^d$ be an open set and $\alpha>0$. Given a function $f:\Omega\to\R^d$, we say that $f\in C^{1+\alpha}(\Omega)$ if $f\in C^1(\Omega)$ and $f'$ is $\alpha$-Hölder continuous on $\Omega$; that is, there exists some constant $C>0$ such that
	$$\big| \, |f'(\x)|-|f'(\y)|  \, \big| \, \leq  \,   C\,|\x-\y|^{\alpha}\, , \qquad \forall \; \x, \,  \y \in \Omega \, . $$
 \end{definition}

 \medskip

 \begin{definition}\label{IFSdef}
     Fix an integer $m\geq2$. We say that $\Phi=\{\varphi_j\}_{1\leq j\leq m}$ is a  \emph{conformal IFS on $\R^d$} if there exists a bounded connected open set $\Omega\subseteq\R^d$  such that each map  $\varphi_j$ is an injective and contractive  conformal map on $\Omega$ satisfying
     \begin{eqnarray}\label{varphi'contr}
         \overline{\varphi_j(\Omega)}\subseteq\Omega\,\,\,\,\,\,\,\,\text{and}\,\,\,\,\,\,\,\,0  \, <  \, \inf_{\x\in\Omega}|\varphi_j'(\x)| \; \leq \; \sup_{\x\in\Omega}|\varphi_j'(\x)|   \, < \, 1.
     \end{eqnarray}
     In particular, we say that $\Phi$ is a $C^{1+\alpha}$ conformal IFS if each $\varphi_j$ $(j=1,2,...,m)$ above belongs to $C^{1+\alpha}(\Omega)$.
 \end{definition}

 It is well known (Hutchinson \cite{hutchinson1981}) that there is a unique compact set $K\subseteq\Omega$ such that
      \begin{eqnarray}\label{defofselfconset}
    K=\bigcup_{j=1}^m\varphi_j(K)  \, .
      \end{eqnarray}
	We call this set $K$ the \emph{self-conformal set generated by $\Phi$}. In particular, we say that $\Phi$ satisfies the \emph{open set condition} (OSC) if there is a nonempty open set $V\subseteq\Omega$ such that $\varphi_j(V)\subseteq V$ and $\varphi_i(V)\cap\varphi_j(V)=\emptyset$ for any $i\neq j\in\{1,...,m\}$.

Here and throughout, let $\Sigma:=\{1,2,...,m\}$ denote a finite alphabet composed of $m$ elements. For any $n\in\N$, the set $\Sigma^n$  consists of  all  words of length $n$ over $\Sigma$, while $\Sigma^*$ represents the collection of all finite words as follows
 $$\Sigma^*:=\bigcup_{k\geq1}\Sigma^k\,.$$ Given a word $I=i_1...i_n$, we define the associated map $\varphi_I$ as the composition $$\varphi_I:=\varphi_{i_1}\circ\cdot\cdot\cdot\circ\varphi_{i_n}   \, ,$$
 and let $$ K_I:=\varphi_I(K)$$ denote the images of $K$ under $\varphi_I$.
The length of a finite word $I\in\Sigma^*$ is denoted by $|I|$.   The  cylinder set associated with $I=i_1\cdot\cdot\cdot i_n$ is defined as $$[I] \, = \, \{J=(j_1j_2\cdot\cdot\cdot)\in\Sigma^{\N}:j_1=i_1,...,j_{n}=i_{n}\}.$$ Additionally, the composition of two finite words $I=i_1i_2\cdot\cdot\cdot i_{n}$ and $J=j_1j_2\cdot\cdot\cdot j_k$ is given by
 \[
 IJ:=i_1\cdot\cdot\cdot i_nj_1\cdot\cdot\cdot j_k\,.
 \]

 \noindent The following statements hold for any $C^{1+\alpha}$ conformal IFS on $\R^d$:
	
	\begin{itemize}
		\item There exists $C_1>1$ such that
		\begin{eqnarray} \label{sv89}
			|\varphi_I'(\x)|\,\leq\, C_1\,|\varphi_I'(\y)|,\,\,\,\,\,\,\,\,\forall\, \x,\y\in\Omega,~\forall\, I\in\Sigma^*,
		\end{eqnarray}
		and hence
		\begin{eqnarray}\label{2.2}
			C_1^{-1}\|\varphi_I'\|\,\leq\,|\varphi_I'(\x)|\,\leq\, C_1\,\|\varphi_I'\|,\,\,\,\,\,\,\,\,\forall\, \x\in\Omega,~\forall \,I\in\Sigma^*,
		\end{eqnarray}
		where $\|f'\|:=\sup_{x\in\Omega}|f'(\x)|$. For a proof, see \cite[Lemma 2.1]{patzschke1997}. \vspace*{2ex}
		
		\item There exists a bounded  open set $U\subseteq\Omega$  and a constant $C_2>1$ such that
\begin{eqnarray*}
\overline{U}\subseteq\Omega,\,\,\,\,\varphi_{j}(U)\subseteq U\ \ \ \ \ (j=1,2,...,m)
\end{eqnarray*}
and
		\begin{eqnarray}\label{p2}	C_2^{-1}\|\varphi_I'\|\cdot|\x-\y|\,\leq\,|\varphi_I(\x)-\varphi_I(\y)|\,\leq\, C_2\,\|\varphi_I'\|\cdot|\x-\y|
		\end{eqnarray}
  for all $\x,\y\in U$ and all $I\in\Sigma^*$. For a proof, see \cite[Lemma 2.2]{patzschke1997}. \vspace*{2ex}
		
		\item It follows directly from (\ref{p2}) that there exists $C_3>1$ such that
		\begin{eqnarray}\label{p3}
			C_3^{-1}\|\varphi_I'\|\,\leq\,|K_I|\,\leq\, C_3\,\|\varphi_I'\|,\,\,\,\,\,\,\,\,\forall\, I\in\Sigma^*.
		\end{eqnarray}
        Therefore,  on letting
        \begin{eqnarray}\label{defkap}
            \kappa:=\max\big\{\|\varphi_j'\|:1\leq j\leq m\big\},
        \end{eqnarray}
         we have that
        \begin{eqnarray}\label{p3'}
            |K_I|\,\leq\, C_3\,\kappa^{|I|},\,\,\,\,\,\,\,\,\forall\, I\in\Sigma^*  \, .
        \end{eqnarray}

		\vspace*{2ex}

		\item By  (\ref{2.2}) and (\ref{p3}) and the fact that the equality $$|(f\circ g)'(\x)|=|f'(g(\x))|\cdot|g'(\x)|$$ holds for any $\x\in\Omega$ and any pair of conformal mappings $f,g:\Omega\to\Omega$, there exists $C_4>1$ such that
		\begin{eqnarray}\label{p4}
			C_4^{-1}|K_I||K_J|\,\leq\, |K_{IJ}|\,\leq\, C_4\,|K_I||K_J|,\,\,\,\,\,\,\,\,\forall\, I,J\in\Sigma^*.
		\end{eqnarray}
	\end{itemize}

\medskip

    \begin{remark}\label{weakrk}
        With reference to Definition \ref{IFSdef}, it can be verified that the second condition  in (\ref{varphi'contr}) can be weakened to  the statement: there exists $n_0\in\N$ such that
        \begin{eqnarray}\label{instead}
            0\,<\,\inf_{\x\in\Omega}|\varphi_I'(\x)|\,\leq\,\sup_{\x\in\Omega}|\varphi_I'(\x)|\,<\,1,\,\,\,\,\,\,\,\,\forall \,I\in\Sigma^{n_0}.
        \end{eqnarray}
        without effecting the results obtained in this paper.  In other words, our theorems  hold for   the corresponding  larger class of self conformal systems coming from the self-conformal IFS $\{\varphi_I\}_{I\in\Sigma^{n_0}}$.  For a concrete example of this  see Remark~\ref{rembigger}.
    \end{remark}

    \begin{remark}
       Let $\Phi=\{\varphi_j:\Omega\to\Omega\}_{1\leq j\leq m}$ be a conformal IFS (not necessarily $C^{1+\alpha}$) on $\R^d$ with $d\geq2$. It is easy to find a bounded connected open set $U\subseteq\Omega$ such that
       \[
       \overline{U}\subseteq\Omega,\ \ \ \ \varphi_j(U)\subseteq U\quad(j=1,2,...,m)\,.
       \]
        Then, by \cite[Proposition 4.2.1]{mauldin2003}, each $\varphi_j$ ($j=1,2,...,m$) is $C^{1+\alpha}$ on $U$ with $\alpha=1$.
    \end{remark}

	\subsection{Ruelle operators on symbolic spaces}\label{ROonsymbolic}
	
	In this section, we introduce Ruelle operators on symbolic spaces. Fix an integer $m\geq2$, let $\Sigma=\{1,2,...,m\}$. Define a metric on $\Sigma^{\N}$ as
	\begin{equation}  \label{qwer} \dist(I,J):=m^{-\sup\left\{k\geq1:i_1=j_1,...,i_k=j_k\right\}},\,\,\,\,\,\,\,\,\forall\, I=i_1i_2\cdot\cdot\cdot,~J=j_1j_2\cdot\cdot\cdot\in\Sigma^{\N},  \end{equation}
	where we set $\sup\emptyset:=0$ and $m^{-\infty}:=0$. It is well known that  $\dist(\cdot,\cdot)$ is an ultrametric and that $(\Sigma^{\N},\dist)$ is a compact metric space. Let $\sigma:\Sigma^{\N}\to\Sigma^{\N}$ be the shift map on $\Sigma^{\N}$, that is $\sigma(i_1i_2\cdot\cdot\cdot)=i_2i_3\cdot\cdot\cdot$ for any $i_1i_2\cdot\cdot\cdot\in\Sigma^{\N}$.
	
	Given $g\in C(\Sigma^{\N})$, we define  $\cS:C(\Sigma^{\N})\to C(\Sigma^{\N})$ to be  the \emph{Ruelle operator with potential $g$} by setting
	$$\cS f(I) \   :=\! \sum_{J\in\sigma^{-1}(I)}g(J)f(J),   \qquad \forall\, f\in C(\Sigma^{\N}),  \  \ \forall\, I\in\Sigma^N.$$
	It can be verified that the iterates of $\cS$ can be  written as
	$$\cS^n f(I) \ = \! \sum_{J\in\sigma^{-1}(I)}g^{(n)}(J)f(J),  \qquad \forall\, n\in\N, \ \forall\, f\in C(\Sigma^{\N}), \ \forall\, I\in\Sigma^{\N},$$
	where $g^{(n)}(I):=g(I)g(\sigma I)\cdot\cdot\cdot g(\sigma^{n-1}I)$ for $I\in\Sigma^{\N}$.
	For any $f\in C(\Sigma^{\N})$, let $$\|f\|_{\infty}:=\sup_{I\in\Sigma^{\N}}|f(I)|\,.$$ The norm of $\cS$ is defined as $$\norm{\cS}_{\infty}:=\sup\big\{\|\cS f\|_{\infty}:f\in C(\Sigma^{\N}),~\|f\|_{\infty}=1\big\}.$$ It is clear that $\norm{\cS}_{\infty}\leq m\|g\|_{\infty}$ and hence $\cS$ is a bounded linear operator on the Banach space $(C(\Sigma^{\N}),\norm{\cdot}_{\infty})$. We define the \emph{spectral radius} of $\cS$ as
    \begin{eqnarray}\label{spera}
R:=\lim_{n\to\infty}\|\cS^n\|_{\infty}^{1/n}.
    \end{eqnarray}
    This limit exists since  $\|\cS^{n+m}\|_{\infty}\leq\|\cS^{n}\|_{\infty}\|\cS^m\|_{\infty}$  (see \cite[Corollary 1.3]{falconer1997}).
	
	Let $\cM(\Sigma^{\N})$ be the collection of all finite Borel signed measures. By the  Riesz Representation Theorem \cite[Theorem 7.17]{folland1999}, we know that $\cM(\Sigma^{\N})$ can be viewed as the dual space of $C(\Sigma^{\N})$. We then define the dual operator $\cS^*:\cM(\Sigma^{\N})\to\cM(\Sigma^{\N})$ of $\cS$ by setting
    \begin{eqnarray}\label{daulsym}
\langle\cS^*\nu,f\rangle=\langle\nu,\cS f\rangle
    \end{eqnarray}
	for any $\nu\in\cM(\Sigma^{\N})$ and $f\in C(\Sigma^{\N})$, where $\displaystyle\langle \nu,f\rangle:=\int_{\Sigma^{\N}} f~\td\nu$.
	
	Let $\beta>0$. Denote by $\cC^{\beta}(\Sigma^{\N})$ the collection of all $\beta$-Hölder continuous functions.  
 We are now in the position to introduce Ruelle's Theorem. In short, it provides us with the existence of   `good' measures on $\Sigma^{\N}$.

	\begin{theorem}[Ruelle {\cite[Theorem 1.5]{baladi2000}}]\label{Theorem 1.5}
		Let $\beta>0$, let $g:\Sigma^{\N}\to\R_{>0}$ be a positive $\beta$-Hölder continuous function,
 let $\cS$ be the Ruelle operator with potential $g$ and let $R$ be the spectral radius of $\cS$. Then, there  exist unique positive $h\in\cC^{\beta}(\Sigma^{\N})$ and Borel probability measure $\nu\in\cM(\Sigma^{\N})$ such that
		\begin{eqnarray}\label{eigen}
			\cS h=Rh, \quad \cS^*\nu=R\nu,  \quad \int_{\Sigma^{\N}}h  \ \td\nu=1.
		\end{eqnarray}
	\end{theorem}
The theorem naturally enables us to define ``good'' measures on symbolic spaces.

    \begin{definition}[Gibbs measure on $\Sigma^{\N}$] \label{gibbssigma}
       Let $\beta>0$, let $g:\Sigma^{\N}\to\R_{>0}$ be a positive $\beta$-Hölder continuous function and
 let $\cS$ be the Ruelle operator with potential $g$.   Let $h$ be the eigenfunction of $\cS$ and $\nu$ be the eigenmeasure of $\cS$ as in (\ref{eigen}).
    We define the \emph{Gibbs measure $\mu$ with respect to the $\beta$-Hölder potential $g$} (or briefly, a \emph{Gibbs measure} when $\beta$ and $g$ are not explicitly relevant) on $\Sigma^{\N}$ by
    $$\td\mu:=h~\td\nu   \, . $$
    \end{definition}
	
	 We now list various  useful elementary  properties regarding   the Gibbs measure $\mu$:
	\begin{itemize}
		\item Let $\widetilde{\cS}$ denote  the normalized Ruelle operator; that is  $$\widetilde{\cS} f(I) \ := \! \sum_{J\in\sigma^{-1}(I)}\tilde{g}(J)f(J)  \quad \text{where}   \quad \tilde{g}:=\frac{1}{R}\cdot\frac{g\cdot h}{ h\circ\sigma} \, . $$  Then we have
		\begin{eqnarray}\label{normal}
			\widetilde{\cS}1\equiv 1   \quad {\rm and } \quad \widetilde{\cS}^*\mu=\mu.
		\end{eqnarray}
		\item $\mu$ is $\sigma$-invariant.  \vspace*{2ex}
		\item (\emph{Gibbs property})  There exists $C_5>1$ such that
		\begin{eqnarray}\label{gibbs}
			C_5^{-1}\tilde{g}^{(|I|)}(J)\,\leq\,\mu([I])\,\leq\, C_5\,\tilde{g}^{(|I|)}(J), \qquad \forall\, I\in\Sigma^*,  \ \forall\, J\in[I]  \, .\vspace{2ex}
		\end{eqnarray}

		\item (\emph{quasi-Bernoulli property}) There exists $C_6>1$ such that
		\begin{eqnarray}\label{quassiber}
			C_6^{-1}\mu([I])\mu([J])\,\leq\,\mu([IJ])\,\leq\, C_6\,\mu([I])\mu([J]),\,\,\,\,\,\,\,\,\forall\, I,J\in\Sigma^*.
		\end{eqnarray}
		
	\end{itemize}
	The identity (\ref{normal}) and the $\sigma$-invariance of $\mu$ can be verified directly by definition. For a proof of the inequality (\ref{gibbs}), we refer to
	\cite{bowen2008}. The quasi-Bernoulli property follows from (\ref{gibbs}) and the fact that $\tilde{g}^{(|I|)}(J_1)\asymp \tilde{g}^{(|I|)}(J_2)$ for any $I\in\Sigma^*$ and any $J_1,J_2\in[I]$. The latter, can be proved by making use of the fact that $\tilde{g}$ is positive    and the $\beta$-Hölder continuity of $\log\tilde{g}$. 

    The following well-known result states that any Gibbs measure on $\Sigma^{\N}$ exhibits exponentially decay of correlations for $\beta$-Hölder continuous functions. To state the result, we  define the $\beta$-Hölder norm for $f\in\cC^{\beta}(\Sigma^{\N})$ as
    \[
    \|f\|_{\beta}:=\|f\|_{\infty}+\sup\left\{\frac{|f(I)-f(J)|}{\dist(I,J)^{\beta}}:I,J\in\Sigma^{\N},\,I\neq J\right\}.
    \]

    \begin{theorem}[{\cite[Theorem 1.6]{baladi2000}}]\label{expmixforhol}
        Let $\beta>0$ and let $\mu$ be a Gibbs measure with respect to a $\beta$-Hölder potential  on $\Sigma^{\N}$. Then there exist $C>0$ and $\gamma\in(0,1)$ such that
        \begin{eqnarray*}
            \left|\int_{\Sigma^{\N}}f_1\cdot f_2\circ\sigma^n\,\td\mu-\int_{\Sigma^{\N}}f_1\,\td\mu\cdot\int_{\Sigma^{\N}}f_2\,\td\mu\right|\,\,\leq\,\, C\,\gamma^n\cdot\|f_1\|_{\beta}\cdot\int_{\Sigma^{\N}}|f_2|\,\td\mu
        \end{eqnarray*}
        for any  $f_1\in\cC^{\beta}(\Sigma^{\N})$,  any $f_2\in L^1(\mu)$ and   any $ n \in \N$.
    \end{theorem}

       The authors in \cite{fan2010} utilized Theorem~\ref{expmixforhol} to prove the following result which says  that any  Gibbs measure  on $\Sigma^{\N}$ is exponentially mixing with respect to $(\sigma,\cC)$, where $\cC$ is the collection of all balls in $\Sigma^{\N}$.

 \begin{theorem}[{\cite[Proposition 7.2]{fan2010}}]\label{empsym}
     Let $\mu$ be a  Gibbs measure
     on $\Sigma^{\N}$. Then there exist $C>0$ and $\gamma\in(0,1)$ such that
     \begin{eqnarray*}
         \big|\mu\big([I]\cap\sigma^{-n}F\big)-\mu([I])\mu(F)\big|\,\,\leq\,\, C\gamma^n\mu(F)
     \end{eqnarray*}
     for  any $I\in\Sigma^*$,  any measurable set $F\subseteq\Sigma^{\N}$ and any $n \in \N$.
 \end{theorem}

   This concludes our discussion concerning Ruelle operators and associated Gibbs measures on symbolic spaces.  We now turn our attention to Ruelle operators on  self-conformal sets.

	\subsection{Ruelle operators on $C^{1+\alpha}$ conformal IFS}\label{ruelle conformal}
 Let $m\in\N_{\geq2}$, let $d\in\N$, let $\alpha>0$ and let $\alpha'>0$. Throughout this section,  let $$\Phi=\{\varphi_j:\Omega\to\Omega\}_{1\leq j\leq m}$$ be a $C^{1+\alpha}$ conformal IFS on $\R^d$,  let $K\subseteq\R^d$ be the self-conformal set generated by $\Phi$, and let $$\{g_j:\varphi_j(\Omega)\to\R_{>0}\}_{1\leq j\leq m}$$ be  positive $\alpha'$-Hölder continuous functions. Define the Ruelle operator $\cL:C(K)\to C(K)$ with  potentials $\{g_j\}_{1\leq j\leq m}$ by setting
 \begin{eqnarray}\label{RuelleRd}
     \left(\cL f\right)(\x):=\sum_{j=1}^{m}g_j\left(\varphi_j(\x)\right)f\left(\varphi_j(\x)\right)
 \end{eqnarray}
 where  $f\in C(K)$  and $\x\in K$. The \emph{spectral radius} of $\cL$ is defined similarly as in (\ref{spera}) with $\cS$ replaced by $\cL$ and $\Sigma^{\N}  $ replaced by $K$. Let $\cM(K)$ be the collection of all finite Borel signed measures on $K$. The dual operator $\cL^*:\cM(K)\to\cM(K)$ corresponding to $\cL$ is defined similarly as in (\ref{daulsym}).   We now construct the Ruelle operator $\cS$ on the symbolic space $\Sigma^{\N}=\{1,...,m\}^{\N}$ associated with $\cL$ via the functions  $\{g_j\}_{1\leq j\leq m}$. In turn,  we  investigate the relationship between the two operators $\cL$ and $\cS$.

Fix $\x_0\in K$. The \emph{coding map  associated} to $\Phi$,  denoted by $\pi:\Sigma^{\N}\to K$, is defined as  follows:
\begin{equation}    \label{cmap} \pi(I):=\lim_{n\to\infty}\varphi_{i_1}\circ\cdot\cdot\cdot\circ\varphi_{i_n}(\x_0)   \quad \text{where}    \ \ I=i_1i_2\cdot\cdot\cdot\in\Sigma^{\N}.
\end{equation}
 The above limit exists  since the maps  $\varphi_j$  $(1\leq j\leq m)$ are contractive, and  its value is independent of the choice of $\x_0\in K$.
The Ruelle operator $\cS:C(\Sigma^{\N})\to C(\Sigma^{\N})$ associated with $\{g_j\}_{1\leq j\leq m}$  (or equivalently induced by $\cL$) is defined by
\begin{eqnarray}\label{inducedruelle}
    \cS \barbelow{f}(I):=\sum_{J\in\sigma^{-1}(I)}\barbelow{g}(J)\barbelow{f}(J)\,\qquad \forall\, \barbelow{f}\in C(\Sigma^{\N}),\,\,\forall\, I\in\Sigma^N,
 \end{eqnarray}
 where we set
 \begin{eqnarray*}
    \barbelow{g}(I):=g_{i_1}(\pi(I)),    \qquad \forall\,I=i_1i_2\cdot\cdot\cdot\in\Sigma^{\N}.
 \end{eqnarray*}

\noindent A direct calculation yields that $\barbelow{g}$   is a $\beta$-Hölder continuous function on $\Sigma^{\N}$, where
\[
\beta:=\frac{-\alpha'\log\kappa}{\log m}\qquad \text{ and  \ $\kappa\in(0,1)$  \ is as in \eqref{defkap}}.
\]
It turns out that $\cL$ and $\cS$ should have some relations. Indeed,
it can be verified that  the operators  $\cL$ and $\cS$ are  naturally related  via the coding map in the following manner:
\begin{eqnarray}\label{relofsandl}
    \cS(f\circ\pi)=(\cL f)\circ\pi  \qquad  \, \forall \,  f\in C(K) \, .
\end{eqnarray}
With this at hand, one can show that  $\cL$ and $\cS$  share the same spectral radius which we denote by $R$. According to Theorem~\ref{Theorem 1.5}, there exist unique positive  $\ubar{h}\in\cC^{\beta}(\Sigma^{\N})$ and Borel probability measure $\ubar{\nu}\in\cM(\Sigma^{\N})$ such that
\begin{eqnarray}\label{eigensym}
    \cS \ubar{h}=R\,\ubar{h},  \quad \cS^*\ubar{\nu}=R\,\ubar{\nu}, \quad \int_{\Sigma^{\N}}  \ubar{h}  \ \td\ubar{\nu}=1.
\end{eqnarray}

In  \cite{fan1999},  Fan and Lau established the existence and uniqueness of the eigenfunction and eigenmeasure of the operator $\cL$  and also revealed  their respective  relationship to $\ubar{h}$ and $\ubar{\nu}$.  More precisely, this is summarised by the following statement.  Clearly the first part is the analogue of Theorem~\ref{Theorem 1.5}  for self-conformal sets.

\begin{theorem}[{\cite[Proof of Theorem 1.1; Theorem 2.2]{fan1999}}]\label{ruellethmrd}
  Let $\cL$ be a Ruelle operator defined as in (\ref{RuelleRd}), and let $R$ be the spectral radius of $\cL$.  Let $\cS$ be the Ruelle operator on $\Sigma^{\N}$ induced by $\cL$ (see (\ref{inducedruelle})), and let
  $\ubar{h}\in\cC^{\beta}(\Sigma^{\N})$ and $\ubar{\nu}\in\cM(\Sigma^{\N})$
  be as in (\ref{eigensym}). Then we have:
    \begin{enumerate}[label=(\roman*)]
        \item\label{2.3.1} There exist unique positive  $h\in C(K)$ and Borel probability measure $\nu\in\cM(K)$ such that
        \[
            \cL h=R\,h,  \quad \cL^*\nu=R\,\nu,  \quad \int_{K}h~\td\nu=1.
        \]
        \item\label{2.3.2} Let $h$ and $\nu$ be the function and measure obtained  in \ref{2.3.1}.  Then $$\ubar{h}=h\circ\pi   \quad  { and } \quad \nu=\ubar{\nu}\circ\pi^{-1}   \, . $$
        \item\label{2.3.3} Let $\nu$ be the measure obtained in \ref{2.3.1}. If  $\Phi$ (the IFS)  satisfies the  open set condition (OSC), then for all $I,J\in\Sigma^*$ with  $|I|=|J|$ and $I\neq J$ we have that $$\nu(K_I\cap K_J)=0   \, . $$
     \end{enumerate}
\end{theorem}

\medskip

In the same way that  Theorem~\ref{Theorem 1.5}   enables us to naturally define ``good'' measures on symbolic spaces, the above theorem enables us to define  ``good'' measures on self-conformal sets.

\begin{definition}[Gibbs measure on $K$]\label{Gibbsmeasonk}
   Let $\alpha'>0$, let $\{g_j:\varphi_j(\Omega)\to\R_{>0}\}_{1\leq j\leq m}$ be positive $\alpha'$-Hölder continuous functions and  let $\cL$ be the Ruelle operator with respect to  $\{g_j\}_{1\leq j\leq m}$ as defined by  (\ref{RuelleRd}). Let $h$ be the eigenfuction of  $\cL$ and $\nu$ be the eigenmeasure of $\cL$  as in  part \ref{2.3.1} of Theorem~\ref{ruellethmrd}.  We define the
  \emph{Gibbs measure $\mu$ with respect to the $\alpha'$-Hölder potentials $\{g_j\}_{1\leq j\leq m}$} (or briefly, a \emph{Gibbs measure}) on $K$ by
\begin{eqnarray}\label{defofgibbs}
     \td{\mu}:=h~\td{\nu}.
 \end{eqnarray}
\end{definition}

\medskip

Now let  $\mu$ be the Gibbs measure with respect to the $\alpha'$-Hölder potentials $\{g_j\}_{1\leq j\leq m}$. In turn, let $\ubar{h}\in\cC^{\beta}(\Sigma^{\N})$ and $\ubar{\nu}\in\cM(\Sigma^{\N})$ be the eigenfunction and eigenmeasure of the Ruelle operator $\cS$ associated with $\{g_j\}_{1\leq j\leq m}$ (see (\ref{eigensym})), and let $\td\ubar{\mu}:=\ubar{h}~\td\ubar{\nu}$ be the Gibbs measure on $\Sigma^{\N}$ with respect to the $\beta$-Hölder potential $\barbelow{g}$ (see Definition~\ref{gibbssigma}). Then, as a consequence of part \ref{2.3.2} of  Theorem~\ref{ruellethmrd},  we have that
 \begin{eqnarray}\label{mubarmupi}
     \mu=\ubar{\mu}\circ\pi^{-1}.
 \end{eqnarray}
 The upshot is that any Gibbs measure $\mu$ on $K$ is an image measure of a Gibbs measure $\ubar{\mu}$ on the symbolic space $\Sigma^{\N}$ under the coding map.

In the following, we assume that $\Phi$ satisfies the OSC.  Let $\widetilde{K}$ be the set of points with unique symbolic representation; that is
$$\widetilde{K}:=\left\{\x\in K:\#\big(\pi^{-1}(\x)\big)=1\right\}.$$  Then by part \ref{2.3.3} of Theorem~\ref{ruellethmrd}, we have that
 \begin{eqnarray}\label{uniqueword}
     \mu\big(\widetilde{K}^c\big)=0.
 \end{eqnarray}
With this in mind, consider the map  $T:\R^d\to \R^d$ defined by
\begin{eqnarray}\label{defoft}
    T(\x):=\left\{
    \begin{aligned}
        &\pi\circ\sigma\circ\pi^{-1}(\x)    \quad \text{if} \ \  \  \x\in\widetilde{K},\\[1ex]
        &\x    \quad  \  \text{if} \ \ \  \x\notin \widetilde{K}.
    \end{aligned}\right.
\end{eqnarray}

 \noindent By (\ref{mubarmupi}), (\ref{uniqueword}) and the $\sigma$-invariance of $\ubar{\mu}$ (see Section \ref{ROonsymbolic}), it can be  verified  that $\mu$ is $T$-invariant. Indeed, since $\mu=\ubar{\mu}\circ\pi^{-1}$, then (\ref{uniqueword}) implies that $\ubar{\mu}((\pi^{-1}\widetilde{K})^c)=0$. Moreover, by the  definitions of $\widetilde{K}$ and $T$, we have  that
 $$\widetilde{K}\cap T^{-1}F=\pi^{-1}\widetilde{K}\cap\sigma^{-1}(\pi^{-1}F)$$
for any  $\mu$-measurable set $F\subseteq \R^d$. This together with the fact that $\ubar{\mu}$ is $\sigma$-invariant in mind, implies  that
\begin{eqnarray*}
    \mu(T^{-1}F)&=&\mu(\widetilde{K}\cap T^{-1}F)\\[4pt]
    &=&\ubar{\mu}(\pi^{-1}\widetilde{K}\cap\sigma^{-1}(\pi^{-1}F))\\[4pt]
    &=&\ubar{\mu}(\sigma^{-1}(\pi^{-1}F))\\[4pt]
    &=&\ubar{\mu}(\pi^{-1}F)\\[4pt]
    &=&\mu(F)
\end{eqnarray*}
 as desired. Next, given any $n\in\N$, any $I\in\Sigma^*$ and any $\mu$-measurable subset $F\subseteq \R^d$, the following relation  follows directly via the  definitions of $\widetilde{K}$ and $T$
$$\pi^{-1}\left(K_I\cap T^{-n}F\cap\widetilde{K}\right)=[I]\cap\sigma^{-n}(\pi^{-1}F)\cap\pi^{-1}\widetilde{K}.$$
Then on using (\ref{uniqueword}), we find that
\begin{eqnarray}\label{mukicaptnf}
\mu\big(K_I\cap T^{-n}F\big)=\ubar{\mu}\left([I]\cap\sigma^{-n}(\pi^{-1}F)\right).
\end{eqnarray}
In particular,  if we put  $F=K$ in (\ref{mukicaptnf}), it follows that
\begin{eqnarray}\label{muubarmu}
    \mu(K_I)=\ubar{\mu}([I]),   \qquad \forall\, I\in\Sigma^*\,.
\end{eqnarray}

The upshot of the above is that on combining (\ref{mubarmupi}),  (\ref{mukicaptnf}), (\ref{muubarmu}) with the $\sigma$-invariance of $\ubar{\mu}$ and Theorem~\ref{empsym}, we obtain the following statement  that plays a crucial role in the ``direct'' proof of Theorem~\ref{Main} given in Appendix \ref{appendix:directproof}. It can be viewed as the analogue of Theorem~\ref{empsym}  for self-conformal sets.

\begin{corollary}\label{emprd}
    Let $\Phi$ be a $C^{1+\alpha}$ conformal IFS satisfying the OSC on $\R^d$, let $K$ be the self-conformal set generated by $\Phi$, let $\mu$ be a  Gibbs measure supported on $K$ and let $T$ be defined as in  (\ref{defoft}). Then there exist $C>0$ and $\gamma\in(0,1)$ such that
 \begin{eqnarray*}
     \left|\mu\big(K_I\cap T^{-n}F\big)-\mu(K_I)\mu(F)\right|\leq C\gamma^n\mu(F)
 \end{eqnarray*}
 for any $n\geq1$, any $I\in\Sigma^*$ and any $\mu$-measurable set $F\subseteq \R^d$.
\end{corollary}

For the sake of completeness, we end this section by  describing  two concrete and well known  examples of  Gibbs measures on self-conformal sets $K$. The first example shows  that   any self-similar measure is a Gibbs measure.  The  second example shows  the existence of Gibbs measures that are  equivalent to the restricted Hausdorff measure on $K$.\\

\begin{example}
    Let $\Phi=\{\varphi\}_{1\leq j\leq m}$ be a self-similar IFS on $\R^d$ and $K$ be the self-similar set generated by $\Phi$. Let $\p=(p_1,...,p_m)$ be a probability vector. Then $\p$ induces on $K$  the  self-similar measure
    \[
        \mu=\sum_{j=1}^{m}p_j\,\mu\circ \varphi_j^{-1}\,.
    \]
     Let $\cL:C(K)\to C(K)$ be the Ruelle operator defined by
    \[\left(\cL f\right)(\x):=\sum_{j=1}^{m}p_jf\left(\varphi_j(\x)\right).\]
    Then  a straightforward calculation shows  that the spectral radius of $\cL$ is $1$ and
    \[
        \cL 1\equiv 1     \quad {\rm and }  \quad \cL^*\mu=\mu.
    \]
   The upshot is that any self-similar measure is a  Gibbs measure. \\

    \begin{example}\label{hausgibbs}
        Let $\Phi=\{\varphi_j\}_{1\leq j\leq m}$ be a $C^{1+\alpha}$ conformal IFS on $\R^d$ satisfying the OSC, and let $K$ be the self-conformal set generated by $\Phi$. Let $\tau:=\dimH K$. Let  $\cL:C(K)\to C(K)$ be the the Ruelle operator defined by
        \[
            \left(\cL f\right)(\x) :=\sum_{j=1}^{m}|\varphi_j'(\x)|^{\tau}f(\varphi_j(\x)).
        \]
Then by \cite[Theorems 2.7 and 2.9]{fan1999}, the spectral radius of $\cL$ is $1$
        and the unique eigenmeasure $\nu \in \cM(K)$  (the existence is guaranteed by part (i) of Theorem~\ref{ruellethmrd}) is given by the normalized restricted Hausdorff measure
        $$\nu:=\frac{\cH^{\tau}|_K}{\cH^{\tau}(K)}    \qquad  {\rm which \  satisfies} \qquad \cL^*\nu=\nu   \, . $$
        By definition, the  corresponding Gibbs measure $\mu$  is given by  $\td\mu=h\,\td\nu$ for some positive continuous function $h\in C(K)$ and it is easily verified that  there exists $C>1$ such that
        \[
            C^{-1}\cH^{\tau}(E)\,\leq\,\mu(E)\,\leq\, C\,\cH^{\tau}(E)    \qquad \text{for all Borel sets}~E\subseteq K.
        \]
In other words, $\mu$ is  equivalent to $ \cH^{\tau}|_K$ - the restricted Hausdorff measure on $K$.

    \end{example}
\end{example}

  \section{Self-conformal sets: rigidity without  separation conditions and  the proof of Theorem~\ref{thmrigidity}}\label{Secrig}

In this section, we will prove Theorem~\ref{thmrigidity}. As alluded in Remark~\ref{thy}, and indeed the title of this section,   we will in fact  prove a stronger statement that does not require the OSC assumption that is implicit in the   definition of a self-conformal system.      Let $\ubar{\mu}$ be a Borel probability measure on the symbolic space $\Sigma^{\N}=\{1,2,...,m\}^{\N}$.  Recall, that  $(\Sigma^{\N}, \dist) $   is a compact metric space  where $\dist$ is given by \eqref{qwer}.   With this in mind,  we say that $\ubar{\mu}$ is \emph{doubling} if there exists $C>1$ such that
\[
0<\ubar{\mu}\big(B(I,2r)\big)\leq C\,\ubar{\mu}\big(B(I,r)\big)<\infty \qquad \forall\ I\in\Sigma^{\N}\quad\text{and}\quad  r>0\,.
\]
It can be verified that this standard definition is equivalent to the statement that  there exists $\eta\in(0,1)$ such that
 \begin{equation}  \label{nsdoub}
 \ubar{\mu}([Ij]) \ >\ \eta\cdot\ubar{\mu}([I])    \qquad   \forall \quad  I\in\Sigma^*    \quad {\rm  and }  \quad  1\leq j\leq m \, , \end{equation}
 \noindent where $[Ij]$ denotes the cylinder set
 \[
 \left\{J=j_1j_2\cdot\cdot\cdot\in\Sigma^{\N}:\,j_1j_2\cdot\cdot\cdot j_{|I|}=I,\,\,j_{|I|+1}=j\,\right\}.
 \]

A direct consequence of (\ref{quassiber})  is that any Gibbs measure on $\Sigma^{\N}$ is doubling and so the following statement  implies Theorem~\ref{thmrigidity}.

\medskip

 \begin{proposition}\label{rigidity}
     Let $\Phi=\{\varphi_j\}_{1\leq j\leq m}$ be a $C^{1+\alpha}$ conformal IFS (without any separation condition) on $\R^d$ with $d\geq2$, let $K$ be  the self-conformal set generated by $\Phi$ and let $\pi$ be the coding map (see \eqref{cmap}). Let  $\ubar{\mu}$ be a doubling Borel probability measure on $\Sigma^{\N}=\{1,...,m\}^{\N}$ and $\mu=\ubar{\mu}\circ \pi^{-1}$. Given any $\ell\in\{1,...,d-1\}$, then one of the following statements hold:
\begin{enumerate}[label=(\roman*)]
         \item $\mu(K\cap M)=0$ for any  $\ell$-dimensional $C^1$ submanifold $M\subseteq\R^d$;
         \vspace*{1ex}
 \item\label{2} $K$ is contained in a $\ell$-dimensional affine subspace or $\ell$-dimensional geometric sphere;
  \vspace*{1ex}
 \item\label{3} $K$ is contained in a finite disjoint union of analytic curves and this may happen  only when $d=2$ and $\ell=1$.
     \end{enumerate}
 \end{proposition}

Before giving the proof we mention a useful generic  fact concerning submanifolds of $\R^d$ that will be exploited in the course of establishing the proposition.   Let $\ell\in\{1,...,d-1\}$ and let $M\subseteq\R^d$ be a $\ell$-dimensional $C^1$ submanifold. Given $\p\in M$, there exists an open set $U\subseteq M$ (under the topology of $M$) with $\p\in U$ such that we can find a diffeomorphism $\varphi: U\to\varphi(U)$  from $U$ into an open subset of $\R^{\ell}$. Let $$\psi:=\varphi^{-1}:\varphi(U)\to U$$ be the inverse map of $\varphi$ on $\varphi(U)$. Then, the tangent space of $M$ at $\p$  is given by
\begin{equation} \label{tanny}
       T_{\p} M:=\left\{\big(\psi'(\varphi(\p))\big)(\x):~\x\in\R^{\ell}\right\}.
\end{equation}
 It is easy to check that this definition is independent of the choice of the diffeomorphism $\varphi$. In the proof of Proposition \ref{rigidity}, we will utilize the following simple observation which follows directly on utilizing   Taylor's formula: \emph{For any $\p\in M$ and any $\epsilon>0$, there exists $r_0=r_0(\p,\epsilon)>0$ such that for any $0<r<r_0$, we have
 \begin{eqnarray}\label{tangent}
     B(\p,r)\cap M\subseteq (\p+T_{\p}M)_{\epsilon r}.
 \end{eqnarray}
 }

 \medskip

 \begin{proof}[Proof of Proposition \ref{rigidity}]
      Let $\ell\in\{1,...,d-1\}$. Throughout the proof, we assume that there exists a $\ell$-dimensional $C^1$ submanifold $M\subseteq\R^d$ such that
      $$\mu(K\cap M)>0   \, . $$ With this in  mind, the proof boils down to showing that either \ref{2} or \ref{3} of Proposition \ref{rigidity} hold.

   Since $\mu=\ubar{\mu}\circ\pi^{-1}$, then in view of the assumption we have that  $\ubar{\mu}\big(\pi^{-1}(K\cap M)\big)>0$. With this at hand,  on combining the doubling property of $\ubar{\mu}$  with the Lebesgue differentiation theorem  for doubling measure \cite[Theorem 1.8]{heinonen2001}, it follows that there exists $I=i_1i_2\cdot\cdot\cdot\in\pi^{-1}(K\cap M)$ such that
   \begin{eqnarray}\label{limden}
       \lim_{n\to\infty}\frac{\ubar{\mu}\big([i_1\cdot\cdot\cdot i_n]\cap\pi^{-1}(K\cap M)\big)}{\ubar{\mu}([i_1\cdot\cdot\cdot i_n])}=1.
   \end{eqnarray}
   Throughout, fix $I=i_1i_2\cdot\cdot\cdot\in\pi^{-1}(K\cap M)$ that satisfies (\ref{limden}) and let $\z_0:=\pi(I)$. Now with the notation and language of Section~\ref{cifs} in mind, for any $n\in\N$, let $I_n:=i_1\cdot\cdot\cdot i_n$ and consider the map given by
		\begin{eqnarray*}
			\psi_{n}(\z):=\|\varphi_{I_n}'\|^{-1}(\z-\varphi_{I_n}(\z_0))+\z_0,\,\,\,\,\,\,\,\,\forall\, \z\in\R^d.
		\end{eqnarray*}
  To ease notation, let $V:=\z_0+T_{\z_0}M$ where the last term is the tangent space at $\z_0$ (see \eqref{tanny} above). We claim that
  \begin{eqnarray}\label{infmax0}
      \inf_{n\in\N}~\max_{\x\in K}~\bd\big(\psi_n\circ\varphi_{I_n}(\x),\psi_n(V)\big)=0.
  \end{eqnarray}

\noindent For the moment let us assume the validity of (\ref{infmax0}). Then,  there exists a subsequence $\{n_k\}_{k\geq1}\subseteq\N$ such that
 \begin{eqnarray}\label{maxto0}
     \max_{\x\in K}~\bd\big(\psi_{n_k}\circ\varphi_{I_{n_k}}(\x),\psi_{n_k}(V)\big)\to0   \quad \text{as}   \quad k\to+\infty.
 \end{eqnarray}
Let $U\subseteq\R^d$ be the bounded connected open set appearing above (\ref{p2}). Then  by (\ref{p2}), we have  that
\begin{eqnarray*}
    C_2^{-1}|\x-\y|\leq|\psi_n\circ\varphi_{I_{n}}(\x)-\psi_n\circ\varphi_{I_{n}}(\y)|\leq C_2|\x-\y| \, ,   \qquad  \forall\, \x,\y\in U, \ ~\forall\, n\in\N.
\end{eqnarray*}
		This together with the fact that  $\z_0$ is a fixed point of $\psi_n\circ\varphi_{I_n}$ implies that $\{\psi_n\circ\varphi_{I_n}\}_{n\in\N}$ is uniformly bounded on $U$. Therefore, in view of  Lemma~\ref{concon} and by passing to a subsequence of $\{\psi_{n_k}\circ\varphi_{n_k}\}_{k\in\N}$, there exists  a conformal map $f:U\to\R^d$  such that $\psi_{n_k}\circ\varphi_{I_{n_k}} \to f$ uniformly on $U$.
         This together with (\ref{maxto0}), implies that
  \begin{eqnarray}\label{maxfto0}
      \max_{\x\in K}~\bd\big(f(\x),\psi_{n_k}(V)\big)\to0\ \ \ \ \text{as}\ \ \ \ k\to+\infty.
  \end{eqnarray}

 \noindent  Now, let $A(d,\ell)$ be the collection of all $\ell$-dimensional affine subspaces in $\R^d$. 
  It is well known that  $A(d,\ell)$ can be viewed as a locally compact metric subspace of $\R^{d^2+d}$,  see \cite[Section 3.16]{mattila1999}. By (\ref{maxfto0}) and the compactness of $f(K)$, the sequence $\{\psi_{n_k}(V)\}_{k\geq1}$ is bounded on $A(d,\ell)$. 
  Then by the locally compactness of $A(d,\ell)$ and passing to a subsequence if necessary, we may assume that there exists $W\in A(d,\ell)$ such that $\psi_{n_k}(V)\to W$.   This together with (\ref{maxfto0}),implies that $f(K)\subseteq W$ and hence $$K\subseteq f^{-1}(W\cap f(U)) \, .$$ Recall, that  $f$ is conformal,   so  $f^{-1}$ is also a conformal map on its domain $f(U)$.

\begin{itemize}
\item[$\circ$]     When $d\geq3$, we note that a conformal map in a connected open set can be extended to a Möbius transformation  in $\overline{\R}^d$. Thus, it follows  that $ f^{-1}(W\cap f(U))$ is contained in either a $\ell$-dimensional affine subspace of $\R^d$ or a $\ell$-dimensional geometric sphere.  \medskip

\item[$\circ$] When $d=2$, we note that a conformal map in a connected open set is a holomorphic (or anti-holomorphic) map.  Thus, it follows that  $f^{-1}(W\cap f(U))$ is contained in countable many analytic curves. Since $K$ is compact, it follows that  $K$ is contained in at most finitely many analytic curves.

\end{itemize}

  The upshot is that  part \ref{2} or \ref{3} of Proposition \ref{rigidity} hold under the assumption that (\ref{infmax0}) is valid.

 We now prove (\ref{infmax0}). If it is not true, then
  \[
      \delta_0:=\inf_{n\in\N}~\max_{\x\in K}~\bd\big(\psi_n\circ\varphi_{I_n}(\x),\psi_n(V)\big)>0.
  \]
  Therefore, for any $n\in\N$, on making use of (\ref{p3}),  it follows that  there exists $\x_n\in K$ such that
  \begin{eqnarray}\label{distpsiv}
      \bd(\varphi_{I_n}(\x_n),V)&=&\|\varphi_{I_n}'\|\cdot \bd(\psi_n\circ\varphi_{I_n}(\x_n),\psi_n(V))\nonumber\\[1ex]
&\geq&\delta_0\cdot\|\varphi_{I_n}'\|\nonumber\\[1ex]
      &\geq&C_3^{-1}\delta_0\cdot|K_{I_n}|  \, .
  \end{eqnarray}
 Fix $\epsilon\in(0,C_3^{-1}\delta_0)$. Then,  by (\ref{tangent}) there exists $N\in\N$ such that for all $n>N$, we have that
  \begin{eqnarray}\label{kcapmsub}
      K_{I_n}\cap M  \ \subseteq \  B(\z_0,|K_{I_n}|)\cap M   \ \subseteq \  (V)_{\epsilon|K_{I_n}|}.
  \end{eqnarray}
  For any $n\in\N$, choose $j^{(n)}_1j^{(n)}_2\cdot\cdot\cdot\in\pi^{-1}(\x_n)$. Then for any $m\in\N$ and $n>N$, by (\ref{p4}) and (\ref{distpsiv}), it follows that
  \begin{eqnarray}\label{dkinjv}
      \bd\big(K_{I_nj^{(n)}_1\cdot\cdot\cdot j^{(n)}_m},V\big)&\geq& \bd(\varphi_{I_n}(\x_n),V)-\big|K_{I_nj^{(n)}_1\cdot\cdot\cdot j^{(n)}_m}\big|\nonumber\\[1ex]
      &\geq&\left(C_3^{-1}\delta_0-C_4\big|K_{j^{(n)}_1\cdot\cdot\cdot j^{(n)}_m}\big|\right)\cdot|K_{I_n}|.
  \end{eqnarray}
  Choose $m_0\in\N$  large enough so that
  \[
      C_3^{-1}\delta_0-C_4\cdot\max_{J\in\Sigma^{m_0}}\big|K_{J}\big|\geq\epsilon \, .
  \]
  Then, it follows that  from (\ref{dkinjv}) and (\ref{kcapmsub}), that
  \begin{eqnarray*}
      \big[I_nj^{(n)}_1\cdot\cdot\cdot j^{(n)}_{m_0}\big]\subseteq\left([I_n]\cap\pi^{-1}(K\cap M)\right)^{c}
  \end{eqnarray*}
  for any $n>N$. Hence, for any $n>N$,
  \[
      \ubar{\mu}\big([I_n]\cap\pi^{-1}(K\cap M)\big) \ \leq  \ \ubar{\mu}([I_n])-\ubar{\mu}([I_nj^{(n)}_1\cdot\cdot\cdot j^{(n)}_{m_0}]).
  \]
  Since $\ubar{\mu}$ is doubling (see  \eqref{nsdoub}), there exists $\eta\in(0,1)$ such that for any $n>N$
  \[
      \ubar{\mu}([I_nj^{(n)}_1\cdot\cdot\cdot j^{(n)}_{m_0}])\geq\eta\cdot\ubar{\mu}([I_n]),
  \]
  and so it follows that
  \[
      \ubar{\mu}\big([I_n]\cap\pi^{-1}(K\cap M)\big)\leq (1-\eta)\cdot\ubar{\mu}([I_n]) \, ,   \qquad  \forall\, n>N.
  \]
  However, this  contradicts  (\ref{limden}) and so   (\ref{infmax0}) is true as desired.
 \end{proof}

\medskip

 As mentioned  in the introduction  (more precisely Remark~\ref{thy}) a similar result to our Theorem~\ref{thmrigidity}, equivalently Proposition~\ref{rigidity} with the OSC and $\mu$ a Gibbs measure,    was obtained by   Käenmäki \cite[Theorem 2.1]{kaenmaki2003}    in which  $\mu$ is  restricted to Hausdorff measures $\cH^s$ and part (iii) of  the theorem is replaced with the statement that $K$ is contained in a single analytic curve. However, this is not true as the following example demonstrates.  In short, it shows that  one can construct a self-conformal set $K\subseteq\R^2$ satisfying the OSC such that there exists a straight line $L$ with $\cH^{\tau}(K\cap L)>0$ (where $\tau=\dimH K$), but $K$ is contained in two different straight lines.  So trivially  (i) and (ii) are  not satisfied and we can not claim that $K$ is contained in a single analytic curve.

 \begin{example} (Counterexample to \cite[Theorem 2.1]{kaenmaki2003}). \label{counterexam}
     In this example, we view $\R^2$ as the complex plane $\bC$. 
    Let $\rmi$ be the imaginary unit, that is the solution of the equation $x^2+1=0$. Given
     $z \in \bC\backslash\{0\}$, there exist unique $r>0$ and $\theta\in[0,2\pi)$ such that $z=re^{\rmi\theta}$, and we  let  $\arg(z):=\theta$.   For any $\theta\in[0,2\pi)$, denote
\[
  L_{\theta}:=\{z\in\bC\backslash\{0\}:\arg(z)=\theta\}.
\]
Let $K$ be the self-similar set generated by $$\Phi=\{\varphi_1,\varphi_2\}=\left\{\frac{1}{3}z-\frac{2}{3}\rmi,\,\frac{1}{3}z+\frac{2}{3}\rmi\right\}.$$
It is straightforward to verify that
\[
\varphi_1\big([-1,1]^2\big)\cap\varphi_2\big([-1,1]^2\big)=\emptyset.
\]
Then $\Phi$ satisfies the strong separation condition and hence the open set condition. In fact, $K$ is exactly a similarity  of the middle-third Cantor set, and so $$\tau:=\dimH K=\log 2/\log 3.$$
Moreover, $K$ satisfies the following facts:
\begin{itemize}
    \item[(i)] $0\notin K$;
    \vspace*{1ex}
    \item[(ii)] $K\subseteq L_{\frac{3\pi}{2}}\cup L_{\frac{\pi}{2}}$; \vspace*{1ex}
    \item[(iii)] $\cH^{\tau}\left(K\cap L_{\frac{3\pi}{2}}\right)>0$ and $\cH^{\tau}\left(K\cap L_{\frac{\pi}{2}}\right)>0$.
\end{itemize}
 Now consider the function $f: z \mapsto \sqrt{z}$.  Then $f$ is a conformal map on $\bC\big\backslash\big([0,+\infty)\times\{0\}\big)$. Let  $K^{'}:=f(K)$.  Then $K^{'}$ is a self-conformal set generated by the conformal IFS $$\Phi'=\{f\circ\varphi_1\circ f^{-1},\,f\circ\varphi_2\circ f^{-1}\}=\left\{\sqrt{\frac{1}{3}z^2-\frac{2}{3}\,\rmi},\,\sqrt{\frac{1}{3}z^2+\frac{2}{3}\,\rmi}\right\}$$
which is defined on the open upper half-plane $$\Omega:=\left\{z=a+b\,\rmi:a\in\R,\,b\in(0,+\infty)\right\}.$$
Clearly, $\Phi'$ satisfies the strong separation condition. Since $f$ is conformal then  it is locally bi-Lipschitz. This implies that $$\dimH K^{'}=\dimH K=\tau.$$
Note that
\begin{eqnarray*}
    f\big(L_{\frac{3\pi}{2}}\big)\subseteq L_{\frac{3\pi}{4}}    \quad {\rm and}  \quad f\big(L_{\frac{\pi}{2}}\big)\subseteq L_{\frac{\pi}{4}}.
\end{eqnarray*}
This together with the above fact (iii) implies tht
\begin{eqnarray*}
    K^{'}\subseteq L_{\frac{3\pi}{4}}\cup L_{\frac{\pi}{4}}\quad {\rm and}  \quad \cH^{\tau}\left(K^{'}\cap L_{\frac{3\pi}{4}}\right)>0, \quad \  \cH^{\tau}\left(K^{'}\cap L_{\frac{\pi}{4}}\right)>0.
\end{eqnarray*}
Thus parts (i) and (ii) of Theorem~\ref{thmrigidity} are not  satisfied and
obviously, the two lines $L_{\frac{3\pi}{4}}$ and $L_{\frac{\pi}{4}}$ are not contained in the same analytic curve.
 \end{example}

 \section{Proof of Theorem~\ref{Main} modulo Theorem~\ref{meaannimpexp}  }\label{SEC3}

With Theorem~\ref{meaannimpexp} at hand, Theorem~\ref{Main} is a direct of the following statement.

\begin{proposition}\label{decofcorselcof}
    Let $(\Phi,K,\mu,T)$ be a self-conformal system, then $\mu$ has exponential decay of correlation with respect to $\cC^{\beta'}(K)$ for some $0<\beta'\leq 1$.
\end{proposition}

\begin{proof}
    Let $\ubar{\mu}$ be the Gibbs measure on the symbolic space $\Sigma^{\N}$ with respect to a $\beta$-H\"{o}lder potential such that $\mu=\ubar{\mu}\circ\pi^{-1}$. Let $m$ be the number of the elements in $\Phi$ and let $\kappa\in(0,1)$ be the constant in \eqref{defkap}. Note that any $\beta$-H\"{o}lder continous function is a $\beta'$-H\"{o}lder continuous function if $0<\beta'\leq\beta$. Then, without loss of generality, we assume that $0<\beta\leq\frac{\log 1/\kappa}{\log m}$. By Theorem~\ref{expmixforhol},  there exist $C>0$ and $\gamma>0$ such that
    \begin{eqnarray}\label{expforholsymbol}
            \left|\int_{\Sigma^{\N}}f_1\cdot f_2\circ\sigma^n\,\td\mu-\int_{\Sigma^{\N}}f_1\,\td\mu\cdot\int_{\Sigma^{\N}}f_2\,\td\mu\right|\,\,\leq\,\, C\,\gamma^n\cdot\|f_1\|_{\beta}\cdot\int_{\Sigma^{\N}}|f_2|\,\td\mu
        \end{eqnarray}
        for any  $f_1\in\cC^{\beta}(\Sigma^{\N})$,  any $f_2\in L^1(\mu)$ and   any $ n \in \N$.

    Let $\beta':=\beta\log m/\log(1/\kappa)$.  Then,  in view of the range of $\beta$, we have that  $0<\beta'\leq 1$.  We now show that $\mu$ has exponential decay of correlations with respect to $\cC^{\beta'}(K)$. With this in mind, let $f\in\cC^{\beta'}$ and
    \[
        M:=\sup\left\{\frac{|f(\x)-f(\y)|}{|\x-\y|^{\beta'}}:\x\neq\y\in K\right\}.
    \]
    Let $C_3>1$ be the constant in \eqref{p3'}.
     Then for any $I\neq J\in\Sigma^{\N}$, we have that
    \begin{eqnarray*}
        \frac{|f(\pi(I))-f(\pi(J))|}{\dist(I,J)^{\beta}}&\leq&\frac{M\cdot|\pi(I)-\pi(J)|^{\beta'}}{\dist(I,J)^{\beta}}
        \ \leq\ \frac{M\cdot C_3\cdot\kappa^{\frac{-\beta'\log\dist(I,J)}{\log m}}}{\dist(I,J)^{\beta}}\\
        &=&\frac{M\cdot C_3\cdot\dist(I,J)^{\frac{\beta'\log1/\kappa}{\log m}}}{\dist(I,J)^{\beta}}\ =\ M\cdot C_3.
    \end{eqnarray*}
    It follows that $\|f\circ\pi\|_{\beta}\leq C_3\cdot\|f\|_{\beta'}$. On combining this with \eqref{expforholsymbol}, we obtain that for any $f_1\in\cC^{\beta'}(K)$, $f_2\in L^1(\mu)$ and $n\in\N$,
    \begin{eqnarray*}
        &~&\hspace{-10ex}\left|\int_{X}f_1\cdot f_2\circ T^n\,\td\mu-\int_{X}f_1\,\td\mu\int_{X}f_2\,\td\mu\right|\\
        &=&\left|\int_{\Sigma^{\N}}f_1\circ\pi\cdot f_2\circ\pi\circ \sigma^n\,\td\ubar{\mu}-\int_{\Sigma^{\N}}f_1\circ\pi\,\td\ubar{\mu}\int_{\Sigma^{\N}}f_2\circ\pi\,\td\ubar{\mu}\right|\\
        &\leq& C\gamma^n\cdot\|f_1\circ\pi\|_{\beta}\cdot\int_{\Sigma^{\N}}|f_2\circ\pi|\,\td\ubar{\mu}\\
        &\leq&CC_3\gamma^n\cdot\|f\|_{\beta'}\cdot\int_{X}|f_2|\,\td\mu.
    \end{eqnarray*}
    This shows  that $\mu$ has exponential decay of correlations with respect to $\cC^{\beta'}(K)$.
\end{proof}

\section{Proof of Theorem~\ref{meaannimpexp}  }\label{SEC3A}

\begin{proof}[Proof of Theorem~\ref{meaannimpexp}]
Without loss of generality, we assume that $|X|=1$ where as usual  $ |X| $ denotes the diameter of X. Since $\mu$ has exponential decay of correlations with respect to $\cC^{\beta}(X)$ for some $0<\beta\leq 1$, there exist  $C>0$ and $\gamma\in(0,1)$ such that the inequality \eqref{expforholfun} holds for any $n\in\N$, $f_1\in\cC^{\beta}(X)$ and $f_2\in L^1(\mu)$. Fix $E\in\cC$, $\rho\in(\gamma,1)$ and $n\in\N$. Consider the function $f_n:\R^d\to\R$ given by
\begin{eqnarray*}
    f_n(\x):=\left\{\begin{aligned}
        &1-\frac{\bd(\x,E)}{\rho^{n}},\quad\text{if $\bd(\x,E)<\rho^n$,}\\
        &0,\qquad\qquad\quad\,\, \ \text{if $\bd(\x,E)\geq\rho^n$.}
    \end{aligned}\right.
\end{eqnarray*}
It is easily verified that
\begin{equation}\label{oneelefnle}
    \one_{E}(\x)\leq  f_n(\x)\leq\one_{(E)_{\rho^n}}(\x),\qquad\forall \ \x\in\R^d
\end{equation}
and thus
\begin{equation}\label{mueleintf}
    \mu(E)\ \leq\ \int_X f_n\,\td\mu\ \leq\ \mu(E)+\mu\big((E)_{\rho^n}\setminus E\big)\ \leq\ \mu(E)+\mu\big((\partial E)_{\rho^n}\big),
\end{equation}
where the last inequality holds since  $(E)_{\rho^n}\setminus E\subseteq(\partial E)_{\rho^n}$.  We next estimate the $\beta$-H\"{o}lder norm of $f_n$. Fix $\x,\y\in X$ and observe that
\begin{itemize}
    \item[$\circ$] if $\bd(\x,E)\geq\rho^n$ and $\bd(\x,E)\geq\rho^n$, then
    \begin{equation*}
        \frac{|f_n(\x)-f_n(\y)|}{|\x-\y|^{\beta}}=0.
    \end{equation*} \medskip
    \item[$\circ$] if $\bd(\x,E)<\rho^n$ and $\bd(\y,E)\geq\rho^n$, then
    \begin{eqnarray*}
        \frac{|f_n(\x)-f_n(\y)|}{|\x-\y|^{\beta}} & = & \frac{1-\frac{\bd(\x,E)}{\rho^n}}{|\x-\y|^{\beta}}\leq\frac{\frac{\bd(\y,E)}{\rho^n}-\frac{\bd(\x,E)}{\rho^n}}{|\x-\y|^{\beta}} \\[2ex] &\leq & \frac{|\x-\y|^{1-\beta}}{\rho^n}\leq\frac{|X|^{1-\beta}}{\rho^n}=\frac{1}{\rho^n}.
    \end{eqnarray*}  \medskip
    \item[$\circ$] if $\bd(\x,E)<\rho^n$ and $\bd(\y,E)<\rho^n$, then
    \begin{eqnarray*}
    \frac{|f_n(\x)-f_n(\y)|}{|\x-\y|^{\beta}}=\frac{\left|\frac{\bd(\y,E)}{\rho^n}-\frac{\bd(\x,E)}{\rho^n}\right|}{|\x-\y|^{\beta}}\leq\frac{|\x-\y|^{1-\beta}}{\rho^n}\leq\frac{|X|^{1-\beta}}{\rho^n}=\frac{1}{\rho^n}.
    \end{eqnarray*}
\end{itemize}

\noindent The upshot of the above is that $\|f_n\|_{\beta}\leq1/\rho^n$. This together with  \eqref{collection}, \eqref{expforholfun},  \eqref{oneelefnle} and \eqref{mueleintf} implies that for any $\mu$-measurable subset $F\subseteq\R^d$,
\begin{eqnarray}\label{muetnfleinf}
\begin{aligned}
     \mu(E\cap T^{-n}F)&\leq\int_X f_n\cdot\one_F\circ T^n\,\td\mu\\[4pt]
    &\leq\int_X f_n\,\td\mu\cdot\mu(F)+Cr^n\cdot\|f_n\|_{\beta}\cdot \mu(F)\\[4pt]
    &\leq\mu(E)\mu(F)+\Big(\mu\big((\partial E)_{\rho^n}\big)+C\cdot(\gamma/\rho)^n\Big)\cdot\mu(F)\\[2ex]
    &\leq\mu(E)\mu(F)+C\big(\rho^{n\delta}+
    \gamma_0^n\big)\cdot\mu(F)\qquad\big(\gamma_0:=\gamma/\rho\in(0,1)\big)\\[2ex]
&\leq\mu(E)\mu(F)+C\gamma_1^n\cdot\mu(F)\qquad \big(\gamma_1:=\max\{\rho^{\delta},\gamma_0\}\in(0,1)\big).
\end{aligned}
\end{eqnarray}

To derive the converse inequality, consider the map $g_n:\R^d\to\R$ defined by
\begin{eqnarray*}
    g_n(\x):=\left\{\begin{aligned}
        &1-\frac{\bd(\x,E\setminus(\partial E)_{\rho^n})}{\rho^n},\quad\text{if $\bd(\x,E\setminus(\partial E)_{\rho^n})<\rho^n$,}\\
        &0,\qquad\qquad\qquad\qquad\ \ \ \ \ \text{if $\bd(\x,E\setminus(\partial E)_{\rho^n})\geq\rho^n$.}
    \end{aligned}\right.
\end{eqnarray*}
Then,  it is easily verified that
\begin{equation*}
    \one_{E\setminus(\partial E)_{\rho^n}}(\x)\leq g_n(\x)\leq\one_{E}(\x),\qquad\forall \ \x\in\R^d,
\end{equation*}
and thus
\[
    \mu(E)\ \geq\ \int_X g_n\,\td\mu\ \geq\ \mu(E)-\mu\big((\partial E)_{\rho^n}\big).
\]
By employing a similar argument  used in deriving \eqref{muetnfleinf} with obvious modifications, we  obtain the desired lower bound; i.e. for any $\mu$-measurable subset $F\subseteq\R^d$
$$
\mu(E\cap T^{-n}F)   \geq     \mu(E)\mu(F) - C\gamma_1^n\cdot\mu(F)   \, .
$$
  This together with \eqref{muetnfleinf} shows that $\mu$ is exponentially mixing with respect to $(T,\cC)$. This completes the proof of Theorem~\ref{meaannimpexp}.
\end{proof}

 \section{Proof of Theorem~\ref{Main2}}\label{proofM2}


Given Theorem~\ref{Main},  the strategy for establishing Theorem~\ref{Main2} is simple enough:   we establish  (\ref{collection}) for balls.   In order to do this we make use of the rigidity theorem (namely Theorem~\ref{thmrigidity}) to prove the following result which provides the desired  upper bound estimate  in   essentially all cases.  Throughout given  a self-conformal system $(\Phi,K,\mu,T)$  on $\R^d$, we let
\begin{eqnarray}\label{defofellk}
    \ell_K:=\min\left\{1\leq\ell\leq d \ \left|   \,\,\,\begin{aligned}
        &\text{There exists a $\ell$-dimensional $C^1$ submanifold }\\ &\text{$M\subseteq\R^d$ such that $\mu(K\cap M)>0$\,.}
    \end{aligned}\right.\right\}\,.
\end{eqnarray}
Note that $\ell_K$ exists since  we always have that  $\mu(K\cap\R^d)=1$. Trivially, when $d=1$  we have that $\ell_K=1=d$. For $d\geq 2$, the statement  $\ell_K=d$ is equivalent to the statement that $K$ satisfies (i) of  Theorem~\ref{thmrigidity} for all $\ell\leq d-1$.

\medskip

\begin{theorem}\label{thmann}
		Let $(\Phi,K,\mu,T)$ be a self-conformal system on $\R^d$.
\begin{itemize}
  \item[(i)] There exists $C>0$, $s>0$ and  such that
  \begin{eqnarray}\label{ballmes}
     \mu(B(\x,r))\leq Cr^{s}  \quad \forall  \; \x\in \R^d,~\forall \; r>0 \, .
  \end{eqnarray}

 \item[(ii)] There exists $C>0$, $\delta>0$ and $r_0>0$ such that
  \begin{eqnarray}\label{thmmeaann}
      \mu\big((\partial B(\x,r))_{\varrho}\big)\leq C\varrho^{\delta}    \qquad \forall \ \x\in K ,    \ \forall \  0<r\leq r_0, \ \forall \ \varrho>0.
  \end{eqnarray}

  \medskip
  \noindent Furthermore, let $\ell_K$ be as in \eqref{defofellk}  and suppose that

  \medskip
  \begin{itemize}
      \item[(a)]   $\ell_K=d$. Then  (\ref{thmmeaann}) holds for any $\x\in\R^d$, any $r>0$ and any $\varrho>0$; \medskip
      \item[(b)]   $\ell_K<d$ where $d\geq2$ and part (ii) of Theorem~\ref{thmrigidity} holds with $\ell=\ell_K$.  Then  (\ref{thmmeaann}) holds for any $\x\in S$, any $r>0$ and any $\varrho>0$ where $S\subseteq\R^d$ is  the $\ell_K$-dimensional affine subspace or geometric sphere associated with  part (ii) of Theorem~\ref{thmrigidity}.
  \end{itemize}
\end{itemize}
\end{theorem}

\medskip

Note that \eqref{thmmeaann} in part (ii) of  Theorem~\ref{thmann}   together   with Theorem~\ref{Main} is not enough to establish  Theorem~\ref{Main2} since we need  \eqref{thmmeaann} to hold for all $\x \in \R^d$ (not just $K$) and all $r > 0$ (not just $r  \le r_0$). Nevertheless,  we shall see in the course of proving  Theorem~\ref{Main2}, that  the furthermore part of  (ii) together with Theorem~\ref{thmrigidity} allows us to do precisely this   in  all cases except when  $d=2$ and $K$ satisfies the statement in part (iii) of Theorem~\ref{thmrigidity}; that is, $K$ is contained within a finite disjoint union of analytic curves.  For this  ``remaining''  case we verify the desired exponentially mixing property directly.

We now establish  Theorem~\ref{Main2} assuming the validity of  Theorem~\ref{thmann}.  The proof of Theorem~\ref{thmann} will be the subject of Section~\ref{annulusec}.


 \begin{proof}[Proof of Theorem~\ref{Main2}  modulo Theorem~\ref{thmann} ]

Throughout, $\cC$ is a  collection of balls in $\R^d$ and  $\ell_K$ is as in \eqref{defofellk}. Then in view of  Theorem~\ref{thmrigidity}, we can split the proof of the theorem  into  three cases:

\begin{itemize}[leftmargin=4.5em]

    \item[Case 1:] $\ell_K=d$;\vspace*{1ex}

   \item[Case 2:] $d\geq2$, $\ell_K<d$ and $K$ satisfies part (ii) of Theorem~\ref{thmrigidity} with $\ell=\ell_K$;
    \vspace{1ex}

    \item[Case 3:] $d=2$, $\ell_K=1$ and $K$ satisfies part (iii) of Theorem~\ref{thmrigidity}.
\end{itemize}

\noindent To see this, simply note  that when $d=1$, then by definition $\ell_K=1$ and  we are in Case 1.

\noindent$\bullet$ \textit{Dealing with Case 1.} In view of part (ii.a) of Theorem \ref{thmann} and the fact that $\mu$ is a Borel measure, any ball in $\R^d$ is $\mu$-measurable and satisfies the upper bound estimate \eqref{collection}. This together with Theorem~\ref{Main} implies that $\mu$ is exponentially mixing with respect to $(T,\cC)$.

\noindent$\bullet$ \textit{Dealing with Case 2.} 
Let $B\subseteq\R^d$ an arbitrary ball in $\R^d$ and let $S\subseteq\R^d$ be the $\ell_K$-dimensional affine subspace or geometric sphere associated with  part (ii) of Theorem~\ref{thmrigidity}.
It can be verified that the intersection $B\cap S$  is either: (i) the empty set; (ii) equal to $S$ (which happens possibly only when $S$ is a geometric sphere); (iii) a single point; (iv) a set with boundary being a $(\ell_K-1)$-dimensional geometric sphere.
\begin{itemize}
    \item[$\circ$]  Suppose that $B\cap S$ satisfies any one of   the first three cases. Since $\mu$ has no atom  and is supported on $K\subseteq S$, we know that $\mu(B)=0$ or $1$. It follows that the term in the left hand side of the inequality (\ref{expmixballs}) equals zero and hence the desired exponential mixing inequality \eqref{expmixballs} trivially holds for this ball.\medskip

    \item[$\circ$]  Suppose that $B\cap S $ satisfies case (iv). Let $m\in\N_{\geq2}$ be the number of elements in the conformal IFS $\Phi$, and let $\gamma\in(0,1)$ be as in Corollary \ref{emprd}. For any $n\in\N$, let
\begin{eqnarray}\label{defqn}
    q_n:=\linte{\frac{-n\log\gamma}{\log m}}
\end{eqnarray}
and let
\[
\cI_n:=\left\{I\in\Sigma^{q_n}:K_I\subseteq B\right\}.
\]
Similar to (\ref{relation1}), it is easily verified that
\[
\bigcup_{I\in\cI_n}K_I~~\subseteq~~ B\cap K~~\subseteq ~~\left(\bigcup_{I\in\cI_n}K_I\right)\cup\big(\partial (B\cap S)\big)_{2C_3\kappa^{q_n}}.
\]
Then on adapting the  arguments used to derive (\ref{uppbd2}) and (\ref{lowbd2})  from (\ref{relation1}), we find that
\begin{eqnarray}\label{expmixint}
    |\mu(B\cap T^{-n}F)-\mu(B)\mu(F)|=O\left(\mu\big((\partial(B\cap S))_{2C_3\kappa^{q_n}}\big)+\gamma^{n/2}\right)\mu(F)
\end{eqnarray}
for all $\mu$-measurable subsets $F\subseteq\R^d$, where the big-$O$ constant does not depend on $B$ and $F$. Since $B$ satisfies (iv), then $\partial(B\cap S)$ is a $\ell_K-1$-dimensional geometric sphere in $S$, and hence there exist $\z\in S$ and $r>0$ such that
\begin{eqnarray}\label{easyobs}
    \partial(B\cap S)\,=\,(\partial B(\z,r))\cap S\,\subseteq\, \partial B(\z,r).
\end{eqnarray}
This together with part (ii.b) of Theorem \ref{thmann} implies that there exists $\delta>0$ for which
\[
\mu\big((\partial(B\cap S))_{2C_3\kappa^{q_n}}\big)\leq\mu\big((\partial B(\z,r))_{2C_3\kappa^{q_n}}\big)=O\left(\kappa^{\delta\cdot q_n}\right).
\]
With this in mind, by (\ref{expmixint}) we have that
\begin{eqnarray*}
    |\mu(B\cap T^{-n}F)-\mu(B)\mu(F)|=O\big(\widetilde{\gamma}^n\big)\,\mu(F)
\end{eqnarray*}
for any $\mu$-measurable subset $F\subseteq\R^d$ and any ball $B\subseteq\R^d$ in  case (iv), where $\widetilde{\gamma}\in(0,1)$ is given by
\[
\widetilde{\gamma}:=\max\left\{\kappa^{\frac{\delta\log(1/\gamma)}{\log m}},\gamma^{1/2}\right\}\,.
\]
\end{itemize}
 The upshot of the above is that in Case 2, the measure $\mu$ is exponentially mixing with respect to $(T,\cC)$.

\noindent$\bullet$ \textit{Dealing with Case 3.}  Let $k\geq1$ be an integer.
Suppose that $$K\subseteq\bigsqcup_{i=1}^k\Gamma_i,$$   where each $\Gamma_i\subseteq\R^2$ ($1\leq i\leq k$) is an analytic curve. It is easily verified that there exists $n_0\in\N$ such that for any $I\in\Sigma^{n_0}$, the corresponding cylinder set $K_I$ is contained within an analytic curve $\Gamma_i$ for some $1\leq i\leq k$. So by iterating the IFS up to $n_0$ if necessary, without loss of generality, we can  assume that each $\varphi_j(K)$ $(j=1,2,...,m)$ is contained within an analytic curve $\Gamma_i$ for some $1\leq i\leq k$. Furthermore, in order to establish the desired exponential mixing inequality (\ref{expmixballs}), without loss of generality, we can assume that $B$ is an open ball in $\R^2$. The point is that the desired inequality for closed ball follows on using the fact that any closed ball can be written as an  intersection of a decreasing sequence of open balls and then applying the obvious limiting argument.

 For any $n\in\N$ and any open ball $B\subseteq\R^2$, let $q_n$ be defined as in (\ref{defqn}) and let
\begin{eqnarray*}
   \cI_n(B)&:=&\left\{I\in\Sigma^{q_n}:K_I\subseteq B\right\},\\[4pt]
         \cJ_n(B)&:=&\left\{J\in\Sigma^{q_n}:K_J\cap B\neq\emptyset,~K_J\cap B^c\neq\emptyset\right\}.
\end{eqnarray*}
Then  it is easily verified that
\[
\bigcup_{I\in\cI_n(B)}K_I~~\subseteq~~ B\cap K~~\subseteq ~~\left(\bigcup_{I\in\cI_n(B)}K_I\right)\cup\left(\bigcup_{J\in\cJ_n(B)}K_J\right).
\]
By adapting the arguments used  in deriving (\ref{uppbd}) and (\ref{lowbd})  from (\ref{relation1}), it follows from the above that
\begin{eqnarray}\label{expmixcyl}
    |\mu(B\cap T^{-n}F)-\mu(B)\mu(F)|=O\left(\sum_{J\in\cJ_n(B)}\mu(K_J)+\gamma^{n/2}\right)\mu(F)
\end{eqnarray}
for all open balls $B\subseteq\R^2$ and all $\mu$-measurable subsets $F\subseteq\R^2$, where the big-$O$ constant does not depend on $B$ and $F$. To estimate the measure sum term in \eqref{expmixcyl}, let
\[
\cK(B):=\left\{1\leq i\leq k:~B\cap\Gamma_i\neq\emptyset\right\}.
\]
Since for any $1\leq i\leq k$, the map $f_i$ is injective on $[0,1]$ and $f_i'(t)\neq0$ for all $t\in[0,1]$, it follows that each map $f_i:[0,1]\to\Gamma_i$ is bi-Lipschitz. With this in mind, it is easy to check that  for any open ball $B\subseteq\R^2$ and any $J\in\cJ_n(B)$,  there exists  $i\in\cK(B)$ so that:
\begin{itemize}
   \item[$\circ$]   $K_J\subseteq \Gamma_i$ and $\Gamma_i\cap\partial B\neq\emptyset$;
    \medskip

    \item[$\circ$]  There exists  $\x\in\Gamma_i\cap\partial B$ and $C>0$ (independent of $B$ and $J$) such that $K_J\subseteq B(\x,C\kappa^{q_n})$, where $\kappa$ is defined as in (\ref{defkap}).
\end{itemize}
 On combining these two facts with part (i) of Theorem~\ref{thmann},  we find that there exists $s>0$ such that for any open ball $B\subseteq\R^2$
\begin{eqnarray}\label{sumofcyl}
    \sum_{J\in\cJ_n(B)}\mu(K_J)= O\Big(\kappa^{s\cdot q_n}\cdot\max\left\{\#(\Gamma_i\cap\partial B):i\in\cK(B)\right\}\Big),
\end{eqnarray}
 where the implied big-$O$ constant does not depend on $B$. We claim that
\begin{eqnarray}\label{supmaxgaicp}
     \sup_{B\subseteq\R^2\,\text{an open ball}}\,\max\{\#(\Gamma_i\cap\partial B):i\in\cK(B)\}<+\infty.
 \end{eqnarray}
If \eqref{supmaxgaicp} is true, then  together with (\ref{expmixcyl}) and (\ref{sumofcyl}) we have that  the measure $\mu$ is exponentially mixing with respect to $(T,\cC)$ and we are done.  The proof of \eqref{supmaxgaicp} is the subject of Lemma~\ref{lemcouopenball} below.
\end{proof}

\begin{lemma}\label{lemcouopenball}
    Let $k\geq1$ be an integer and let $\Gamma_i $  ($1\leq i\leq k$)   be disjoint analytic curves in $\R^d$. For any set $E\subseteq\R^2$, denote
    \[
    \cK(E):=\left\{1\leq i\leq k:\Gamma_i\cap E\neq\emptyset\right\}\,.
    \]
    Then we have
\begin{eqnarray}\label{supbopenmaxcoun}
        \sup_{B\subseteq\R^2\,\text{an open ball}}\,\max\{\#(\Gamma_i\cap\partial B):i\in\cK(B)\}<+\infty\,.
    \end{eqnarray}
\end{lemma}
 \begin{proof}
Since $\Gamma_i$ ($1\leq i\leq k$) are analytic curves, we note that for each  $1\leq i\leq k$ we can write
$$\Gamma_i=g_i([0,1]\times\{0\})   \quad  {\rm where } \quad g_i:\cO\to\R^2$$ is  a conformal map on an open set $\cO\subseteq\R^2$ containing $[0,1]\times\{0\}$. In turn, for each  $1\leq i\leq k$,   we define the map $f_i:t\mapsto g_i(t,0)$ on the unit interval. Then each $f_i$ is an injective real analytic map with $f_i'(t)\neq0~(t\in[0,1])$.

Now observe that  given any open ball $B\subseteq\R^2$ and any $1\leq i\leq k$ for which $\Gamma_i$ is an arc of a circle,  the set $\Gamma_i\cap\partial B$ is either:
 (i) equal to $\Gamma_i$, (ii) the empty set, (iii) a single point, (iv) a set with two points.
 \begin{itemize}
     \item[$\circ$]   If (i) is the case,  $\Gamma_i\cap B=\emptyset$ since  $B$ is open and so  $i\notin\cK(B)$.
     \medskip

     \item[$\circ$] In the last three cases, we have that  $\#(\Gamma_i\cap \partial B)\leq 2$.
 \end{itemize}
  The upshot of this  is that  if $\Gamma_i$ is a part of a circle, then $$\#(\Gamma_i\cap\partial B)\leq2$$
  for any open ball $B\subseteq\R^2$ for which $i\in\cK(B)$. In particular,  it  shows that \eqref{supbopenmaxcoun} is valid  if every  $\Gamma_i$ ($i=1,2,...,k$) is an arc of a  circle.  Thus, from this  point onwards, we assume that not every  $\Gamma_i$ ($i=1,2,...,k$) and so  proving  (\ref{supbopenmaxcoun}) boils down to showing that
\begin{eqnarray}\label{aimcount}
    \sup\left\{\#(\Gamma_i\cap \partial B):~B~\text{is an open ball in}~\R^2\right\}<+\infty
\end{eqnarray}
for any $1\leq i\leq k$ for which $\Gamma_i$ is not  contained in any circle. Fix such an $i$, call it $i_0$.
We prove, by contradiction, that (\ref{aimcount}) is true for $\Gamma_{i_0}$. Suppose that (\ref{aimcount}) is not true, then for any $n\in\N$, there are $\x_n\in\R^2$ and $r_n>0$ such that
\begin{eqnarray}\label{nmbgeqn}
    \#\big(\Gamma_{i_0}\cap\partial B(\x_n,r_n)\big)~\geq~n. ,
\end{eqnarray}
or equivalently, for any $n\in\N$, there are $0\leq t_1^{(n)}<t_2^{(n)}<\cdot\cdot\cdot<t_n^{(n)}\leq 1$ that satisfy
\begin{eqnarray} \label{gamcapparbsup}
     \Gamma_{i_0}\cap\partial B(\x_n,r_n)~~\supseteq~~\left\{f_{i_0}(t_1^{(n)}),f_{i_0}(t_2^{(n)}),...,f_{i_0}(t_n^{(n)})\right\}.
\end{eqnarray}
By passing to a subsequence if necessary, we may assume that $$\x_n\to \x_{\infty}\in \R^2\cup\{\infty\}\,\,\,\,\,\,\,\,\&\,\,\,\,\,\,\,\,r_n\to r_{\infty}\in[0,+\infty]\,\,\,\,\,\,\,\
(n\to+\infty)\,.$$  If $(\x_{\infty},r_{\infty})\in\R^2\times\{+\infty\}$ or $(\x_{\infty},r_{\infty})\in\{\infty\}\times[0,+\infty)$,  then it is easy to verify that $\Gamma_{i_0}\cap\,\partial   B(\x_n,r_n)=\emptyset$ for $n$  sufficiently large, which contradicts (\ref{nmbgeqn}). Therefore, for (\ref{nmbgeqn}) to hold, it is necessary that $$(\x_{\infty},r_{\infty})\in\R^2\times[0,+\infty)   \quad {\rm or}  \quad \x_{\infty}=\infty   \ {\rm and} \   r_{\infty}=+\infty \, . $$
\noindent We deal with these two case separately.

 \noindent$\bullet$ \emph{Case (i): $(\x_{\infty},r_{\infty})\in\R^2\times[0,+\infty)$. }For any $n\in\N$, let $F_n:[0,1]\to\R_{\geq0}$ be defined by setting
\begin{eqnarray*}
    F_n(t):=|f_{i_0}(t)-\x_n|^2\qquad(t\in[0,1])\,.
\end{eqnarray*}
Also consider the map on [0,1] given by
\begin{eqnarray*}
    F_{\infty}:\ t\mapsto|f_{i_0}(t)-\x_{\infty}|^2 \qquad(t\in[0,1])\,.
\end{eqnarray*}
Clearly, $F_{\infty}$ and  $\{F_n\}_{n\in\N}$ are  analytic functions on $[0,1]$. Moreover, in view of the fact that $\x_n\to\x_{\infty}$ as $n\to\infty$ and the analyticity of $f_{i_0}$, it is easily verified that for any $k\in\Z_{\geq0}$, the limit \begin{eqnarray}\label{unilimfndkfn}
   \frac{\td^kF_n}{\td t^k}(t)\to\frac{\td^k F_{\infty}}{\td t^k}(t)\qquad(n\to\infty)
\end{eqnarray}
holds uniformly with respect to $t\in[0,1]$. In view of \eqref{gamcapparbsup},  we have $F_n(t_j^{(n)})=r_n^2$ for any $n\in\N$ and any $j\in\{1,...,n\}$. For any $j\in\N$, let $\cT_j$ be the set of limit points of $\{t_j^{(n)}\}_{n\geq j}$, and let $\cT$ represents the union of $\cT_j$ over $j\in\N$. Then, for any $t\in\cT$, there exist $j_0=j_0(t)\in\N$ and $n_1<n_2<n_3<\cdot\cdot\cdot$ such that $t_{j_0}^{(n_k)}\to t$ $(k\to\infty)$, and hence
\begin{eqnarray}\label{Finftequrinf2}
    F_{\infty}(t)=\lim_{k\to\infty}|f_{i_0}(t_{j_0}^{(n_k)})-\x_{n_k}|^2=\lim_{k\to\infty}r_{n_k}^2=r_{\infty}^2\,.
\end{eqnarray}

\emph{Subcase (i): $\#\cT=+\infty$. }In view of  \eqref{Finftequrinf2}, in this subcase there are infinitely many $t\in[0,1]$ that satisfy $F_{\infty}(t)=r_{\infty}^2$. Then,  since the function $F_{\infty}$ is analytic it follows 
$$F_{\infty}(t)\equiv r_{\infty}^2\qquad (t\in[0,1])\,.$$ This implies that $\Gamma_{i_0}$ is a subset of the circle $\partial B(\x_{\infty},r_{\infty})$, which is a contradiction.

\emph{Subcase (ii): $\#\cT<+\infty$. }In this subcase, there exists $t_0\in\cT$ such that $t_0\in\cT_j$ for infinitely many $j\in\N$. Without the loss of generality, we assume that $t_0\in\cT_j$ for all $j\in\N$. Throughout this subcase, fix an arbitrary  integer $k\in\N$. Then there exists a subsequence $n_1<n_2<n_3<\cdot\cdot\cdot$ such that $t_j^{(n_{\ell})}\to t_0~(\ell\to\infty)$ for any $j=1,2,...,k+1$. Recall that $F_n(t_j^{(n)})=r_n^2$ for any $n,j\in\N$ with $1\leq j\leq n$.  So  by Rolle's theorem, there exists $\xi_n\in(t_1^{(n)},t_{k+1}^{(n)})$ for each $n\geq k+1$ such that
\[
\frac{\td^k F_n}{\td t^k}(\xi_n)=0.
\]
   In view of the fact that $\xi_{n_{\ell}}\in(t_1^{(n_{\ell})},t_{k+1}^{(n_{\ell})})$ and $t_1^{(n_{\ell})}\to t_0$,  $t_{k+1}^{(n_{\ell})}\to t_0$ as $\ell\to\infty$, we have that $\xi_{n_{\ell}}\to t_0$ as  $\ell\to\infty$.  This together with  the uniformly convergent property of the $k$-th derivatives of $F_n$  (see (\ref{unilimfndkfn})) implies that
\begin{eqnarray*}
    \frac{\td^kF_{\infty}}{\td t^k}(t_0)  \ = \ \lim_{\ell\to\infty}\frac{\td^kF_{n_{\ell}}}{\td t^k}(\xi_{n_{\ell}})=0.
\end{eqnarray*}
Now $k \in \N $ is arbitrary and $F_{\infty} $ is analytic within $[0,1]$, so it follows that
$$F_{\infty}(t)\equiv F_{\infty}(t_0)=r_{\infty}^2\qquad (t\in[0,1]),$$ which implies that $\Gamma_{i_0}$ is a subset of the circle $\partial B(\x_{\infty},r_{\infty})$. This contradicts the assumption that $\Gamma_{i_0}$ is not an arc of any circle.

\noindent$\bullet$ \emph{Case (ii): $\x_{\infty}=\infty$ and $r_{\infty}=+\infty$. } We identify $\R^2$ with the complex plane $\bC$.   For any $n\in\N$, fix a point $\z_n\in\Gamma_{i_0}\cap\partial B(\x_n,r_n)$. It is easily verified that $$\Gamma_{i_0}\cap\partial B(\x_n,r_n)\,\subseteq\,B(\z_n,2|\Gamma_{i_0}|)\cap\partial B(\x_n,r_n)$$
and so the right hand side is an arc of the circle $\partial B(\x_n,r_n)$. Since $r_n\to+\infty$ as $n\to\infty$, it follows that for  all sufficiently large $n\in\N$, there exist $\theta_1^{(n)},\,\theta_2^{(n)}\in[-2\pi,2\pi]$ such that  $$0<\theta_2^{(n)}-\theta_1^{(n)}<\pi$$ and
\[
h_n\big((0,1)\big)=B(\z_n,2|\Gamma_{i_0}|)\cap \partial B(\x_n,r_n),
\]
where we set
\[
h_n(s)=:\x_n+r_ne^{\rmi\big((1-s)\,\theta_1^{(n)}\,+\,s\,
\theta_2^{(n)}\big)}\qquad (s\in\R)\,.
\]
Without the loss of generality, we suppose that this fact holds for all $n\in\N$.  We claim that
\begin{eqnarray}\label{rnthetaupp}
    0<\inf\left\{r_n\big(\theta_2^{(n)}-\theta_1^{(n)}\big):\,n\in\N\right\}\leq\sup\left\{r_n\big(\theta_2^{(n)}-\theta_1^{(n)}\big):\,n\in\N\right\}\,<\,+\infty
\end{eqnarray}
and hence
\begin{eqnarray}\label{thedif}
    \theta_2^{(n)}-\theta_1^{(n)}\to0\qquad (n\to\infty).
\end{eqnarray}
To prove the upper bound of the supremum term in (\ref{rnthetaupp}), just note that
\begin{eqnarray}
    r_n\,(\theta_2^{(n)}-\theta_1^{(n)})&\leq&\pi\, r_n\sin{\frac{\theta_2^{(n)}-\theta_1^{(n)}}{2}}\label{leasin}\\[4pt]
    &=&\frac{\pi}{2}\,|h_n(1)-h_n(0)|\nonumber\\[4pt]
    &\leq&2\,\pi\,|\Gamma_{i_0}|\label{leqdiam}
\end{eqnarray}
for all  $n\in\N$, where we use the fact that $\sin x\geq 2x/\pi$ ($x\in[0,\pi/2]$) in deriving (\ref{leasin}), and inequality (\ref{leqdiam}) holds since  the points $h_n(0)$ and $h_n(1)$ belong to the closure of $B(\z_n,2|\Gamma_{i_0}|)$.     For the lower bound of the  infimum term in (\ref{rnthetaupp}), we note that
\begin{eqnarray*}
    r_n\,(\theta_2^{(n)}-\theta_1^{(n)})&=&\cH^1(h_n(0,1))\\[4pt]
    &\geq&|h_n(1)-h_n(0)|\\[4pt]
    &=&2\cdot\sqrt{4\,|\Gamma_{i_0}|^2-\Big(r_n-r_n\cos\frac{\theta_2^{(n)}-\theta_1^{(n)}}{2}\Big)^2}\\[4pt]
    &\asymp&2\cdot\sqrt{4\,|\Gamma_{i_0}|^2-\frac{1}{64}\,(r_n\,(\theta_2^{(n)}-\theta_1^{(n)}))^2\cdot(\theta_2^{(n)}-\theta_1^{(n)})^2}\\[4pt]
    &\to&4\,|\Gamma_{i_0}|\,\,\,\,(n\to\infty).
\end{eqnarray*}

With (\ref{rnthetaupp}) and (\ref{thedif}) at hand, we now show,  by passing to a subsequence if necessary, that there exist $\ba_{\infty}\in\bC$ and  $\bfb_{\infty}\in\bC\backslash\{0\}$ such that for any $k\in\Z_{\geq0}$, the following limit
\begin{eqnarray}\label{hntohinf}
  \lim_{n\to\infty}\frac{\td^k\,h_n}{\td s^k}(s)=\frac{\td^k\,h_{\infty}}{\td s^k}(s)\qquad (h_{\infty}(s):=\ba_{\infty}+\bfb_{\infty}\,s)
\end{eqnarray}
holds uniformly with respect to $s\in[0,1]$. To prove this statement, for any $n\in\N$ and $s\in[0,1]$,  note that by the definition of $h_n$,
\begin{eqnarray}
    h_n(s)-h_n(0)
&=&r_n\left(e^{\rmi\big((1-s)\,\theta_1^{(n)}+s\,\theta_2^{(n)}\big)}-e^{\rmi\,\theta_1^{(n)}}\right)\nonumber\\
    &=&r_n\cdot s\cdot(\theta_2^{(n)}-\theta_1^{(n)})\cdot\frac{e^{\rmi\big((1-s)\,\theta_1^{(n)}+s\,\theta_2^{(n)}\big)}-e^{\rmi\,\theta_1^{(n)}}}{s\cdot(\theta_2^{(n)}-\theta_1^{(n)})}\,.\label{hn-h0}
\end{eqnarray}
By (\ref{rnthetaupp}),    (\ref{thedif}) and the boundedness of $\{\theta_1^{(n)}\}_{n\in\N}$ and $\{h_n(0)\}_{n\in\N}$, there exists a subsequence $n_1<n_2<\cdot\cdot\cdot$ on $\N$ that satisfies
\begin{eqnarray}\label{conver1}
    r_{n_k}(\theta_2^{(n_k)}-\theta_1^{(n_k)})\to\ell_{\infty},\,\,\,\,\theta_1^{(n_k)}\to\theta_{\infty},\,\,\,\,h_{n_k}(0)\to\z_{\infty}\,\,\,\,(k\to\infty)
\end{eqnarray}
for some $\ell_{\infty}>0$ and $\theta_{\infty}\in[-2\pi,2\pi]$, and that the limit
\begin{eqnarray}\label{conver2}
    \frac{e^{\rmi\big((1-s)\,\theta_1^{(n_k)}+s\,\theta_2^{(n_k)}\big)}-e^{\rmi\,\theta_1^{(n_k)}}}{s\cdot(\theta_2^{(n_k)}-\theta_1^{(n_k)})}\to\left.\frac{\td(e^{\rmi\,\theta})}{\td\, \theta}\right|_{\theta=\theta_{\infty}}=\,\,\,\,e^{\rmi\left(\theta_{\infty}+\frac{\pi}{2}\right)}\,\,\,\,\,\,\,\,\,(k\to\infty)
\end{eqnarray}
holds uniformly with respect to $s\in[0,1]$. Then on combining (\ref{hn-h0}), (\ref{conver1}) and  (\ref{conver2}), we obtain that the limit
\[
h_{n_k}(t)\to\ba_{\infty}+\bfb_{\infty}\,t \qquad (n\to\infty)
\]
holds uniformly with respect to $t\in[0,1]$, where $\ba_{\infty}\in\bC$ and $\bfb_{\infty}\in\bC\backslash\{0\}$ is defined as
\[
\ba_{\infty}:=\z_{\infty},\,\,\,\,\bfb_{\infty}:=\ell_{\infty}\,e^{\rmi\left(\theta_{\infty}+\frac{\pi}{2}\right)}.
\]
This proves \eqref{hntohinf} when  $k=0$. In view of (\ref{conver1}) and (\ref{conver2}), it is easily verified that for any integer $j\geq 2$, the limits
\begin{eqnarray*}
    \lim_{k\to\infty}\frac{\td\,h_{n_k}}{\td\,t}(t)
    =\bfb_{\infty}=\frac{\td\,h_{{\infty}}}{\td\,t}(t), \qquad \lim_{k\to\infty}\frac{\td^j\,h_{n_k}}{\td\,t^j}(t)=0=\frac{\td^j\,h_{\infty}}{\td\,t^j}(t)
\end{eqnarray*}
hold uniformly with respect to $t\in[0,1]$. The upshot is that the desired limit (\ref{hntohinf}) is true  on the subsequence $\{h_{n_k}\}_{k\in\N}$. Without the loss of generality, assume that (\ref{hntohinf}) holds for $\{h_n\}_{n\in\N}$.

The foundations are now in place to show that
\begin{equation} \label{vgtrf}
\Gamma_{i_0} \,  \subset \,  L:=h_{\infty}(\R) \, .
\end{equation}
By definition,  $L$ is a straight line  so the above implies that  $\#(\Gamma_{i_0}\cap\partial B)\leq 2$ for any ball $B$ in $\bC$, which contradicts (\ref{nmbgeqn}) and so (\ref{aimcount}) is true as desired. To prove  \eqref{vgtrf},  we start by recalling  that  $\cT_j$ ($j\in\N$)  is the set of limit points of $\{t_j^{(n)}\}_{n\geq j}$ and  $\cT$ represents the union of $\cT_j$ over $j\in\N$. For each $j\in\N$ and $n\geq j$, since $f_{i_0}(t_{j}^{(n)})\in h_n\big((0,1)\big)$, then there exists $s_j^{(n)}\in(0,1)$ such that $f_{i_0}(t_{j}^{(n)})=h_{n}(s_{j}^{(n)})$.
Since $s_j^{(n)}$ are bounded, we know that for any $t\in\cT$, there exist $j_0=j_0(t)\in\N$ and $n_1<n_2<n_3<\cdot\cdot\cdot$ such that
\begin{eqnarray}\label{tnktotsnktos}
    t_{j_0}^{(n_k)}\to t\qquad\text{and}\qquad s_{j_0}^{(n_k)}\to s\qquad(k\to\infty)
\end{eqnarray}
 for some $s\in[0,1]$. Then, by the uniformly convergent property (\ref{hntohinf}) associated with $\{h_n\}_{n\in\N}$, we have
\begin{eqnarray}\label{fi0t=hinfs}
    f_{i_0}(t)=\lim_{k\to\infty}f_{i_0}(t_{j_0}^{(n_k)})=\lim_{k\to\infty}h_{n_k}(s_{j_0}^{(n_k)})=h_{\infty}(s)\in L\,.
\end{eqnarray}
In view of \eqref{fi0t=hinfs} and the fact that $h_{\infty}$ is invertible, we have $s=h_{\infty}^{-1}\circ f_{i_0}(t)$. We continue by considering the following two subcase.

\emph{Subcase (i): $\#\cT=+\infty$.} In this subcase, with (\ref{fi0t=hinfs}) in mind, there are infinitely many $t\in[0,1]$ such that $f_{i_0}(t)\in L$, then  by the analyticity of $f_{i_0}$, the curve $\Gamma_{i_0}=f_{i_0}([0,1])$ is contained in the straight line $L$.

\emph{Subcase (ii): $\#\cT<+\infty$.} In this subcase, there exists $t_0\in\cT$ such that there are  infinitely many $j\in\N$  for  which $t_0\in\cT_j$. Without the loss of generality, we assume that $t_0\in\cT_j$ for all $j\in\N$. Let $s_0:=h_{\infty}^{-1}\circ f_{i_0}(t_0)$.  Recall that $$g_{i_0}:\cO\to g_{i_0}(\cO)$$ is a conformal map on an open set $$\cO\supseteq[0,1]\times\{0\}$$ such that $g_{i_0}(t,0)=f_{i_0}(t)$ $(t\in[0,1])$ and $g_{i_0}^{-1}$ is also conformal on the domain $g_{i_0}(\cO)$.  With  the uniformly convergent property \eqref{hntohinf} associated with $\{h_n\}_{n\in\N}$ and the fact that $h_{\infty}(s_0)\in g_{i_0}(\cO)$ in mind, it is easily verified  that there exists $\delta>0$ such that $$h_{n}\big([s_0-\delta,s_0+\delta]\big)\subseteq g_{i_0}(\cO),\quad h_{\infty}\big([s_0-\delta,s_0+\delta]\big)\subseteq g_{i_0}(\cO) \quad (n\in\N)\,.$$
It follows that the functions
\[
G_n(s):=\Ima\left(g_{i_0}^{-1}\circ h_n(s)\right),\quad G_{\infty}(s):=\Ima\left(g_{i_0}^{-1}\circ h_{\infty}(s)\right)\qquad (n\in\N)
\]
are real analytic with respect to $s\in[s_0-\delta,s_0+\delta]$, where $\Ima(z)$ denotes the imaginary part of  $z\in\bC$.

Now fix a $n_0\in\N$. Recall  our assumption that $t_0\in\cT_j$ for any $j\geq1$; that is to say that $t_0$ is a limit point of $\{t_j^{(n)}\}_{n\geq j}$ for any $j\geq1$. Then there exists  a subsequence $k_1<k_2<k_3<\cdot\cdot\cdot$ such that $t_j^{(k_{\ell})}\to t_0$ ~$(\ell\to\infty)$ for any $j=1,2,...,n_0+1$. On the other hand, concerning the sequence $s_j^{(k_{\ell})}$, by passing to a subsequence if necessary, we know that for any $j=1,2,...,n_0+1$, there exists $s_j\in[0,1]$ such that $s_j^{(k_{\ell})}\to s_j$~$(\ell\to\infty)$.  In view of \eqref{fi0t=hinfs} and the fact that $h_{\infty}$ is invertible, we obtain that $s_j=h_{\infty}^{-1}\circ f_{i_0}(t_0)=s_0$ for any $j=1,2,...,n_0+1$.  Therefore, it follows that when $\ell$ is sufficiently large, we have $s_j^{(k_{\ell})}\in[s_0-\delta,s_0+\delta]$ for any $j=1,2,...,n_0+1$, which means that $s_j^{(k_{\ell})}$ lies in the domains of $G_n$ and $G_{\infty}$. For any such $\ell\in\N$, by the definitions of $s_{j}^{(k_{\ell})}$ and $t_{j}^{(k_{\ell})}$, we have
\begin{eqnarray*}
    G_{k_{\ell}}(s_{j}^{(k_{\ell})})=\Ima\left(g_{i_0}^{-1}\circ h_{k_{\ell}}(s_{j}^{(k_{\ell})})\right)=\Ima(t_{j}^{(k_{\ell})})=0,\qquad\forall\,j=1,2,...,n_0+1\,.
\end{eqnarray*}
With this in mind, on applying Rolle's theorem $n_0$ times   to the function $G_{k_{\ell}}$, it follows that  there exists $\xi_{k_{\ell}}\in\big(\min_{1\leq j\leq n_0+1}s_{j}^{(k_{\ell})},\max_{1\leq j\leq n_0+1}s_{j}^{(k_{\ell})}\big)$  such that
\begin{eqnarray*}
    \frac{\td^{n_0}\,G_{k_{\ell}}}{\td s^{n_0}}(\xi_{k_{\ell}})=0\,.\ \ \ \
\end{eqnarray*}
Since $\xi_{k_\ell}\to s_0$ $(\ell\to\infty)$ and the limit $\displaystyle\frac{\td^{j} \,h_k}{\td s^{j}}\to \frac{\td^{j} \,h_{\infty}}{\td s^{j}}$ $(k\to\infty)$ holds uniformly on $[0,1]$ for any $j=0,1,...,n_0$,  we have that
\begin{eqnarray*}
    \frac{\td^{n_0} \,G_{\infty}}{\td s^{n_0}}(s_0)=0.
\end{eqnarray*}
Now since $n_0\in\N$ is arbitrary and $ G_{\infty} $ is analytic in the domain of interest, 
we obtain that
\begin{eqnarray*}
    G_{\infty}(s)\equiv G_{\infty}(s_0)=0,\qquad \forall\,\, s\in[s_0-\delta,s_0+\delta].
\end{eqnarray*}
The upshot of this is that there exists $\epsilon>0$ such that $f_{i_0}([t_0-\epsilon,t_0+\epsilon])\subseteq L$. Since $f_{i_0}$ is analytic, the curve $\Gamma_{i_0}=f_{i_0}([0,1])$ is contained in the straight line $L$.
 \end{proof}

 \section{Proof of Theorem~\ref{thmann}}\label{annulusec}

 Theorem~\ref{thmann} is easily seen to be  a direct  consequence of the following proposition by
 On observing that any Gibbs measure on $\Sigma^{\N}$ is doubling and $K$ is not a singleton if the OSC is satisfied.
	\begin{proposition}\label{Prop3.1}
		Let $\Phi=\{\varphi_j\}_{1\leq j\leq m}$ be a $C^{1+\alpha}$ conformal IFS (without any separation condition) on $\R^d$, let $K$ be  the self-conformal set generated by $\Phi$ and let $\pi$ be the coding map. Let  $\ubar{\mu}$ be a doubling Borel probability measure on $\Sigma^{\N}=\{1,...,m\}^{\N}$ and $\mu:=\ubar{\mu}\circ \pi^{-1}$.
\begin{itemize}
  \item[(i)]  If $K$ is not a singleton, then there exists $C>0$ and $s>0$ such that
        \begin{eqnarray}\label{Crs}
            \mu(B(\x,r))\leq Cr^{s}  \quad \forall  \; \x\in \R^d,~\forall \; r>0 \, .
        \end{eqnarray}

  \item[(ii)]  If $K$ is not a singleton,
		then there exist $C>0$, $\delta>0$ and $r_0>0$ such that
  \begin{eqnarray}\label{annpro}
      \mu\big((\partial B(\x,r))_{\varrho}\big)\leq C\varrho^{\delta}, \quad \forall  \; \x\in K,~\forall \; 0<r\leq r_0,~\forall\; \varrho>0.
  \end{eqnarray}

   \medskip
  \noindent Furthermore, let $\ell_K$ be as in \eqref{defofellk}  and suppose that

  \medskip
  \begin{itemize}
      \item[(a)]   $\ell_K=d$. Then  (\ref{annpro}) holds for any $\x\in\R^d$, any $r>0$ and any $\varrho>0$; \medskip
      \item[(b)]   $\ell_K<d$ where $d\geq2$ and part (ii) of Proposition \ref{rigidity} holds with $\ell=\ell_K$.  Then  (\ref{annpro}) holds for any $\x\in S$, any $r>0$ and any $\varrho>0$ where $S\subseteq\R^d$ is  the $\ell_K$-dimensional affine subspace or geometric sphere associated with  part (ii) of
      Proposition \ref{rigidity}.
      \medskip

      \item[(c)] $d=2$, $\ell_K=1$ and $K$ is contained in a single analytic curve. Then (\ref{annpro}) holds for any $\x\in\Gamma$, any $r>0$ and any $\varrho>0$, where $\Gamma\subseteq\R^2$ is the corresponding analytic curve.
  \end{itemize}
\end{itemize}
       \end{proposition}

\medskip

Before providing the proof we introduce some useful notation.    For any two subsets $A$, $B\subseteq\R^d$, define the distance between $A$ and $B$ as
	\[
	\bd(A,B):=\inf\big\{\bd(\x,B):\x\in A\big\}.
	\]
 For any $\varrho>0$,  let
	$$\Lambda_{\varrho}:=\left\{I\in\Sigma^*:|K_I|<\varrho\leq|K_{I^-}|\right\},$$
where
\begin{eqnarray*}
I^-:=\left\{\begin{aligned}
    &i_1i_2\cdot\cdot\cdot i_{n-1},~~ \  \ \text{if} \ ~I=i_1i_2\cdot\cdot\cdot i_n   \ \ {\rm and } \ \ n\geq2, \\[2ex]
    &\emptyset,~~ \  \ \ \ \ \ \ \ \ \ \ \ \ \ \ \,  \text{if} \ ~|I|=1  \ .
\end{aligned}\right.
\end{eqnarray*}
In the above, the symbol $\emptyset$ is used to denote the empty word  and we  define  $K_{\emptyset}:=K$.
	Given any $I\in\Sigma^*$ and $\varrho>0$, we let
	\[
	\Lambda_{\varrho}(I):=\left\{J\in\Sigma^*:IJ\in\Lambda_{\varrho\,|K_{I^-}|}\right\}.
	\]
	Let $C_4>1$ be the constant appearing in (\ref{p4}) and let
    \begin{eqnarray}\label{defofvarrho0}
        \varrho_0:=C_4^{-1}\cdot\min_{1\leq i\leq m}|K_i|\,.
    \end{eqnarray}
    Then in view of (\ref{p4}), given any $0<\varrho<\varrho_0$, it is easily verified that
    \begin{eqnarray}\label{cunldisjointuni}
        \Lambda_{\varrho}(I)\neq\emptyset\ \ \ \  \& \ \ \ \  [I]=\bigsqcup_{J\in\Lambda_{\varrho}(I)}[IJ]\ \ \ \ \ \ \ \ \forall\,I\in\Sigma^*\,,
    \end{eqnarray}
     where  we use the symbol `$\sqcup$' to denote a disjoint union.
Furthermore, denote
\[
M_{\varrho}(I):=\sup\{|J|:J\in\Lambda_{\varrho}(I)\}\,.
\]
Then by \eqref{p3'} and \eqref{p4}, we have
\begin{eqnarray}\label{supmirhoi}
    \sup_{I\in\Sigma^*}M_{\varrho}(I)<+\infty\,.
\end{eqnarray}

 \subsection{Proof of Proposition \ref{Prop3.1}: part~(i)} \label{onedim}
With the above notation in mind, we start by proving the following statement regarding the distance between points in $\R^d$ and cylinder sets within the self-conformal set $K$. It is essential for proving part (i) of Proposition \ref{Prop3.1}.  Throughout, let $\varrho_0>0$ be as in \eqref{defofvarrho0}.

\begin{lemma}\label{lemdistxk}
    Under the setting of Proposition \ref{Prop3.1}, there exists       $\varrho\in(0,\varrho_0)$ that satisfies the following statement: given any $I\in\Sigma^*$ and any $\x\in\R^d$, there exists $J\in\Lambda_{\varrho}(I)$ such that $$\bd(\x,K_{IJ}) \; >  \: \varrho \, |K_{I^-}|  \, . $$
\end{lemma}
\begin{proof}
    Fix  $\z_0\in K$. For any $I\in\Sigma^*$, denote
    \begin{eqnarray}\label{psi}
			\psi_I(\z):=\|\varphi_I'\|^{-1}(\z-\varphi_I(\z_0))+\z_0,\ \ \ \ \forall\, \z\in\R^d.
		\end{eqnarray}
 We  claim that
  \begin{eqnarray}\label{distlow}
      \delta:=\inf_{I\in\Sigma^*}\inf_{\x\in\R^d}\sup_{\z\in K}|\x-\psi_I\circ\varphi_I(\z)|>0.
  \end{eqnarray}
  Indeed, if (\ref{distlow}) is not true, then there exist $\{I_k\}_{k\geq1}\subseteq\Sigma^*$ and $\{\x_k\}_{k\geq1}\subseteq\R^d$ such that
  \begin{eqnarray}\label{supto0}
      \sup_{\z\in K}|\x_k-\psi_{I_k}\circ\varphi_{I_k}(\z)|\to0   \quad \text{as}   \quad k\to\infty.
  \end{eqnarray}
Let $U\subseteq\R^d$ be the bounded connected open set associated with \eqref{p2}.  In view of (\ref{p2}),  it follows that
\begin{eqnarray}\label{biblip}
    C_2^{-1}|\x-\y|\leq|\psi_I\circ\varphi_I(\x)-\psi_I\circ\varphi_I(\y)|\leq C_2|\x-\y| \qquad (\x,\y\in U,~ I\in\Sigma^*)\,.
\end{eqnarray}
		Moreover, since $\z_0$ is a fixed point of $\psi_I\circ\varphi_I$, then $\{\psi_I\circ\varphi_I\}_{I\in\Sigma^{*}}$ is uniformly bounded on $U$. Therefore, according to  Lemma~\ref{concon} and by passing to a subsequence if necessary, we may assume that there exists  a conformal map $f:U\to\R^d$  such that $\psi_{I_k}\circ\varphi_{I_k} \to f$ uniformly on $U$. 	With this and (\ref{supto0}) in mind, we have
  \begin{eqnarray}\label{supfto0}
      \sup_{\z\in K}|\x_k-f(\z)|\to0   \quad \text{as}   \quad k\to\infty.
  \end{eqnarray}
  It follows that the sequence $\{\x_k\}_{k\geq1}$ is bounded. Then, by passing to a subsequence, we may assume that $\x_k\to\x$ for some $\x\in\R^d$.  This together with (\ref{supfto0}) implies that $f(K)=\{\x\}$.  Now,  with this and the fact that  $f$ is conformal in mind, we conclude that $K$ is a singleton, which contradicts our setting. This proves (\ref{distlow}).

Now fix an arbitrary $\x\in\R^d$ and  $I\in\Sigma^*$. Then by the compactness of $K$ and the definition of $\delta$, there exists $\z\in K$ such that
		\begin{eqnarray}
			|\x-\varphi_I(\z)|&=&\|\varphi_I'\|\cdot |\psi_I(\x)-\psi_I\circ\varphi_I(\z)|\nonumber\\[4pt]
	&\geq&\delta\,\|\varphi_I'\|\nonumber\\[4pt]
			&\geq& C_3^{-1}\delta\,|K_I|\label{varphii'geqki}\\[4pt]
			&\geq& C_3^{-1}C_4^{-1}\delta\cdot\big(\min_{1\leq i\leq m}|K_i|\big)\cdot|K_{I^-}|,\label{distxz}
		\end{eqnarray}
		where inequality \eqref{varphii'geqki} follows  from (\ref{p3}) and \eqref{distxz} follows  from (\ref{p4}). Let $\varrho>0$ be a small number that will be determined later. For any $j_1j_2\cdot\cdot\cdot\in\pi^{-1}(\z)$, there exists unique $n_0\in\N$ such that $J=j_1\cdot\cdot\cdot j_{n_0}\in\Lambda_{\varrho}(I)$. By (\ref{distxz}), we have
		\begin{eqnarray}\label{dist2xz}
			\bd(\x,K_{IJ})&\geq& |\x-\varphi_I(\z)|-|K_{Ij_1\cdot\cdot\cdot j_{n_0}}|\nonumber\\[4pt]
			&>& |\x-\varphi_I(\z)|-\varrho\, |K_{I^{-}}|\nonumber\\
			&\geq&\left(C_3^{-1}C_4^{-1}\delta\cdot\big(\min_{1\leq i\leq m}|K_i|\big)-\varrho\right)\cdot |K_{I^-}|\,.
		\end{eqnarray}
		Now choose  $\varrho\in(0,\varrho_0)$ to be a sufficiently  small  number (independent of the choices of $\x\in\R^d$ and $I\in\Sigma^*$) such that
		\[
		C_3^{-1}C_4^{-1}\delta\cdot\big(\min_{1\leq i\leq m}|K_i|\big) \ \geq \  2\, \varrho.
		\]
		Then in view of  the inequality (\ref{dist2xz}), we obtain the desired lower bound $$\bd(\x,K_{IJ})>\varrho \,|K_{I^-}|\,.$$
        This completes the proof of Lemma~\ref{lemdistxk}.
\end{proof}



We are now in the position to establish part (i) of Proposition~\ref{Prop3.1}.   Let $\varrho\in(0,\varrho_0)$ be as  in Lemma~\ref{lemdistxk}. To prove \eqref{Crs}, we first show that there exist $s>0$ and $N\in\N$ such that
\begin{equation}\label{mubxrhok}
    \mu(B(\x,\varrho^k))\leq \varrho^{ks}\qquad(\x\in\R^d,\ k\geq N)\,.
\end{equation}
Throughout, fix an arbitrary $\x\in\R^d$ and let
  \begin{eqnarray*}
      \cA_1&
  :=\left\{I\in\Lambda_{\varrho}:K_I\cap B(\x,\varrho)\neq\emptyset\right\}\,.
  \end{eqnarray*}
  For any $I\in\Sigma^*$ and any integer $k\geq1$, let
  \begin{eqnarray*}
      \cA_k(I)&
      :=&\left\{J\in\Lambda_{\varrho}(I):K_{IJ}\cap B(\x,\varrho^k)\neq\emptyset\right\},\\[0.4em]
E_k&:=&\bigcup_{I_1\in\cA_1}\bigcup_{I_2\in\cA_{2}(I_1)}\cdot\cdot\cdot\bigcup_{I_k\in\cA_k(I_1I_2\cdot\cdot\cdot I_{k-1})}[I_1I_2\cdot\cdot\cdot I_k].
  \end{eqnarray*}
		Roughly speaking, $E_k$ is the union of those cylinder sets $[I]$ in the symbolic space such that the associated cylinder sets $K_I$ within the self-conformal set satisfy
        \[
        |K_I|\asymp\varrho^k\qquad  {\rm and } \qquad K_I\cap B(\x,\varrho^k)\neq\emptyset.
        \]
        In view of \eqref{cunldisjointuni}, it can be verified that the unions in the definition of $E_k$ are  disjoint. Also, it is easily seen that $E_{k+1}\subseteq E_k$ for any $k\geq1$. Now given any $k\in\N$ and any
        \[
        I_1\in\cA_1,\,I_2\in\cA_2(I_1),\,...\,,\,I_k\in\cA_k(I_1I_2\cdot\cdot\cdot I_{k-1}),
        \]
        on applying Lemma~\ref{lemdistxk} to $I=I_1I_2\cdot\cdot\cdot I_k$ and $\x$, we obtain that there exists $J\in\Lambda_{\varrho}(I_1I_2\cdot\cdot\cdot I_k)$ such that
        \begin{eqnarray}\label{dxkgeqvarrhoki}
            \bd(\x,K_{I_1\cdot\cdot\cdot I_kJ})>\varrho\,|K_{I_1I_2\cdot\cdot\cdot I_k^{-}}|\,.
        \end{eqnarray}
        It can be verified that $|K_{I_1I_2\cdot\cdot\cdot I_k^{-}}|\geq\varrho^{k}$ by the choices of $I_1,I_2,...,I_k$. Then by \eqref{dxkgeqvarrhoki}, we have that $$\bd(\x,K_{I_1\cdot\cdot\cdot I_kJ})>\varrho^{k+1}$$
        which implies that
	$$[I_1I_2\cdot\cdot\cdot I_kJ]\cap E_{k+1}=\emptyset.$$ The upshot is that \begin{eqnarray}\label{intersectionball}
	\ubar{\mu}\big([I_1I_2\cdot\cdot\cdot I_k]\cap E_{k+1}\big)\leq\ubar{\mu}([I_1I_2\cdot\cdot\cdot I_k])-\ubar{\mu}([I_1I_2\cdot\cdot\cdot I_kJ])\,.
		\end{eqnarray}
		In view of  the fact that $\ubar{\mu}$ is doubling and that the sequence $\{M_{\varrho}(I)\}_{I\in\Sigma^*}$ is bounded (see \eqref{supmirhoi}), there exists $\eta\in(0,1)$ such that
        $$\ubar{\mu}([IJ])\geq\eta\cdot\ubar{\mu}([I]),\ \ \ \ \ \ \ \ \forall\,I\in\Sigma^*,\,\,\,\forall\,J\in\Lambda_{\varrho}(I).$$
        This together with (\ref{intersectionball}) implies that
		\begin{eqnarray}\label{mui1ikcaples}
			\ubar{\mu}\big([I_1I_2\cdot\cdot\cdot I_k]\cap E_{k+1}\big)\leq(1-\eta)~\ubar{\mu}([I_1I_2\cdot\cdot\cdot I_k]).
		\end{eqnarray}
	   On combining \eqref{mui1ikcaples} with the fact that $E_{k+1}\subseteq E_k$, we have that  for any $k \ge 1$
		\begin{eqnarray}
			\ubar{\mu}(E_{k+1})&=&\mu(E_k\cap E_{k+1})\nonumber\\[3pt]
&=&\sum_{I_1\in\cA_1}\sum_{I_2\in\cA_2(I_1)}\cdot\cdot\cdot\sum_{I_k\in\cA_{k}(I_1I_2\cdot\cdot\cdot I_{k-1})}\ubar{\mu}([I_1I_2\cdot\cdot\cdot I_k]\cap E_{k+1})\nonumber\\[3pt]
			&\leq&(1-\eta)\cdot\sum_{I_1\in\cA_1}\sum_{I_2\in\cA_2(I_1)}\cdot\cdot\cdot\sum_{I_k\in\cA_{k}(I_1I_2\cdot\cdot\cdot I_{k-1})}\ubar{\mu}([I_1I_2\cdot\cdot\cdot I_k])\nonumber\\[3pt]
			&=&(1-\eta)\cdot\ubar{\mu}(E_k)    \, .\label{Ek+1}
		\end{eqnarray}
		Then, by iterating  inequality (\ref{Ek+1}), we obtain that
        \begin{eqnarray}\label{ubarmuek}
            \ubar{\mu}(E_k)\leq(1-\eta)^{k-1}  \qquad \forall\,k\geq1.
        \end{eqnarray}
         Furthermore, by the definition of $E_k$, we have
		\[
		\pi^{-1}\big(B(\x,\varrho^k)\cap K\big)\subseteq E_k,\,\,\,\,\,\,\,\,\forall\, k\geq1\,.
		\]
		This together with (\ref{ubarmuek}) and the fact that $\mu=\ubar{\mu}\circ\pi^{-1}$ implies that
		\begin{eqnarray}
			\mu\big(B(\x,\varrho^k)\big)&=&\ubar{\mu}\left(\pi^{-1}\big(K\cap B(\x,\varrho^k)\big)\right)\nonumber\\[2ex]
			&\leq&\nonumber\ubar{\mu}(E_k)\\[3pt]
			&\leq&(1-\eta)^{k-1}\nonumber\\[3pt]
			&=&\varrho^{k\cdot\frac{(k-1)\log(1-\eta)}{k\log\varrho}}\nonumber\\
            &\leq&\varrho^{k\cdot\frac{\log(1-\eta)}{2\log\varrho}}    \qquad  \forall~ k\geq2 \, . \label{rhok}
		\end{eqnarray}
        This proves \eqref{mubxrhok} with $N=2$  and
       \begin{equation}  \label{junjieW}
           s:=\frac{\log(1-\eta)}{2\log\varrho}  \, .
       \end{equation}

    To complete the proof, we need to consider  the $\mu$-measure of balls  $B(\x,r))$ with arbitrary radius $r>0$.  For this we consider the following two cases.
        \begin{itemize}
            \item[$\circ$]  If $r\leq\varrho^2$, then there exists unique integer $k\geq2$ such that $\varrho^{k+1}<r\leq\varrho^k$. From this and the inequality $(\ref{rhok})$, we have
    \begin{eqnarray}\label{rsmall}
            \mu\big(B(\x,r)\big)~\leq~\left(\frac{r}{\varrho}\right)^{\frac{\log(1-\eta)}{2\log\varrho}}.
        \end{eqnarray}

\item[$\circ$]    If $r>\varrho^2$, then
        \begin{eqnarray}\label{lrho2}
            \mu\big(B(\x,r)\big)~\leq~ 1~<\left(\frac{r}{\varrho^2}\right)^{\frac{\log(1-\eta)}{2\log\varrho}}.
        \end{eqnarray}
        \end{itemize}
        On combining (\ref{rsmall}) and (\ref{lrho2}), we obtain the desired inequality (\ref{Crs}) with $s$ given by \eqref{junjieW} and
        $$
        C=\varrho^{-\log(1-\eta)/\log\varrho} \, .
        $$
        This completes the proof of part (i) of Proposition \ref{Prop3.1}.

	\subsection{Proof of Proposition \ref{Prop3.1}: \!part (ii) }\label{hidim}   We first establish the ``Furthermore'' part of the statement.  As we  shall see in Section~\ref{junjieZ},  this together with the  rigidity statement Proposition~\ref{rigidity}  implies  \eqref{annpro}  in essentially all situations.


 \subsubsection{Proof of part (a) of Proposition \ref{Prop3.1} (ii)}\label{pfofpaofii}
Let $d=1$. Then,  for any $x\in\R$, $r>0$ and $\varrho>0$, we have $(\partial B(x,r))_{\varrho}=B(x-r,\varrho)\cup B(x+r,\varrho)$.  Thus,
part (i) of Proposition~\ref{Prop3.1} implies that
$$
\mu \big( (\partial B(x,r))_{\varrho} \big)  \leq 2 C\cdot\varrho^s  \, .
$$
This is precisely  the desired inequality  \eqref{annpro}  for all $x\in\R$, all $r > 0$ and all $\varrho>0$.

Without loss of generality, we assume that $d\geq2$. Then, in view of the definition of $\ell_K$ and Proposition~\ref{rigidity},  the statement $\ell_K=d$ is equivalent to that $K$ is not contained in any $(d-1)$-dimensional $C^1$ submanifold.

We start the $d \ge 2$ proof with establishing the following lemma concerning the tangent plane of a geometric sphere. It is required in estimating the lower bound of the distance between cylinder   sets and the boundary of balls   (namely in proving Lemma~\ref{3.2}).

	\begin{lemma}\label{lem3.1}
		Let $d \ge 2$, $\ell\in\{1,...,d-1\}$ and $S\subseteq\R^d$ be a $\ell$-dimensional geometric sphere with radius $R$.  Then  then for any $\varrho\in(0,1)$, any $r\in(0,\varrho R]$ and any $\x\in S$, we have
		\begin{eqnarray*}
			S\cap B(\x,r)\subseteq (\x+T_{\x}S)_{\varrho\, r}\,.
		\end{eqnarray*}
	\end{lemma}
 \begin{proof}
     Let $S$ be a $\ell$-dimensional geometric sphere with radius $R$ and let $\x\in S$. After applying  an isometry if necessary, we may assume that
     \begin{eqnarray*}
         S=\left\{\y=(y_1,...,y_{\ell+1},0,...,0)\in\R^d:y_1^2+y_2^2+\cdot\cdot\cdot+y_{\ell+1}^2=R^2\right\}
     \end{eqnarray*}
     and that $\x=(R,0,...,0)$. Then
     \[
         \x+T_{\x}S=\left\{\y\in\R^d:y_1=R,\,y_{\ell+2}=\cdot\cdot\cdot=y_d=0\right\}.
     \]
     Let $\varrho\in(0,1)$ and $r\in(0,\varrho R)$. For any $\y\in S\cap B(\x,r)$, a straightforward calculation yields that
     \begin{eqnarray*}
         \bd(\y,\x+T_{\x}S)&=&R-y_1
            \ =   \ R-\sqrt{R^2-y_2^2-\cdot\cdot\cdot-y_{\ell+1}^2}\\[2ex]
         &=&\frac{y_2^2+\cdot\cdot\cdot y_{\ell+1}^2}{R+\sqrt{R^2-y_2^2-\cdot\cdot\cdot-y_{\ell+1}^2}}
         \ \leq \  \frac{r^2}{R}
         \  \leq  \ \varrho\,r\,.
     \end{eqnarray*}
     This completes the proof.
 \end{proof}

 The following result can be viewed as an analogue of Lemma~\ref{lemdistxk}. Throughout, let $\varrho_0>0$ be as in \eqref{defofvarrho0}.

	\begin{lemma}\label{3.2}
		Suppose that $K$ is not contained in  any $(d-1)$-dimensional $C^1$ submanifold, then there exists $\varrho\in(0,\varrho_0)$ that satisfies the following statement:  given any  $I\in\Sigma^*$ and any  ball $B\subseteq\R^d$, there exists $J\in\Lambda_{\varrho}(I)$ such that  $$\bd(K_{IJ},\partial B)>\varrho\,|K_{I^-}|\,.$$
	\end{lemma}
	\begin{proof}
		We adapt the proof of Lemma~\ref{lemdistxk}. Fix  $\z_0\in K$. For any $I\in\Sigma^*$, define $\psi_I:\R^d\to\R^d$ as in (\ref{psi}).
		We first show  under the setting of  Lemma \ref{3.2} that
		\begin{eqnarray}\label{fm3.1}
			\delta:=\inf_{I\in\Sigma^*}\inf_{\z\in \R^d}~\inf_{r>0}~\sup_{\x\in K}~\bd(\psi_I\circ\varphi_I(\x),\partial B(\z,r))>0\,.
		\end{eqnarray}
		Suppose by contradiction that  (\ref{fm3.1}) is not true, then there exist $\{I_k\}_{k\geq1}\subseteq\Sigma^*$, $\{\z_k\}_{k\geq1}\subseteq\R^d$ and $\{r_k\}_{k\geq1}\subseteq(0,+\infty)$ such that
		\begin{eqnarray}\label{limit}
			\sup_{\x\in K}~\bd(\psi_{I_k}\circ\varphi_{I_k}(\x),\partial B(\z_k,r_k))\to 0\quad\text{as}\quad k\to\infty.
		\end{eqnarray}
Let $U\subseteq\R^d$ be the bounded connected open set associated with \eqref{p2}. In view of  (\ref{biblip}), the uniformly boundedness of $\{\psi_I\circ\varphi_I\}_{I\in\Sigma^*}$ and Lemma~\ref{concon},  we may assume that there exists  a conformal map $f:U\to\R^d$  such that $\psi_{I_k}\circ\varphi_{I_k} \to f$ uniformly on $U$. 		Then it follows from (\ref{limit}) that
  \begin{eqnarray}\label{limit2}
      \sup_{x\in K}~\bd(f(\x),\partial B(\z_k,r_k))\to0\quad\text{as}\quad k\to\infty.
  \end{eqnarray}
  By passing to a subsequence if necessary, we may assume that
  \[
  \z_k\to\z_{\infty},\qquad r_k\to r_{\infty}\qquad\text{as}\qquad k\to\infty
  \]
  for some $\z_{\infty}\in\R^d\cup\{\infty\}$ and $r_{\infty}\in[0,+\infty]$.
  If $(\z_{\infty},r_{\infty})\in\R^d\times\{+\infty\}$  or $(\z_{\infty},r_{\infty})\in\{\infty\}\times[0,+\infty)$, then it can be verified that
  \begin{eqnarray}\label{contra}
      \limsup_{k\to\infty}~\sup_{\x\in K}~\bd(f(\x),\partial B(\z_k,r_k))=+\infty
  \end{eqnarray}
		which contradicts (\ref{limit2}). Therefore, for  \eqref{limit} to hold, it is necessary   $$(\z_{\infty},r_{\infty})\in\R^d\times[0,+\infty)   \quad {\rm or}  \quad \z_{\infty}=\infty   \ {\rm and} \   r_{\infty}=+\infty \, . $$

\noindent We deal with these two case separately

\begin{itemize}
    \item[$\circ$] Suppose that $\z_{\infty}\in\R^d$ and $r_{\infty}\in[0,+\infty)$. Then it follows from~(\ref{limit2}) that $$\sup_{\x\in K}\bd(f(\x),\partial B(\z_{\infty},r_{\infty}))=0\,,$$ which implies that $K\subseteq f^{-1}(\partial B(\z,r)\cap f(U))$. In turn, this means  that $K$ is a subset of  a $(d-1)$-dimensional $C^1$ manifold which contradicts the hypothesis  of the lemma.
    \medskip

    \item[$\circ$] Suppose that $\z_{\infty}=\infty$ and $r_{\infty}=+\infty$. In this case,
  \begin{eqnarray}\label{limtoinf}
      \bd(\z_k,f(K))\to +\infty\quad\text{and}\quad r_k\to+\infty,~~\text{as}~~k\to+\infty.
  \end{eqnarray}
 Then, on combining (\ref{limit2}) and (\ref{limtoinf}) with Lemma~\ref{lem3.1}, we can find a sequence $\{\x_k\}_{k\in\N}\subseteq\R^d$  that satisfies  the following two statements:
 \medskip

 \begin{itemize}
     \item[(i)] for each $k\geq1$, we have that  $\x_k\in\partial B(\z_k,r_k)$ .
     \medskip

     \item[(ii)]  for any $\epsilon>0$, there exists $N>0$ such that for all $k>N$, we have $\bd(\x_k,f(K))<\epsilon$ and
		\begin{eqnarray}\label{fksubseteqxktxk}
			f(K)\subseteq\big(\x_k+T_{\x_k}\partial B(\z_k,r_k)\big)_{\epsilon}\,.
		\end{eqnarray}
 \end{itemize}
		Note that the sequence $\{\x_k\}_{k\geq1}$ is bounded.  Thus,   the sequence $$\{\x_k+T_{\x_k}\partial B(\z_k,r_k)\}_{k\geq1}$$ is  bounded in $A(d,d-1)$, where $A(d,d-1)$ denotes the collection of all $(d-1)$-dimensional affine subspaces in $\R^d$ (see for example \cite[Section 3.16]{mattila1999}). With this and \eqref{fksubseteqxktxk} in mind,  it follows that there exists a $(d-1)$-dimensional affine subspace $V\subseteq\R^d$ such that $f(K)\subseteq (V)_{\epsilon}$ for any $\epsilon>0$. Letting $\epsilon$ approach to zero, we have $f(K)\subseteq V$ and thus $K\subseteq f^{-1}(V\cap f(U))$. It implies that $K$ is contained in a $(d-1)$-dimensional $C^1$ submanifold, which is a contradiction.
\end{itemize}

		The upshot of the above  is that inequality (\ref{fm3.1}) is true.
		 Let $B\subseteq\R^d$ be a ball and fix $I\in\Sigma^*$. Then by the definition of $\delta$, there exists $\mathbf{x}\in K$ such that
		\begin{eqnarray}
			\bd(\varphi_I(\mathbf{x}),\partial B)&=&\|\varphi_I'\|\cdot \bd(\psi_I\circ\varphi_I(\x),\psi_I(\partial B))\nonumber\\[4pt]
			&\geq&\delta\|\varphi_I'\|\nonumber\\[4pt]
			&\geq& C_3^{-1}\delta|K_I|\label{phinorm}\\[4pt]
			&\geq& C_3^{-1}C_4^{-1}\delta\cdot\big(\min_{1\leq i\leq m}|K_i|\big)\cdot|K_{I^-}|\,,\label{dist1}
		\end{eqnarray}
		where inequality \eqref{phinorm} follows from (\ref{p3}) and the  inequality \eqref{dist1} is a consequence of (\ref{p4}). Let $\varrho\in(0,\varrho_0)$ be a  small number that will be determined later. For any infinite word $j_1j_2\cdot\cdot\cdot\in\pi^{-1}(\x)$, there exists unique $n_0\in\N$ such that the finite word $J=j_1\cdot\cdot\cdot j_{n_0}\in\Lambda_{\varrho}(I)$. By  (\ref{dist1}) and the definition of $\Lambda_{\varrho}(I)$, we have
		\begin{eqnarray}\label{dist2'}
			\bd(K_{IJ},~\partial B)&\geq& \bd(\varphi_{I}(\x),\partial B)-|K_{IJ}|\nonumber\\[4pt]
			&>& \bd(\varphi_{I}(\x),\partial B)-\varrho |K_{I^{-}}|\nonumber\\[4pt]
			&\geq&\left(C_3^{-1}C_4^{-1}\delta\cdot\big(\min_{1\leq i\leq m}|K_i|\big)-\varrho\right)\cdot |K_{I^-}|\,.
		\end{eqnarray}
		Let $\varrho\in(0,\varrho_0)$ be sufficiently same so that
		\[
		C_3^{-1}C_4^{-1}\delta\cdot\big(\min_{1\leq i\leq m}|K_i|\big)-\varrho\geq\varrho.
		\]
		Then the inequality (\ref{dist2'}) yields the desired lower bound $\bd(K_{IJ},~\partial B)>\varrho\,|K_{I^-}|$.
	\end{proof}

		We are now in the position  to establish part (a) of Proposition \ref{Prop3.1} (ii) when $d \ge 2$.
        The  proof is an adaptation of the proof of Proposition \ref{Prop3.1} (i).
 Recall that under the setting of part (a) with $d\geq2$, the self-conformal set $K$ is not contained in any $d-1$  dimensional $C^1$ submanifold. This means that $K$ satisfies  the  hypothesis of Lemma~\ref{3.2}.  Let $\varrho>0$ be the   constant associated with that lemma.   Note that, to prove part (a),  it suffices to find $\delta>0$ and $N\in\N$ such that
 \begin{eqnarray}\label{muparbrhok}
     \mu((\partial B)_{\varrho^k})\leq \varrho^{k\delta}
 \end{eqnarray}
  for any  ball $B\subseteq\R^d$ and any $k\geq N$ (for arbitrary $r>0$ we simply follow the arguments used at the end of the proof of Proposition~\ref{Prop3.1}  (i)).  So with this in mind, fix a ball $B\subseteq\R^d$  and define
  \begin{eqnarray*}
      \cA_1
  :=\left\{I\in\Lambda_{\varrho}:K_I\cap (\partial B)_{\varrho}\neq\emptyset\right\}.
  \end{eqnarray*}
  For any $I\in\Sigma^*$ and positive integer $k\geq2$, define
  \begin{eqnarray*}
      \cA_k(I)
      &:=&\left\{J\in\Lambda_{\varrho}(I):K_{IJ}\cap(\partial B)_{\varrho^k}\neq\emptyset\right\},\\[0.4em]
E_k&:=&\bigcup_{I_1\in\cA_1}\bigcup_{I_2\in\cA_{2}(I_1)}\cdot\cdot\cdot\bigcup_{I_k\in\cA_k(I_1I_2\cdot\cdot\cdot I_{k-1})}[I_1I_2\cdot\cdot\cdot I_k]\,.
  \end{eqnarray*}
		Note that the above definitions of $\cA_1$, $\cA_k(I)$ and $E_k$ are slightly different from those appearing in the proof of Proposition \ref{Prop3.1} (i) even though we use the same symbols.
 In terms of \eqref{cunldisjointuni}, it can be verified that the  unions in the definition of $E_k$ are all disjoint. Also, it is easily seen that $E_{k+1}\subseteq E_k$ for any $k\geq1$. Now, by Lemma~\ref{3.2},
   we know  that for any $k\in\N$ and any
 \[
 I_1\in\cA_1,\,I_2\in\cA_2(I_1),\,...\,,\,I_k\in\cA_k(I_1I_2\cdot\cdot\cdot I_{k-1}),
 \] there exists $J\in\Lambda_{\varrho}(I_1I_2\cdot\cdot\cdot I_k)$ such that $$\bd(K_{I_1\cdot\cdot\cdot I_kJ},~\partial B)>\varrho\,|K_{I_1I_2\cdot\cdot\cdot I_k^{-}}|.$$ Furthermore, we have that $|K_{I_1I_2\cdot\cdot\cdot I_k^{-}}|\geq\varrho^{k}$ by the choices of $I_1,I_2,...,I_k$. Then
  $$\bd(K_{I_1\cdot\cdot\cdot I_kJ},~\partial B)>\varrho^{k+1}$$ and thus
  $J\notin \cA_{k+1}(I_1\cdot\cdot\cdot I_k)$, which in turn implies that
		 $$[I_1I_2\cdot\cdot\cdot I_kJ]\cap E_{k+1}=\emptyset\,.$$ The above discussion shows that
	\begin{eqnarray}\label{intersection}
	\ubar{\mu}\big([I_1I_2\cdot\cdot\cdot I_k]\cap E_{k+1}\big)\leq\ubar{\mu}([I_1I_2\cdot\cdot\cdot I_k])-\ubar{\mu}([I_1I_2\cdot\cdot\cdot I_kJ])\,.
		\end{eqnarray}
		Now since $\ubar{\mu}$ is doubling and the sequence $\{M_{\varrho}(I)\}_{I\in\Sigma^*}$ is  bounded (see \eqref{supmirhoi}),  there exists $\eta\in(0,1)$ such that $$\ubar{\mu}([IJ])\geq\eta\cdot\ubar{\mu}([I]),\,\,\,\,\,\,\,\,\forall\,I\in\Sigma^*,\,\forall\,J\in\Lambda_{\varrho}(I).$$ With this and (\ref{intersection}) in mind,  it follows that
		\begin{eqnarray*}
			\ubar{\mu}\big([I_1I_2\cdot\cdot\cdot I_k]\cap E_{k+1}\big)\leq(1-\eta)~\ubar{\mu}([I_1I_2\cdot\cdot\cdot I_k]).
		\end{eqnarray*}
		Since $E_{k+1}\subseteq E_k$,  the upshot of the above is that
		\begin{eqnarray*}
			\ubar{\mu}(E_{k+1})&=&\mu(E_k\cap E_{k+1})\\[3pt]
&=&\sum_{I_1\in\cA_1}\sum_{I_2\in\cA_2(I_1)}\cdot\cdot\cdot\sum_{I_k\in\cA_{k}(I_1I_2\cdot\cdot\cdot I_{k-1})}\ubar{\mu}([I_1I_2\cdot\cdot\cdot I_k]\cap E_{k+1})\\[3pt]
			&\leq&(1-\eta)\cdot\sum_{I_1\in\cA_1}\sum_{I_2\in\cA_2(I_1)}\cdot\cdot\cdot\sum_{I_k\in\cA_{k}(I_1I_2\cdot\cdot\cdot I_{k-1})}\ubar{\mu}([I_1I_2\cdot\cdot\cdot I_k])\\[3pt]
			&=&(1-\eta)\cdot\ubar{\mu}(E_k).
		\end{eqnarray*}
		Therefore, for any $k\geq1$,  $$\ubar{\mu}(E_k)\leq(1-\eta)^{k-1}\,.$$ Combining this inequality with the following easily verified inclusion
		\[
		\pi^{-1}\big(K\cap(\partial B)_{\varrho^k}\big)\subseteq E_k,   \qquad \forall\, k\geq1,
		\]
		we obtain that
		\begin{eqnarray*}
			\mu\big((\partial B)_{\varrho^k}\big)&=&\ubar{\mu}\left(\pi^{-1}\big(K\cap(\partial B)_{\varrho^k}\big)\right)\\[3pt]
			&\leq&\ubar{\mu}(E_k)\\[3pt]
			&\leq&(1-\eta)^{k-1}\\[3pt]
			&=&\varrho^{k\cdot\frac{(k-1)\log(1-\eta)}{k\log\varrho}}\\
            &\leq&\varrho^{k\cdot\frac{\log(1-\eta)}{2\log\varrho}}
            ,\,\,\,\,\forall~k\geq2.
		\end{eqnarray*}
       This shows \eqref{muparbrhok} with $\delta=\log(1-\eta)/(2\log\varrho)$ and $N=2$, and thus the proof of part (a) of Proposition \ref{Prop3.1} (ii) is complete.

	 \subsubsection{Proof of part (b) of Proposition \ref{Prop3.1} (ii)} Under the setting of  part (b), the self-conformal set $K$ is contained in a $\ell_K$ dimensional affine space or $\ell_K$ dimensional geometric sphere, and   not contained in any $(\ell_K-1)$-dimensional $C^1$ submanifold.

       The following result can be viewed as an analogue of Lemma~\ref{3.2}. Throughout, we define $\bd(A,\emptyset):=+\infty$ for a non-empty subset $A\subseteq\R^d$ and let $\varrho_0>0$ be as in \eqref{defofvarrho0}.
	\begin{lemma}\label{lem3.3}
		Let $d\geq2$ and let $\ell\in\{1,...,d-1\}$. Suppose that $K\subseteq S$ where $S$ is a $\ell$-dimensional geometric sphere or a $\ell$-dimensional affine hyperplane in $\R^d$, and $K$ is not contained in any $(\ell-1)$-dimensional $C^1$ submanifold. Then there exists $\varrho\in(0,\varrho_0)$ that satisfies the following statement: given any $I\in\Sigma^*$ and any ball $B\subseteq\R^d$ centered at $K$, there exists $J\in\Lambda_{\varrho}(I)$ such that
        \begin{eqnarray}\label{dkijparbsge}
            \bd(K_{IJ},(\partial B)\cap S)>\varrho\,|K_{I^-}|\,.
        \end{eqnarray}
	\end{lemma}
	\begin{proof}
		The proof is similar to that of Lemma~\ref{3.2}. Firstly, under the setting of Lemma~\ref{lem3.3}, we shall show that
		\begin{eqnarray}\label{case2dist}
			\delta := \inf_{I\in\Sigma^*}\inf_{\z\in K}~\inf_{r>0}~\sup_{\x\in K}~\bd(\psi_I\circ\varphi_I(\x),\psi_I(\partial B(\z,r)\cap S))>0,
		\end{eqnarray}
		where $\psi_I$ is defined as in (\ref{psi}). Let $S\subseteq\R^d$ be the set described in the statement of Lemma~\ref{lem3.3}. When $S$ is a $\ell$-dimensional affine hyperplane,  the  set $K$ can be viewed  as a self-conformal set in $\R^{\ell}$ and thus the  proof of (\ref{case2dist}) is the same as that of (\ref{fm3.1}). When $S$ is a $\ell$-dimensional geometric sphere, the proof remains largely the same with only minor modifications. Nevertheless, we sketch the proof of (\ref{case2dist}) in the latter case.
		
		Assume that $K\subseteq S$ where $S$ is a $\ell$-dimensional geometric sphere with radius $R$, and $K$ is not contained in any $\ell-1$-dimensional $C^1$ submanifold. Suppose in contradiction that (\ref{case2dist}) is not true, then there exist $\{I_k\}_{k\geq1}\subseteq\Sigma^*$, $\{\z_k\}_{k\geq1}\subseteq K$ and $\{r_k\}_{k\geq}\subseteq (0,+\infty)$ such that
		\begin{eqnarray*}
			\sup_{\x\in K} \bd(\psi_{I_k}\circ\varphi_{I_k}(\x),\psi_{I_k}(\partial B(\z_k,r_k)\cap S))\to 0\quad\text{as}\quad k\to\infty.
		\end{eqnarray*}
		 Let $U\subseteq\R^d$ be the bounded connected open set associated with (\ref{p2}). By (\ref{biblip}), the uniformly boundedness of $\{\psi_I\circ\varphi_I\}_{I\in\Sigma^*}$ and Lemma~\ref{concon},  we may assume that there exists  a conformal map $f:U\to\R^d$  such that $\psi_{I_k}\circ\varphi_{I_k} \to f$ uniformly on $U$.   Then
\begin{eqnarray}\label{dist2}
			\sup_{\x\in K} \bd(f(\x),\psi_{I_k}(\partial B(\z_k,r_k)\cap S))\to 0\quad\text{as}\quad k\to\infty.
		\end{eqnarray}
		Since $\#\big(\partial B(\z,r)\cap S\big)\leq 1$  when $\z\in S$ and $r\geq 2R$, we may assume that $\{r_k\}_{k\geq1}\subseteq (0,2R)$. Thus, $\partial B(\z_k,r_k)\cap S$ is a $(\ell-1)$-dimensional geometric sphere for any $k\geq1$. Let
		\[
		S^{\ell-1}:=\left\{(x_1,...,x_{\ell},0,...,0)\in\R^d:\sum_{i=1}^{\ell}x_i^2=1\right\}.
		\]
		Let $O(d)$ denote the collection of all $d\times d$  orthogonal matrices. Then for any $k\geq1$, there exist $\rho_k>0$, $O_k\in O(d)$ and $\bb_k\in\R^d$ such that
		\[
		\psi_{I_k}(\partial B(\z_k,r_k)\cap S)=\rho_k\, O_k(S^{\ell-1})+\bb_k.
		\]
		 By passing to a subsequence if necessary, we may assume that
         \[
         \rho_k\to\rho_{\infty},\qquad\bb_k\to\bb_{\infty}, \qquad O_k\to O_{\infty}\qquad\text{as}\qquad k\to\infty
         \]
         for some $\rho_{\infty}\in[0,+\infty]$, $\bb_{\infty}\in\R^d\cup\{\infty\}$ and $O_{\infty}\in O(d)$. In view of (\ref{dist2}), this is only possible if $(\rho_{\infty},\bb_{\infty})\in[0,+\infty)\times\R^d$ or $(\rho_{\infty},\bb_{\infty})\in\{+\infty\}\times\{\infty\}$. On following  the arguments used after \eqref{contra}, in both these cases  we find that $K$ is contained in a $\ell-1$-dimensional $C^1$ submanifold. This is a contradiction and thus establishes (\ref{case2dist}).
         With inequality (\ref{case2dist}) at hand, the desired inequality \eqref{dkijparbsge} can be proved for  $\varrho\in(0,\varrho_0)$ sufficiently  small enough via calculations analogous to those used in deriving (\ref{dist1}) and (\ref{dist2'}).
	\end{proof}
	
	\medskip

		 We are in a position to establish  part (b) of Proposition \ref{Prop3.1} (ii). Under the setting of part (b),  recall that $K\subseteq S$ where $S$ is a $\ell_K$-dimensional affine hyperplane or a $\ell_K$-dimensional geometric sphere, but $K$ is not contained in any $(\ell_K-1)$-dimensional $C^1$ submanifold. To ease notation, in the following we simply write $\ell$ for $\ell_K$. 

  \noindent If $S$ is a $\ell$-dimensional affine hyperplane, then we can view $K$ as a self-conformal set in $\R^{\ell}$ and  the proof to establish part (b)  follows the same line of argument as that appearing in Section~\ref{pfofpaofii}. Thus, without loss of generality,  we assume that $S$ is a $\ell$-dimensional sphere with radius $R$.
		
		First, we show that there exists $C>0$ such that for any $\z\in S$, $r\leq 2R$ and $\varrho>0$
		\begin{eqnarray}\label{annulus}
			(\partial B(\z,r))_{\varrho}\cap S\subseteq (\partial B(\z,r)\cap S)_{5\sqrt{2R}\sqrt{\varrho}}   \, .
		\end{eqnarray}
		 With this immediate goal in mind, fix an arbitrary point $\x\in(\partial B(\z,r))_{\varrho}\cap S$. By applying  an isometric mapping on $\R^d$ if necessary, we can assume, without loss of generality, that
		\[
		S=\left\{(y_1,...,y_{\ell+1},0,...,0)\in\R^d:\sum_{i=1}^{\ell+1}y_i^2=R^2\right\}
		\]
		and that
		\[
		\z=(0,R,0,...,0),   \quad {\rm and }  \quad \x=(x_1,x_2,0,...,0)   \ \ {\rm with }  x_1\geq0  \, .
		\]
		Then since  $\x\in S$,    a straightforward calculation shows that
		\[
		\x=\left(|\x-\z|\cdot\sqrt{1-\frac{|\x-\z|^2}{4R^2}},R-\frac{|\x-\z|^2}{2R},0,...,0\right).
		\]
		To prove (\ref{annulus}), we need to show  that
        \begin{eqnarray}\label{xinparbcap}
            \x\in\big(\partial B(\z,r)\cap S\big)_{2\sqrt{2R}\sqrt{\varrho}}\,;
        \end{eqnarray}
         that is to find  $\y\in\partial B(\z,r)\cap S$ such that $|\x-\y|\leq 5\sqrt{2R}\sqrt{\varrho}$. To do this, let
		\[
		\y=\left(r\cdot\sqrt{1-\frac{r^2}{4R^2}},R-\frac{r^2}{2R},0,...,0\right)  \, .
		\]
  Then $\y\in\partial B(\z,r)\cap S$ and
	\begin{eqnarray}
		|\x-\y|&\leq&\left||\x-\z|\cdot\sqrt{1-\frac{|\x-\z|^2}{4R^2}}-r\cdot\sqrt{1-\frac{r^2}{4R^2}}\right|+\frac{\big||\x-\z|-r\big|\cdot\big(|\x-\z|+r\big)}{2R}\nonumber\\[5pt]
		&\leq&r\cdot\left|\sqrt{1-\frac{|\x-\z|^2}{4R^2}}-\sqrt{1-\frac{r^2}{4R^2}}~\right|+\big||\x-\z|-r\big|\cdot\sqrt{1-\frac{|\x-\z|^2}{4R^2}}\nonumber\\[5pt]
        &~&+\,2\,\big||\x-\z|-r\big|\nonumber\\[5pt]
		&\leq&\frac{\big||\x-\z|^2-r^2\big|}{\sqrt{4R^2-|\x-\z|^2}+\sqrt{4R^2-r^2}}+3\big||\x-\z|-r\big|\nonumber\\[5pt]
        &=&\frac{\big||\x-\z|+r\big|\cdot\big||\x-\z|-r\big|}{\sqrt{(2R+|\x-\z|)(2R-|\x-\z|)}+\sqrt{(2R+r)(2R-r)}}+3\big||\x-\z|-r\big|\nonumber\\[5pt]
		&\leq&2\sqrt{2R}\cdot\frac{\big||\x-\z|-r\big|}{\sqrt{2R-|\x-\z|}+\sqrt{2R-r}}+3\big||\x-\z|-r\big|\label{lastthree}\\[5pt]
		&\leq&5\sqrt{2R}\sqrt{\big||\x-\z|-r\big|}\label{lasttwo}\\[5pt]
		&\leq&5\sqrt{2R}\sqrt{\varrho}\label{lastone}\,  .
	\end{eqnarray}
	In the above,   inequality \eqref{lastthree}  follows since
    \[
    |\x-\z|+r\leq 4R,\quad 2R+|\x-\z|\geq 2R,\quad 2R+r\geq2R\qquad(\x,\z\in S,\,r\leq 2R)
    \]
    and inequality \eqref{lasttwo} is a consequence of the fact that
    \begin{eqnarray*}
        \frac{\big||\x-\z|-r\big|}{\sqrt{2R-|\x-\z|}+\sqrt{2R-r}}\leq \sqrt{\big||\x-\z|-r\big|}
    \end{eqnarray*}
    and
    \begin{eqnarray*}
        \big||\x-\z|-r\big|=\left(\sqrt{\big||\x-\z|-r\big|}\right)^2\leq\sqrt{2R}\cdot\sqrt{\big||\x-\z|-r\big|}\, .
    \end{eqnarray*}
   The last inequality \eqref{lastone} follows since  $\x\in\big(\partial B(\z,r)\big)_{\varrho}$ and so
    $
    \big||\x-\z|-r\big|\leq\varrho\,.
$
    The upshot of the above is  that      \eqref{xinparbcap}  is true and thus we have established (\ref{annulus})  as desired.

 Now by  Lemma~\ref{lem3.3} and on following the  arguments  used  in the proof of part (a) of Proposition~\ref{Prop3.1}~(ii) with $\partial B$ replaced by $\partial B(\z,r)\cap S$, we find that  there exists constants $C>0$ and $\delta>0$  such that
	\[
	\mu\big((\partial B(\z,r)\cap S)_{\varrho}\big)\leq C \varrho^{\delta}
	\]
	for any $\z\in S$, any $r>0$ and any $\varrho>0$. This together  with (\ref{annulus}), implies that
	\begin{eqnarray}\label{annmea}
		\mu\big((\partial B(\z,r))_{\varrho}\big)\leq C\cdot 5^{\delta} (2R)^{\delta/2}\varrho^{\delta/2}
	\end{eqnarray}
	for any $\z\in S$, any $r\leq 2R$ and any $\varrho>0$. If $\z\in S$ and $r>2R$, note that
	\begin{eqnarray*}
		(\partial B(\z,r))_{\varrho}\cap S=\left\{
		\begin{aligned}
			&\emptyset,\qquad \qquad \qquad \qquad \quad\,\,\,\text{if} \  \ r-\varrho>2R,\\[2ex]
			&(\partial B(\z,2R))_{\varrho-r+2R}\cap S, \ \ \text{if}~r-\varrho\leq 2R
		\end{aligned}
		\right.
	\end{eqnarray*}
	and thus (\ref{annmea}) also holds when $r>2R$.  In other words (\ref{annpro}) holds for any $\z\in S$, any $r>0$ and any $\varrho>0$.  This  completes the proof of part (b) of Proposition~\ref{Prop3.1}~(ii).

	\subsubsection{Proof of part (c) of Proposition \ref{Prop3.1} (ii)}\label{pfofprcofproii}
	We start by  establishing various  preliminary lemmas regarding the properties of multivariate analytic functions. The first result  is concerned with the cardinality of  the level sets associated with such functions.
	\begin{lemma}\label{zeros}
		Let $F:[a,b]\times[c,d]\to\R$ be real analytic.
        \begin{itemize}
            \item[(i)] If $F(t,\cdot)\not\equiv0$ for any $t\in[a,b]$, then
		\begin{eqnarray}\label{count1}
			\sup_{t\in[a,b]}\#\big\{s\in[c,d]:F(t,s)=0\big\}<+\infty.
		\end{eqnarray}

        \item[(ii)] If $F(t,\cdot)$ is not a constant function for any $t\in[a,b]$, then
        \begin{eqnarray}\label{countptsle}
            \sup_{p\in\R}\sup_{t\in[a,b]}\#\left\{s\in[c,d]:F(t,s)=p\right\}<+\infty.
        \end{eqnarray}
        \end{itemize}
	\end{lemma}
	\begin{proof}
		(i) Suppose on the contrary that the left-hand-side of (\ref{count1}) equals infinity. Then for any $n\in\N$, there exist $t_n\in[a,b]$ and distinct points $s_1^{(n)},...,s_n^{(n)}\in[c,d]$ such that
		\begin{eqnarray}\label{count2}
			F\big(t_n,s_i^{(n)}\big)=0,\qquad~\forall\, i=1,2,...,n.
		\end{eqnarray}
		Without  loss of generality, we assume that $t_n\to t_0$ for some $t_0\in[a,b]$ as $n\to\infty$. Given an integer $i\geq1$, denote by $\cS_i$ the set of limit points of  $\{s_i^{(n)}:n\geq i\}$. Let $$\cS=\bigcup_{i=1}^{\infty}\cS_i \, .$$ Since $F$ is real analytic, it follows  by (\ref{count2}) and the definition of $\cS$, that
		\begin{eqnarray}\label{equal0}
			F(t_0,s)=0,\qquad\forall\, s\in\cS.
		\end{eqnarray}
		We now proceed by considering two case.

		\emph{Case 1: $\#\cS=+\infty$}. In this case, there exist infinitely many $s\in[c,d]$ such that $F(t_0,s)=0$. Note that $F(t_0,\cdot)$ is real analytic on $[c,d]$, then in view of  it follows that  $F(t_0,\cdot)\equiv0$ on $[c,d]$, which is a contradiction.
		
		\emph{Case 2: $\#\cS<+\infty$}. In this case, there exists $s_0\in\cS$ such that $s_0\in\cS_i$ for infinitely many $i\in\N$. Without loss of generality, we assume that $s_0\in\cS_i$ for all $i\geq1$.   We now show that
		\begin{eqnarray}\label{deuql0}
			\frac{\partial^k F}{\partial s^k}(t_0,s_0)=0 \, , \qquad \forall\ k\geq0.
		\end{eqnarray}
	 By (\ref{equal0}), the formula in (\ref{deuql0}) holds when $k=0$. Fix any $k\geq1$. By passing to a subsequence, we assume that
		\begin{eqnarray}\label{limit3}
			s_i^{(n)}\to s_0\,\,\,\,(n\to\infty),\,\,\,\,\,\,\,\,\forall\, i=1,2,...,k+1.
		\end{eqnarray}
		We denote
		\[
		u_n=\min\big\{s_i^{(n)}:i=1,2,...,k+1\big\},~~v_n=\max\big\{s_i^{(n)}:i=1,2,...,k+1\big\}
		\]
		for any $n\geq k+1$. By (\ref{limit3}), we have
		\begin{eqnarray}\label{limit4}
			u_n\to s_0,~~v_n\to s_0~~~\text{as}~~~n\to\infty.
		\end{eqnarray}
		Given $n\in\N_{\geq k+1}$, note that $s_1^{(n)},s_2^{(n)},...,s_{k+1}^{(n)}$ are all distinct, then by (\ref{count2}) and Rolle's Theorem, there exists $\xi_n\in(u_n,v_n)$ such that
		\[
		\frac{\partial^kF}{\partial s^k}(t_n,\xi_n)=0.
		\]
		Letting $n\to\infty$, we obtain from (\ref{limit4}) and the continuity of $\frac{\partial^k F}{\partial s^k}$ that
		\[
		\frac{\partial^kF}{\partial s^k}(t_0,s_0)=0  \, .
		\]
        This establishes \eqref{deuql0}. Now since  $k$ is arbitrary, \eqref{deuql0} together with the fact that $F(t_0,\cdot)$  is analytic on $[c,d]$ implies that  $F(t_0,\cdot)\equiv0$ on $[c,d]$.  This contradicts the hypothesis  of part (i) and so we are done.
\medskip

        (ii) The proof is a modification of that of part (i). Suppose that \eqref{countptsle} is not true. Then for any $n\in\N$, there exist $p_n\in\R$ and $t_n\in[a,b]$ such that
        \[
        \#\left\{s\in[c,d]:F(t_n,s)=p_n\right\}\geq n\,;
        \]
        that is to say, there exist distinct points $s_1^{(n)}$, $s_{2}^{(n)}$, ... , $s_n^{(n)}\in[c,d]$ such that
        \begin{eqnarray}\label{ftnsinpn}
            F(t_n,s_i^{(n)})=p_n,\qquad\forall\,i=1,2,...,n.
        \end{eqnarray}
         Note that $F$ is real analytic on a compact set, hence the range of $F$ is bounded and so is $\{p_n\}_{n\in\N}$. Therefore, without loss of generality, we may assume that
        \[
        t_n\to t_0,\qquad p_n\to p_0\qquad(n\to\infty)
        \]
        for some $t_0\in[a,b]$ and $p_0\in\R$. Given an integer $i\geq1$, denote by $\cS_i$ the set of limit points of  $\{s_i^{(n)}:n\geq i\}$. Let $\cS$ be the union of $\cS_i$\,\,$(i\geq1)$. Since $F$ is real analytic, it follows  by (\ref{ftnsinpn}) and the definition of $\cS$, that
        \begin{eqnarray}\label{ft0sp0}
            F(t_0,s)=p_0,\qquad\forall\,s\in\cS.
        \end{eqnarray}
As in the proof of (i), we proceed by considering two case.

        \emph{Case 1: $\#\cS=+\infty$}. In this case, by \eqref{ft0sp0}, there exist infinitely many $s\in[c,d]$ such that $F(t_0,s)=p_0$. Note that $s\mapsto F(t_0,s)$ is real analytic on $[c,d]$, then it follows that  $F(t_0,\cdot)\equiv p_0$ on $[c,d]$, which is a contradiction.

        \emph{Case 2: $\#\cS<+\infty$.} Similar to Case 2 in the proof of part (i), there exists $s_0\in\cS$ such that
        \[
        \frac{\partial^k F}{\partial s^k}(t_0,s_0)=0,\qquad\forall\,k\geq1.
        \]
        Then by the analyticity of the map $s\mapsto F(t_0,s)$  $(s\in[c,d])$, we have $F(t_0,s)=F(t_0,s_0)=p_0$ for any $s\in[c,d]$, which is a contradiction.
	\end{proof}

    The next lemma is concerned  with bounding from below the  derivatives of analytic functions.
	
	\begin{lemma}\label{dermin}
		If $F:[a,b]\times[c,d]\to\R$ is real analytic and $F(t,\cdot)\not\equiv0$ for any $t\in[a,b]$, then there exist positive integer $n_0\in\N$ and $\delta_0>0$ such that
		\begin{eqnarray*}
			\max_{0\leq i\leq n_0}\Big|\frac{\partial^i F}{\partial s^i}(t,s)\Big|\geq\delta_0,\,\,\,\,\,\,\,\,\forall\, (s,t)\in[a,b]\times[c,d].
		\end{eqnarray*}
	\end{lemma}
	\begin{proof}
		Suppose on the contrary the  result is not  true. Then   for any integer $n\geq1$, we can find $(s_n,t_n)\in[a,b]\times[c,d]$ such that
  \begin{eqnarray}\label{deeas1}
      \Big|\frac{\partial^iF}{\partial s^i}(t_n,s_n)\Big|<\frac{1}{n},\,\,\,\,\,\,\,\,\forall\, i=0,1,...,n.
  \end{eqnarray}
		By passing to a subsequence if necessary, we can assume that $(t_n,s_n)\to(t_0,s_0)$ for some $(t_0,s_0)\in[a,b]\times[c,d]$. Then letting $n\to\infty$ on both sides of (\ref{deeas1}),  gives
		\[
		\frac{\partial^iF}{\partial s^i}(t_0,s_0)=0,\,\,\,\,\,\,\,\,\forall\, i\geq0.
		\]
        This together with  the analyticity of $F(t_0,\cdot)$  implies that  $F(t_0,\cdot)\equiv0$, which contradicts the assumption that  $F(t,\cdot)\not\equiv0$ for any $t\in[a,b]$.
	\end{proof}

    \medskip

	Given  a function $F:[a,b]\times[c,d]\to\R$, for any $t\in[a,b]$ and any $-\infty<r_1<r_2<+\infty$, let
	\[
	\cF(t,r_1,r_2):=\left\{s\in[c,d]:r_1<F(t,s)<r_2\right\}
	\]
	  and denote by  $\cI(t,r_1,r_2)$ the collection of all connected components of $\cF(t,r_1,r_2)$. The next result show that  cardinalities of the sets $\cI(t,r_1,r_2)$ associated with analytic functions are bounded.
	
	\begin{lemma}\label{cccount}
		If $F:[a,b]\times[c,d]\to\R$ is real analytic and  $F(t,\cdot):[c,d]\to\R$ is not a constant function for any $t\in[a,b]$,  then
		\begin{eqnarray*}
			\sup\big\{\#\cI(t,r_1,r_2):t\in[a,b],-\infty<r_1<r_2<+\infty\big\}<+\infty.
		\end{eqnarray*}
	\end{lemma}
 \begin{proof}
     Fix $t\in[a,b]$ and fix $r_1<r_2$. Let  $r_0\in(r_1,r_2)$. Note that since  $F$ is analytic,   $\cI(t,r_1,r_2)$ is a collection of disjoint (relatively) open intervals on $[c,d]$.  For brevity of notation, in the following  we write $\cI= \cI(t,r_1,r_2)$.

     The goal is to estimate the number of elements in $\cI$. To do this, we partition  $\cI$ into three parts:
     \[
         \cI=\cI_1\cup\cI_2\cup\cI_3
     \]
     where
     \begin{eqnarray*}
         \cI_1&:=&\left\{I\in\cI:c\in \overline{I}~\text{or}~d\in \overline{I} \right\},\\[4pt]
         \cI_2&:=&\left\{I\in\cI\backslash\cI_1:\exists\, s\in I\ \ \text{s.t.}\ \ F(t,s)=r_0\right\},\\[4pt]
         \cI_3&:=&\cI\backslash(\cI_1\cup\cI_2).
     \end{eqnarray*}
     It is obvious that $\#\cI_1\leq2$. Note that $F(t',\cdot)$ is not a constant function for any $t'\in[a,b]$. Thus, by   the definition of $\cI_2$ and applying part (ii) of Lemma~\ref{zeros} to the function $F$, we have
     \begin{eqnarray*}
\#\cI_2&\leq&\sup_{t'\in[a,b]}\#\left\{s\in[c,d]:F(t',s)=r_0\right\}\\[4pt]
&\leq&\sup_{p\in\R}\sup_{t'\in[a,b]}\#\left\{s\in[c,d]:F(t',s)=p\right\}\\[4pt]
&<&+\infty.
     \end{eqnarray*}

     It remains to estimate the number of elements  of $\cI_3$. If $I\in\cI_3$, then its closure $\overline{I}\subseteq(c,d)$ and $r_0\notin F(t,I)$. We claim that  there must be $s\in I$ such that $$\frac{\partial F}{\partial s}(t,s)=0  \, .  $$ If not, then $F(t,\cdot)$ is monotone on $I$. Write $I=(p,q)$ where $c<p<q<d$. Without the loss of generality, we assume that $F(t,\cdot)$ is increasing on $I$.  Under this assumption, since  $r_0\notin F(t,I)$, we have
     \[
     F(t,s)\geq r_0>r_1 \qquad \text{or} \qquad F(t,s)\leq r_0<r_2
     \qquad(s\in I)\,.
     \]
     In the above former case, by the continuity of $F(t,\cdot)$, there exists $\delta>0$ such that $(p-\delta,q)\subseteq (c,d)$ and $F(t,(p-\delta,q))\subseteq (r_1,r_2)$, which contradicts the fact that $I=(p,q)$ is a connected component of $\cF(t,r_1,r_2)$. The latter case can be treated similarly. So we have proven the claim  that any interval in $\cI_3$ must contain at least one zero of $\frac{\partial F}{\partial s}(t,\cdot)$.  Recall that $F(t^{'},\cdot)$ is not a constant function for any $t^{'}\in[a,b]$, so $\frac{\partial F}{\partial s}(t^{'},\cdot)\not\equiv0$ for any $t^{'}\in[a,b]$. With this in mind, on applying part (i) of Lemma~\ref{zeros} to the real analytic function $\frac{\partial F}{\partial s}$, we have
     \[
         \#\cI_3\leq\sup_{t^{'}\in[a,b]}\#\left\{s\in[c,d]:\frac{\partial F}{\partial s}(t^{'},s)=0\right\}<+\infty.
     \]

     The upshot of the above  is  that
     \begin{eqnarray*}
         \#\cI&\leq&2+\sup_{p\in\R}\sup_{t^{'}\in[a,b]}\#\left\{s\in[c,d]:F(t^{'},s)=p\right\}\\&~&\,\,\,\,+\sup_{t^{'}\in[a,b]}\#\left\{s\in[c,d]:\frac{\partial F}{\partial s}(t^{'},s)=0\right\}\\[4pt]
         &<&+\infty,
     \end{eqnarray*}
     which  completes the proof.
 \end{proof}

The following rather elementary  result while be useful  in obtaining lower bounds for the integral of analytic functions (namely in proving the subsequent  Lemma~\ref{intlowbd}).

\begin{lemma}\label{normequi}
    Let $n\geq1$ be an integer. Then there exists $\gamma=\gamma(n)>0$ such that
    \[
        \int_{-1}^{1}\big|a_0+a_1x+\cdot\cdot\cdot+a_nx^n\big|~\td x\geq\gamma\cdot\max_{0\leq i\leq n}|a_i|
    \]
    for all $(a_0,a_1,...,a_n)\in\R^{n+1}$.
\end{lemma}
\begin{proof}
    Let $\cP_n([-1,1])$ be the collection of all polynomials on $[-1,1]$ with coefficients in $\R$ and degrees at most $n$, then $\cP_n([-1,1])$ is a $(n+1)$-dimensional (real)  vector space under  addition and scalar multiplication. Given $f(x)=a_0+a_1x+\cdot\cdot\cdot+a_nx^n\in\cP_n([-1,1])$, we define
    \begin{eqnarray*}
        \|f\|_1&:=&\int_{-1}^{1}|f(x)|~\td x,\\
        \|f\|_2&:=&\max_{0\leq i\leq n}|a_i|.
    \end{eqnarray*}
    It is easy to show that both $\|\cdot\|_1$ and $\|\cdot\|_2$ are norms on $\cP_{n}([-1,1])$.  Since any two norms on a finite dimensional vector space are equivalent, there exists $\gamma>0$ such that
    \[
       \|f\|_1\geq \gamma\|f\|_2\,,\qquad\forall\ f\in\cP_n([-1,1])\,,
    \]
    which proves the lemma.
\end{proof}

\medskip

\begin{lemma}\label{intlowbd}
    Let $F:[a,b]\times[c,d]\to\R$
    be real analytic and suppose that $F(t,\cdot)\not\equiv0$ for any $t\in[a,b]$. Then there exist $C>0$ and $n_0\in\N$ such that for any $t\in[a,b]$ and any sub-interval $I\subseteq[c,d]$, we have
    \begin{eqnarray}\label{intlow}
        \int_I|F(t,s)|~\td s\geq C|I|^{n_0}.
    \end{eqnarray}
\end{lemma}

\begin{proof}
    By Lemma~\ref{dermin}, there exist positive integer $n_0\geq2$ and $\delta_0>0$ such that
\begin{eqnarray}\label{derest}
    \max_{0\leq i\leq n_0-1}\Big|\frac{\partial^{i}F}{\partial s^i}(t,s)\Big|\geq\delta_0,\ \ \ \ \forall\, (t,s)\in[a,b]\times[c,d].
\end{eqnarray}
    Let
    \[
        M:=\max\left\{\Big|\frac{\partial^{n_0}F}{\partial s^{n_0}}(t,s)\Big|:(t,s)\in[a,b]\times[c,d]\right\}.
    \]
    Throughout, let $I$ be a sub-interval of $[c,d]$   and let $s_I$ be the center of $I$. Also, we fix an arbitrary point $t\in[a,b]$.
    By Taylor's theorem,
\begin{eqnarray}\label{taylor}
        F(t,s)=\sum_{i=0}^{n_0-1}\frac{1}{i!}\cdot\frac{\partial^iF}{\partial s^i}(t,s_I)\cdot(s-s_I)^i+\epsilon(t,s),\,\,\,\,\,\,\,\,\forall\, (t,s)\in[a,b]\times[c,d],
    \end{eqnarray}
    where the error term $\epsilon(t,s)$ satisfies
    \[
        |\epsilon(t,s)|\leq\frac{M}{(n_0+1)!}\cdot|s-s_I|^{n_0}.
    \]
    To ease notation, let
    \[
        G(t,s):=\sum_{i=1}^{n_0-1}\frac{1}{i!}\cdot\frac{\partial^i F}{\partial s^i}(t,s_I)\cdot(s-s_I)^i.
    \]
    Then  by (\ref{taylor}) and the triangle inequality, we have
    \begin{eqnarray}\label{I1-I2}
        \int_I |F(t,s)|~\td s &\geq& \int_I |G(t,s)|~\td s-\int_I|\epsilon(t,s)|~\td s\nonumber\\[0.5em]
        &=:&I_1-I_2,
    \end{eqnarray}
    where we set
    \[
    I_1:=\int_I |G(t,s)|~\td s,\ \ \ \ \ \ I_2:=\int_I|\epsilon(t,s)|~\td s\,.
    \]

    We first estimate the lower bound of $I_1$. Let $\gamma>0$ be  as in  Lemma~\ref{normequi} with $n=n_0-1$. Then on combining    (\ref{derest}) and Lemma~\ref{normequi}, we have
    \begin{eqnarray}\label{I1int}
         I_1&=&\int_I|G(t,s)|~\td s\nonumber\\[0.5em]
        &=&\frac{|I|}{2}\int_{-1}^{1}\big|G\Big(t,s_I+\frac{|I|}{2}s\Big)\big|~\td s\nonumber\\[0.5em]
&=&\frac{|I|}{2}\int_{-1}^{1}\big|\sum_{i=0}^{n_0-1}\frac{1}{i!}\cdot\frac{\partial^i F}{\partial s^i}(t,s_I)\cdot\Big(\frac{|I|}{2}s\Big)^i\big|~\td s\nonumber\\[0.5em]
&\geq&\frac{|I|}{2}\cdot\gamma\cdot\max_{0\leq i\leq n_0-1}\Big\{\frac{1}{i!}\cdot\big|\frac{\partial^i F}{\partial s^i}(t,s_I)\big|\cdot\big(\frac{|I|}{2}\big)^i\Big\}\nonumber\\[0.5em]
&\geq&\frac{\gamma}{2^{n_0}(n_0-1)!}\cdot\max_{0\leq i\leq n_0-1}\Big\{\big|\frac{\partial^i F}{\partial s^i}(t,s_I)\big|\Big\}\cdot \min\{|I|^{n_0},|I|\}\nonumber\\[0.5em]
&\geq&\frac{\gamma\delta_0}{2^{n_0}(n_0-1)!}\cdot\min\{|I|^{n_0},|I|\}.
\end{eqnarray}

Next we obtain an upper bound for   $I_2$:
\begin{eqnarray}\label{I2int}
    I_2&=&\int_{I}|\epsilon(t,s)|~\td s\nonumber\\[0.5em]
&\leq&\int_{I}\frac{M}{n_0!}\cdot|s-s_I|^{n_0}~\td s\nonumber\\[0.5em]
&=&\frac{M}{2^{n_0}(n_0+1)!}\cdot|I|^{n_0+1}.
\end{eqnarray}

\noindent In the following, let
\[
r_0 \, :=   \, \min\big\{\gamma\delta_0(n_0+1)(n_0+2)/(2M),d-c,1\big\}\,.
\]
We continue with estimating the integral of $|F(t,\cdot)|$ over $I$  by considering  two cases:
\begin{itemize}
    \item[$\circ$]  If $|I|\leq r_0$, then by \eqref{I1-I2},       (\ref{I1int}) and (\ref{I2int}), we have that
\begin{eqnarray}
    \int_{I}|F(t,s)|~\td s ~&\geq&~ \frac{1}{2^{n_0}(n_0-1)!}\left(\gamma\delta_0-\frac{M|I|}{n_0(n_0+1)}\right)\cdot|I|^{n_0}\nonumber\\[4pt]
    &\geq&\frac{\gamma\delta_0}{2^{n_0+1}(n_0-1)!}\cdot |I|^{n_0}\,.\label{Ismall}
\end{eqnarray}
\medskip

\item[$\circ$] If $|I|>r_0$, let $I'\subseteq I$ be a subinterval with $|I'|=r_0$. Then on applying  (\ref{Ismall}) to the interval $I'$, we have that
\begin{eqnarray}\label{Ibig}
    \int_I |F(t,s)|~\td s&\geq&\int_{I'}|F(t,s)|~\td s\nonumber\\[0.5em]
    &\geq&\frac{\gamma\delta_0}{2^{n_0+1}(n_0-1)!}\cdot r_0^{n_0}\nonumber\\[0.5em]
    &=&\frac{\gamma\delta_0}{2^{n_0+1}(n_0-1)!}\cdot\left(\frac{r_0}{|I|}\right)^{n_0}\cdot|I|^{n_0}\nonumber\\[0.5em]
    &\geq&\frac{\gamma\delta_0}{2^{n_0+1}(n_0-1)!}\cdot\left(\frac{r_0}{d-c}\right)^{n_0}\cdot|I|^{n_0}.
\end{eqnarray}
\end{itemize}
The proof of (\ref{intlow}) is complete by combining (\ref{Ismall}) and (\ref{Ibig}).
\end{proof}

\medskip

We are now  finally in the position to exploit the preparatory lemmas to prove part (c) of Proposition \ref{Prop3.1} (ii).
		 Recall that if $g:\cO\to\R^2$ is a conformal map on an open set $\cO\supseteq[0,1]\times\{0\}$, then the map  $t\mapsto g(t,0)$ is injective and real analytic with non-zero derivative with respect to $t\in[0,1]$. With this in mind, within the context of   part (c),  there exists an one-to-one real analytic function $f:[0,1]\to\R^2$ such that $f'(t)\neq0$ for any $t\in[0,1]$ and $K\subseteq\Gamma:=f([0,1])$. 
        Let $f(t)=(f_1(t),f_2(t))$ and define $F:[0,1]\times[0,1]\to\R$ as  $$F(t,s):=\big|f(t)-f(s)\big|^2=\big(f_1(t)-f_1(s)\big)^2+\big(f_2(t)-f_2(s)\big)^2.$$
  It is clear that $F$ is also real analytic. Furthermore, $F(t,\cdot)$ is not a constant function for any $t\in[0,1]$. Otherwise,   there exists $t_0\in[0,1]$ such that $F(t_0,s)=F(t_0,t_0)=0$ for all $s\in[0,1]$ and thus $f(t)\equiv f(t_0)$ for all $ t\in[0,1]$, which is a contradiction.  For any $t\in[0,1]$ and $r_1<r_2$, let
  \[
      \cF(t,r_1,r_2):=\big\{s\in[0,1]:r_1<F(t,s)<r_2\big\}
  \]
  and let $\cI(t,r_1,r_2)$ be the collection of connected components of $\cF(t,r_1,r_2)$.

  Throughout, fix $\x\in \Gamma$. Let $t\in[0,1]$  so  $\x=f(t)$. Then for any $r>0$ and $\varrho>0$, we have
  \[
      (\partial B(\x,r))_{\varrho}\cap K\  \subseteq \ \bigcup_{I\in\cI\big(t,~(r-\varrho)_{+}^2,~(r+\varrho)^2\big)}f(I)\cup\{\x\},
  \]
  where, as usual, $(x)_{+}:=\max\{x,0\}$. Hence
  \begin{eqnarray}\label{annmeas1}
       \mu\big((\partial B(\x,r))_{\varrho}\big)\,\,\leq\sum_{I\in\cI\big(t,~(r-\varrho)_{+}^2,~(r+\varrho)^2\big)}\mu\big(f(I)\big).
  \end{eqnarray}
  Recall, our goal is to establish \eqref{annpro} and so n view of \eqref{annmeas1} we now   estimate $\mu\big(f(I)\big)$ for any $I\in \cI\big(t,~(r-\varrho)_{+}^2,~(r+\varrho)^2\big)$.

  The one-dimensional version of the well-known Area Formula (see for example \cite[Theorem 3.7]{morgan2016}) states that if $g:[0,1]\to\R$ is Lipschitz, then for any measurable set $A\subseteq[0,1]$, we have
  \begin{eqnarray}\label{areaformula}
      \int_A |g'(s)|\,\td s=\int_{\R}\#\left\{s\in A:g(s)=p\right\}\,\td p\,.
  \end{eqnarray}
  Applying \eqref{areaformula} to $A=\cF\big(t,~(r-\varrho)_{+}^2,~(r+\varrho)^2\big)$ and $g(s)=F(t,s)$, we obtain that
  \begin{eqnarray}\label{coarea}
      \int_{\cF\big(t,~(r-\varrho)_{+}^2,~(r+\varrho)^2\big)}\Big|\frac{\partial F}{\partial s}(t,s)\Big|~\td s
      =\int_{(r-\varrho)_+^2}^{(r+\varrho)^2}\#\big\{s\in[0,1]:F(t,s)=p\big\}~\td p\,.
  \end{eqnarray}
Next, we estimate the  two integrals appearing in \eqref{coarea}. Recall that $F(t^{'},\cdot)$ is not a constant function for any $t^{'}\in[0,1]$,  then by part (ii) of Lemma \ref{zeros}, we find that
\[
M:=\sup_{p\in\R}\sup_{t^{'}\in[0,1]}\#\big\{s\in[0,1]:F(t^{'},s)=p\big\}<+\infty\,.
\]
 Furthermore, note that $|F(t^{'},s)|\leq |\Gamma|^2$ for all $(t^{'},s)\in[0,1]^2$, so it follows that \begin{eqnarray}\label{countint}
    {\rm r.h.s. \ of \ } \eqref{coarea} &= &\int_{[(r-\varrho)_+^2,(r+\varrho)^2]\,\cap\,[0,|\Gamma|^2]}\#\big\{s\in[0,1]:F(t,s)=p\big\}~\td p\nonumber\\[5pt]
    & \leq  &  M\cdot\cL\big([(r-\varrho)_+^2,(r+\varrho)^2]\cap[0,|\Gamma|^2]\big)\nonumber\\[5pt]
      & \leq   &   M\cdot\max\big\{|\Gamma|,\varrho\big\}\cdot\varrho\,,
  \end{eqnarray}
  where $\cL$ denotes the Lebesgue measure on $[0,1]$.
  On the other hand, by Lemma~\ref{intlowbd}, there exists $n_0\in\N$ and $C>0$ (independent of $t\in[0,1]$, $r>0$ and $\varrho>0$) such that
  \begin{eqnarray}\label{derint}
      {\rm l.h.s. \ of \ } \eqref{coarea} &= &
    \sum_{I\in\cI\big(t,~(r-\varrho)_+^2,~(r+\varrho)^2\big)}\int_{I}~\Big|\frac{\partial F}{\partial s}(t,s)\Big|~\td s\nonumber\\[2ex]
      &\geq&C\cdot\sum_{I\in\cI\big(t,~(r-\varrho)_+^2,~(r+\varrho)^2\big)}|I|^{n_0}.
  \end{eqnarray}
  On combining (\ref{coarea}), (\ref{countint}) and (\ref{derint}) with the fact that $|I| \le 1$ , it follows that
 \begin{eqnarray*}
     |I|  \, \ll \,  \varrho^{1/n_0}
 \end{eqnarray*}
 for any $I\in\cI\big(t,~(r-\varrho)_+^2,~(r+\varrho)^2\big)$. Since $f$ is injective and $f^{'}(t)\neq0$ for any $t\in[0,1]$, it is easily verified that $f:[0,1]\to\Gamma$ is bi-Lipschitz, and thus $f(I)$ is contained in a ball with radius approximately $\varrho^{1/n_0}$. Then by part (i) of Proposition \ref{Prop3.1},
 there exists $s>0$ such that
 \begin{eqnarray}\label{measfI}
\mu\big(f(I)\big)   \, \ll  \, \varrho^{s/n_0}
 \end{eqnarray}
 for any $I\in\cI(t,~(r-\varrho)_+^2,~(r+\varrho)^2)$. In view  of Lemma~\ref{cccount}, we know that
 \[
 \sup\left\{\#\cI(t,~(r-\varrho)_+^2,~(r+\varrho)^2):t\in[0,1],\,r>0,\,\varrho>0\right\}<+\infty.
 \]
 This together with (\ref{annmeas1}) and (\ref{measfI}) shows that
 \begin{eqnarray}\label{annmeaana}
     \mu\big((\partial B(\x,r))_{\varrho}\big)  \, \ll \, \varrho^{s/n_0}
 \end{eqnarray}
 for any $\x\in\Gamma$, any $r>0$ and any $\varrho>0$,
where the implied  constant does not depend on $\x$, $r$ and $\varrho$. This completes the proof of part (c) of Proposition~\ref{Prop3.1} (ii).

\medskip

In order to complete the proof of Proposition~\ref{Prop3.1} (ii) it remains to establish \eqref{annpro} in the cases not covered by parts (a), (b) or (c).

 \subsubsection{Completing the proof of Proposition~\ref{Prop3.1} (ii)} \label{junjieZ}

 Let $\ell_K$ be as  in \eqref{defofellk}.  Then, in view of   Proposition~\ref{rigidity}, we divide the proof of   \eqref{annpro} into the following  cases:

 \begin{itemize}[leftmargin=5em]
     \item[Case A:]\label{A} $\ell_K=d$; \vspace*{1ex}

     \item[Case B:]\label{B} $d\geq2$,
     $\ell_K<d$ and part (ii) of Proposition \ref{rigidity} holds with $\ell=\ell_K$;
     \vspace*{1ex}

     \item[Case C:]\label{C}

     $d=2$ and $K$ is contained in an analytic curve;
     \vspace*{1ex}

     \item[Case D:]\label{D} $d=2$ and $K$ is contained in a disjoint union of at least two analytic curves.
 \end{itemize}
To see this, simply note that when $d=1$, then $\ell_K=1$ and thus we are in Case A. The first three cases have been considered respectively in Sections \ref{pfofpaofii} -- \ref{pfofprcofproii}. Thus it remains  to establish the desired inequality \eqref{annpro} for Case D. So suppose
 that
 \[
     K\subseteq\bigsqcup_{i=1}^k\Gamma_i\,,
 \]
 where $k\geq2$ and each $\Gamma_i$ $(i=1,2,...,k)$ is an analytic curve. For each $i=1,2,...,k$, denote by $\mu_i:=\mu|_{\Gamma_i}$ the restriction of $\mu$ supported on the analytic curve $\Gamma_i$.  Then, on naturally adapting  the arguments used in deriving \eqref{annmeaana}, we find  that there exist $C>0$ and $\delta>0$ such that
  \begin{eqnarray*}
     \mu_i\big((\partial B(\x,r))_{\varrho}\big)\leq C\varrho^{\delta}
 \end{eqnarray*}
  for any $1\leq i\leq k$, any $\x\in \Gamma_i$, any $r>0$, and any $\varrho>0$.
  Note that  $\Gamma_i$ ($i=1,2,...,k$) are disjoint closed sets, thus there exists $r_0>0$ such that
  \[
  \bd(\Gamma_i,\Gamma_j)>2\,r_0,\qquad \forall\ 1\leq i\neq j\leq k.
  \]
  With this in mind, it follows that for any $1\leq i\leq k$ and any $\x\in\Gamma_i$\,,
 \[
     (\partial B(\x,r))_{\varrho}\cap K\ \subseteq\ (\partial B(\x,r))_{\varrho}\cap\Gamma_i,\qquad \forall~0<r\leq r_0\,,\ \ \forall~ 0<\varrho\leq r_0\,.
 \]
 The upshot of the above is that for any $\x\in K$,
 \[
     \mu\big((\partial B(\x,r))_{\varrho})\leq C\varrho^{\delta},\qquad \forall~ 0< r\leq r_0,~\forall~ 0<\varrho\leq r_0.
 \]
 If $\varrho>r_0$, then since $\mu$ is a probability measure, we have that for any $r>0$,
 \begin{eqnarray*}
     \mu\big((\partial B(\x,r))_{\varrho})\leq1<\left(\frac{\varrho}{r_0}\right)^{\delta}\,.
 \end{eqnarray*}
 Thus, it follows that for any $\x\in K$, any $0<r\leq r_0$ and any $\varrho>0$, we obtain the desired inequality
 \[
 \mu\big((\partial B(\x,r))_{\varrho}\big)\leq \widetilde{C}\,\varrho^\delta\quad\text{where}\quad\widetilde{C}:=\max\left\{C,\frac{1}{r_0^{\delta}}\right\}\,.
 \]


 \section{Applications: proving the statements appearing  in Section~\ref{appintro}  }\label{SEC5}

In this section we establish  the applications of Theorem~\ref{Main2} to the recurrent sets stated in Section~\ref{appintro}.  This will involve establishing a more versatile form of the standard quantitative  Borel-Cantelli Lemma (see Lemma~\ref{countlem})  that is required when proving Theorem~\ref{toprove}.  The final section is devoted to providing the details of the two counterexamples to Claim~F discussed in Section~\ref{appintro}.

For the sake of clarity and convenience, we list several facts concerning self-conformal systems $(\Phi,K,\mu,T)$ on $\R^d$ that will be frequently used in the proof of Theorems~\ref{quantrec} $\&$  \ref{toprove}.  In the following, with \eqref{mubarmupi} in mind, $\ubar{\mu}$ is the Gibbs measure with respect to a $\beta$-Hölder potential on the symbolic space $\Sigma^{\N}$ such that $\mu=\ubar{\mu}\circ\pi^{-1}$.

\begin{enumerate}[label=(P\arabic*)]
   \item\label{toprodecayP1}  In view of Theorem~\ref{expmixforhol}, Corollary \ref{emprd} and Theorem~\ref{Main2}, there exist $C>0$ and $\gamma\in(0,1)$ that satisfy the following:\vspace*{1ex}
    \begin{itemize}
        \item[(i)]  For any $f_1\in\cC^{\beta}(\Sigma^{\N})$, any $f_2\in L^1(\mu)$ and any $n\in\N$,
        \begin{eqnarray}\label{decayHolder}
~ \hspace*{8ex} \left|\int_{\Sigma^{\N}}f_1\cdot f_2\circ\sigma^n\,\td\ubar{\mu}-\int_{\Sigma^{\N}}f_1\,\td\ubar{\mu}\cdot\int_{\Sigma^{\N}}f_2\,\td\ubar{\mu}\right|\,\,\leq\,\, C\,\gamma^n\cdot\|f_1\|_{\beta}\cdot\int_{\Sigma^{\N}}|f_2|\,\td\ubar{\mu}\,.
        \end{eqnarray}

        \item[(ii)] For any $I\in\Sigma^{*}$, any measurable subset $F\subseteq \R^d$ and any $n\in\N$,
        \begin{eqnarray}\label{KIT-nF}
            \big|\mu(K_I\cap T^{-n}F)-\mu(K_I)\mu(F)\big|\,\leq\,C\,\gamma^n\mu(F).
        \end{eqnarray}

        \item[(iii)] For any  ball $B\subseteq\R^d$, any measurable subset $F\subseteq\R^d$ and any $n\in\N$,
        \begin{eqnarray}\label{BT-nF}
            \big|\mu(B\cap T^{-n}F)-\mu(B)\mu(F)\big|\,\leq\,C\,\gamma^n\mu(F).
        \end{eqnarray}
    \end{itemize}
    \item\label{toprovediamP2} Let $\kappa\in(0,1)$ be as in (\ref{defkap}).  Then \eqref{p3'}  states that there is a constant  $C_3>1$ such that
    \begin{eqnarray}\label{diaofkiP2}
        |K_I|\,\leq\, C_3\,\kappa^{|I|},\,\,\,\,\,\,\,\,\forall\,I\in\Sigma^*.
    \end{eqnarray}
    \item\label{keyproP3} Part (ii) of Theorem~\ref{thmann} states that  there exist $r_0>0$, $C>0$ and $\delta>0$ such that for any $\x \in K$,
\begin{eqnarray}\label{muannlusP3}
        \mu\big((\partial B(\x,r))_{\varrho}\big)\,\leq\,C\varrho^{\delta},\,\,\,\,\,\,\,\,\forall\,\,0<r\leq r_0\,,\,\forall\,\varrho>0.
 \end{eqnarray}
\end{enumerate}

\medskip

The following  is essentially a consequence of (\ref{KIT-nF}).

\begin{lemma}\label{lemtwoseccylin}
Let $C>0$ and $\gamma\in (0,1)$ be as in \ref{toprodecayP1}.  Then for any $n_1$, $n_2\in\N$, any $I,\,J\in\Sigma^k$ with $k\geq n_1$, and any Borel set $F\subseteq\R^d$
\begin{eqnarray}\label{twoseccylinder}
    \Big|\mu(K_I\cap T^{-n_1}K_J\cap T^{-n_2}F)-\mu(K_I\cap T^{-n_1}K_J)\,\mu(F)\Big|\,\leq\,C\,\gamma^{n_2}\mu(F)  \, .
\end{eqnarray}
\end{lemma}

\begin{proof}
 Recall that $\widetilde{K}$ is the set of those $\x\in K$ with which $\#\big(\pi^{-1}(\x)\big)=1$ and that $T(\x)=\pi\circ\sigma\circ\pi^{-1}(\x)$ when $\x\in\widetilde{K}$. With this in mind,  it follows that
\begin{eqnarray}\label{kicaptkjkti}
 K_I\cap T^{-n_1}K_J\cap\widetilde{K} =
  \left\{
\begin{aligned}
    &K_{Ij_{k-n_1+1}\cdot\cdot\cdot j_{k}}\cap\widetilde{K},\ \ \ \ \text{if}~i_{n_1+1}\cdot\cdot\cdot i_{k}=j_{1}\cdot\cdot\cdot j_{k-n_1}\\[2ex]
    &\emptyset,\ \ \ \ \qquad \qquad\qquad\ \ \, \text{if}~i_{n_1+1}\cdot\cdot\cdot i_{k}\neq j_{1}\cdot\cdot\cdot j_{k-n_1}
\end{aligned}
    \right.
\end{eqnarray}
for any  $n_1,\, k\in\N$ with $k\geq n_1$ and any $I=i_1\cdot\cdot\cdot i_{k}$, $J=j_1\cdot\cdot\cdot j_{k}\in\Sigma^{k}$.  Then (\ref{twoseccylinder}) is an immediate consequence of (\ref{KIT-nF}) and (\ref{kicaptkjkti}).
\end{proof}

\bigskip

\subsection{Proof of Theorem~\ref{quantrec}} ~\label{CRR}

The following  statement~\cite[Lemma 1.5]{harman1998} represents an important  tool in the theory of metric Diophantine approximation for establishing counting statements.  It has its bases in the familiar variance method of probability theory and can be viewed as the quantitative form of the (divergence)  Borel-Cantelli Lemma \cite[Lemma~2.2]{BRV2016}.  As we shall see it is an essential ingredient  in the proof of  Theorem~\ref{quantrec}.

\begin{lemma}\label{countlem}
    Let $(X,\cB,\mu)$ be a probability space. Let $\{f_n(x)\}_{n\in\N}$ be a sequence of measurable functions on $X$, and $\{f_n\}_{n\in\N}$, $\{\phi_n\}_{\N}$ be sequences of numbers such that
    \[
        0\leq f_n\leq \phi_n~~(n=1,2,...).
    \]
    Suppose that there exists $C>0$ such that for any pair of positive integers $a<b$, we have
    \[
        \int_X~\left(\sum_{n=a}^{b}(f_n(x)-f_n)\right)^2~\td \mu(x)\leq C\sum_{n=a}^b\phi_n.
    \]
    Then for any $\epsilon>0$, we have
    \[
        \sum_{n=1}^{N}f_n(x)=\sum_{n=1}^{N}f_n+O\left(\Psi(N)^{1/2}(\log(\Psi(N)+1))^{\frac{3}{2}+\epsilon}+\max_{1\leq k\leq N}f_k\right)
    \]
    for $\mu$-almost every $x\in X$, where $\displaystyle \Psi(N):=\sum_{n=1}^N\phi_n$.
\end{lemma}

\noindent In the next section we shall state and prove a  more general  form (namely Lemma~\ref{genharmanlem})  of this well known  statement.

\medskip

We now lay the foundations for applying Lemma~\ref{countlem} within the context  of  Theorem~\ref{quantrec}.    With this in mind, we first show that the function  $t_n$  appearing in the statement of Theorem~\ref{quantrec}  is Lipschitz continuous.

\begin{lemma}\label{lem5.1}
    Assume the setting in Theorem~\ref{quantrec}. Then,  for any $ \,  \x,\y\in K$, we have that  $$|t_n(\x)-t_n(\y)|\leq |\x-\y|  \, . $$ Moreover, there exists   {$N \in \N$} such that $\mu(B(\x,t_n(\x)))=\psi(n)$ for all $\x\in K$ and $n>N$.
\end{lemma}
\begin{proof}

Fix any $\x,\y\in K$ with $\x\neq\y$ and  fix $n\in\N$. Note that $B(\x,t_n(\x))\subseteq B(\y,t_n(\x)+|\x-\y|)$, then by the definition of $t_n(\cdot)$, we have $$\mu(B(\y,t_n(\x)+|\x-\y|))\geq \psi(n),$$
    and hence $t_n(\y)\leq t_n(\x)+|\x-\y|$, again  by definition. Similarly, we find  that $t_n(\x)\leq t_n(\y)+|\x-\y|$  and this complete the proof of first part.

    To prove the moreover part, note that by part (i) of Theorem~\ref{thmann}
    and the fact that $\psi(x)\to0$ $(x\to\infty)$, there exists $N>0$ such that $t_n(\x)<r_0$ for all $n>N$ and $\x\in K$, where $r_0$ is as in \eqref{thmmeaann}.  
    On the other hand, by    \eqref{thmmeaann}, we have $\mu(\partial B(\x,r))=0$ for all $0<r\leq r_0$ and all $\x\in K$, and hence for any fixed $\x\in K$, the map $r\mapsto\mu(B(\x,r))$ is continuous over the interval $[0,r_0)$. On combining these observations, we obtain that $\mu(B(\x,t_n(\x)))=\psi(n)$ for all $\x\in K$ and  $n>N$.
\end{proof}

For any $n\in\N$, let  $$\hat{R}_n=\hat{R}_n(\psi):=\big\{\x\in K: |T^n\x-\x|<t_n(\x)\big\}.$$

\noindent The following  is an  extremely useful statement   and is a straightforward application of the triangle inequality.    In short, it provides a   mechanism for ``locally'' representing $R_n$ as the inverse image of a ball.  In turn this allows us to exploit mixing.   Throughout, given $x \in \R$ we let $$(x)_+:=\max\{x,0\} \, . $$

\begin{lemma}   \label{ohffs}
    Given any $I\in\Sigma^*$, fix a point $\z_I\in K_I$. Then for any $n\in\N$ and $I\in\Sigma^*$, we have that
    \begin{eqnarray}\label{relation}
        K_I\cap T^{-n}B(\z_I,(t_n(\z_I)-|K_I|)_+) \ \subseteq  \ K_I\cap\hat{R}_n  \ \subseteq \  K_I\cap T^{-n}B(\z_I,t_n(\z_I)+2|K_I|)\, .
    \end{eqnarray}

\end{lemma}
\begin{proof}
     Let $\x\in K_I\cap\hat{R}_n$.  Then, by   the triangle inequality we have that
    \begin{eqnarray*}
        |T^n\x-\z_I|&\leq&|T^n\x-\x|+|\x-\z_I|\\[1ex]
        &\leq&t_n(\x)+|\x-\z_I|\\[1ex]
        &\leq&t_n(\z_I)+2|\x-\z_I|\qquad(\text{by Lemma~\ref{lem5.1}})\\[1ex]
        &\leq&t_n(\z_I)+2|K_I|.
    \end{eqnarray*}
    In other words, $T^(\x)   \in B(\z_I,t_n(\z_I)+2|K_I|)    $  and so $ \x   \in  T^{-n}B(\z_I,t_n(\z_I)+2|K_I|)  $.   Hence,
     $K_I\cap\hat{R}_n\subseteq K_I\cap T^{-n}B(\z_I,t_n(\z_I)+2|K_I|)$  which is precisely the right hand side of the desired statement. A similar calculation yields the left hand side inclusion  in (\ref{relation}).
\end{proof}

The next two lemmas are concerned with ``precisely'' estimating  the   $\mu$-measure of $\hat{R}_n$ and their pairwise intersections.    They are  crucial in successfully being able to apply  Lemma~\ref{lem5.1} in order to prove  Theorem~\ref{quantrec}.    Let $r_0$ be as in \ref{keyproP3}. Then, without the loss of generality, we may assume that
\begin{eqnarray}\label{uptn}
    t_n(\x)\leq r_0,\,\,\,\,\forall\, n\geq1,\,\forall\,\x\in K,
\end{eqnarray}
so that we can use (\ref{muannlusP3}) freely. Indeed,  as we have already observed in the proof of Lemma~\ref{lem5.1}, since $\psi(x)\to0$ $(x\to\infty)$ in the setting in Theorem~\ref{quantrec}, then by part (i) of Theorem~\ref{thmann} and the definition of $t_n$, there exists $N\in\N$ for  which (\ref{uptn}) holds for any $\x\in K$ and any $n\geq N$.

\begin{lemma}\label{measan}
    Assume the setting in Theorem~\ref{quantrec}. Then there exists $\widetilde{\gamma}\in (0,1)$ such that
    \begin{eqnarray*}
        \mu(\hat{R}_n)=\psi(n)+O(\widetilde{\gamma}^n).
    \end{eqnarray*}
\end{lemma}
\begin{proof}
    For any $I\in\Sigma^*$, fix a  point  $\z_I\in K_I$.
Let $\delta>0$ be as in \ref{keyproP3}. Note that by  Lemma~\ref{lem5.1} and (\ref{muannlusP3}), we have that for any $n\in\N$, any $\x\in K$ and any $\varrho>0$,
 \begin{eqnarray}\label{mubzitn}
\mu\big(B(\x,t_n(\x)+\varrho)\big)&=&\mu\big(B(\x,t_n(\x))\big)+O\Big(\mu\big((\partial B(\x,t_n(\x))_{\varrho}\big)\Big)\nonumber\\
&=&\psi(n)+O\big(\varrho^{\delta}\big)
 \end{eqnarray}
 This together with  (\ref{KIT-nF}), \eqref{diaofkiP2} and the right hand side  of (\ref{relation}) implies that for any $n\in\N$, any $\ell\in\N$ and any $I\in\Sigma^{\ell}$,
 \begin{eqnarray*}
     \mu(K_I\cap\hat{R}_n)&\leq&\mu\left(K_I\cap T^{-n}B(\z_I,t_n(\z_I)+2|K_I|)\right)\nonumber\\[4pt]
     &=&\left(\mu(K_I)+O(\gamma^n)\right)\mu\big(B(\z_I,t_n(\z_I)+2|K_I|)\big)\nonumber\\[4pt]
     &=&\big(\mu(K_I)+O(\gamma^n)\big)\,\big(\psi(n)+O(|K_I|^{\delta})\big)\\[4pt]
&=&\left(\mu(K_I)+O(\gamma^n)\right)\big(\psi(n)+O(\kappa^{\delta\ell})\big)\nonumber\\[4pt]
     &=&\mu(K_I)\big(\psi(n)+O(\kappa^{\delta\ell})\big)+O\big(\gamma^n\big),
 \end{eqnarray*}
 where the big-$O$ constants are independent of $n\in\N$, $\ell\in\N$ and $I\in\Sigma^{\ell}$. On exploiting the left hand side  of (\ref{relation}),  a similar calculation shows that
 \begin{eqnarray*}
\mu(K_I\cap\hat{R}_n)\geq\mu(K_I)\big(\psi(n)+O(\kappa^{\delta\ell})\big)+O\big(\gamma^n\big)
 \end{eqnarray*}
 for any $I\in\Sigma^{\ell}$ and $n\in\N$. Therefore, we conclude that
 \begin{eqnarray*}
\mu(K_I\cap\hat{R}_n)=\mu(K_I)\big(\psi(n)+O(\kappa^{\delta\ell})\big)+O\big(\gamma^n\big),\qquad\forall\, I\in\Sigma^{\ell}.
 \end{eqnarray*}
 On  summing over all $I\in\Sigma^{\ell}$, it follows that
  \begin{eqnarray}\label{muan7.9}
\mu(\hat{R}_n)=\psi(n)+O(\kappa^{\delta\ell}+m^{\ell}\gamma^n),
  \end{eqnarray}
  where $m\in\N$ is the number of elements in the  $C^{1+\alpha}$ conformal IFS $\Phi$ under consideration.
  Now let  $$\ell=\ell(n):=\linte{\frac{n\log(1/\gamma)}{2\log m}}  \,  .   $$  Then $m^{\ell}\gamma^n\asymp\gamma^{n/2}$  and it follows  via  \eqref{muan7.9} that
 $\mu(\hat{R}_n)=\psi(n)+O(\widetilde{\gamma}^n)$, where
 $$\widetilde{\gamma}=\max\big\{\gamma^{1/2},\kappa^{\frac{\delta\log(1/\gamma)}{2\log m}}\big\} \, $$
and thereby completes the proof.
\end{proof}

\begin{lemma}\label{capcapineq}
    Assume the setting in Theorem~\ref{quantrec}. Then there exist $C>0$ and $\eta\in(0,1)$ such that for any pair of positive integers $a<b$, we have
    \begin{eqnarray}\label{keyine}
        \sum_{a\leq n_1<n_2\leq b}\mu(\hat{R}_{n_1}\cap\hat{R}_{n_2})~\leq\sum_{a\leq n_1<n_2\leq b}\mu(\hat{R}_{n_1})\mu(\hat{R}_{n_2})+C\cdot\sum_{n=a}^{b}\big(\mu(\hat{R}_n)+\eta^n\big).
    \end{eqnarray}
\end{lemma}
\begin{proof}
     Let $\gamma\in(0,1)$ be as in \ref{toprodecayP1} and $m\in\N$ be the number of elements in  conformal IFS $\Phi$. Let $$M:=\frac{\log(1/\gamma)}{4\log m} \, .$$ Without the loss of generality, we may assume that $M<1$ since  $\gamma$ can be taken to be as close to $1$ as we wish. For any $I\in\Sigma^*$, fix a point $\z_I\in K_I$. To ease notation, we write \[B_p(I)\equiv B(\z_I,t_p(\z_I)+2|K_I|)\] for any $I\in\Sigma^*$ and $p\in\N$. Let $\delta>0$ be as in \ref{keyproP3}. Then it follows from  \eqref{diaofkiP2} and \eqref{mubzitn} that
\begin{eqnarray}\label{mubpipsip}
\mu(B_p(I))=\psi(p)+O(\kappa^{\delta|I|})\,.
     \end{eqnarray}

     Given positive integers $n$, $k\in\N$, the overarching goal is to obtain a sharp  upper bound for  $\mu(\hat{R}_n\cap\hat{R}_{n+k})$. With this in mind,  we start by observing the right hand side of (\ref{relation}) implies that for any $I\in\Sigma^*$ and any $n,k\in\N$,
    \begin{eqnarray}\label{twocap}
K_I\cap\hat{R}_n\cap\hat{R}_{n+k}\subseteq K_I\cap T^{-n}B_n(I)\cap T^{-(n+k)}B_{n+k}(I).
    \end{eqnarray}
    We proceed by  considering two cases depending on the size of $k$.

\noindent$\bullet$ \emph{Estimating  $\mu(\hat{R}_n\cap\hat{R}_{n+k})$ when $k>\frac{n+1}{M}$. } Let $$\ell_1=\ell_1(n,k):=\linte{(n+k)M}  \, . $$ It is easily verified that $$n<\ell_1   \quad  {\rm and}  \quad   m^{2\ell_1}\asymp \gamma^{-\frac{n+k}{2}}   \, .$$
For any $p\in\N$ and $I\in\Sigma^{*}$, let
\begin{eqnarray*}
    \cJ(p,I)&:=&\left\{J\in\Sigma^{|I|}:K_J\cap B_p(I) \neq\emptyset\right\}.
\end{eqnarray*}
Then by (\ref{twocap}),  for any $I\in\Sigma^{\ell_1}$ we have that
\begin{eqnarray}
\mu\left(K_I\cap\hat{R}_n\cap\hat{R}_{n+k}\right)&\leq&\mu\left(K_I\cap T^{-n}B_n(I)\cap T^{-(n+k)}B_{n+k}(I)\right)\nonumber\\[4pt]
&\leq&\sum_{J\in\cJ(n,I)}\mu\left(K_I\cap T^{-n}K_J\cap T^{-(n+k)}B_{n+k}(I)\right)\label{eas2cap}.
\end{eqnarray}
  Then  Lemma~ \ref{lemtwoseccylin} together with \eqref{mubpipsip},  (\ref{eas2cap})  and the fact that $n<\ell_1$,  implies that
\begin{eqnarray*}
\mu\left(K_I\cap\hat{R}_n\cap\hat{R}_{n+k}\right)&\leq&\sum_{J\in\cJ(n,I)}\left(\mu(K_I\cap T^{-n}K_J)+O(\gamma^{n+k})\right)\mu(B_{n+k}(I))\nonumber\\
&=&\sum_{J\in\cJ(n,I)}\left(\mu(K_I\cap T^{-n}K_J)+O(\gamma^{n+k})\right)\left(\psi(n+k)+O\big(\kappa^{\delta\ell_1}\big)\right).
\end{eqnarray*}
Then on summing over $I\in\Sigma^{\ell_1}$ and using the fact that $m^{2\ell_1}\asymp\gamma^{-\frac{n+k}{2}}$, it follows  that
\medskip
\begin{eqnarray}
~ \hspace*{-5ex} \mu(\hat{R}_n\cap\hat{R}_{n+k})\nonumber\\[4pt]
&~&   \hspace*{-10ex}  \le \left(\sum_{I\in\Sigma^{\ell_1}}\sum_{J\in\cJ(n,I)}\mu(K_I\cap T^{-n}K_J)+O(m^{2\ell_1}\gamma^{n+k})\right)\Big(\psi(n+k)+O(\kappa^{\delta\ell_1})\Big)\nonumber\\[4pt]
&~&   \hspace*{-10ex}  = \left(\mu\Big(\bigcup_{I\in\Sigma^{\ell_1}}\bigcup_{J\in\cJ(n,I)}(K_I\cap T^{-n}K_J)\Big)+O(\gamma^{\frac{n+k}{2}})\right)\Big(\psi(n+k)+O(\kappa^{\delta\ell_1})\Big).
\label{meaofnk}
\end{eqnarray}
For the moment, we focus on estimating the  $\mu$-measure term appearing within the first bracket in (\ref{meaofnk}). Let $$\ell_2=\ell_2(n):=\linte{2nM}  \, . $$ It can be verified that $$\ell_1>\ell_2  \quad {\rm and }  \quad m^{\ell_2}\gamma^n\asymp\gamma^{-n/2}$$
Let $C_3>0$ be the constant in \ref{toprovediamP2}. Then, for any $I\in\Sigma^{\ell_1}$, by the definition of $\cJ(n, I)$ and   the inequality (\ref{diaofkiP2}), we have
\begin{eqnarray}\label{inclukij}
    \bigcup_{J\in\cJ(n,I)}(K_I\cap T^{-n}K_J)~\subseteq~ K_I\cap T^{-n}B(\z_I,t_n(\z_I)+3C_3\kappa^{\ell_1}).
\end{eqnarray}
Given any positive integers $p\leq q$ and any $I=i_1\cdot\cdot\cdot i_{q}\in\Sigma^{q}$, let
    $I_{p}:=i_1\cdot\cdot\cdot i_{p}$. With this notation in minf, if $\x$ is in the set on the right-hand-side of (\ref{inclukij}), then $\x\in K_{I_{\ell_2}}$ (since $\ell_1>\ell_2$) and
\begin{eqnarray*}
    |T^n\x-\z_{I_{\ell_2}}|&\leq&|T^n\x-\z_{I}|+|\z_{I}-\z_{I_{\ell_2}}|\\[1ex]
    &\leq&t_n(\z_I)+3C_3\kappa^{\ell_1}+|K_{I_{\ell_2}}|\\[1ex]
&\leq&t_n(\z_I)+3C_3\kappa^{\ell_1}+C_3\kappa^{\ell_2}\\[1ex]
&\leq&t_n(\z_{I_{\ell_2}})+|\z_{I_{\ell_2}}-\z_I|+4C_3\kappa^{\ell_2} \quad (\text{by Lemma~\ref{lem5.1} })\\[1ex]
&\leq&t_n(\z_{I_{\ell_2}})+5C_3\kappa^{\ell_2}.
\end{eqnarray*}
The upshot of above is that
\medskip
\begin{eqnarray}\label{cupcupkitnkj}
\bigcup_{I\in\Sigma^{\ell_1}}\bigcup_{J\in\cJ(n,I)}(K_I\cap T^{-n}K_J)~\subseteq~\bigcup_{I\in\Sigma^{\ell_2}}\big(K_I\cap T^{-n}B(\z_I,t_n(\z_I)+5C_3\kappa^{\ell_2})\big),
\end{eqnarray}
\medskip
 Then by (\ref{KIT-nF}), \eqref{mubzitn},  \eqref{cupcupkitnkj}, and the fact that $m^{\ell_2}\gamma^n\asymp\gamma^{-n/2}$,it follows that
\begin{eqnarray}
\mu\left(\bigcup_{I\in\Sigma^{\ell_1}}\bigcup_{J\in\cJ(n,I)}(K_I\cap T^{-n}K_J)\right)&\leq&\sum_{I\in\Sigma^{\ell_2}}\mu\big(K_I\cap T^{-n}B(\z_I,t_n(\z_I)+5C_3\kappa^{\ell_2})\big)\nonumber\\[1ex]
&=&\sum_{I\in\Sigma^{\ell_2}}\left(\mu(K_I)+O(\gamma^n)\right)\,\mu\big(B(\z_I,t_n(\z_I)+5C_3\kappa^{\ell_2})\big)\nonumber\\[1ex]
&=&\sum_{I\in\Sigma^{\ell_2}}\left(\mu(K_I)+O(\gamma^n)\right)\left(\psi(n)+O(\kappa^{\delta\ell_2})\right)\nonumber\\[1ex]
&=&\left(1+O(m^{\ell_2}\gamma^n)\right)\left(\psi(n)+O(\kappa^{\delta\ell_2})\right)\nonumber\\[1ex]
&=&\left(1+O(\gamma^{n/2})\right)\left(\psi(n)+O(\kappa^{\delta\ell_2})\right).\label{cupcup}
\end{eqnarray}
Let $\widetilde{\gamma}\in(0,1)$ be as in Lemma~\ref{measan}. Now feeding  (\ref{cupcup}) into (\ref{meaofnk}) and then using Lemma~\ref{measan}, we find that
\begin{eqnarray}
    \mu(\hat{R}_n\cap\hat{R}_{n+k})&\leq&\left(\psi(n)+O(\gamma^{n/2}+\kappa^{\delta\ell_2})\right)\cdot\left(\psi(n+k)+O(\kappa^{\delta\ell_1})\right)\nonumber\\[2ex]
    &=&\left(\mu(\hat{R}_{n})+O(\gamma_1^n)\right)\cdot\big(\mu(\hat{R}_{n+k})+O(\gamma_1^{n+k})\big),\label{capcapkln}
\end{eqnarray}
where $\gamma_1:=\max\{\widetilde{\gamma},\gamma^{1/2},\kappa^{\delta M}\}\in(0,1)$.

\medskip

\noindent$\bullet$ \emph{Estimating  $\mu(\hat{R}_n\cap\hat{R}_{n+k})$ when $1 \le k \le \frac{n+1}{M}$. }   Recall that $\ell_2=\ell_2(n):=\linte{2nM}$ and   $m^{\ell_2}\gamma^n\asymp\gamma^{-n/2}$.  So given the range of $k$ under consideration, it follows   that
 \begin{eqnarray}\label{ell2compnk}
     \ell_2=\frac{2M^2}{M+1}(n+k)+O(1)\,.
 \end{eqnarray}
  For any $n,k\in\N$ with $1 \le k\le\frac{n+1}{M}$ and any $I\in\Sigma^{\ell_2}$, by (\ref{KIT-nF}), (\ref{BT-nF}) \eqref{mubpipsip} and (\ref{twocap}),  we have that
\begin{eqnarray*}
    \mu\left(K_I\cap\hat{R}_n\cap\hat{R}_{n+k}\right)&\leq&\mu\left(K_I\cap T^{-n}B_n(I)\cap T^{-(n+k)}B_{n+k}(I)\right)\\[1ex]
    &=&\left(\mu(K_I)+O(\gamma^n)\right)\cdot\mu\left(B_n(I)\cap T^{-k}B_{n+k}(I))\right)\\[1ex]
    &=&\left(\mu(K_I)+O(\gamma^n)\right)\cdot\left(\mu(B_n(I))+O(\gamma^k)\right)\cdot\mu(B_{n+k}(I))\\[1ex]
    &=&\left(\mu(K_I)+O(\gamma^n)\right)\cdot\left(\psi(n)+O(\kappa^{\delta\ell_2}+\gamma^k)\right)\cdot\left(\psi(n+k)+O(\kappa^{\delta\ell_2})\right).
\end{eqnarray*}
On summing over $I\in\Sigma^{\ell_2}$,  we obtain that
\medskip
\begin{eqnarray}    \mu(\hat{R}_{n}\cap\hat{R}_{n+k})
&\leq&\left(1+O(m^{\ell_2}\gamma^n)\right)\cdot\left(\psi(n)+O(\kappa^{\delta\ell_2}+\gamma^k)\right)\cdot\left(\psi(n+k)+O(\kappa^{\delta\ell_2})\right)\nonumber\\[1ex]
&=&\left(1+O(\gamma^{n/2})\right)\cdot\left(\psi(n)+O(\kappa^{\delta\ell_2}+\gamma^k)\right)\cdot\left(\psi(n+k)+O(\kappa^{\delta\ell_2})\right)\nonumber\\[1ex]
&=&\left(\psi(n)+O(\gamma^{n/2}+\kappa^{\delta\ell_2}+\gamma^k)\right)\cdot\left(\psi(n+k)+O(\kappa^{\delta\ell_2})\right)\nonumber\\[1ex]
&=&\left(\mu(\hat{R}_n)+O(\widetilde{\gamma}^n+\gamma^{n/2}+\kappa^{\delta\ell_2}+\gamma^k)\right)\cdot\left(\mu(\hat{R}_{n+k})+O(\widetilde{\gamma}^{n+k}+\kappa^{\delta\ell_2})\right)\nonumber\\[1ex]
&=&\left(\mu(\hat{R}_n)+O(\widetilde{\gamma}^{kM}+\gamma^{kM/2}+\kappa^{\delta\cdot\frac{2M^2}{M+1}(n+k)}+\gamma^k)\right)\nonumber\\[1ex]
&~&\times\left(\mu(\hat{R}_{n+k})+O(\widetilde{\gamma}^{n+k}+\kappa^{\delta\cdot\frac{2M^2}{M+1}(n+k)})\right)\nonumber\\[1ex]
&=&\left(\mu(\hat{R}_n)+O\big(\gamma_2^k\big)\right)\cdot\left(\mu(\hat{R}_{n+k})+O\big(\gamma_2^{n+k}\big)\right),\label{capcapkgn}
\end{eqnarray}
where $$\gamma_2:=\max\{\widetilde{\gamma}^M,\gamma^{\min\{1,M/2\}},\kappa^{2\delta M^2/(M+1)}\}\in(0,1)  \,. $$

\noindent In the above we  use the fact that $m^{\ell_2}\asymp\gamma^{-n/2}$ to go from the first to second line.  Then we  use Lemma \ref{measan} to from the third to the fourth line and  finally we  use  \eqref{ell2compnk} and the fact that $n\geq kM-1$  to go from the fourth to the fifth line.
\bigskip

Everything is now in place to prove  the desired pairwise independent on average inequality (\ref{keyine}).     For any $n\in\N$, let
\[
    \cF_1(n):=\left[1,\frac{n+1}{M}\right]\cap\N    \quad  {\rm and }  \quad   \cF_2(n):=\left(\frac{n+1}{M},+\infty\right)\cap\N.
\]
Then for any pair of positive integers $a,b\in\N$ with $a<b$, by (\ref{capcapkln}) and (\ref{capcapkgn}), it follows that
\begin{eqnarray*}
    \sum_{a\leq n_1<n_2\leq b}\mu(\hat{R}_{n_1} \hspace*{-5ex} &~& \cap  \  \hat{R}_{n_2})  \ =  \  \sum_{n=a}^{b-1}\sum_{k=1}^{b-n}\mu(\hat{R}_{n}\cap\hat{R}_{n+k})\\[1ex]
    &=&\sum_{n=a}^{b-1}~\sum_{k\in[1,b-n]\cap\cF_1(n)}\mu(\hat{R}_{n}\cap\hat{R}_{n+k})+\sum_{n=a}^{b-1}~\sum_{k\in[1,b-n]\cap\cF_2(n)}\mu(\hat{R}_{n}\cap\hat{R}_{n+k})
    \\[2ex]
    &\leq&\sum_{n=a}^{b-1}~\sum_{k\in[1,b-n]\cap\cF_1(n)}\left(\mu(\hat{R}_n)+O\big(\gamma_2^k\big)\right)\cdot\left(\mu(\hat{R}_{n+k})+O\big(\gamma_2^{n+k}\big)\right)\\[1ex]
    &~& \qquad +  \ \sum_{n=a}^{b-1}~\sum_{k\in[1,b-n]\cap\cF_2(n)}\left(\mu(\hat{R}_{n})+O(\gamma_1^n)\right)\cdot\big(\mu(\hat{R}_{n+k})+O(\gamma_1^{n+k})\big)\\[2ex]
    &=&\sum_{a\leq n_1<n_2\leq b}\mu(\hat{R}_{n_1})\mu(\hat{R}_{n_2})+O\left(\sum_{n=a}^{b}\big(\mu(\hat{R}_n)+\eta^n\big)\right),
\end{eqnarray*}
where $\eta=\max\{\gamma_1,\gamma_2\}\in(0,1)$. The proof is complete.
\end{proof}

We are now in a position to prove Theorem~\ref{quantrec} utilizing Lemma~\ref{countlem}.
   Let $\eta\in(0,1)$ be as in Lemma~\ref{capcapineq}. We shall use Lemma~\ref{countlem} with $X=K$ and
    \[
        f_n(\x)=\one_{\hat{R}_n}(\x), \quad f_n=\mu(\hat{R}_n),  \quad \phi_n=\mu(\hat{R}_n)+\eta^n   \qquad (n=1,2,...).
    \]
    By 
    Lemma~\ref{capcapineq}, there exist a  constant $C>0$  such that for any positive integers $a<b$, we have
    \begin{eqnarray*}
        \int_K~\left(\sum_{n=1}^b\big(\one_{\hat{R}_n}(\x)   -  \mu(\hat{R}_n)\big)\right)^2\td\mu(\x)  & = & \sum_{a\leq n_1\leq n_2\leq b}\mu(\hat{R}_{n_1}\cap\hat{R}_{n_2})-\left(\sum_{n=a}^b\mu(\hat{R}_n)\right)^2\\[2ex]
         &~\hspace{-35ex}=&  ~\hspace{-17ex}   \sum_{n=a}^b\mu(\hat{R}_n)+2\sum_{a\leq n_1<n_2\leq b}\mu(\hat{R}_{n_1}\cap\hat{R}_{n_2})\\[1ex]
        &~&-\sum_{n=1}^b\mu(\hat{R}_n)^2-2\sum_{a\leq n_1<n_2\leq b}\mu(\hat{R}_{n_1})\mu(\hat{R}_{n_2})\\[2ex]
         &~\hspace{-35ex}  \le &  ~\hspace{-17ex}  \sum_{n=a}^b\mu(\hat{R}_n)+2\sum_{a\leq n_1<n_2\leq b}\mu(\hat{R}_{n_1})\mu(\hat{R}_{n_2})\\[1ex]
        &~&+~C\sum_{n=a}^b\big(\mu(\hat{R}_n)+\eta^n\big)-\sum_{n=1}^b\mu(\hat{R}_n)^2\\[1ex]
        &~&-~2\sum_{a\leq n_1<n_2\leq b}\mu(\hat{R}_{n_1})\mu(\hat{R}_{n_2})\\[2ex]
       &~\hspace{-35ex}=&  ~\hspace{-17ex} \sum_{n=a}^b\mu(\hat{R}_n)-\sum_{n=1}^b\mu(\hat{R}_n)^2+C\sum_{n=a}^b\big(\mu(\hat{R}_n)+\eta^n\big)\\
       &~\hspace{-35ex}\le&  ~\hspace{-17ex} (1+C)\sum_{n=a}^{b}\big(\mu(\hat{R}_n)+\eta^n\big).
    \end{eqnarray*}
    By Lemma~\ref{countlem}, for any given $\epsilon>0$, we have
    \begin{eqnarray}\label{resquant}
        \sum_{n=1}^{N}\one_{B(\x,t_n(\x))}(T^n\x)=\sum_{n=1}^N\one_{\hat{R}_n}(\x)=\Psi(N)+O\left(\Psi(N)^{1/2}\log^{\frac{3}{2}+\epsilon}(\Psi(N))\right)
    \end{eqnarray}
    for $\mu$-almost every $\x\in K$, where $\displaystyle\Psi(N):=\sum_{n=1}^N\big(\mu(\hat{R}_n)+\eta^n\big)$. However, by Lemma~\ref{measan}, we have $$\Psi(N)=\sum_{n=1}^N\psi(n)+O(1).$$ So  the term $\Psi(N)$ in (\ref{resquant}) can be replaced by the summation $\sum_{n=1}^N\psi(N)$. This completes the proof of Theorem~\ref{quantrec}.

\bigskip

\subsection{Proof of Theorem~\ref{toprove} }\label{apptoprove} ~

Although the proof of Theorem~\ref{toprove} follows the same line of attack as that used in establishing Theorem~\ref{quantrec}, it is more involved chiefly due to the fact that the asymptotic behaviour of the counting function is dependant on $\x \in K$.   Indeed,   the quantitative form of the Borel-Cantelli Lemma (i.e. Lemma~\ref{countlem}),  which  is a key ingredient in the proof of  Theorem~\ref{quantrec},  is not applicable as it stands.    In short  we need to work with  more versatile form  in which the  sequence $\{f_n\}_{n \in \N}$ in Lemma~\ref{countlem} is allowed to  depend on $\x \in K$.

\begin{lemma}\label{genharmanlem}
    Let $(X,\cB,\mu)$ be a probability space. let $\{f_n(x)\}_{n\in\N}$ and $\{g_n(x)\}_{n\in\N}$ be  sequences of measurable functions on $X$, and  let $\{\phi_n\}_{n\in\N}\subseteq\R$ be a sequence of real numbers. Suppose that
    \begin{eqnarray}\label{nonneg}
        0\leq g_n(x)\leq \phi_n,\,\,\,\,\,\,\,\,\forall\,n\in\N,\,\,\,\,\forall\,x\in X
    \end{eqnarray}
    and that there exists $C>0$ with which
    \begin{eqnarray}\label{keycont}
        \int_{X}\left(\sum_{n=a}^b(f_n(x)-g_n(x))\right)^2\,\td\mu(x)\leq C\sum_{n=a}^b\phi_n
    \end{eqnarray}
    for any pair of integers $0<a<b$. Then for any $\epsilon>0$, we have
    \begin{eqnarray}\label{quantestim}
        \sum_{n=1}^N f_n(x)=\sum_{n=1}^N g_n(x)+O\left(\Psi(N)^{\frac{1}{2}}(\log(\Psi(N)))^{\frac{3}{2}+\epsilon}+\max_{1\leq k\leq N}g_k(x)\right)
    \end{eqnarray}
    for $\mu$-almost every $x\in X$, where $\displaystyle\Psi(N):=\sum_{n=1}^N\phi_n$.
\end{lemma}

In the case $\{g_n(x)\}_{n\in\N}$ is a sequence independent of $x$, the above lemma coincides with Lemma~\ref{countlem}. The proof of Lemma~\ref{genharmanlem} essentially follows that of Lemma~\ref{countlem} with natural modification.  For completeness, the proof is given in Appendix~\ref{HJ}.

Let $\psi:\R\to   \R_{\ge 0} $ be a real positive function such that $\psi(x)\to0$  as $x\to\infty$.    For any $n\in\N$, as in the introduction (see \eqref{anrpsidef}), let
\[
R_n(\psi)\,:=\,\left\{\x\in K:T^n\x\in B(\x,\psi(n))\right\}.
\]
Let $r_0$ be as in \ref{keyproP3}. Then, without the loss of generality, we may assume that
\begin{eqnarray*}\label{uptnB}
   \psi(n) \leq \min\{r_0,1\},\,\,\,\,\forall\, n\geq1  ,
\end{eqnarray*}
so that we can use (\ref{muannlusP3}) freely.  The following is the analogue of Lemma~\ref{ohffs} and allows us to ``locally'' represent $R_n$ as the inverse image of a ball.

\begin{lemma}\label{relalemtoprove}
     For each $I\in\Sigma^*$, fix $\z_I\in K_I$. Then for any $n\in\N$ and any $I\in\Sigma^*$, we have that
    \begin{equation*}
K_I\cap T^{-n}\big(B(\z_I,(\psi(n)-|K_I|)_+  )\big)    \ \subseteq \  K_I\cap R_n(\psi) \,    \subseteq \ K_I\cap T^{-n}\big(B(\z_I,\psi(n)+|K_I|)\big) \, .
    \end{equation*}

\end{lemma}
\begin{proof}
    The proof  makes use of the triangle inequality and is similar to   that used    to prove   \eqref{relation}. So we omit the details. 
\end{proof}

The overarching goal is to obtain precise enough estimates on the
  $\mu$-measures of $R_n(\psi)$ and their pairwise intersections, and the integrals of the functions
  \medskip
\begin{equation}
    \label{junjieF}\x\mapsto\one_{R_{n_1}(\psi)}(\x)\cdot\mu(B(\x,\psi(n_ 2)))  \qquad (n_1, n_2  \in \N)
\end{equation}

\noindent   in order to apply Lemma~\ref{genharmanlem} to
prove  Theorem~\ref{toprove}.  We start with dealing with the  $\mu$-measure of $R_n(\psi)$.

\begin{lemma}\label{topromuofrnpsi}
    There exist $C>0$ and $\widetilde{\gamma}\in(0,1)$ such that for any $n\in\N$, we have
    \begin{eqnarray*}
        \left|\mu\big(R_n(\psi)\big)-\int_K\mu\big(B(\x,\psi(n))\big)\,\td\mu(\x)\right|\,\leq\,C\,\widetilde{\gamma}^n.
    \end{eqnarray*}
\end{lemma}
\begin{proof}
    Let $\gamma\in(0,1)$ be as in \ref{toprodecayP1}, let $\delta>0$ be as in (\ref{muannlusP3}) and let $m\in\N_{\geq2}$ be the number of elements in the conformal IFS $\Phi$. For each $n\in\N$, denote
    \[
    q_n:=\linte{\frac{n\log(1/\gamma)}{2\log m}}.
    \]
    By Lemma~\ref{relalemtoprove}, we obtain that
    \[
    R_n(\psi)\,\subseteq\,\bigcup_{I\in\Sigma^{q_n}}K_I\cap T^{-n}\big(B(\z_I,\psi(n)+|K_I|)\big)
    \]
    and
    \[
    R_n(\psi)\,\supseteq\,\bigcup_{I\in\Sigma^{q_n}}K_I\cap T^{-n}\big(B(\z_I,(\psi(n)-|K_I|)_+  )\big)  \, .
    \]
    These inclusions together with 
    \ref{toprodecayP1} --  \ref{keyproP3}
    imply that
    \begin{eqnarray}
        \mu(R_n(\psi))&=&\sum_{I\in\Sigma^{q_n}}\mu\Big(K_I\cap T^{-n}\big(B(\z_I,\psi(n)  )\big)\Big)\nonumber\\
        &\,&\hspace{10ex}+\,\,O\left(\sum_{I\in\Sigma^{q_n}}\mu\Big(K_I\cap T^{-n}\big(\partial B(\z_I,\psi(n))\big)_{|K_I|}\Big)\right)\nonumber\\[2ex]
        &=&\sum_{I\in\Sigma^{q_n}}\big(\mu(K_I)+O(\gamma^n)\big)\cdot\mu\big(B(\z_I,\psi(n))\big)\nonumber\\
        &~&\hspace{10ex}+\,\,O\left(\sum_{I\in\Sigma^{q_n}}\big(\mu(K_I)+\gamma^n\big)\cdot\mu\Big((\partial B(\z_I,\psi(n)))_{|K_I|}\Big)\right)\nonumber\\[2ex]
        &=&\sum_{I\in\Sigma^{q_n}}\mu(K_I)\cdot\mu\big(B(\z_I,\psi(n))\big)+O\Big(m^{q_n}\gamma^n+\kappa^{\delta q_n}+m^{q_n}\gamma^n\kappa^{\delta q_n}\Big)\nonumber\\[2ex]
        &=&\sum_{I\in\Sigma^{q_n}}\int_{K_I}\mu\big(B(\z_I,\psi(n))\big)\,\td\mu+O\Big(m^{q_n}\gamma^n+\kappa^{\delta q_n}+m^{q_n}\gamma^n\kappa^{\delta q_n}\Big)\nonumber\\[2ex]
        &=&\sum_{I\in\Sigma^{q_n}}\int_{K_I}\Big(\mu(B(\x,\psi(n)))+O\big(\kappa^{\delta q_n}\big)\Big)\,\td\mu(\x)\nonumber\\
        &~&\hspace{15ex}+\,\,O\Big(m^{q_n}\gamma^n+\kappa^{\delta q_n}+m^{q_n}\gamma^n\kappa^{\delta q_n}\Big)\nonumber\\[2ex]
        &=&\int_K\mu\big(B(\x,\psi(n))\big)\,\td\mu(\x)+O\Big(m^{q_n}\gamma^n+\kappa^{\delta q_n}+m^{q_n}\gamma^n\kappa^{\delta q_n}\Big)\label{murnpsiintk}
    \end{eqnarray}
    Note that by the definition of $q_n$, we have $m^{q_n}\gamma^n\asymp\gamma^{n/2}$. Let $$\widetilde{\gamma}:=\max\left\{\kappa^{\frac{\delta\log(1/\gamma)}{2\log m}},\gamma^{\frac{1}{2}}\right\}.$$ Then by (\ref{murnpsiintk}), we obtain that
    \[
    \mu\big(R_n(\psi)\big)=\int_{K_I}\mu\big(B(\x,\psi(n))\big)\,\td\mu(\x)+O\big(\widetilde{\gamma}^n\big)
    \]
    as desired.
\end{proof}

As usual, let  $\tau:=\dimH K$. For any $n\in\N$, define the density function
\begin{eqnarray*}
l_n(\x):=\frac{\mu\big(B(\x,\psi(n))\big)}{\psi(n)^{\tau}}\,\,\,\,\,\,\,\,(n\in\N,\,\x\in K)\,.
\end{eqnarray*}The next  result is a technical  lemma concerning the H\"{o}lder norm of the map $l_n\circ\pi$ on the symbolic space.  It will  be subsequently  required  in  calculating the integral  of the functions appearing in \eqref{junjieF} (namely in proving Lemma~\ref{toprolemonern1mebn2}) and in calculating the
  $\mu$-measure of the  pairwise intersections of the sets
$R_n(\psi)$ (namely in proving  Lemma~\ref{topromuofrnpsirnkpsi}).

\begin{lemma}\label{bddalholder}
    Assume the setting in Theorem~\ref{toprove}. Let $\kappa\in(0,1)$ be as in \ref{toprovediamP2}, $\delta>0$ be as in \ref{keyproP3} and $m$ be the number of elements in $\Phi$. Then for any $\eta\in(0,1)$ and any
$0<\widetilde{\beta}\leq\frac{-\eta\delta\log\kappa}{2\log m}$, we have
\begin{eqnarray}\label{desibetano}  \|l_n\circ\pi\cdot\one_{[I]}\|_{\widetilde{\beta}
    }\leq C\cdot(m^{\widetilde{\beta}|I|}+\psi(n)^{-\eta\tau}),\qquad\forall\ n\in\N,\ \forall \ I\in\Sigma^*,
\end{eqnarray}
where $C>0$ is a constant that does not depend on $\eta$ and  $\widetilde{\beta}$.
\end{lemma}


\begin{proof}
     Under the setting of
     Theorem~\ref{toprove}, the notion of  $\mu$ being equivalent to $\cH^{\tau}|_K$ and $\mu$ being $\tau$-Ahlfors regular coincide (see  Remark~\ref{newrem}\,(i)).  Thus, by definition there exists $M>1$ such that
    \[
        M^{-1}r^{\tau}\,\leq\, \mu(B(\x,r))\,\leq\, M\,r^{\tau},\,\,\,\,\,\,\,\,\forall\,\, \x\in K,~\forall\,\, 0\leq r\leq |K|\,.
    \]
    From this, we conclude that
    \begin{eqnarray}\label{bddoflnx}
         M^{-1}\leq l_n(\x)\leq M,\qquad \forall\ n\in\N,\ \forall\ \x\in K.
    \end{eqnarray}

    Throughout, fix $\eta>0$, $n\in\N$ and
    \begin{equation} \label{bingyork}
        0<\widetilde{\beta}\leq\frac{-\eta\delta\log\kappa}{2\log m}  \, .
    \end{equation} Now we start to estimate the $\widetilde{\beta}$-H\"{o}lder norm of $l_n\circ\pi$. For any given  $J_1\neq J_2\in\Sigma^{\N}$ and $I\in\Sigma^*$, observe that
    \begin{itemize}
        \item[(i)] if $J_1,J_2\notin[I]$, then \[
        |l_n\circ\pi(J_1)\cdot\one_{[I]}(J_1)-l_n\circ\pi(J_2)\cdot\one_{[I]}(J_2)|=0.\vspace*{1ex}
        \]

        \item[(ii)] if $J_1\in[I]$ and $J_2\notin[I]$, then $\dist(J_1,J_2)>m^{-|I|}$. Combining this with \eqref{bddoflnx}, we have
        \[
        \frac{|l_n\circ\pi(J_1)\cdot\one_{[I]}(J_1)-l_n\circ\pi(J_2)\cdot\one_{I}(J_2)|}{\dist(J_1,J_2)^{\widetilde{\beta}}} < \frac{l_n\circ\pi(J_1)}{m^{-|I|}}\leq M\cdot m^{\widetilde{\beta}|I|}.\vspace*{1ex}
        \]

        \item[(iii)] if $J_1,J_2\in[I]$, then
        \begin{equation}\label{diffequali}
            \hspace{5ex}|l_n\circ\pi(J_1)\cdot\one_{[I]}(J_1)-l_n\circ\pi(J_2)\cdot\one_{I}(J_2)|=|l_n\circ\pi(J_1)-l_n\circ\pi(J_2)|.
        \end{equation}
        For any  $ J   \neq J' \in\Sigma^{\N}$, define $J\wedge J':=j_1j_2\cdot\cdot\cdot j_k$ where $k$ is the largest number such that $j_1j_2\cdot\cdot\cdot j_k=j'_1j'_2\cdot\cdot\cdot j'_k$, and we adopt the convention that $J\wedge J'=\emptyset$ if $j_1\neq j_1'$. With this standard notation in mind, note that
        \vspace{2ex}
        \begin{itemize}
            \item[$\circ$] if $\dist(J_1,J_2)\leq\psi(n)^{\tau\cdot\frac{2\log m}{-\delta\log\kappa}}$, then by \ref{toprovediamP2} and \ref{keyproP3}, we obtain that
            \medskip
            \begin{equation*}
                \begin{aligned}
                    \text{r.h.s. of ~\eqref{diffequali}}&= \frac{|\mu(B(\pi(J_1),\psi(n)))-\mu(B(\pi(J_1),\psi(n)))|}{\psi(n)^{\tau}}\\[2ex]
                    &\leq\frac{\mu\left(\big(\partial B(\pi(J_1),\psi(n))\big)_{|\pi(J_1)-\pi(J_2)|}\right)}{\psi(n)^{\tau}}\\[2ex]
                    &\ll \frac{|\pi(J_1)-\pi(J_2)|^{\delta}}{\psi(n)^{\tau}}\leq\frac{|K_{J_1\wedge J_2}|^{\delta}}{\psi(n)^{\tau}}\ll \frac{\kappa^{\delta|J_1\wedge J_2|}}{\psi(n)^{\tau}}\\[2ex]
                    &=\frac{\kappa^{\delta\cdot\frac{-\log\dist(J_1,J_2)}{\log m}}}{\psi(n)^{\tau}}=\frac{\dist(J_1,J_2)^{\frac{-\delta\log\kappa}{\log m}}}{\psi(n)^{\tau}}\\[2ex]
                    &\leq\dist(J_1,J_2)^{\frac{-\delta\log\kappa}{2\log m}},
                \end{aligned}
            \end{equation*}
            where the implied constants within `$\ll$' are all independent of $J_1,J_2$ and $n$. In view of \eqref{bingyork}, it follows that
            \vspace*{1ex}
            \begin{eqnarray*}
                \hspace{15ex}\frac{|l_n\circ\pi(J_1)\cdot\one_{[I]}(J_1)-l_n\circ\pi(J_2)\cdot\one_{[I]}(J_2)|}{\dist(J_1,J_2)^{\widetilde{\beta}}}\ll\dist(J_1,J_2)^{\frac{-\delta\log\kappa}{2\log m}-\widetilde{\beta}}\leq 1,
            \end{eqnarray*}
            where the implied constant within `$\ll$' is independent of $J_1,J_2$, $n$ and $\widetilde{\beta}$.
            \vspace*{2ex}

            \item[$\circ$] if $\dist(J_1,J_2)>\psi(n)^{\tau\cdot\frac{2\log m}{-\delta\log\kappa}}$, then by \eqref{bddoflnx}, we have that
            \[
            \text{r.h.s. of \eqref{diffequali}}\leq 2M.
            \]
            This together with \eqref{bingyork}  implies that
            \vspace*{1ex}
            \begin{equation*}
                \begin{aligned}
                    \hspace{13ex}\frac{|l_n\circ\pi(J_1)\cdot\one_{[I]}(J_1)-l_n\circ\pi(J_2)\cdot\one_{[I]}(J_2)|}{\dist(J_1,J_2)^{\widetilde{\beta}}}<\frac{2M}{\psi(n)^{\tau\widetilde{\beta}\cdot\frac{2\log m}{-\delta\log\kappa}}}\leq 2M\psi(n)^{-\eta\tau}.
                \end{aligned}
            \end{equation*}
        \end{itemize}
    \end{itemize}
    \smallskip
    \noindent The above cases  (i), (ii) and (iii) imply that there exists $C>0$ such that for any $\eta>0$ and $0<\widetilde{\beta}\leq\frac{-\eta\delta\log\kappa}{2\log m}$, we have that
    \[
    \frac{|l_n\circ\pi(J_1)\cdot\one_{[I]}(J_1)-l_n\circ\pi(J_2)\cdot\one_{I}(J_2)|}{\dist(J_1,J_2)^{\widetilde{\beta}}}\leq C(m^{\widetilde{\beta}|I|}+\psi(n)^{-\eta\tau})
    \]
    for any $n\in\N$, $I\in\Sigma^*$ and $J_1\neq J_2\in\Sigma^{\N}$. Then, on combining this and the definition of $\widetilde{\beta}$-H\"{o}lder norm, the desired inequality \eqref{desibetano} is proved.
\end{proof}

\begin{lemma}\label{toprolemonern1mebn2}
    Assume the setting in Theorem~\ref{toprove}. Let $\tau:=\dimH K$.  Then, for any $\eta\in(0,1)$, there exists  $\widetilde{\gamma}=\widetilde{\gamma}(\eta)\in(0,1)$  such that for any $n_1,\,n_2\in\N$
    \begin{eqnarray*}
    \int_{K}\one_{R_{n_1}(\psi)}(\x) & \hspace*{-4ex} \cdot  & \hspace*{-4ex}   \mu\big(B(\x,\psi(n_2))\big)\,\td\mu(\x)\\
        &=&\int_K\mu\big(B(\x,\psi(n_1))\big)\cdot\mu\big(B(\x,\psi(n_2))\big)\,\td\mu(\x)
        + O\big(\widetilde{\gamma}^{n_1}\big)\,\psi(n_2)^{(1-\eta)\tau}\, ,
    \end{eqnarray*}
    where the  implied  big $O$ constant  does not depend on the choices of $n_1,n_2\in\N$.
\end{lemma}
\begin{proof}
    Recall that $\mu=\ubar{\mu}\circ\pi^{-1}$ where $\ubar{\mu}$ is the Gibbs measure with respect to a $\beta$-Hölder potential on symbolic space $\Sigma^{\N}$.  Note that for any $\widetilde{\beta}\in(0,\beta]$,  any  $\beta$-H\"{o}lder continuous function  is also a $\widetilde{\beta}$-H\"{o}lder continuous function. Therefore,  $\ubar{\mu}$ is also a Gibbs measure with respect to a $\widetilde{\beta}$-H\"{o}lder potential. It follows that \eqref{decayHolder}  holds for any $0<\widetilde{\beta}\leq\beta$. More precisely, for any $0<\widetilde{\beta}\leq\beta$, there exist $C_{\widetilde{\beta}}>0$ and $\gamma_{\widetilde{\beta}}\in(0,1)$ such that
    \begin{equation}\label{decaywiltubeta}
        \left|\int_{\Sigma^{\N}}f_1\cdot f_2\circ\sigma^n\,\td\ubar{\mu}-\int_{\Sigma^{\N}}f_1\,\td\ubar{\mu}\cdot\int_{\Sigma^{\N}}f_2\,\td\ubar{\mu}\right|\,\,\leq\,\, C_{\widetilde{\beta}}\cdot(\gamma_{\widetilde{\beta}})^n\cdot\|f_1\|_{\widetilde{\beta}}\cdot\int_{\Sigma^{\N}}|f_2|\,\td\ubar{\mu}
    \end{equation}
    for any $f_1,f_2\in\cC^{\widetilde{\beta}}(\Sigma^{\N})$ and $n\in\N$.

Throughout, fix $\eta\in(0,1)$ and let $\widetilde{\beta}=\widetilde{\beta}(\eta):=\min\left\{\beta,\frac{-\eta\delta\log\kappa}{2\log m}\right\}$.  Let $\delta>0$ be as in (\ref{muannlusP3}).   In turn, for any $n\in\N$, let
\[
q_n=q_n(\eta):=\linte{\frac{n\log(1/\gamma_{\widetilde{\beta}})}{2(\widetilde{\beta}+1)\log m}}.
\]

On using  the left-hand side inclusion in Lemma~\ref{relalemtoprove}, we obtain that for any $n_1,n_2\in\N$,
\begin{eqnarray}
    &~&\hspace{-10ex}\int_K\one_{R_{n_1}(\psi)}\,(\x)\cdot\mu\big(B(\x,\psi(n_2))\big)\,\td\mu(\x) \nonumber\\[2ex]
    &~&\hspace{3ex}=\sum_{I\in\Sigma^{q_{n_{1}}}}\int_K\one_{K_I\,\cap\, R_{n_1}(\psi)}\,(\x)\cdot\mu\big(B(\x,\psi(n_2))\big)\,\td\mu(\x)  \nonumber \\[2ex]
&~&\hspace{3ex}\geq\sum_{I\in\Sigma^{q_{n_{1}}}}\int_K\one_{K_I\,\cap\, T^{-n_1}B(\z_I,\psi(n_1)-|K_I|)}\,(\x)\cdot\mu\big(B(\x,\psi(n_2))\big)\,\td\mu(\x) \, .  \label{junjieA}
\end{eqnarray}
where for $I\in\Sigma^*$ we fix some  point $\z_I\in K_I$ and $B(\x,r) $ is the empty set if $r \le 0$.
On combining  \eqref{junjieA} with Lemma~\ref{bddalholder},  the equality $\mu=\ubar{\mu}\circ\pi^{-1}$, the properties  \ref{toprovediamP2} --  \ref{keyproP3} and \eqref{decaywiltubeta}, we obtain that for any $n_1,n_2\in\N$,
\begin{eqnarray*}
 &~&\hspace{-3ex}\int_K\one_{R_{n_1}(\psi)}\,(\x)\cdot\mu\big(B(\x,\psi(n_2))\big)\,\td\mu(\x) \nonumber\\[2ex]
    &~&\hspace{3ex}\geq\sum_{I\in\Sigma^{q_{n_1}}}\left(\int_K\big(l_{n_2}(\x)\cdot\one_{K_I}(\x)\big)\cdot\big(\one_{B(\z_I,\psi(n_1)-|K_I|)}\circ T^{n_1}\big)\,(\x)\,\,\td\mu(\x)\right)\cdot\psi(n_2)^{\tau}\\[2ex]
&~&\hspace{3ex} =\sum_{I\in\Sigma^{q_{n_1}}}\left(\int_{\Sigma^{\N}}\big(l_{n_2}\circ\pi\cdot\one_{[I]}\big)\cdot\big(\one_{\pi^{-1}B(\z_I,\psi(n_1)-|K_I|)}\circ\sigma^{n_1}\big)\,\,\td\ubar{\mu}\right)\cdot\psi(n_2)^{\tau}\\[2ex]
    &~&\hspace{3ex}=  \sum_{I\in\Sigma^{q_{n_1}}}\left(\int_{\Sigma^{\N}} l_{n_2}\circ\pi\,(J)\cdot\one_{[I]}\,(J)\,\,\td\ubar{\mu}(J)+O\big(\|l_{n_2}\circ\pi\cdot\one_{[I]}\|_{\widetilde{\beta}}\cdot(\gamma_{\widetilde{\beta}})^{n_1}\big)\right)\\[1pt]
    &~&\hspace{30ex}\times \ \left(\int_{\Sigma^{\N}}\one_{\pi^{-1}B(\z_I,\psi(n_1)-|K_I|)}\circ\sigma^{n_1}\,\td\ubar{\mu}\right)
    \cdot\psi(n_2)^{\tau}\\[2ex]
    &~&\hspace{3ex}=\sum_{I\in\Sigma^{q_{n_1}}}\left(\int_{K} l_{n_2}(\x)\cdot\one_{K_I}(\x)\,\td\mu(\x)+O\big(m^{\widetilde{\beta}\cdot q_{n_1}}\cdot\psi(n_2)^{-\eta\tau}\cdot(\gamma_{\widetilde{\beta}})^{n_1}\big)\right)\\[2pt]
    &~&\hspace{30ex} \times  \  \mu\big(B(\z_I,\psi(n_1)-|K_I|)\big)\cdot\psi(n_2)^{\tau}\\[2ex]
    &~&\hspace{3ex}=\left(\sum_{I\in\Sigma^{q_{n_1}}}\int_{k_I}\mu\big(B(\z_I,\psi(n_1)-|K_I|)\big)\cdot\mu\big(B(\x,\psi(n_2))\big)\,\td\mu(\x)\right)\\[1pt]
    &~&\hspace{30ex} \   +\,\,O\big(m^{(\widetilde{\beta}+1) q_{n_1}}\cdot(\gamma_{\widetilde{\beta}})^{n_1}\big)\cdot\psi(n_2)^{(1-\eta)\tau}\\[2ex]
&~&\hspace{3ex}=\left(\sum_{I\in\Sigma^{q_{n_1}}}\int_{k_I}\Big(\mu\big(B(\x,\psi(n_1))\big)
+O\big(\kappa^{\delta \cdot q_{n_1}}\big)\Big)
\cdot\mu\big(B(\x,\psi(n_2))\big)\,\td\mu(\x)\right)\\[1pt]
    &~&\hspace{30ex}  +\,\,O\big(m^{(\widetilde{\beta}+1)q_{n_1}}\cdot(\gamma_{\widetilde{\beta}})^{n_1}\big)\cdot\psi(n_2)^{(1-\eta)\tau}\\[2ex]
    &~&\hspace{3ex}=\left(\sum_{I\in\Sigma^{q_{n_1}}}\int_{k_I}\mu\big(B(\x,\psi(n_1))\big)\cdot\mu\big(B(\x,\psi(n_2))\big)\,\td\mu(\x)\right)\\[1pt]
    &~&\hspace{20ex}+\ O\big(m^{(\widetilde{\beta}+1)\cdot q_{n_1}}\cdot(\gamma_{\widetilde{\beta}})^{n_1}+\kappa^{\delta\cdot q_{n_1}}\big)\cdot\psi(n_2)^{(1-\eta)\tau}\,.
\end{eqnarray*}

\noindent Now by the definition of $q_n$, we have $m^{(\widetilde{\beta}+1)\cdot q_n}(\gamma_{\widetilde{\beta}})^n\asymp(\gamma_{\widetilde{\beta}})^{n/2}$. Let $$\widetilde{\gamma}=\widetilde{\gamma}(\eta):=\max\left\{\kappa^{\frac{\delta\log(1/\gamma_{\widetilde{\beta}})}{2(\widetilde{\beta}+1)\log m}},(\gamma_{\widetilde{\beta}})^{\frac{1}{2}}\right\}.$$
Then the above calculation ensures the existence of $C=C(\eta)>0$ such that
\begin{eqnarray*}
        &~&\hspace{-5ex}\int_{K}\one_{R_{n_1}(\psi)}(\x)\cdot\mu\big(B(\x,\psi(n_2))\big)\,\td\mu(\x)\\
        &~&\hspace{7ex}\geq\int_K\mu\big(B(\x,\psi(n_1))\big)\cdot\mu\big(B(\x,\psi(n_2))\big)\,\td\mu(\x)
        -C\, \widetilde{\gamma}^{n_1}\,\psi(n_2)^{(1-\eta)\tau}.
    \end{eqnarray*}
The  complementary upper bound
follows on using the same argument but with the left  hand side inclusion in Lemma~\ref{relalemtoprove}   replaced by  the right hand side inclusion.
\end{proof}

We are now in the position to  provide a good upper bound the measure of the pairwise intersection of the sets $R_n(\psi)$.

\begin{lemma}\label{topromuofrnpsirnkpsi}
    Assume the setting in Theorem~\ref{toprove}. Let $\tau:=\dimH K$. Then  for any $\eta\in(0,1)$,   there exist $C=C(\eta)>0$ and $\widetilde{\gamma}=\widetilde{\gamma}(\eta)\in(0,1)$ such that for any $n,\,k\in\N$, we have that
    \begin{eqnarray*}
        \mu\big(R_n(\psi)\cap R_{n+k}(\psi)\big)
&\leq&\int_K\mu\big(B(\x,\psi(n))\big)\,\mu\big(B(\x,\psi(n+k))\big)\,\td\mu(\x)\\[4pt]
&~&\hspace{10ex}+\,C\,\big((\widetilde{\gamma}^n+\widetilde{\gamma}^k)\,\psi(n+k)^{(1-\eta)\tau}+\widetilde{\gamma}^{n+k}\big).
    \end{eqnarray*}
\end{lemma}

\begin{proof}
    Under the setting of Theorem~\ref{toprove}, the Gibbs measure $\mu$ is $\tau$-Ahlfors regular. Throughout, fix $\eta\in(0,1)$.
 Let $\gamma\in(0,1)$ be as in \ref{toprodecayP1}. For any $n\in\N$, let
    \[
q_n:=\linte{\frac{n\log(1/\gamma)}{4\log m}}\,.
    \]
    For each $I\in\Sigma^*$, fix a point $\z_I\in K_I$. Then by Lemma~\ref{relalemtoprove},  for any  $n_1,n_2,\ell\in\N$  it follows that
    \begin{eqnarray}
        &~&\hspace{-15ex}K_I\cap R_{n_1}(\psi)\cap R_{n_2}(\psi)\nonumber\\[2ex]
        &\subseteq& K_I\cap T^{-n_1}B(\z_I,\psi(n_1)+|K_I|)\cap T^{-n_2}B(\z_I,\psi(n_2)+|K_I|)\label{twocaprel1}\\[2ex]
        &\subseteq&\bigcup_{J\in\cJ(\ell,n_1,I)}K_I\cap T^{-n_1}K_J\cap T^{-n_2}B(\z_I,\psi(n_2)+|K_I|)\label{twocaprel2}
    \end{eqnarray}
 where
 \begin{eqnarray*}
     \cJ(\ell,n,I):=\left\{J\in\Sigma^{\ell}:K_I\cap B(\z_I,\psi(n)+|K_I|)\neq\emptyset\right\}   \qquad (\ell,n\in\N,\,I\in\Sigma^*)\,.\medskip
 \end{eqnarray*}

 We now estimate the  measure of  $R_n(\psi)\cap R_{n+k}(\psi)$   by considering two cases.

\noindent$\bullet$ \emph{Estimating  $\mu({R}_n\cap{R}_{n+k})$ when $q_{n+k}\geq n$. }  In this case, by means of the inclusion (\ref{twocaprel2}), Lemma~\ref{lemtwoseccylin} (with $k$ ``equal'' to $q_{n+k}$ and $n_1 =n $) and the inequalities (\ref{diaofkiP2}) and (\ref{muannlusP3}), it follows that
\medskip
    \begin{eqnarray}
        &~&\hspace{-5ex}\mu(R_n(\psi)\cap R_{n+k}(\psi))\nonumber\\[2ex]
&~&\hspace{3ex}\leq\sum_{I\in\Sigma^{q_{n+k}}}\sum_{J\in\cJ(q_{n+k},n,I)}\mu\Big(K_I\cap T^{-n}K_J\cap T^{-(n+k))}B(\z_I,\psi(n+k)+|K_I|)\Big)\nonumber\\[4pt]
&~&\hspace{3ex}=\sum_{I\in\Sigma^{q_{n+k}}}\sum_{J\in\cJ(q_{n+k},n,I)}\Big(\mu(K_I\cap T^{-n}K_J)+O\big(\gamma^{n+k}\big)\Big)\,\mu\big(B(\z_I,\psi(n+k)+|K_I|)\big)\nonumber\\[4pt]
&&\hspace{3ex}=\sum_{I\in\Sigma^{q_{n+k}}}\sum_{J\in\cJ(q_{n+k},n,I)}\Big(\mu(K_I\cap T^{-n}K_J)+O\big(\gamma^{n+k}\big)\Big)\nonumber\\[4pt]
&~& \hspace{15ex} \times\,\,\Big(\mu\big(B(\z_I,\psi(n+k))\big)+O\big(\kappa^{\delta\,q_{n+k}}\big)\Big)\nonumber\\[4pt]
&~&\hspace{3ex}=\sum_{I\in\Sigma^{q_{n+k}}}\sum_{J\in\cJ(q_{n+k},n,I)}\mu\left(K_I\cap T^{-n}K_J\right)\, \mu\big(B(\z_I,\psi(n+k))\big)\nonumber\\[4pt]
&~& \hspace{15ex} +\,\,\,\,O\big(\kappa^{\delta\,q_{n+k}}\big)\,\mu\left(\bigcup_{I\in\Sigma^{q_{n+k}}}\bigcup_{J\in\cJ(q_{n+k},n,J)}K_I\cap T^{-n}K_J\right)\nonumber\\[4pt]
&~& \hspace{15ex} +\,\,\,\,O\left(m^{2q_{n+k}}\gamma^{n+k}\,\big(\psi(n+k)^{\tau}+\kappa^{\delta\,q_{n+k}}\big)\right)\nonumber\\[2ex]
&~&\hspace{3ex}=\sum_{I\in\Sigma^{q_{n+k}}}\sum_{J\in\cJ(q_{n+k},n,I)}\mu\left(K_I\cap T^{-n}K_J\right)\, \mu\big(B(\z_I,\psi(n+k))\big)\nonumber\\[4pt]
&~& \hspace{15ex} +\,\,\,\, O\left(m^{2q_{n+k}}\gamma^{n+k}\,\big(\psi(n+k)^{\tau}+\kappa^{\delta\,q_{n+k}}\big)+\kappa^{\delta\, q_{n+k}}\right)\label{qnkgeqnrnrnk}
    \end{eqnarray}

    \noindent Next we estimate from above the  main term in \eqref{qnkgeqnrnrnk};   that is
    \[
S_{n+k}:=\sum_{I\in\Sigma^{q_{n+k}}}\sum_{J\in\cJ(q_{n+k},n,I)}\mu\left(K_I\cap T^{-n}K_J\right)\, \mu\big(B(\z_I,\psi(n+k))\big)\,.
    \]
     By (\ref{diaofkiP2}) and (\ref{muannlusP3}), we have that
   \begin{eqnarray*}
           S_{n+k}&=& \sum_{I\in\Sigma^{q_{n+k}}}\sum_{J\in\cJ(q_{n+k},n,I)}\int_{K_I \, \cap\, T^{-n}K_J}\mu\big(B(\z_I,\psi(n+k))\big)\,\,\td\mu(\x)\\[4pt]
           &=&\sum_{I\in\Sigma^{q_{n+k}}}\sum_{J\in\cJ(q_{n+k},n,I)}\int_{K_I\, \cap \, T^{-n}K_J}\Big(\mu\big(B(\x,\psi(n+k))\big)+O(\kappa^{\delta\,q_{n+k}})\Big)\,\,\td\mu(\x)\\[4pt]
           &=&\int_{\bigcup_{I\in\Sigma^{q_{n+k}}}\bigcup_{J\in\cJ(q_{n+k},n,I)}K_I\, \cap \, T^{-n}K_J} \!\!\!\!\!\!\hspace*{-16ex} \mu\big(B(\x,\psi(n+k))\big)\,\,\td\mu(\x)+O(\kappa^{\delta\,q_{n+k}})
   \end{eqnarray*}
   For any
   \[
   \x\in K_I\, \cap \, T^{-n}K_J\qquad{} (I\in\Sigma^{q_{n+k}},\,J\in\cJ(q_{n+k},n,I))\,,
   \]
    by (\ref{diaofkiP2}) and triangle inequality we have
   \begin{eqnarray*}
       |T^n\x-\x|&\leq&|\x-\z_I|+|\z_I-T^n\x|\\[4pt]
       &<&|K_I|+\psi(n)+|K_J|\\[4pt]
       &\leq&\psi(n)+2\,C_3\,\kappa^{q_{n+k}}.
   \end{eqnarray*}
  Therefore, on letting $\widetilde{\psi}(n):=\psi(n)+2\,C_3\,\kappa^{q_n}$ $(n\in\N)$, we obtain that
  \[
  \bigcup_{I\in\Sigma^{q_{n+k}}}\bigcup_{J\in\cJ(q_{n+k},n,I)}K_I\, \cap \, T^{-n}K_J\ \subseteq\ R_{n}(\widetilde{\psi})\,.
  \]
  Then by Lemma~\ref{toprolemonern1mebn2} there exists $\gamma_1=\gamma_1(\eta)\in(0,1)$ such that
  \begin{eqnarray}\label{upbddsnk}
      S_{n+k}&\leq&\int_{R_{n}(\widetilde{\psi})}\mu\big(B(\x,\widetilde{\psi}(n+k))\big)\,\,\td\mu(\x)+O(\kappa^{\delta\,q_{n+k}})\nonumber\\[4pt]
      &=&\int_K \,\mu\big(B(\x,\widetilde{\psi}(n))\big)\,\mu\big(B(\x,\widetilde{\psi}(n+k))\big)\,\,\td\mu(\x)\nonumber\\[4pt]
      &~&\hspace{8ex}+\,\,O\big(\gamma_1^n\,\widetilde{\psi}(n+k)^{(1-\eta)\tau}+\kappa^{\delta\,q_{n+k}}\big)\nonumber\\[4pt]
    &=&\int_K\,\Big(\mu\big(B(\x,\psi(n))\big)+O\big(\kappa^{\delta\,q_n}\big)\Big)\nonumber\\[4pt]
    &~&\hspace{8ex}\times\,\,\Big(\mu\big(B(\x,\psi(n+k))\big)+O\big(\kappa^{\delta\,q_{n+k}}\big)\Big)\,\,\td\mu(\x)\nonumber\\[4pt]
    &~&\hspace{25ex}+\,\,O\big(\gamma_1^n\,\widetilde{\psi}(n+k)^{(1-\eta)\tau}+\kappa^{\delta\,q_{n+k}}\big)\nonumber\\[4pt]
    &=&\int_K\,\mu\big(B(\x,\psi(n))\big)\,\mu\big(B(\x,\psi(n+k))\,\,\td\mu(\x)\nonumber\\[4pt]
    &~&\hspace{8ex}+\,\,O\big((\gamma_1^n+\kappa^{\delta\,q_n})\,\psi(n+k)^{(1-\eta)\tau}+\kappa^{\delta\,q_{n+k}}\big)\,.
  \end{eqnarray}
  Recall that by the definition of $q_{n}$, we have $m^{2q_n}\gamma^{n}\asymp\gamma^{n/2}$. Let
  \[
  \gamma_2=\gamma_2(\eta):=\max\left\{\gamma^{1/2},\kappa^{\frac{\delta\cdot\log(1/\gamma)}{4\cdot\log m}},\gamma_1(\eta)\right\}\in(0,1)\,.
  \]
  Then on combining (\ref{qnkgeqnrnrnk}) and (\ref{upbddsnk}) we find that there exists a constant $C>0$ so that
  \begin{eqnarray*}
      \mu\big(R_{n}(\psi)\cap R_{n+k}(\psi)\big)&\leq&\int_K\,\mu\big(B(\x,\psi(n))\big)\,\mu\big(B(\x,\psi(n+k))\,\,\td\mu(\x)  \\[4pt]
      &~&\hspace{15ex}+ \ \ C\  \big(\gamma_2^{n}\,\psi(n+k)^{(1-\eta)\tau}+\gamma_2^{n+k}\big)
  \end{eqnarray*}
  for all $n,k\in\N$ with $q_{n+k}\geq n$.

 \noindent$\bullet$ \emph{Estimating  $\mu({R}_n\cap{R}_{n+k})$ when $q_{n+k} <  n$. }  By the inclusion  (\ref{twocaprel2}) and
the properties \ref{toprodecayP1} -- \ref{keyproP3}, namely (\ref{KIT-nF}),  (\ref{BT-nF}),  \eqref{diaofkiP2}, \eqref{muannlusP3},   we have that
 \medskip
  \begin{eqnarray}
      &~&\hspace{-4ex}\mu\big(R_n(\psi)\cap R_{n+k}(\psi)\big)  \nonumber \\[2ex]
      &~&\hspace{4ex}\leq\sum_{I\in\Sigma^{q_n}}\mu\big(K_I\cap T^{-n}B(\z_I,\psi(n)+|K_I|)\cap T^{-(n+k)}B(\z_I,\psi(n+k)+|K_I|)\big)   \nonumber \\[4pt]
      &~&\hspace{4ex}=\sum_{I\in\Sigma^{q_n}}\big(\mu(K_I)+O(\gamma^n)\big)\,\mu\big(B(\z_I,\psi(n)+|K_I|)\cap T^{-k}B(\z_I,\psi(n+k)+|K_I|)\big) \nonumber  \\[4pt]
      &~&\hspace{4ex}=\sum_{I\in\Sigma^{q_n}}\big(\mu(K_I)+O(\gamma^n)\big)\,\big(\mu(B(\z_I,\psi(n)+|K_I|))+O(\gamma^{k})\big)  \nonumber \\
      &~&\hspace{18ex}\times\,\,\mu\big(B(\z_i,\psi(n+k)+|K_I|)\big)  \nonumber \\[6pt]
&~&\hspace{4ex}=\sum_{I\in\Sigma^{q_n}}\big(\mu(K_I)+O(\gamma^n)\big)\,\big(\mu(B(\z_I,\psi(n)))+O(\gamma^{k}+\kappa^{\delta\,q_n})\big) \nonumber \\
      &~&\hspace{18ex}\times\,\,\big(\mu(B(\z_I,\psi(n+k)))+O(\kappa^{\delta\,q_n})\big)  \nonumber \\[6pt]
&~&\hspace{4ex}=\sum_{I\in\Sigma^{q_n}}\int_{K_I}\,\mu\big(B(\z_I,\psi(n))\big)\,\mu\big(B(\z_I,\psi(n+k))\big)\,\,\td\mu(\x) \nonumber
\\[4pt]
      &~&\hspace{18ex} +\,\,O\big(\gamma^k\,\psi(n+k)^{\tau}+\kappa^{\delta\,q_n}+m^{q_n}\,\gamma^n\big)  \nonumber \\[4pt]
      &~&\hspace{4ex}=\sum_{I\in\Sigma^{q_n}}\int_{K_I}\big(\mu(B(\x,\psi(n)))+O(\kappa^{\delta\,q_n})\big)  \nonumber \\
      &~&\hspace{18ex} \times\,\,\big(\mu(B(\x,\psi(n+k)))+O(\kappa^{\delta\,q_{n+k}})\big)\,\,\td\mu(\x) \nonumber   \\[6pt]
      &~&\hspace{26ex}+\,\,O\big(\gamma^k\,\psi(n+k)^{\tau}+\kappa^{\delta\,q_n}+m^{q_{n}}\,\gamma^n\big)  \nonumber  \\[4pt]
      &~&\hspace{4ex}=\int_{K}\mu\big(B(\x,\psi(n))\big)\,\mu\big(B(\x,\psi(n+k))\big)\,\,\td\mu(\x)  \nonumber  \\[4pt]
      &~&\hspace{18ex}+\,\,O\big(\gamma^k\,\psi(n+k)^{\tau}+\kappa^{\delta\,q_n}+m^{q_n}\,\gamma^n\big)  \nonumber \\[4pt]
      &~&\hspace{4ex}=\int_{K}\mu\big(B(\x,\psi(n))\big)\,\mu\big(B(\x,\psi(n+k))\big)\,\,\td\mu(\x)   \nonumber \\[4pt]
      &~&\hspace{18ex}                                             +\,\,O\big(\gamma^k\,\psi(n+k)^{\tau}+\kappa^{\delta\,q_n}+\gamma^{n/2}\big)  \label{asdfg}
  \end{eqnarray}

  \noindent In the case $q_{n+k}<n$, we have that
  \[
  n>\frac{(n+k)\log(1/\gamma)}{4\log m}-1\,
  \]

 \noindent  and so it follows  that
  \[
  \kappa^{\delta\,q_n}=O\left(\kappa^{\delta\cdot\left(\frac{\log(1/\gamma)}{4\log m}\right)\cdot(n+k)}\right)  \ \quad {\rm and } \ \quad  \gamma^{n/2}=O\left(\gamma^{\frac{\log(1/\gamma)}{8\log m}\cdot(n+k)}\right)   \, .
  \]
 The upshot of this and  \eqref{asdfg} is that there exists a constant $C>0$ so that
  \begin{eqnarray*}
      \mu\big(R_n(\psi)\cap R_{n+k}(\psi)\big)&\leq&\int_{K}\mu\big(B(\x,\psi(n))\big)\,\mu\big(B(\x,\psi(n+k))\big)\,\,\td\mu(\x)\\[4pt]
      &~&\hspace{5ex}+\,\,C\,\big(\gamma_3^k\,\psi(n+k)^{\tau}+\gamma_3^{n+k}\big)
  \end{eqnarray*}
  for all $n,k\in\N$ with $q_{n+k}<n$, where
  \[
  \gamma_3:=\max\left\{\gamma,\,\kappa^{\delta\cdot\left(\frac{\log(1/\gamma)}{4\log m}\right)},\gamma^{\frac{\log(1/\gamma)}{8\log m}}\right\}.
  \]

  \medskip

 The above two cases imply the desired upper bound estimate of the measure of the set $R_n(\psi)\cap R_{n+k}(\psi)$ for any $n,k\in\N$.
\end{proof}

Now we are in a position to prove Theorem~\ref{toprove} utilizing Lemma~\ref{genharmanlem}.
    Fix $\eta\in(0,1)$. With Lemma~\ref{topromuofrnpsi}, Lemma~\ref{toprolemonern1mebn2} and Lemma~\ref{topromuofrnpsirnkpsi} at hand, it follows that there exist $C=C(\eta )>0     $ and $\widetilde{\gamma}=\widetilde{\gamma}(\eta)\in(0,1)$ such that for any $a<b\in\N$
    \medskip
    \begin{eqnarray*}
        &~&\hspace{-5ex}\int_K\left(\sum_{a\leq n\leq b}\big(\one_{R_n(\psi)}(\x)-\mu(B(\x,\psi(n)))\big)\right)^2\,\td\mu(\x)\\[4pt]
        &~&\hspace{3ex}=\  \  \  \sum_{n=a}^b\mu(R_n(\psi))+2\sum_{n=a}^{b-1}\sum_{k=1}^{b-n}\mu(R_{n}(\psi)\cap R_{n+k}(\psi))\\[4pt]
        &~&\hspace{12ex}-\,\,2\sum_{a\leq n_1,n_2\leq b}\int_K\one_{R_{n_1}(\psi)}(\x)\cdot\mu(B(\x,\psi(n_2)))\,\td\mu(\x)\\[4pt]
        &~&\hspace{12ex}+\,\,\sum_{a\leq n_1,n_2\leq b}\int_K\mu(B(\x,\psi(n_1)))\,\mu(B(\x,\psi(n_2)))\,\td\mu(\x)\\[4pt]
        &~&\hspace{3ex}\leq\ \ \   \sum_{n=a}^b\Big(\int_K\mu(B(\x,\psi(n)))\,\td\mu(\x)+C\,\widetilde{\gamma}^n\Big)\\[4pt]
        &~&\hspace{12ex}+\,\,2\sum_{n=a}^{b-1}\sum_{k=1}^{b-n}\Big(\int_K\mu(B(\x,\psi(n)))\,\mu(B(\x,\psi(n+k)))\,\td\mu(\x)\\[4pt]
        &~&\hspace{18ex}\qquad \qquad\qquad+\,\,C\,((\widetilde{\gamma}^n+\widetilde{\gamma}^k)\,\psi(n+k)^{(1-\eta)\tau}+\widetilde{\gamma}^{n+k})\Big)\\[4pt]
        &~&\hspace{12ex}-\,\,2\sum_{a\leq n_1,n_2\leq b}\Big(\int_K\mu(B(\x,\psi(n_1)))\,\mu(B(\x,\psi(n_2)))\,\td\mu(\x)\\[4pt]
        &~&\hspace{18ex}\qquad\qquad\qquad+C\,\widetilde{\gamma}^{n_1}\,\psi(n_2)^{(1-\eta)\tau}\Big)\\[4pt]
        &~&\hspace{12ex}+\,\,\sum_{a\leq n_1,n_2\leq b}\int_K\mu(B(\x,\psi(n_1)))\,\mu(B(\x,\psi(n_2)))\,\td\mu(\x)\\[4pt]
        &~&\hspace{3ex} \ll  \  \ \sum_{n=a}^b\big(\psi(n)^{(1-\eta)\tau}+\widetilde{\gamma}^{n}\big)\\[2ex]
        &~&\hspace{3ex}  \asymp \  \ \sum_{n=a}^b\phi_n    \qquad {\rm where }  \qquad  \phi_n=\phi_n(\eta):= \max_{\x\in K}\mu(B(\x,\psi(n)))^{1-\eta}  \ + \ \widetilde{\gamma}^n.
    \end{eqnarray*}

   \noindent Applying Lemma~\ref{genharmanlem} with  $X=K$ and
   $$
   f_n(\x)=\one_{R_n(\psi)}(\x),  \quad g_{n}(\x)=\mu(B(\x,\psi(n))),    \quad   \phi_n =  \max_{\x\in K}\mu(B(\x,\psi(n)))^{1-\eta}  \, + \,  \widetilde{\gamma}^n,$$
    we obtain that for any $\epsilon>0$,
   \begin{equation}  \label{junjieD}
       \sum_{n=1}^N\one_{R_n(\psi)}(\x)=\sum_{n=1}^{N}\mu(B(\x,\psi(n)))+O\left(\Psi_{\eta}(N)^{1/2}\big(\log(\Psi_{\eta}(N))\big)^{\frac{3}{2}+\epsilon}\right)
   \end{equation}
   for $\mu$-almost every $\x\in K$, where $\Psi_{\eta}(N):=\sum_{n=1}^N\phi_n$.  However, since $\mu $ is  $\tau$-Ahlfors regular,  we have $$\Psi_{\eta}(N)  \asymp \sum_{n=1}^N\psi(n)^{(1-\eta)\tau} \, , $$   so the term $\Psi_{\eta}(N)$ appearing in the error term of  (\ref{junjieD}) can be replaced by the summation $\sum_{n=1}^N\psi(n)^{(1-\eta)\tau}$.  This completes the proof of Theorem~\ref{toprove}.

\bigskip

\subsection{Proof of Theorem~\ref{quantcount} } \label{weakquantcount}
Under the setting of Theorem~\ref{quantcount}, the measure $\mu$ is $\tau$-Ahlfors regular (see Remark~\ref{newrem}\,(i)).
    Hence,  for any $\x\in K$, $$0<\Theta^{\tau}_*(\mu,\psi,\x)\leq \Theta^{*\tau}(\mu,\psi,\x)<+\infty\,.$$   Let $\Q_+$ be the set of all positive rational numbers. For each $q\in\Q_+$ and $n\in\N$, define $\psi_q(n):=q\cdot\psi(n)^{\tau}$. With this notation in mind, suppose that the sum of the sequence $\psi(n)^{\tau}$ is divergent.  Then by Theorem~\ref{quantrec} and the fact that $\Q_+$ is countable, there exists $K^{'}\subseteq K$ such that $\mu(K^{'})=1$ and
    \begin{eqnarray}\label{counttnxpsiq}
        \lim_{N\to\infty}\frac{\sum_{n=1}^N\one_{B(\x,t_n(\x,\psi_q))}(T^n\x)}{\sum_{n=1}^Nq\psi(n)^{\tau}}=1,\qquad\forall\,\, \x\in K^{'},~\forall\,\, q\in\Q_+.
    \end{eqnarray}

    Fix an arbitrary point $\x\in K^{'}$ and any $\epsilon>0$. Since $\Q_+$ is dense in $(0,+\infty)$, there exist $p=p(\x,\epsilon),\,q=q(\x,\epsilon)\in\Q_+$ such that
    \[
        \Theta^{*\tau}(\mu,\psi,\x)<p<(1+\epsilon)\cdot\Theta^{*\tau}(\mu,\psi,\x),  \qquad (1-\epsilon)\cdot\Theta^{\tau}_*(\mu,\psi,\x)<q<\Theta^{\tau}_*(\mu,\psi,\x).
    \]
    With this in mind, it is easy to verify that there exists $n_0=n_0(\x)\in\N$ such that for all $n>n_0$, we have
    \[
    \mu\big(B(\x,t_n(\x,\psi_q))\big)\leq\mu\big(B(\x,\psi(n))\big)\leq \mu\big(B(\x,t_n(\x,\psi_q))\big)
    \]
    and thus
    \[
        B(\x,t_n(\x,\psi_q))\subseteq B(\x,\psi(n))\subseteq B(\x,t_n(\x,\psi_p))\,.
    \]
   This together with \eqref{counttnxpsiq} implies that

    \begin{eqnarray*}
    ~ \hspace*{4ex} \limsup_{N\to\infty}\frac{\sum_{n=1}^N\one_{B(\x,\psi(n))}(T^n\x)}{\sum_{n=1}^N\mu(B(\x,\psi(n)))}\\[2ex]
    & ~ &   \hspace*{-14ex}  \leq \ \  \limsup_{N\to\infty}\left(\frac{\sum_{n=1}^N\one_{B(\x,t_n(\x,\psi_p))}(T^n\x)}{\sum_{n=1}^N\mu(B(\x,t_n(\x,\psi_p)))}\cdot\frac{\sum_{n=1}^N\mu(B(\x,t_n(\x,\psi_p)))}{\sum_{n=1}^N\mu(B(\x,t_n(\x,\psi_q)))}\right)\\[2ex]
    & ~ &   \hspace*{-14ex}  =  \ \ \limsup_{N\to\infty}\left(\frac{\sum_{n=1}^N\one_{B(\x,t_n(\x,\psi_p))}(T^n\x)}{\sum_{n=1}^Np\psi(n)^{\tau}}\cdot\frac{\sum_{n=1}^Np\psi(n)^{\tau}}{\sum_{n=1}^Nq\psi(n)^{\tau}}\right)\\[2ex]
   & ~ &   \hspace*{-14ex}  =  \ \  \limsup_{N\to\infty}\frac{\sum_{n=1}^Np\psi(n)^{\tau}}{\sum_{n=1}^Nq\psi(n)^{\tau}}\\[2ex]
    & ~ &   \hspace*{-14ex}  \leq  \ \  \frac{(1+\epsilon)\cdot\Theta^{*\tau}(\mu,\psi,\x)}{(1-\epsilon)\cdot\Theta^{\tau}_*(\mu,\psi,\x)}\,.
    \end{eqnarray*}
   Since  $\epsilon>0$ is arbitrary, it follows that
    \[
        \limsup_{N\to\infty}\frac{\sum_{n=1}^N\one_{B(\x,\psi(n))}(T^n\x)}{\sum_{n=1}^N\mu(B(\x,\psi(n)))}\leq\frac{\Theta^{*\tau}(\mu,\psi,\x)}{\Theta^{\tau}_*(\mu,\psi,\x)}.
    \]
    A similar argument, with obvious modifications, shows that
    \[
        \liminf_{N\to\infty}\frac{\sum_{n=1}^N\one_{B(\x,\psi(n))}(T^n\x)}{\sum_{n=1}^N\mu(B(\x,\psi(n)))}\geq\frac{\Theta^{\tau}_*(\mu,\psi,\x)}{\Theta^{*\tau}(\mu,\psi,\x)}  \, ,
    \]
    and thereby completes the proof.

\bigskip

\subsection{Providing counterexamples to Claim~F } \label{examplesec}

We provide the details of the two counterexamples to Claim~F discussed in Section~\ref{appintro}.

\begin{example}  \label{egcantor}
Consider the functions   $\varphi_1:[0,1]\to[0,1/3]$ and $\varphi_2:[0,1]\to[2/3,1]$  given by
    \[
    \varphi_1(x)=\frac{x}{3},  \qquad \varphi_2(x)=\frac{x+2}{3}  \qquad \forall \  x\in[0,1].
    \]
 Then  $\Phi=\{\varphi_1,\varphi_2\}$ is the well-known  conformal IFS with the attractor $K$ being the standard middle-third Cantor set.    Let $\mu:=\cH^{\tau}|_K$  ($\tau=\log2/\log 3$) be the standard Cantor measure, and let $T:[0,1]\to[0,1]$ be the~$\times 3$ map: $$ Tx=3x~\tmod~1   \, . $$
  Consider the constant function $\psi:\R\to\R_{\geq0}$ given by $$\psi(x) \, := \,  \frac{1}{3}+\frac{2}{3^2}.$$
    Then, for $\mu$--almost all  $x\in K$, we have
   \begin{eqnarray} \label{klj}
      \lim_{N\to\infty}\frac{\sum_{n=1}^{N}
 \mathbbm{1}_{B(x,\psi(n))}(T^nx)}{\sum_{n=1}^{N}\mu(R_n(\psi))}
      =\left\{
\begin{aligned}
    &\textstyle{\frac{4}{5}  \quad \text{if}  \quad x\in\left(\big[0,\frac{1}{9}\big]\cup\big[\frac{8}{9},1\big]\right)\cap K} ,\\[2ex]
    &\textstyle{\frac{6}{5}   \quad \text{if} \quad x\in\left(\big[\frac{2}{9},\frac{1}{3}\big]\cup\big[\frac{2}{3},\frac{7}{9}\big]\right)\cap K}.
\end{aligned}
      \right.
   \end{eqnarray}
\end{example}

\begin{proof}[Proof of \eqref{klj}] We start with an observation.  The condition that $\psi(x)\to0$ as $ x \to \infty$  imposed in the statement of Theorem~\ref{toprove}
is used in its proof only to ensure that inequality \eqref{muannlusP3} is valid; that is to guarantee  that $\psi(n)  \leq r_0$  for $n$ large enough. However, for the particular example under consideration, this inequality is valid for all $r>0$ and thus  the conclusion of the theorem  and its corollary  are valid for any constant function $\psi$.  With this in mind, it follows via Theorem~\ref{toprove}  that:  for $\mu-$almost all $x \in K$
    \[
        \lim_{N\to\infty}\frac{\sum_{n=1}^N\one_{B(x,\psi(n))}(T^nx)}{\sum_{n=1}^N\mu(B(x,\psi(n)))}=1  \, .
    \]
    Thus, for $\mu-$almost all $\x \in K$
   \begin{equation}  \label{hty}
      \lim_{N\to\infty}  \frac{\sum_{n=1}^N\one_{B(x,\psi(n))}(T^nx)}{\sum_{n=1}^N\mu(R_n(\psi))}   \ =   \ \lim_{N\to\infty} \frac{\sum_{n=1}^N\mu(B(x,\psi(n)))}{\sum_{n=1}^N\mu(R_n(\psi))}  \, .
    \end{equation}

        We now
estimate $\mu(B(x,\psi(n)))$ and $\mu(R_n(\psi))$.
    It follows from the definition of $\psi$,  that
    \begin{eqnarray*}
        B(x,\psi(n)) \cap K=\left\{
\begin{aligned}
    & \ [0,1/3]\cap K,  \hspace*{15ex}  \;  \text{if} \ x\in[0,1/9]\cap K\\[0.3ex]
    & \ \big([0,1/3]\cup[2/3,7/9]\big)\cap K,\,\,\,\,\text{if}\ x\in[2/9,1/3]\cap K\\[0.3ex]
    & \ \big([2/9,1/3]\cup[2/3,1]\big)\cap K,\,\,\,\,\text{if}\ x\in[2/3,7/9]\cap K\\[0.3ex]
    & \ [2/3,1]\cap K,\hspace*{15ex}  \; \text{if}\ x\in[8/9,1]\cap K
\end{aligned}
        \right.
    \end{eqnarray*}
    for all $n\in\N$.  Thus,
    \begin{eqnarray*}\label{mball}
    \mu(B(x,\psi(n)))=\left\{
\begin{aligned}
    &\frac{1}{2},\,\,\,\,\text{if}\ x\in\big([0,1/9]\cup [8/9,1]\big)\cap K,\\[1ex]
    &\frac{3}{4},\,\,\,\,\text{if}\ x\in\big([2/9,1/3]\cup[2/3,7/9]\big)\cap K
\end{aligned}
    \right.
    \end{eqnarray*}

\noindent  for all $n\in\N$.  This together with Lemma~\ref{topromuofrnpsi} implies the existence of $\gamma\in(0,1)$ such that
\begin{eqnarray*}
    \mu(R_n(\psi))&=&\int_{0}^1\mu\Big(B\big(x,\psi(n)\big)\Big)~\td\mu(x)+O(\gamma^n)\\
    &=&\int_{\big([0,1/9]\cup[8/9,1]\big)\cap K}\mu\Big(B\big(x,\psi(n)\big)\Big)~\td\mu(x)\\&~&\hspace{8ex}+\int_{\big([2/9,1/3]\cup[2/3,7/9]\big)\cap K}\mu\Big(B\big(x,\psi(n)\big)\Big)~\td\mu(x)+O(\gamma^n)\\
    &=&\frac{5}{8}+O(\gamma^n).
\end{eqnarray*}
It thus follows via  \eqref{hty} that: for $\mu-$almost all $x \in K$
\begin{eqnarray*}
\lim_{N\to\infty}  \frac{\sum_{n=1}^N\one_{B(x,\psi(n))}(T^nx)}{\sum_{n=1}^N\mu(R_n(\psi))}
     &=&\lim_{N\to\infty}\frac{N\cdot\mu\big(B(x,1/3+2/3^2)\big)}{N\cdot5/8+O(1)}\\[0.7em]
     &=&\frac{\mu\big(B(x,1/3+2/3^2)\big)}{5/8}\\[0.7em]
     &=& \left\{
\begin{aligned}
    &\textstyle{\frac{4}{5}  \quad \text{if}  \quad x\in\left(\big[0,\frac{1}{9}\big]\cup\big[\frac{8}{9},1\big]\right)\cap K} ,\\[2ex]
    &\textstyle{\frac{6}{5}   \quad \text{if} \quad x\in\left(\big[\frac{2}{9},\frac{1}{3}\big]\cup\big[\frac{2}{3},\frac{7}{9}\big]\right)\cap K}.
\end{aligned}
      \right.
   \end{eqnarray*}
\end{proof}

\medskip

As mentioned in the introduction (Section~\ref{appintro}), the  second example is slightly more sophisticated but it  removes the need for the  function $\psi$  to be constant.

 \begin{example}  \label{egfull}
 Let $\Phi=\{\varphi_1,\varphi_2,\varphi_3,\varphi_4\}$ be a conformal IFS on $[0,1]$ given by
 \begin{eqnarray*}
     \varphi_1(x)=\frac{1}{4}x,\ \varphi_2(x)=\frac{1}{2(1+x)},\ \varphi_3(x)=\frac{1+x}{2+x},\ \varphi_4(x)=\frac{2}{2+x}\,.
 \end{eqnarray*}
    Then  it can be easily   verified that the self-conformal set generated by $\Phi$ is exactly the unit interval $[0,1]$; that is
    \[
[0,1]=\bigcup_{i=1}^4\varphi_i([0,1]).
    \]
    Moreover, it is easily verified that
    \[
    \varphi_i((0,1))\cap\varphi_j((0,1))=\emptyset,\qquad\forall\,\,1\leq i\neq j\leq 4
    \]
    and hence $\Phi$ satisfies open set condition.
     Define the Ruelle operator $\cL:C([0,1])\to C([0,1])$ by setting
    \[
    \left(\cL f\right)(x):=\sum_{i=1}^4|\varphi_i'(x)|f(\varphi_i(x)) \qquad \forall\, f\in C([0,1])\   \ \text{and}\  \  x\in[0,1].    \]
    By Example~\ref{hausgibbs}, the spectral radius of $\cL$ is $1$. Let
    \begin{equation}  \label{plmk}
    h(x):=\frac{1}{\log 2}\cdot\frac{1}{1+x}
    \end{equation} and let $\lambda$ denote the Lebesgue measure on $[0,1]$. Then a straightforward calculation shows that $$\cL h=h,\quad \cL^*\lambda=\lambda,\quad\int_{[0,1]}h(x)\,\td x=1; $$  that is to say that  $h$ and $\lambda$ are the unique eigenfunction and eigenmeasure of $\cL$ guaranteed by part (i) of Theorem~\ref{ruellethmrd}.    Hence by  definition,  the measure $\td\mu:=h~\td\lambda$ is the Gibbs measure with respect to $\cL$. The IFS $\Phi$ induces a transformation $T:[0,1]\to[0,1]$ given by
    \begin{eqnarray*}
        Tx=\left\{
        \begin{aligned}
            &4x,  \qquad \  \text{if}\ \ 0\leq x<\frac{1}{4},\\
            &\frac{1}{2x}-1,~~~~\text{if}\ \ \frac{1}{4}\leq x< \frac{1}{2},\\[4pt]
            &\frac{2x-1}{1-x},~~~~\text{if}\ \ \frac{1}{2}\leq x<\frac{2}{3},\\[4pt]
            &\frac{2}{x}-2, \quad \text{if}\ \ \frac{2}{3}\leq x\leq 1.
        \end{aligned}\right.
    \end{eqnarray*}

   \noindent The upshot of the above discussion is that the IFS $\Phi$ gives rise  to a self-conformal system $(\Phi,[0,1],\mu,T)$ in which $\mu$ is absolutely continuous with respect to Lebesgue measure with density $h$ given by \eqref{plmk}.

   Let $\psi:\R\to\R_{\geq0}$ be a real positive function such that  $\psi(x)\to0$ as $x\to\infty$ and  $\sum_{n=1}^{\infty}\psi(n)=+\infty$.
    For each  $n\in\N$, let $R_n(\psi)$ be the set  defined by  (\ref{anrpsidef}) with $X=[0,1]$.  Now regrading the measure of   $R_n(\psi)$, by Lemma~\ref{topromuofrnpsi}, we obtain that there exists $\widetilde{\gamma}\in(0,1)$ such that
\begin{eqnarray}
    \mu(R_n(\psi))&=&\int_{0}^1\mu\Big(B\big(x,\psi(n)\big)\Big)~\td\mu(x)+O(\widetilde{\gamma}^n)\nonumber\\[1ex]
    &=&\int_{0}^{\psi(n)}h(x)\left(\int_{0}^{x+\psi(n)}h(y)~\td y\right)\td x\nonumber\\[1ex]
    &~& \qquad\qquad + \ \int_{\psi(n)}^{1-\psi(n)}h(x)\left(\int_{x-\psi(n)}^{x+\psi(n)}h(y)~\td y\right)\td x\nonumber\\[1ex]
    &~&  \qquad\qquad + \ \int_{1-\psi(n)}^{1}h(x)\left(\int_{x-\psi(n)}^{1}h(y)~\td y\right)\td x+O(\widetilde{\gamma}^n)\nonumber\\[2ex]
    &=&\frac{1}{(\log 2)^2}\int_{0}^1\frac{1}{1+x}\left(\int_{x-\psi(n)}^{x+\psi(n)}\frac{1}{1+y}~\td y\right)\td x+O(\widetilde{\gamma}^n+\psi(n)^2)\nonumber\\[2ex]
    &=&\frac{\psi(n)}{(\log 2)^2}+O\left(\widetilde{\gamma}^n+(\psi(n))^2\right).\label{muan}
\end{eqnarray}

We are now in the position to put everything together and show that the self-conformal system $(\Phi,[0,1],\mu,T)$ provides a counterexample to Claim~F.    Indeed, by Corollary \ref{quantcountabcont}, we have that for $\mu-$almost all $x \in [0,1]$
\begin{eqnarray}\label{countrecur}
    \lim_{N\to\infty}
    \frac{\sum_{n=1}^N\one_{B(x,\psi(n))}(T^nx)}{\sum_{n=1}^N2h(x)\psi(n)}=1  \, .
\end{eqnarray}
Recall that if $\{a_n\}_{n\in\N}$ is a positive sequence of real numbers such that $a_n\to0$ and $\sum_{n=1}^{\infty}a_n=+\infty$, then
\begin{eqnarray*}
    \lim_{N\to\infty}\frac{\sum_{n=1}^Na_n^2}{\sum_{n=1}^Na_n}=0.
\end{eqnarray*}
With this in mind, on combining \eqref{plmk}, (\ref{muan}) and (\ref{countrecur}),  we find that for $\mu-$almost all $x \in K:=[0,1]$
\begin{eqnarray*}
     \lim_{N\to\infty}
      \frac{\sum_{n=1}^N\one_{B(x,\psi(n))}(T^nx)}{\sum_{n=1}^N\mu(R_n(\psi))}&=&\lim_{N\to\infty} \frac{\sum_{n=1}^N\one_{B(x,\psi(n))}(T^nx)}{\sum_{n=1}^N2h(x)\psi(n)}\cdot\lim_{N\to\infty}\frac{\sum_{n=1}^N2h(x)\psi(n)}{\sum_{n=1}^N\mu(R_n(\psi))}\\[4pt]
     &=&\lim_{N\to\infty}\frac{\sum_{n=1}^N2h(x)\psi(n)}{\sum_{n=1}^N\left(\frac{\psi(n)}{(\log2)^2}+O\left(\gamma^n+(\psi(n))^2\right)\right)}\\[4pt]
     &=&2(\log 2)^2\,h(x)\\[4pt]
     &=&\frac{2\log 2}{1+x}   \,  .
\end{eqnarray*}
The upshot is that Claim~F is false.
\end{example}

\bigskip

\begin{remark} \label{rembigger}
  In view of Remark~\ref{weakrk},  we could replace the IFS  $\Phi$ in Example~\ref{egfull} by the simpler IFS   $\Phi^{'}=\{\phi_1,\phi_2\}$ where
   $\phi_1:[0,1]\to[0,1/2]$ and $\phi_2:[0,1]\to[1/2,1]$ are  given by
    $$
    \phi_1(x)=\textstyle{ \frac{x}{2}},  \qquad \textstyle{  \phi_2(x)=\frac{1}{1+x} } \qquad \forall \  x\in[0,1]  .
    $$
     It is easily seen that $\Phi$ corresponds to  a single iteration of $\Phi^{'}$ and that the latter gives rise to a self-conformal system $(\Phi^{'},[0,1],\mu,T)$ in which $\mu$ is as in Example~\ref{egfull} and  the transformation  $T:[0,1]\to[0,1]$ is given by
    \begin{eqnarray*}
        Tx=\left\{
        \begin{aligned}
            &2x,\qquad\,\, \ \ \text{if  \ $0\leq x<\frac{1}{2}$},\\
            &\frac{1}{x}-1,\,\,\,\, \ \ \,\text{if  \ $\frac{1}{2}\leq x\leq 1$}.
        \end{aligned}
        \right.
    \end{eqnarray*}

\noindent Note that  $\Phi^{'}$ fails to meet the second condition of \eqref{varphi'contr} since  $|\phi_2^{'}(0)|=1$  but it does satisfy the weaker condition \eqref{instead} with $n_0=2$.  The point  being made in  Remark~\ref{weakrk} is that the results in this paper (such as  Corollary~\ref{quantcountabcont}) are in fact  applicable to systems such as  $(\Phi^{'},[0,1],\mu,T)$  that do not necessarily satisfy   the second condition of \eqref{varphi'contr} but do satisfy the weaker condition \eqref{instead}.
\end{remark}

\appendix
\section{Exponentially mixing in product  systems}\label{AppA}

This appendix is motivated by the  discussion centred around Remark~\ref{proddy} in the introduction, namely Section~\ref{bandi}.

Throughout, let $k\geq2$ be an integer. For each $1\leq i\leq k$, let $(X_i,d_i)$ be a metric space  and let $(X_i,\cB_i,\mu_i,T_i)$ be a measure-preserving system with $\mu$ being exponentially mixing with respect to $(T_i,\cC_i)$, where $\cC_i$ is the collection of balls in $X_i$. Now construct the   metric space $(X,d)$ and the measure-preserving system $(X,\cB,\mu,T)$ via $\{(X_i,d_i)\}_{1\leq i\leq k}$ and $\{(X_i,\cB_i,\mu_i,T_i)\}_{1\leq i\leq k}$ respectively as follows:
\begin{itemize}
    \item $\displaystyle X:=\prod_{i=1}^k X_i$.
    \medskip

    \item For any $\x=(x_1,...,x_k)$, $\y=(y_1,...,y_k)\in X$
    \[
    d(\x,\y):=\max\{d_i(x_i,y_i):i=1,2,...,k\}\,.
    \]
    \item $\displaystyle \cB:=\bigotimes_{i=1}^k\cB_i$ is the $\sigma$-algebra on $X$ generated by the sets of the form
    \[
    \left\{\prod_{i=1}^kF_i:F_i\in\cB_i,\,i=1,2,...,k\right\}\,.
    \]
\item $\displaystyle\mu:=\bigotimes_{i=1}^k\mu_i$ is the product measure on $X$.
\medskip

\item The map $T:X\to X$ is given  by
\[
T\x:=(T_1x_1,\,T_2x_2,\,...\,,\,T_kx_k),\ \ \ \ \forall\,\,\x=(x_1,x_2,...,x_k)\in X\,.
\]
\end{itemize}

\medskip

\begin{theorem}
    Let $(X,d)$  and $(X,\cB,\mu,T)$ be mentioned above. Then $\mu$ is exponentially mixing with respect to $(T,\cC)$, where $\cC$ is the collection of balls in $(X,d)$.
\end{theorem}

\begin{proof}
    Since each $\mu_i$ is exponentially mixing with respect to $(T_i,\cC_i)$, then there exist $C\geq1$ and $\gamma\in(0,1)$ such that for any $i=1,2,...,k$,
    \begin{eqnarray}\label{muiexpmixing}
        \big|\mu_i(B_i\cap T^{-n}F_i)-\mu_i(B_i)\mu_i(F_i)\big|\leq C\,\gamma^n\,\mu_i(F_i)
    \end{eqnarray}
    holds for any ball $B_i\subseteq X_i$, any $F_i\in\cB_i$ and any $n\in\N$. Throughout, fix an arbitrary ball $B\subseteq X$. In view of the definition of metric in $X$, it is clear that $B$ can be written as $$B=\prod_{i=1}^kB_i\,,$$ where $B_i\subseteq X_i$ is a ball in $X_i$ for each $i=1,2,...,k$. Next,  consider the set
    \begin{eqnarray*}
        F=\prod_{i=1}^k F_i\,,
    \end{eqnarray*}
    where $F_i\in\cB_i$ $(i=1,2,...,k)$.
    Then, for any $n\in\N$,
    \begin{eqnarray*}
        \mu\big(B\cap T^{-n}F\big)&=&\mu\left(\prod_{i=1}^k B_i\cap T_i^{-n}F_i\right)\\
        &=&\prod_{i=1}^k\mu_i\big(B_i\cap T_i^{-n}F_i\big)\\
        &\leq&\prod_{i=1}^k\left(\mu_i(B_i)+C\gamma^n\right)\,\mu_i(F_i)\qquad\text{(by \eqref{muiexpmixing})}\\
&=&\left(\prod_{i=1}^k\left(\mu_i(B_i)+C\gamma^n\right)\right)\cdot\mu(F)\\
&\leq&\left(\prod_{i=1}^k\mu_i(B_i)+2^kC^k\gamma^n\right)\cdot\mu(F)\\[4pt]
&=&\big(\mu(B)+2^k\,C^k\,\gamma^n\big)\,\mu(F)\,.
    \end{eqnarray*}
A similar argument with obvious modifications shows that for any $n\in\N$,
\begin{eqnarray*}
    \mu\big(B\cap T^{-n}F\big)~\geq~\big(\mu(B)-2^k\,C^k\,\gamma^n\big)\,\mu(F)\,.
\end{eqnarray*}
The upshot of the above discussions is that for any $F\in\cB$ of the form
\[
F=\prod_{i=1}^kF_i\,\qquad(F_i\in\cB_i\,,\,i=1,2,...,k)\,,
\]
we have
\begin{eqnarray}\label{muexpmixprod}
    \big|\mu(B\cap T^{-n}F)-\mu(B)\mu(F)\big|\,\leq\, 2^kC^k\gamma^n\,\mu(F)\,,\qquad\forall\,n\in\N\,.
\end{eqnarray}

We are going to prove \eqref{muexpmixprod} for any $F\in\cB$. To do this, let
\[
    \cA_B:=\left\{F\in\cB:\text{the inequality \eqref{muexpmixprod} holds for $F$}\right\}\,
\]
and let $\cG$ be the collection of finite disjoint unions of the sets in
\begin{eqnarray*}
    \left\{F=\prod_{i=1}^kF_i:F_i\in\cB_i,\,i=1,2,...,k\right\}.
\end{eqnarray*}
Various properties of $\cA_B$ and $\cG$ are listed below:
\begin{itemize}
    \item[$\circ$] As was discussed above, we have $\cA_{B}\supseteq\cG$.
    \medskip

    \item[$\circ$] $\cG$ is an algebra which generates the $\sigma$-algebra $\cB$.
    \medskip

    \item[$\circ$] $\cA_B$ is a monotone class, that is to say that $\cA_B$ is closed under countable increasing unions and countable decreasing intersections.
\end{itemize}
On combining the above facts with the Monotone Class Lemma \cite[Lemma 2.35]{folland1999}, we have $\cA_B=\cB$, which is equivalent to the statement that any $F\in\cB$ satisfies the desired
inequality \eqref{muexpmixprod}. Since $B\subseteq X$ is an arbitrary ball, then we have proved that $\mu$ is exponentially mixing with respect to $(T,\cC)$.
\end{proof}


\section{Absolutely friendly measures and exponential mixing}   \label{AppH}

In this appendix we explore the connection between exponential mixing statements and  the class of so-called absolutely friendly measures \cite{PVabs}, which are supported on fractal subsets of $\mathbb{R}^d$.   This class forms a subclass of the friendly measures introduced in \cite{kleinbock2005}, and includes, for example, measures supported on self-similar sets satisfying the open set condition (see Theorem~\ref{thmklw} below for the precise statement).
Friendly measures provide a robust framework for extending classical results in the theory of metric Diophantine approximation -- such as Khintchine-type theorems -- from Lebesgue measure to more singular or fractal settings. In order to describe the framework,  let  $\mu$ be a probability measure supported on a bounded subset  $X$ of $\R^d$.  In turn, let $\spt(\mu)$ denote the support of $\mu$.  Recall that $\mu$ is said to be \emph{doubling}
 if there exist
strictly positive constants $C$ and $r_0$ such that $$\mu(B(\x,2r))
\ \leq \ C \, \mu(B(\x,r)) \hspace{9mm} \forall \  \x \in \spt(\mu)
\qquad \forall \ r < r_0 \ ,   $$  where  $B(\x,r)$ is a ball in $\R^d$ with centre $\x$  and radius $r$.  Next, let $H$ denote a generic $(d-1)$--dimensional hyperplane of $\R^d$.   Then, we say that
$\mu$ is \emph{absolutely decaying} if there exist strictly positive constants $C$, $r_0$  and $\delta>0$ such that for any  hyperplane $H$ and $\varrho>0$, we have that
\begin{equation}\label{absoludecay}
\mu\big((H)_{\varrho}\cap B(\x,r)\big)\leq C\left(\frac{\varrho}{r}\right)^{\delta}\mu(B(\x,r)) \hspace{7mm} \forall \ \x \in \spt(\mu)
\qquad  \forall \ r < r_0 \ .
\end{equation}
The measure  $\mu$ is said to be  \emph{absolutely friendly} if it is doubling and absolutely decaying.

Now let  $ E\subseteq\R^d$ be any hypercube. It is clear that the  boundary $\partial E$ of the hypercube is contained in a union of $2d$  hyperplanes. It is  easily verified that this trivial observation  together with the definition of absolutely decaying implies
the existence of  constants  $C>0$ and $\delta>0$ such that
$$
\mu\big((\partial E)_{\varrho}\big)  \ \leq  \ C  \, \varrho^{\delta}   \qquad \forall \ \varrho>0  . $$
    In turn, this together with   Theorem~\ref{Main}  implies the following exponentially mixing statement.

\begin{proposition}
    \label{expmixforcube}
    Let $d\geq1$ be an integer.  Let $(\Phi,K,\mu,T)$ be a self-conformal system on $\R^d$. If $\mu$ is absolutely decaying, then $\mu$ is exponentially mixing with respect to $(T,\cC)$, where $\cC$ is a collection of hypercubes in $\R^d$.
\end{proposition}

The upshot is that in the setting of self-conformal systems,  if we additionally assume that the measure $\mu$ is absolutely decaying, then proving exponential mixing with respect to hypercubes becomes essentially trivial.
This simplification, however, does not seem to carry over to the case of balls, where the argument remains substantially more intricate. Nevertheless, as we now demonstrate, the absolute decay condition still enables a significant simplification of the analysis required to establish the full general version of the result, namely Theorem~\ref{Main2}.   With this in mind,  recall that Theorem~\ref{thmann} plays the key role in the proof of Theorem~\ref{Main2}.
The first part of Theorem~\ref{thmann} is relatively straightforward to establish (see Section~\ref{onedim}), whereas the second part contains the main technical difficulty. What follows shows that, assuming the “easy” part (i) of Theorem~\ref{thmann}, the absolute decay assumption enables a comparatively direct proof of the more challenging part (ii).
Moreover, with reference to \eqref{thmmeaann} we are able to establish the desired inequality for all $\x \in \R^d$ (not just $K$) and all $r > 0$ (not just $r  \le r_0$).  As a result,  the analogue of Proposition~\ref{expmixforcube} for balls follows directly from  the following statement, together with  part (i) of Theorem~\ref{thmann} and  Theorem~\ref{Main}.

\begin{proposition}\label{ADimplyMAB}
Let  $\mu$ be a probability measure supported on a bounded subset  $X$ of $\R^d$.   Suppose that $\mu$ is absolutely decaying and there exist $C'>0$, $\delta'>0$ such that
     \begin{equation}\label{mubxrlecrd}
         \mu\big(B(\x,r)\big)\leq C'r^{\delta'},\qquad \forall \ \x\in\R^d,\ \ \forall \ r>0.
     \end{equation}
     Then there exist constants $C>0$ and $\delta>0$ so that for any ball $B\subseteq\R^d$, we have
     \begin{equation}\label{meaofannballnei}
         \mu\big((\partial B)_{\varrho}\big)\leq C\varrho^{\delta},\qquad\forall\ \varrho>0.
     \end{equation}
\end{proposition}

\begin{remark}
Observe that the inequality corresponding to \eqref{mubxrlecrd} is necessary, as it follows from \eqref{meaofannballnei}.  To see this, let $\x\in\R^d$, $r>0$ and $\varrho=r$, then $(\partial B(\x,r))_{\varrho}=B(\x,2r)$. By \eqref{meaofannballnei}, we have
\[
\mu(B(\x,r))\leq\mu(B(\x,2r))=\mu((\partial B(\x,r))_{r})\leq C r^{\delta}.
\]
\end{remark}

\medskip

\begin{proof}[Proof of Proposition~\ref{ADimplyMAB}]
     Throughout, write $K=\spt(\mu)$. Since $\mu$ is absolutely decaying, there exist constants  $C>0$, $r_0>0$ and  $\delta_1>0$ such that \eqref{absoludecay} holds for any hyperplane $H$ with $\delta$ `equal' to $\delta_1$.  For any $\x\in\R^d$ and $r>0$, let $\overline{B}(\x,r)$ denote the associated  closed ball. Since $\overline{B}(\x,r)$ is the intersection of the open balls $\{B(\x,r+1/n)\}_{n\in\N}$, it follows via  \eqref{absoludecay}  that
     \begin{equation}\label{meaofhycab}
         \mu\big((H)_{\varrho}\cap \overline{B}(\x,r)\big)\leq C\left(\frac{\varrho}{r}\right)^{\delta_1}\mu\big(\overline{B}(\x,r)\big)
     \end{equation}
     for any hyperplane $H$, $\x\in K$, $0<r<r_0$ and $\varrho>0$.

     Now for any fixed $\y\in\R^d$, $R>0$ and $\varrho>0$, we  estimate the $\mu$-measure of annulus  $(\partial B(\y,R))_{\varrho}$ by considering two cases.
\medskip

       \noindent \emph{$\bullet$ Case 1: Estimating $\mu\big((\partial B(\y,R))_{\varrho}\big)$ when $\varrho\geq\min\big\{\frac{R^4}{16},r_0^2\big\}$}. In this case,  by \eqref{mubxrlecrd} we have that
         \[
             \mu\big((\partial B(\y,R))_{\varrho}\big)\leq\mu\big(B(\y,R+\varrho)\big)\leq \min\{1, C'(R+\varrho)^{\delta'}\}\leq 3^{\delta'}\cdot\max\{C',1\}\cdot\varrho^{\delta'/4}.\medskip
         \]

\noindent\emph{$\bullet$   Case 2: Estimating $\mu\big((\partial B(\y,R))_{\varrho}\big)$ when $\varrho<\min\big\{\frac{R^4}{16},r_0^2\big\}$}. In this case,  by  Lemma~\ref{lem3.1}  (a generic geometric statement concerning the tangent plane of a geometric sphere), we obtain that for any $\x\in(\partial B(\y,R))_{\varrho}$
         \begin{equation}\label{parbcapclbsuhy}
             (\partial B(\y,R))_{\varrho}\cap\overline{B}(\x,\sqrt{\varrho})\subseteq\big(\x+T_{\x}(\partial B(\y,|\x-\y|))\big)_{3\varrho^{3/4}}.
         \end{equation}
         Note that the  collection $$\left\{\overline{B}(\x,\sqrt{\varrho}):\x\in K\cap(\partial B(\y,R))_{\varrho}\right\}$$ of closed balls  is a  cover of the set $K\cap(\partial B(\y,R))_{\varrho}$. On applying the Besicovitch covering theorem \cite[Theorem 2.7]{mattila1999}   to this cover, there exists an integer $N\geq1$ that only depends on $d$ and $\cX_1,\cX_2,...,\cX_N\subseteq K\cap(\partial B(\y,R))_{\varrho}$ such that
         \begin{itemize}
             \item[(i)] for any $1\leq i\leq N$ and any $\x\neq\x'\in\cX_i$,
             \[
             \overline{B}(\x,\sqrt{\varrho})\cap\overline{B}(\x',\sqrt{\varrho})=\emptyset.
             \]
             \item[(ii)] the  collection
             \[
                 \left\{\overline{B}(\x,\sqrt{\varrho}):1\leq i\leq N,\,\x\in\cX_i\right\}
             \]
             of closed balls is a cover of the set $K\cap(\partial B(\y,R))_{\varrho}$.
             \end{itemize}
            On combining   properties (i) and (ii) with \eqref{meaofhycab} and \eqref{parbcapclbsuhy}, it follows that
             \begin{eqnarray*}
                 \mu\big((\partial B(\y,R))_{\varrho} \big)&\leq&\sum_{i=1}^N\sum_{\x\in\cX_i}\mu\big(\partial B(\y,R))_{\varrho}\cap\overline{B}(\x,\sqrt{\varrho})\big)\\
                 &\leq&\sum_{i=1}^N\sum_{\x\in\cX_i}\mu\Big(\big(\x+T_{\x}(\partial B(\y,|\x-\y|))\big)_{3\varrho^{3/4}}\cap\overline{B}(\x,\sqrt{\varrho})\Big)\\
                 &\leq&C\sum_{i=1}^N\sum_{\x\in\cX_i}\left(\frac{3\varrho^{3/4}}{\sqrt{\varrho}}\right)^{\delta_1}\mu\big(\overline{B}(\x,\sqrt{\varrho})\big)\\
                 &=&3^{\delta_1}C\varrho^{\delta_1/4}\sum_{i=1}^{N}\mu\left(\bigsqcup_{\x\in\cX_i}\overline{B}(\x,\sqrt{\varrho})\right)\\
                 &\leq&3^{\delta_1}CN\varrho^{\delta_1/4}.\medskip
             \end{eqnarray*}

             \noindent The upshot of the above is  that the inequality \eqref{meaofannballnei} holds for any ball $B\subseteq\R^d$ and $\varrho>0$ with $\delta=\min\{\delta_1,\delta'\}/4$.
\end{proof}

\noindent For the sake of completeness we formally state what has already been mentioned informally above: namely, that Proposition~\ref{ADimplyMAB}, in conjunction  with part (i) of Theorem~\ref{Main} and Theorem~\ref{thmann}, yields  the following analogue of Proposition~\ref{expmixforcube} for balls.

\begin{proposition}\label{absoluexp}
Let $d\geq1$ be an integer.  Let $(\Phi,K,\mu,T)$ be a self-conformal system on $\R^d$. If $\mu$ is absolutely decaying, then $\mu$ is exponentially mixing with respect to $(T,\cC)$, where $\cC$ is a collection of balls in $\R^d$.
    \end{proposition}

We emphasize  that Theorem~\ref{Main2} shows that the proposition is in fact true for any measure $\mu$ associated with a self-conformal system on $\R^d$, without requiring  the assumption that it is absolutely decaying.  In the remainder of this section, we investigate sufficient conditions under which $\mu$ is absolutely decaying, and we see that not all such measures satisfy this property. With this in mind,   the following statement for self-similar systems satisfying the open set condition  was established in  \cite[Theorem 2.3]{kleinbock2005}.

\begin{theorem}\label{thmklw}
 Let $(\Phi,K,\mu,T)$ be a self-conformal system on $\R^d$ where $\Phi$ is an irreducible self-similar IFS   and  $\mu$ is the restricted Hausdorff measure  $\frac{\cH^s|_K}{\cH^s(K)}$ ($s=\dimH K$).  Then $\mu$ is absolutely decaying. In particular, $\mu$ is absolutely friendly.
\end{theorem}

\noindent Here, and hereafter, we say  that a conformal IFS $\Phi=\{\varphi_j\}_{1\leq j\leq m}$ on $\R^d$ is \emph{irreducible} if the associated self-conformal set $K\subseteq\R^d$ is not contained in any $(d-1)$-dimensional $C^1$ submanifold.    This  notion  of irreducibility is a simpler form of that appearing as Definition 1.4 in \cite{urbanski2005}.  The fact that they  are equivalent is a consequence of Theorem~\ref{thmrigidity}.  We also note that the theorem implies that  $\Phi$ is irreducible if and only if $\ell_K=d$, where $\ell_K$ is given by  \eqref{defofellk}.   In the case the IFS is self similar,  the notion of irreducible is equivalent  to the familiar notion that the associated self-similar set $K\subseteq\R^d$ is not contained in a $(d-1)$--dimensional hyperplane of $\R^d$.  As is mentioned in \cite{das2021}  -- just before the statement of their Theorem 1.5, the statement that the self-similar set $K$ is not contained in a finite collection of $(d-1)$--dimensional hyperplanes of
$\R^d$ is  equivalent to the more natural and formally weaker assumption
that $K$ is not contained in any $(d-1)$--dimensional hyperplane.  Regardless of which notion of irreducible one works with, it is reasonably straightforward to see that it is a necessary assumption in the above statement (see  Example~B.1 below for an obvious example).

For general self-conformal systems, we have the following result, which extends Theorem~\ref{thmklw}. It originally appeared in \cite[Theorem 1.5]{urbanski2005}, though with a minor oversight that was later corrected in \cite[Theorem 1.10]{das2021}.

\begin{theorem} \label{selfcondou}
    Let $(\Phi,K,\mu,T)$ be a self-conformal system. If $\Phi$ is irreducible and $\mu$ is doubling, then $\mu$ is absolutely decaying.  In particular, $\mu$ is absolutely friendly.
\end{theorem}

\medskip

\begin{remark}\label{rkofdoubling}
    The doubling property plays an important role in the proof of Theorem~\ref{selfcondou}, and Example B.2  below highlights that it is a necessary assumption.  For a self-conformal system $(\Phi, K, \mu, T)$, we outline two situations in which the measure $\mu$ is guaranteed to be doubling:
\begin{itemize}
    \item[(i)]  If $\Phi$ satisfies the strong separation property, then all Gibbs measures on $K$ are doubling (see \cite[Proposition 4.8]{anttila2023}).

    \medskip

    \item[(ii)] If $\mu$ is equivalent to $\cH^{s}|_K$   ($s=\dimH K$), then $\mu$ is $s$-Ahlfors regular  (see
    Remark~\ref{newrem} in Section~\ref{appintro}), and thus doubling.
\end{itemize}
\end{remark}


\noindent With Theorem~\ref{selfcondou} and Remark~\ref{rkofdoubling} in hand, we have well-established criteria for identifying self-conformal systems $(\Phi, K, \mu, T)$ for which the measure $\mu$ is absolutely decaying. However, as we demonstrate below, it is not difficult to construct simple examples where the conditions of Theorem~\ref{selfcondou} fail to hold, and moreover, the measure $\mu$ is not absolutely decaying. In such cases, we cannot directly apply Proposition~\ref{expmixforcube} or Proposition~\ref{absoluexp} to deduce that $\mu$ exhibits exponential mixing with respect to hypercubes or balls. Nonetheless, since Theorem~\ref{Main2} does not rely on any additional assumptions, it ensures that the desired dynamical properties still hold even when the measure is not friendly.

We now present the examples referenced above. 

\noindent  \noindent $\bullet$ \textbf{Example B.1.}  \  Here is  a concrete two-dimensional example of a self-conformal IFS that is clearly  not irreducible. Let $\Phi=\{\varphi_1,\varphi_2\}$  be a self-similar IFS on $[0,1]^2$ consisting of the maps
 \[
 \textstyle{\varphi_1(\x)=\frac{1}{3}\x,\quad\varphi_2(\x)=\frac{1}{3}\x+\Big(\frac{2}{3},0\Big),\quad\forall\ \x\in  [0,1]^2.}
  \]
 Then the associated self-similar set $K$ is the middle-third Cantor set embedded in the line $\R\times\{0\}$, and hence $\Phi$ is not irreducible.  It is easily checked that  if we take $H  = \R\times\{0\}$, then the inequality appearing in \eqref{absoludecay} reads $ 1 \le C (\rho/r)^{\delta}$  for any measure $\mu$ supported on $K$.  In turn, it follows that given any $r<1$ we can choose $ \rho$ sufficiently small so that the inequality fails.  Hence,  $\mu$ cannot be absolutely decaying.  For a more sophisticated example  see Example \ref{counterexam}. More  generally, for any self-conformal system $(\Phi, K, \mu, T)$,  if $\Phi$ is not irreducible, then by definition, the associated self-conformal set $K$ is contained in some $(d-1)$-dimensional $C^1$ submanifold $M\subseteq\R^d$.  Consequently, the measure $\mu$ is supported on $M$. Since $M$ can be locally approximated by a hyperplane, one can adapt the argument  used in the proof of Proposition~\ref{ADimplyMAB} to show that $\mu(M) = 0$ under the assumption  that $\mu$ is absolutely decaying.  This  contradicts  the fact that $\mu$ is supported on $M$.  Therefore,  $\mu$ cannot be absolutely decaying if $\Phi$ is not irreducible.

\medskip

As mentioned  in Remark~\ref{rkofdoubling} , if $\Phi$ satisfies  the  strong separation condition, then $\mu$ is doubling.  However, if $\Phi$ does not satisfy the strong separation condition, the measure  $\mu$ may fail to be doubling, as we now demonstrate.

\noindent  \noindent $\bullet$ \textbf{Example B.2.}  \ Let $K\subseteq\R^2$ be the Sierpi\'{n}ski triangle generated by the self-similar IFS $\Phi=\{\varphi_1,\varphi_2,\varphi_3\}$ where $\varphi_i(\x)=(\x+\mathbf{q}_i)/2$ ($i=1,2,3$) and $\mathbf{q}_1,\mathbf{q}_2,\mathbf{q}_3\subseteq\R^2$ are the  vertices of an equilateral triangle, and let $\mu$ be the self-similar measure defined by
     \[
         \mu:=\sum_{i=1}^3 p_i\,\mu\circ\varphi_i^{-1},\qquad\text{where  \ \ \ $p_i\geq0$ $(i=1,2,3)$ \  and \  $\sum_{i=1}^3 p_i=1$}.
     \]
     It is well-known that $\Phi$ does not satisfies the strong separation condition. Regarding doubling,  it is shown in \cite[Proposition~1.3 (b)]{yung2007} that $\mu$ is not doubling if $(p_1,p_2,p_3)\neq(1/3,1/3,1/3)$.
     The upshot is that, in this case, Theorem~\ref{selfcondou} cannot be used to coclude that that $\mu$ is absolutely decaying.   To reiterate, this means that Proposition~\ref{absoluexp} cannot be applied to deduce that $\mu$ is exponential mixing with respect to balls. 
     We now go on to show that $\mu$ is in fact  not absolutely decaying,   which, in particular, confirms that the doubling assumption in Theorem~\ref{selfcondou} is necessary.
     With this in mind, without loss of generality, we assume that the diameter of $K$ equals one. If $\#\{i=1,2,3:p_i\neq0\}=1$, then $\mu$ is supported at a single point and thus not absolutely decaying. For this reason, we  assume that $\#\{i=1,2,3:p_i\neq0\}\geq2$ in the following. Let $p_{\max}=\max\{p_1,p_2,p_3\}$ and $p_{\min}=\min\{p_1,p_2,p_3\}$. Without loss of generality, we assume that $p_1=p_{\min}$ and $p_2=p_{\max}$. Since $(p_1,p_2,p_3)\neq(1/3,1/3,1/3)$, we have $p_{\min}< p_{\max}$.  Given   $n\in\N$,
         let $$m=m(n):=\linte{\frac{n\log1/p_{\min}}{\log1/p_{\max}}}     \qquad {\rm and } \qquad  M=M(n):=\frac{\log1/p_{\min}}{\log1/p_{\max}}-\frac{1}{n}   \, .$$ Then, by definition, it follows that
         \begin{equation}   \label{SVB}
         m>M\cdot n \qquad {\rm and } \qquad p_{\min}^n\asymp p_{\max}^m \, .
         \end{equation}
          Also observe that since  $p_{\min}<p_{\max}$, we have   $M(n)>1$ for  sufficiently large $n$. Now let $\pi:\{1,2,3\}^{\N}\to K$ be the coding map  associated  to $\Phi$. In turn,  let $\x=\x(n):=\pi(21^n2^{\infty})$ and $\y:=\pi(12^{\infty})$, where $21^n2^{\infty}$ denotes the word
         \[2\underbrace{1\cdot\cdot\cdot 1}_{n}222 \ldots \]
         and $12^{\infty}:= 1222 \ldots $. Observe that for any large $n\in\N$, we have that
        \begin{eqnarray}\label{Bcenxcapk}
          \textstyle{   B\Big(\x,\frac{1}{2^{n+1}}+\frac{1}{2^m}\Big)    \cap  K \ \subseteq  \ K_{21^{n-1}}\cup K_{12^{m-3}}\cup K_{21^{n-2}21^{m-2-n}}  }
        \end{eqnarray}
        and
\begin{eqnarray}\label{Bcenycapk}
           \textstyle{ B\Big(\x,\frac{1}{2^{n+1}}+\frac{1}{2^m}\Big)\cap B\Big(\y,\frac{1}{2^m}\Big)  \ \supseteq \  K_{12^{m-1}} \, ; }
        \end{eqnarray}
        see Figure \ref{figureST} for a geometric insight of
        \eqref{Bcenxcapk} and \eqref{Bcenycapk}.
        \begin{figure}[h]
    \centering
\includegraphics[scale=0.4]{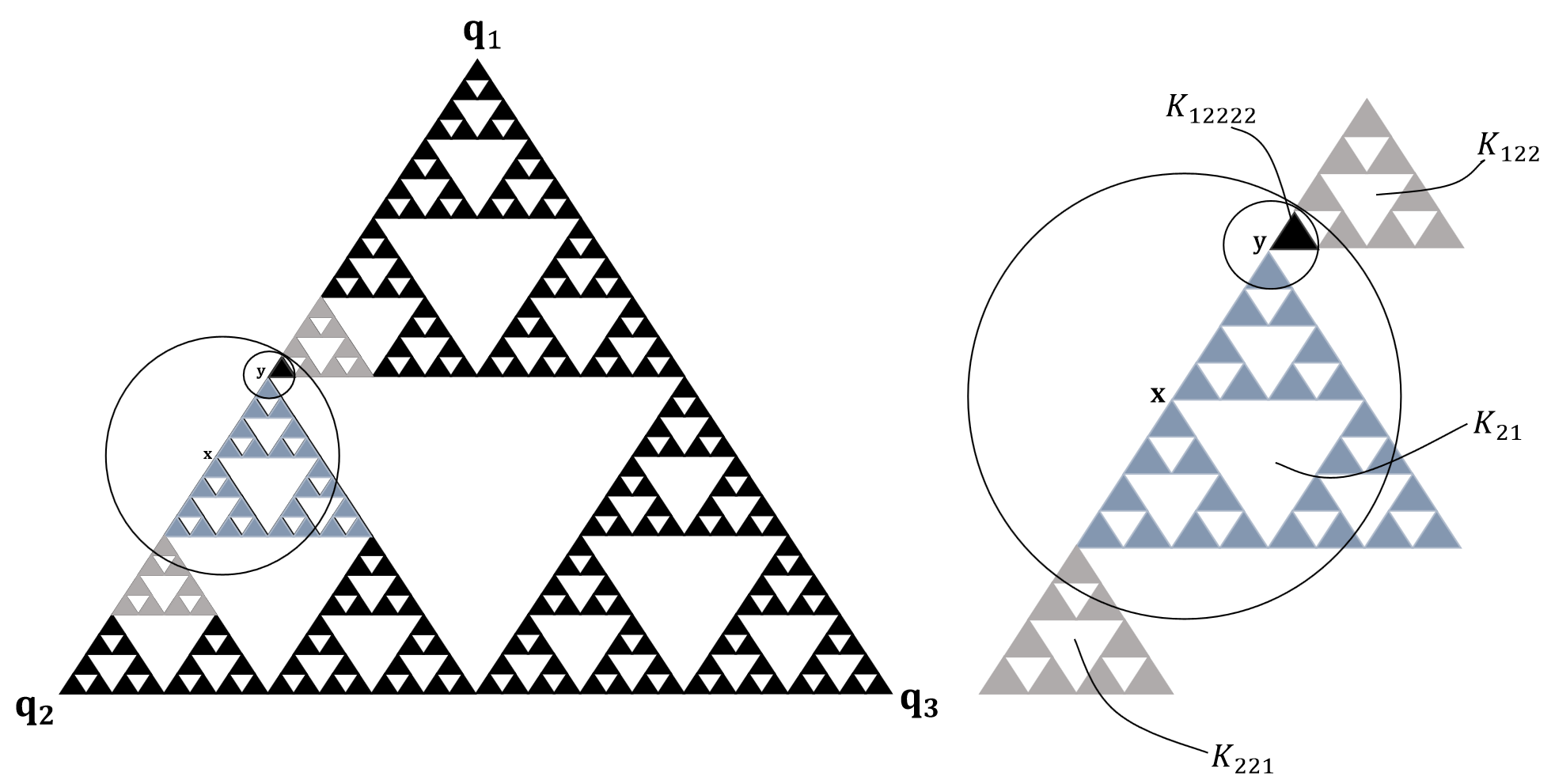}
\caption{Let $n=2$ and $m=5$.  This figure illustrates that $B(\x,1/2^3+1/2^5)\cap K\subseteq K_{21}\cup K_{122}\cup K_{221}$ and $B(\x,1/2^3+1/2^5)\cap B(\y,1/2^5)\supseteq K_{12222}$, where $\x=\pi(2112^{\infty})$ and $\y:=\pi(12^{\infty})$.   This coincides with \eqref{Bcenxcapk} and \eqref{Bcenycapk} in the case $n=2$ and $m=5$.}
\label{figureST}
\end{figure}
Let $H$ be any line through the point $\y$ and note  that $(H)_{1/2^m}\supseteq B(\y,1/2^m)$. We are now in a position to show that $\mu$ is not absolutely decaying. Let $r=r(n):=1/2^{n+1}+1/2^m$ and $\rho=\rho(n):=1/2^m$.  On combining \eqref{SVB}, \eqref{Bcenxcapk} and  \eqref{Bcenycapk}, we find that for any $\gamma>0$,
\begin{eqnarray*}
\frac{\mu\left(B(\x,r\big)\cap (H)_{\rho}\right)}{\left(\frac{\rho}{r}\right)^{\gamma}\mu\left(B(\x,r)\right)}&\geq&\frac{\mu\big(K_{12^{m-1}}\big)}{\left(\frac{\rho}{r}\right)^{\gamma}\big(\mu(K_{21^{n-1}})+\mu\big( K_{12^{m-2}}\big)+\mu\big( K_{21^{n-1}21^{m-2-n}})\big)}\\[2ex]
&\geq&\frac{(1+2^{m-n-1})^{\gamma}\cdot p_{\min}p_{\max}^{m-1}}{p_{\max}p_{\min}^{n-1}+2p_{\max}^{m-2}}\\[2ex]
&\geq& 2^{((M-1)n-1)\cdot\gamma}\cdot\frac{1}{\frac{p_{\min}^{n-2}}{p_{\max}^{m-1}}+\frac{2}{p_{\min}p_{\max}}}\\[2ex]
&\asymp&2^{(M-1)n\gamma}  \ \to \ \infty\quad\text{as  \ \ $n  \to \infty$  \, ,
}
\end{eqnarray*}
 since  $M-1>0$ when $n$ is sufficiently large. Thus,  $\mu$ is not absolutely decaying as claimed.

\medskip

With Example~B.2 in mind, we note that it is easier to construct an example of a self-conformal system on $\R$ that is irreducible but not absolutely decaying. However, we have chosen to present an example in $\R^2$ to remain consistent with the higher-dimensional focus of this paper.

\section{Direct proof of Theorem~\ref{Main}} \label{appendix:directproof}

 \begin{proof}[Direct proof of Theorem~\ref{Main}]
Assume the hypothesis in Theorem~\ref{Main}.     Let $\mu$ be a  Gibbs measure on $K$,  and $\cC$ be a collection of  $\mu$-measurable sets in $\R^d$ that satisfy the inequality (\ref{collection}).
Working on the associated symbolic space (see Section~\ref{cifs}), with $\kappa$ defined  as in (\ref{defkap}) we have
via  (\ref{p3'}) that $$|K_I|\leq C_3\,\kappa^{|I|}  \, , \qquad \forall\, I\in\Sigma^*.$$  Throughout, fix $n\in\N$ and $E\in\cC$. Let $q=q(n)\in\N$ that will be determined later and let $\cI,\cJ\subseteq\Sigma^q$ be given by
     \begin{eqnarray*}
         \cI&:=&\left\{I\in\Sigma^q:K_I\subseteq E\right\},\\[4pt]
         \cJ&:=&\left\{J\in\Sigma^q:K_J\cap E\neq\emptyset,~K_J\cap E^c\neq\emptyset\right\}.
     \end{eqnarray*}
Then it is easily verified that
     \begin{eqnarray*}
         K_J  \ \subseteq  \ (\partial E)_{2\,|K_J|} \ \subseteq  \ (\partial E)_{2\,C_3\,\kappa^q}  \, ,   \qquad  \forall\, J\in\cJ
     \end{eqnarray*}
     and hence
     \begin{eqnarray}\label{relation1}
         \bigcup_{I\in\cI}K_I~~\subseteq~~ E\cap K~~\subseteq ~~\left(\bigcup_{I\in\cI}K_I\right)\cup(\partial E)_{2\,C_3\,\kappa^q}  \, .
     \end{eqnarray}

Now let $\gamma\in(0,1)$ be  as  in Corollary~\ref{emprd}, and let $\delta>0$ be as in inequality (\ref{collection}). In the following, we do not distinguish between the constants $C>0$ appearing in the inequality (\ref{collection})  and the Corollary~\ref{emprd}. Let $F\subseteq\R^d$ be a $\mu$-measurable set. On combining  (\ref{collection}),  part (iii) of Theorem~\ref{ruellethmrd}, Corollary \ref{emprd} and (\ref{relation1}), we obtain that
     \begin{eqnarray}
         \mu(E\cap T^{-n}F)&\leq&\sum_{I\in\cI\cup\cJ}\mu(K_I\cap T^{-n}F)\nonumber\\[4pt]
&\leq&\Big(\sum_{I\in\cI\cup\cJ}\mu(K_I)+Cm^q\gamma^n\Big)\,\mu(F)\label{uppbd}\\[4pt]
&\leq&\left(\mu(E)+\mu\big((\partial E)_{2\,C_3\,\kappa^q}\big)+Cm^q\gamma^n\right)\mu(F)\label{uppbd2}\\[4pt]
&\leq&\left(\mu(E)+ C\cdot(2C_3)^{\delta}\kappa^{\delta\, q}+Cm^q\gamma^n\right)\mu(F)\nonumber
     \end{eqnarray}
and that
\begin{eqnarray}
    \mu(E\cap T^{-n}F)&\geq&\sum_{I\in\cI}\mu(K_I\cap T^{-n}F)\nonumber\\[4pt]
    &\geq&\Big(\sum_{I\in\cI}\mu(K_I)-Cm^q\gamma^n\Big)\mu(F)\label{lowbd}\\[4pt]
    &\geq&\left(\mu(E)-\mu\big((\partial E)_{2\,C_3\,\kappa^q}\big)-Cm^q\gamma^n\right)\mu(F)\label{lowbd2}\\[4pt]
    &\geq&\left(\mu(E)-C\cdot(2C_3)^{\delta}\kappa^{\delta\, q}-Cm^q\gamma^n\right)\mu(F).\nonumber
\end{eqnarray}
The upshot is  that
\begin{eqnarray}\label{empes}
    \left|\mu(E\cap T^{-n}F)-\mu(E)\mu(F)\right|\leq C\left((2C_3)^{\delta}\kappa^{\delta\, q}+m^q\gamma^n\right)\mu(F).
\end{eqnarray}
We now set  $$q=q(n):=\linte{\frac{-n\log \gamma}{2\log m}},$$ where $\linte{x}$ denote the largest integer that is not greater than $x$.  Then, it follows that   $$m^q\gamma^n\leq \gamma^{n/2}\ \ \ \ {\rm and  \ } \ \ \ \ \kappa^{\delta\,q}\leq\kappa^{\delta\cdot(\frac{-n\log\gamma}{2\log m}-1)}  \, , $$ 
and this together with (\ref{empes}) implies that
\begin{eqnarray*}
    \left|\mu(E\cap T^{-n}F)-\mu(E)\mu(F)\right|\leq \widetilde{C}\,\widetilde{\gamma}^n\mu(F),
\end{eqnarray*}
where we set
\[
\widetilde{C}:=2\,C\cdot\max\{(2C_3\,\kappa^{-1})^{\delta},1\} > 0 \, , \qquad \widetilde{\gamma}:=\max\big\{\kappa^{\frac{-\delta\log \gamma}{2\log m}},\gamma^{1/2}\big\}\in(0,1).
\]
This completes the direct proof of Theorem~\ref{Main}.
 \end{proof}

\section{Proof of Lemma~\ref{genharmanlem}}   \label{HJ}

For convenience we restate the lemma under consideration.

\begin{lemma}\label{genharmanlemA}
    Let $(X,\cB,\mu)$ be a probability space. let $\{f_n(x)\}_{n\in\N}$ and $\{g_n(x)\}_{n\in\N}$ be  sequences of measurable functions on $X$, and  let $\{\phi_n\}_{n\in\N}\subseteq\R$ be a sequence of real numbers. Suppose that
    \begin{eqnarray}\label{nonnegA}
        0\leq g_n(x)\leq \phi_n,\,\,\,\,\,\,\,\,\forall\,n\in\N,\,\,\,\,\forall\,x\in X
    \end{eqnarray}
    and that there exists $C>0$ with which
    \begin{eqnarray}\label{keycontA}
        \int_{X}\left(\sum_{n=a}^b(f_n(x)-g_n(x))\right)^2\,\td\mu(x)\leq C\sum_{n=a}^b\phi_n
    \end{eqnarray}
    for any pair of integers $0<a<b$. Then for any $\epsilon>0$, we have
    \begin{eqnarray}\label{quantestimA}
        \sum_{n=1}^N f_n(x)=\sum_{n=1}^N g_n(x)+O\left(\Psi(N)^{\frac{1}{2}}(\log(\Psi(N)))^{\frac{3}{2}+\epsilon}+\max_{1\leq k\leq N}g_k(x)\right)
    \end{eqnarray}
    for $\mu$-almost every $x\in X$, where $\displaystyle\Psi(N):=\sum_{n=1}^N\phi_n$.
\end{lemma}
\begin{proof}
    The proof is divided into two cases.

    \emph{Case (i). $\sup\{\Psi(N):N\in\N\}<+\infty$}. By the inequalities (\ref{nonnegA}) and (\ref{keycontA}) and Fatou's lemma, we have
    \begin{eqnarray*}
        \Big\|\sum_{n=1}^{\infty}f_n\Big\|_{L^2(\mu)}&\leq&\Big\|\sum_{n=1}^{\infty}(f_n-g_n)\Big\|_{L^2(\mu)}+\,\,\Big\|\sum_{n=1}^{\infty}g_n\Big\|_{L^2(\mu)}\\
        &\leq&\liminf_{N\to\infty}\Big\|\sum_{n=1}^{N}(f_n-g_n)\Big\|_{L^2(\mu)}+\,\,\sum_{n=1}^{\infty}\phi_n\\
        &\leq& \Big(C\cdot\sup_{N\in\N}\Psi(N)\Big)^{1/2}+\,\,\sup_{N\in\N}\Psi(N)\\[5pt]
        &<&+\infty.
    \end{eqnarray*}
    It implies that the sum of the sequence $\{f_n(x)\}_{n\in\N}$  converges for $\mu$-almost every $x\in X$. Then the asymptotic formula (\ref{quantestimA}) holds trivially.

    \emph{Case (ii). $\sup\{\Psi(N):N\in\N\}=+\infty$.}  We start by introducing some useful notation. For each $r\in\N$, define a collection of dyadic intervals with integer endpoints as
    \[
    L_r:=\left\{(t\cdot 2^s,(t+1)\cdot 2^s]:s=0,1,...,r\,\,\text{and}\,\,t=0,1,...,2^{r-s}-1\right\}.
    \]
    Given an interval $I\subseteq[0,+\infty)$ and $x\in X$, denote
    \[
    F(I,x):=\sum_{k:\Psi(k)\in I}(f_k(x)-g_k(x))\,.
    \]
    For any $r\in\N$ and $x\in X$, define
    \[
    G(r,x):=\sum_{I\in L_r}|F(I,x)|^2\,.
    \]
    Given any $j\in\N$, let \[     n_j:=\max\left\{k\in\N:\Psi(k)\leq j\right\} \]
    and consider the binary expansion
    \begin{eqnarray}\label{binary}
        j=\sum_{s=0}^{\linte{\log_2j}}b(j,s)\cdot2^s,
    \end{eqnarray}
    where $b(j,s)=0$ or $1$ for each $s=0,1,...,\linte{\log_2j}$. Denote by
    \[
    B(j):=\left\{s\in\Z_{\geq0}\cap[0,\linte{\log_2j}]:b(j,s)=1\right\}=\left\{s_{1,j}<s_{2,j}<\cdot\cdot\cdot<s_{k_j,j}\right\}.
    \]
    By the equality (\ref{binary}), we obtain the following partition of the interval $(0,j]$:
    \begin{eqnarray}\label{intpartition}
        (0,j]=(0,2^{s_{k_j,j}}]\sqcup\left(\bigsqcup_{\ell=1}^{k_j-1}\Big(\sum_{i=\ell+1}^{k_j}2^{s_{i,j}},\sum_{i=\ell}^{k_j}2^{s_{i,j}}\Big]\right).
    \end{eqnarray}
    If we let
    \[
    \cI(j):=\left\{(t_{i,j}\cdot 2^{s_{i,j}},(t_{i,j}+1)\cdot 2^{s_{i,j}}]:i=1,2,...,k_j-1\right\}\cup\big\{(0,2^{k_j,j}]\big\},
    \]
    where $t_{i,j}$ $(i=1,2,...,k_j-1)$ is defined by
    \[
    t_{i,j}:=\sum_{\ell=i+1}^{k_j}2^{s_{\ell,j}-s_{i,j}},
    \]
    then the partition (\ref{intpartition}) can be rewritten as
    \begin{eqnarray}\label{partitionof0j}
        (0,j]=\bigsqcup_{I\in \cI(j)}I.
    \end{eqnarray}
    For convenience, let $\Psi(0):=0$. Throughout, we fix  $\epsilon>0$.

     In view of  inequality (\ref{keycontA}), it follows that
     \begin{eqnarray*}
         \int_XG(r,x)\,\td\mu(x)&=&\sum_{I\in L_r}\int_X|F(I,x)|^2\,\td\mu(x)\\[4pt]
         &\leq&C\cdot\sum_{I\in L_r}\sum_{k:\Psi(k)\in I}\phi_k\\[4pt]
         &=&C\cdot\sum_{s=0}^r\sum_{t=0}^{2^{r-s}-1}\Big(\Psi\big(n_{(t+1)\cdot 2^s}\big)-\Psi\big(n_{t\cdot 2^s}\big)\Big)\\[4pt]
         &=&C\cdot(r+1)\cdot\Psi\big(n_{2^r}\big)\\[4pt]
         &\leq&C\cdot(r+1)\cdot2^r
     \end{eqnarray*}
     for any $r\in\N$. This together with Markov's inequality, implies that for any $r\in\N$,
     \[
     \mu\Big(\{x\in X:G(r,x)>2^r\cdot r^{2+\epsilon}\}\Big)\, \ll \,  \frac{1}{r^{1+2\epsilon}}  ,
     \]
      where the implied constant is independent of $r\in\N$.  Now  $
     \sum_{r=1}^{\infty}\frac{1}{r^{1+2\epsilon}}<+\infty $,
      so by the (convergent)  Borel-Cantelli Lemma, it follows that  for $\mu$-almost every $x\in X$, there exists $N=N(x)\in\N$ such that for all $r>N(x)$, we have
     \begin{eqnarray}\label{uppofGrx}
         G(r,x)\,\leq\,2^r\cdot r^{2+2\epsilon}.
     \end{eqnarray}

     Given any $j\in\N$ and $x\in X$, we have
     \begin{eqnarray}
         \left|\sum_{n=1}^{n_j}(f_n(x)-g_n(x))\right|&\leq&\sum_{I\in\cI(j)}\left|\sum_{k:\Psi(k)\in I}(f_k(x)-g_k(x))\right|\label{7.24}\\[4pt]
         &=&\sum_{I\in\cI(j)}|F(I,x)|\nonumber\\
&\leq&\Big(\#\cI(j)\Big)^{1/2}\cdot\left(\sum_{I\in\cI(j)}|F(I,x)|^2\right)^{1/2}\label{7.25}\\[4pt]
         &\leq&(\log_2j)^{1/2}\cdot\left(\sum_{I\in\cI(j)}|F(I,x)|^2\right)^{1/2}\label{7.26},
     \end{eqnarray}
     where the inequality (\ref{7.24}) is a consequence of the definition of $n_j$ and the partition (\ref{partitionof0j}),  the inequality (\ref{7.25}) is a consequence of the Cauchy-Schwarz inequality, and (\ref{7.26}) holds since  $\#\cI(j)\leq\linte{\log_2j}$. Note that
     \[
     \cI(j)\subseteq L_{\linte{\log_2j}+1}
     \]
     and so
     \begin{eqnarray}
         \left(\sum_{I\in\cI(j)}|F(I,x)|^2\right)^{1/2}&\leq&\left(\sum_{I\in L_{\linte{\log_2j}+1}}|F(I,x)|^2\right)^{1/2}\nonumber\\[4pt]
         &=&G\big(\linte{\log_2j}+1,x\big)^{1/2}\label{uppofF2}.
     \end{eqnarray}
     Combining (\ref{uppofGrx}), (\ref{7.26}) and (\ref{uppofF2}), we obtain that
     \begin{eqnarray}\label{orderoffn-gn}
         \left|\sum_{n=1}^{n_j}(f_n(x)-g_n(x))\right|=O\left(j^{\frac{1}{2}}(\log j)^{\frac{3}{2}+\epsilon}\right)
     \end{eqnarray}
     for $\mu$-almost every $x\in X$. Since $\sup\{\Psi(N):N\in\N\}=+\infty$, we have that $\Phi(n_j)>0$ when $j\in\N$ is sufficiently large. For such $j\in\N$, there exists a positive integer $r=r(j)\in\N$ such that
     \[
     r-1<\Psi(n_j)\leq r.
     \]
     Then by the definitions of $n_j$ and $n_r$, we have $n_j=n_r$.   Hence together with  (\ref{orderoffn-gn}), it follows that
     \begin{eqnarray}
         \left|\sum_{n=1}^{n_j}(f_n(x)-g_n(x))\right|&=&\left|\sum_{n=1}^{n_r}(f_n(x)-g_n(x))\right|\nonumber\\
         &=&O\left(r^{\frac{1}{2}}(\log r)^{\frac{3}{2}+\epsilon}\right)\nonumber\\
         &=&O\left(\Psi(n_j)^{\frac{1}{2}}(\log (\Psi(n_j)+1))^{\frac{3}{2}+\epsilon}\right)\label{orderfornj}
     \end{eqnarray}
     for $\mu$-almost every $x\in X$. This proves (\ref{quantestim}) when $N=n_j$ $(j\in\N)$.

     It remains to prove (\ref{quantestim}) for all $N\in\N\setminus\{n_j:j\in\N\}$. Fix $x\in X$  that satisfies (\ref{orderfornj}) and let $N\in\N\setminus\{n_j:j\in\N\}$. Then there exists $j\in\N$ for  which
     \[
     n_j<N<n_{j+1}.
     \]
     It follows by the definitions of $n_j$ and $n_{j+1}$ that
     \begin{eqnarray}\label{psiNnj}
         \Psi(n_{j+1})\leq\Psi(n_j+1)+1.
     \end{eqnarray}
     In view of (\ref{orderfornj}) and (\ref{psiNnj}),  there exists  $C=C(x)>0$ such that
     \begin{eqnarray*}
         \sum_{n=1}^N(f_n(x)-g_n(x))&\leq&\sum_{n=1}^{n_{j+1}}f_n(x)-\sum_{n=1}^Ng_n(x)\\[4pt]
         &=&\sum_{n=1}^{n_{j+1}} \big(f_n(x)-g_n(x) \big)\,\,+\sum_{n=N+1}^{n_{j+1}}g_n(x)\\[4pt]
         &\leq& C\cdot\left(\Psi(n_{j+1})^{\frac{1}{2}}\big(\log(\Psi(n_{j+1})+1)\big)^{\frac{3}{2}+\epsilon}+\Psi(n_{j+1})-\Psi(n_j+1)\right)\\[4pt]
         &\leq&C\cdot\left((\Psi(N)+1)^{\frac{1}{2}}\big(\log(\Psi(N)+2)\big)^{\frac{3}{2}+\epsilon}+1\right)
     \end{eqnarray*}
     and that
     \begin{eqnarray*}
         \sum_{n=1}^N(f_n(x)-g_n(x))&\geq&\sum_{n=1}^{n_j}f_n(x)-\sum_{n=1}^Ng_n(x)\\[4pt]
         &=&\sum_{n=1}^{n_j} \big(f_n(x)-g_n(x)\big)-\sum_{n=n_j+1}^Ng_n(x)\\[4pt]
         &\geq&-C\cdot\left(\Psi(n_j)^{\frac{1}{2}}\big(\log(\Psi(n_j)+1)\big)^{\frac{3}{2}+\epsilon}+\big(\Psi(n_{j+1})-\Psi(n_j)\big)\right)\\[4pt]
         &\geq&-C\cdot\left(\Psi(N)^{\frac{1}{2}}\big(\log(\Psi(N)+1)\big)^{\frac{3}{2}+\epsilon}+\phi_{n_{j}+1}+1\right)\\[4pt]
         &\geq&-C\cdot\left(\Psi(N)^{\frac{1}{2}}\big(\log(\Psi(N)+1)\big)^{\frac{3}{2}+\epsilon}+\max_{1\leq k\leq N}\phi_{k}+1\right).
     \end{eqnarray*}
     The proof is complete.
\end{proof}

\section{Example~ABB is not a counterexample}   \label{appendixABB}

In this appendix we show that Example~ABB is not a counterexample to  Claim~0-1 without the  Ahlfors regular assumption  -- see Remark~\ref{nonono} in Section~\ref{appintro} for the context.

\medskip

\begin{proposition}   \label{propnocounter}
      Let $\Phi=\{\varphi_1(x)=\frac{1}{3}x,\,\varphi_2(x)=\frac{1}{3}x+\frac{2}{3}\}$ and let $K$ be the self-similar set generated by $\Phi$. Recall, $K$ is the
standard middle-third Cantor.  Let $\mu:=\ubar{\mu}\circ\pi^{-1}$ where $\ubar{\mu}$ is the Bernoulli measure on $\Sigma^{\N}:=\{1,2\}^{\N}$associated with the probability vector  $(p_1,p_2)$ with $p_1\neq p_2$. Let $\alpha>0$ and let $\psi_{\alpha}(n)=3^{-\linte{\alpha\log n}}$. If $$\alpha>\frac{1}{-(p_1\log p_1+p_2\log p_2)}\,,$$ then for $\mu$--almost all $x \in X$
    \[
    \sum_{n=1}^{\infty}\mu(B(x,\psi_{\alpha}(n)))<+\infty
    \]

\end{proposition}

\medskip

\begin{proof}
     In the setup of third-middle Cantor set, note that for any $I=i_1i_2\cdot\cdot\cdot\in\Sigma^{\N}=\{1,2\}^{\N}$  and $n\in\N$, we have that
    $$\mu\big(B(\pi(I),3^{-n})\big)=\mu\big(K_{i_1i_2\cdot\cdot\cdot i_n}\big)=p_{i_1}p_{i_2}\cdot\cdot\cdot p_{i_n}\,.$$
    With this in mind, for any $n\in\N$, we obtain that
    \begin{eqnarray}\label{mubpii3n}
        \mu\big(B(\pi(I),\psi_{\alpha}(n))\big)=p_{i_1}p_{i_2}\cdot\cdot\cdot p_{\linte{\alpha\log n}}\,.
    \end{eqnarray}
     Consider the map $f:\Sigma^{\N}\to\R$ given by
    \[
    f(I)=\log p_{i_1}\qquad(I=i_1i_2\cdot\cdot\cdot\in\Sigma^{\N})
    \]
    and let $\sigma$ be the shift map on $\{0,1\}^{\N}$. Then it follows from \eqref{mubpii3n} that
    \begin{eqnarray}\label{logmubpii}
        \log\mu\big(B(\pi(I),\psi_{\alpha}(n))\big)=\sum_{i=1}^{\linte{\alpha\log n}}f(\sigma^{i-1}I)\,.
    \end{eqnarray}
    By the Birkoff Ergodic Theorem, we have that for $\ubar{\mu}$-almost every $I\in\Sigma^{\N}$,
    \begin{eqnarray}\label{birkhofferg}
        \lim_{n\to\infty}\frac{1}{\linte{\alpha\log n}}\sum_{i=1}^{\linte{\alpha\log n}}f(\sigma^{i-1}I)=\int_{\Sigma^{\N}}f(I)\,\td\ubar{\mu}(I)=p_1\log p_1+p_2\log p_2\,.
    \end{eqnarray}
    In view of the range of $\alpha$ under consideration, we can find $\epsilon\in(0,1)$ such that
    \begin{eqnarray}\label{epsilonrang}
        (1-\epsilon)\cdot\alpha\cdot(p_1\log p_1+p_2\log p_2)<-1-\epsilon\,.
    \end{eqnarray}
    Now we fix $I\in\Sigma^{\N}$ that satisfies \eqref{birkhofferg} and fix $\epsilon\in(0,1)$ that satisfies \eqref{epsilonrang}.  Then  on combining \eqref{logmubpii}, \eqref{birkhofferg} and \eqref{epsilonrang},  there exists $N\in\N$ such that for all $n>N$, we have
    \begin{eqnarray*}
        \mu\big(B(\pi(I),\psi_{\alpha}(n))\big)&=&e^{\log\mu(B(\pi(I),\psi_{\alpha}(n)))}\\[4pt]
        &\leq&e^{(1-\epsilon)\cdot\linte{\alpha\log n}\cdot(p_1\log p_1+p_2\log p_2)}\\[4pt]
        &\leq&e^{(1-\epsilon)\cdot(\alpha\log n-1)\cdot(p_1\log p_1+p_2\log p_2)}\\[4pt]
        &\leq& n^{-1-\epsilon}\cdot e^{-(1-\epsilon)\cdot(p_1\log p_1+p_2\log p_2)}
    \end{eqnarray*}
    and thus
    \[
    \sum_{n=1}^{\infty}\mu\big(B(\pi(I),\psi_{\alpha}(n))\big)\ <\ +\infty\,.
    \]
    The upshot of the above is that for $\mu$-almost every $x\in K$, the sum of   $\mu(B(x,\psi_{\alpha}(n)))$ is convergent. This thereby completes the proof of Proposition~\ref{propnocounter}.
\end{proof}

\bigskip

\noindent{\it Acknowledgments}. Huang and Li were supported by National Key R\&D Program of China (No. 2024YFA1013700), NSFC 12271176 and Guangdong Natural Science Foundation 2024A1515010946.    Velani would like to thank Bing Li for the kind invitation to Guangzhou as part of the excellent South China University of Technology (SCUT) Overseas Lecturer Programme. He particularly appreciated the thoughtful engagement of the audience during the lectures — indeed, he found himself learning new things while responding to their insightful questions. The final stages of this paper were completed during this visit.

\bibliographystyle{abbrv}
\bibliography{MainNew.bib}

\end{document}